\documentclass[11pt]{amsart}

\usepackage[dvipsnames,usenames]{xcolor}
\usepackage{amsmath,
amssymb,
amscd,
amsfonts,
array,
comment,
enumitem,
epsfig,
float,
graphicx,
latexsym,
longtable,
mathrsfs,
mathtools,
pdfpages,
psfrag,
rotating,
transparent,
url,
wrapfig,
xr}
\usepackage[linktoc=page]{hyperref}
\hypersetup{
colorlinks,
citecolor=blue,
filecolor=blue,
linkcolor=blue,
urlcolor=blue
}
\usepackage[all]{xy}
\usepackage{amsaddr}
\usepackage[left=1.5in,right=1.5in,top=1.5in,bottom=1.5in+13pt,footskip=0.3in+13pt,headheight=13pt,headsep=0.3in]{geometry}
\usepackage[utf8]{inputenc}
\usepackage{mathtools}
\usepackage{relsize}
\usepackage[hang,flushmargin]{footmisc}
\usepackage[margin=1cm]{caption}

\usepackage{fancyhdr}
\newtheorem{theorem}{Theorem}[section]
\newtheorem{corollary}[theorem]{Corollary}

\newtheorem{proposition}[theorem]{Proposition}
\newtheorem{defprop}[theorem]{Definition-Proposition}
\newtheorem{lemma}[theorem]{Lemma}

\newtheorem{scholium}[theorem]{Scholium}
\theoremstyle{definition}
\newtheorem{definition}[theorem]{Definition}

\newtheorem{example}[theorem]{Example}
\newtheorem{expectation}[theorem]{Expectation}
\theoremstyle{remark}
\newtheorem{remark}[theorem]{Remark}

\theoremstyle{plain}
\newcommand{\thistheoremname}{}
\newtheorem{genericthm}[theorem]{\thistheoremname}

\newtheorem*{genericthm*}{\thistheoremname}
\newenvironment{namedthm*}[1]
{\renewcommand{\thistheoremname}{#1}%
\begin{genericthm*}}
{\end{genericthm*}}

\newcommand\cA{\mathcal{A}}

\newcommand\cC{\mathcal{C}}

\newcommand\cF{\mathcal{F}}
\newcommand\cG{\mathcal{G}}
\newcommand\cH{\mathcal{H}}

\newcommand\cK{\mathcal{K}}

\newcommand\cM{\mathcal{M}}

\newcommand\cP{\mathcal{P}}

\newcommand\cT{\mathcal{T}}

\newcommand{\bR}{\mathbb{R}}

\newcommand\ba{\mathbf{a}}

\newcommand\bn{\mathbf{n}}

\newcommand\bx{\mathbf{x}}

\newcommand{\sB}{\mathscr{B}}

\newcommand{\sC}{\mathscr{C}}
\newcommand{\sD}{\mathscr{D}}

\newcommand{\ccD}{\Delta_{\cT}}

\newcommand{\fX}{\mathfrak{X}}

\newcommand{\on}{\operatorname}

\newcommand{\comp}{C^2}

\renewcommand{\comp}{\text{comp}}

\newcommand{\incom}{\text{in}}

\newcommand{\inte}{{\on{int}}}

\renewcommand{\comp}{{\on{comp}}}

\newcommand{\br}{{\on{br}}}
\renewcommand{\min}{{\on{min}}}

\newcommand{\met}{{\on{met}}}

\newcommand\qu{/\kern-.7ex/} 
\newcommand\lqu{\backslash \kern-.7ex \backslash}

\newcommand{\ol}{\overline}

\newcommand{\wh}{\widehat}
\newcommand{\wt}{\widetilde}

\newcommand{\eps}{\epsilon}

\begin{document}
\sloppy

\begin{abstract}
The second author introduced 2-associahedra as a tool for investigating functoriality properties of Fukaya categories, and he conjectured that they could be realized as face posets of convex polytopes.
We introduce a family of posets called categorical $n$-associahedra, which naturally extend the second author's 2-associahedra and the classical associahedra.
Categorical $n$-associahedra give a combinatorial model for the poset of strata of a compactified real moduli space of a tree arrangement of affine coordinate subspaces.
We construct a family of complete polyhedral fans, called velocity fans, whose coordinates encode the relative velocities of pairs of colliding coordinate subspaces, and whose face posets are the categorical $n$-associahedra.
In particular, this gives the first fan realization of 2-associahedra.
In the case of the classical associahedron, the velocity fan specializes to the normal fan of Loday's realization of the associahedron.

In order to prove that the velocity fan is a fan, we first construct a cone complex of metric $n$-bracketings and then exhibit a piecewise-linear isomorphism from this complex to the velocity fan.
We demonstrate that the velocity fan, which is not simplicial, admits a canonical smooth flag triangulation on the same set of rays, and we describe a second, finer triangulation which provides a new extension of the braid arrangement.
Although $n$-associahedra are typically too large to be realized as coarsenings of the braid arrangement, we describe piecewise-unimodular maps on the velocity fan such that the image of each cone is a union of cones in the braid arrangement, and we highlight a connection to the theory of building sets and nestohedra.  We explore the local iterated fiber product structure of categorical $n$-associahedra and the extent to which this structure is realized by the velocity fan.
For the class of concentrated $n$-associahedra we exhibit generalized permutahedra having velocity fans as their normal fans recovering Loday's associahedron and Forcey's multiplihedron.
In future work we will investigate projectivity for general velocity fans.

\end{abstract}

\title{Higher\hspace{.2ex}\raisebox{.2ex}{-}\hspace{.2ex}Categorical Associahedra}
\author{Spencer Backman, Nathaniel Bottman, and Daria Poliakova}
\maketitle

\setcounter{tocdepth}{2}

\tableofcontents
\pagebreak

\section{Introduction}

Given a word of length $n$, the associahedron $\mathcal{K}_n$ is the set of parenthesizations of this word, partially ordered by inclusion.
There is a rich history of realizing $\mathcal{K}_n$ as the face poset of a convex polytope \cite{c507c782-a486-3e42-a801-97778df5e634, huguet1978structure, haimanconstruct, lee1989associahedron, gel1989newton, gel1991discriminants, shnider1993quantum,stasheff1997operads,chapoton2002polytopal,rote2003expansive,fomin2003systems,Lodayassociahedron,santos2004catalan,reading2006cambrian,hohlweg2007realizations,bergeron2007isometry,buchstaber2008lectures,postnikovgp,hohlweg2011permutahedra,stella2013polyhedral,ceballos2015many,arkani2018scattering,padrol2023associahedra,black2023polyhedral,gekhtman2024associahedra}.\footnote{See \cite{ceballos2015many} for 
a thorough account of the history up to 2015.}\,\footnote{The associahedron is often interpreted in the language of other collections of objects whose maximal elements are enumerated by the Catalan numbers, e.g.\ the subdivisions of a polygon.}

The associahedron makes fundamental appearances in several different parts of mathematics such as homotopy theory, representation theory, operad theory, category theory, Grassmannians, Coxeter groups, $A$-discriminants, cluster algebras, polytopal subdivisions, Catalan combinatorics, wonderful compactifications, scattering amplitudes, and tropical geometry.
These appearances motivate many interesting generalizations of the associahedron, which often admit polytopal realizations 
\cite{billera1992fiber,kapranov1993permutoassociahedron, gelfand1994discriminants, reiner1994coxeter,jonsson2003generalized,
troptotposgrass2005, ardila2005bergman, ardila2006positive, postnikovgp, pilaud2012brick, oppermann2012higher, escobar2014brick, devadoss2015convex, santos2017noncrossing, arkani2014amplituhedron, postnikov2018positive, arkani2018scattering, ceballos2019geometry, devadoss2020colorful, palu2021non, arkani2021cluster, galashin2021p,
damgaard2021momentum, highersecondarypolytopes}.

There is an area of mathematics where the associahedron plays a distinguished role which may be less well-known to researchers in algebraic and geometric combinatorics: the Fukaya category of a symplectic manifold. 
The associahedra form an example of an operad, and in the study of operads there is a notion of a category over an operad, where the operad controls the algebraic structure of the morphisms.
Thus, a linear $A_\infty$-category is defined as a category over the operad of cellular chains on the associahedra.
(See \cite[Prop.\ 2.20]{abouzaid_bottman}.\footnote{Note that versions of the associahedra were already used to define $A_\infty$-algebras in J.P.\ May's seminal \cite{may:book}, while H-spaces of \cite{c507c782-a486-3e42-a801-97778df5e634}, \cite{stasheff1963homotopy} can be viewed as a topological version of $A_\infty$-algebras.})
The Fukaya category of a symplectic manifold $M$ is a collection of Lagrangian submanifolds of $M$ equipped with morphisms given by Floer cochains, and this category is an important example of an $A_\infty$-category.

In 2017 the second author introduced a family of posets called 2-associahedra for investigating functoriality properties of Fukaya categories \cite{b:2-associahedra}.
He and Carmeli proved in \cite{bottman_carmeli} that that the 2-associahedra form a relative 2-operad and can be used to define $(A_\infty,2)$-categories.
As described in \cite{abouzaid_bottman}, the functoriality properties of the Fukaya category can be encoded in the construction of an $(A_\infty,2)$-category whose objects are symplectic manifolds.
In the context of Fukaya categories, the relevant interpretation of the associahedron is the poset of strata of a compactified moduli space of points on a real line, a construction due to Kapranov \cite{kapranov1993permutoassociahedron} building on ideas of Drinfeld \cite{drinfeld1989quasi,drinfeld1990quantum} (see Figure \ref{fig:K4_models}).\footnote{We note that there are other connections between associahedra and symplectic geometry which do not pass through the Fukaya category \cite{oppermann2012higher,dyckerhoff2021symplectic,gekhtman2024associahedra}.}  Similarly, the second author proved that a 2-associahedron is the poset of strata of a compactified moduli space of marked vertical lines in $\mathbb{R}^2$ \cite{bottman2017moduli} (see Figure \ref{fig:W_21}).
\begin{figure}[ht]
\centering
\def\svgwidth{1.0\columnwidth}
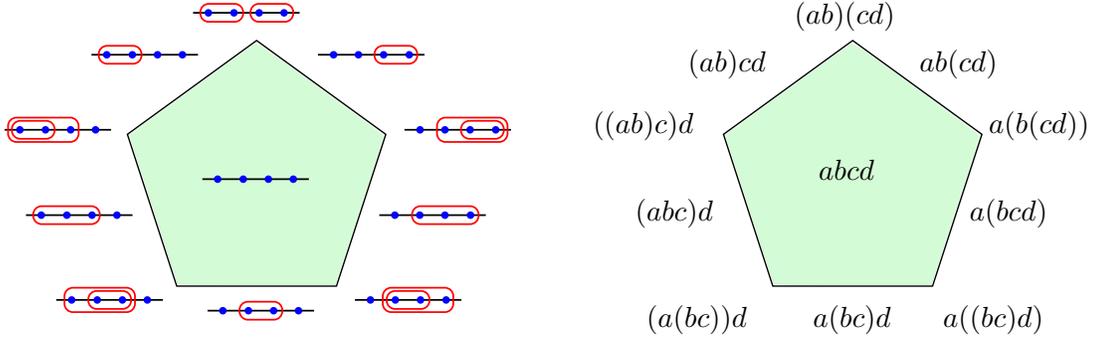
\caption{
\label{fig:K4_models}
$K_4$, presented in terms of 1-bracketings (left) and parenthesizations (right).
}
\end{figure}
\begin{figure}[ht]
\centering
\includegraphics[width=0.7\textwidth]{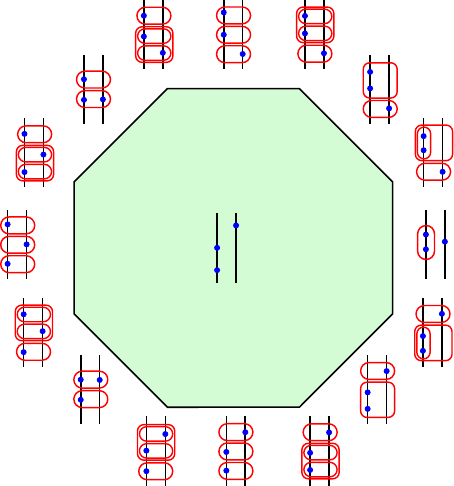}
\caption{
\label{fig:W_21}
The 2-associahedron associated to 2 lines with 2 points on the first line and 1 point on the second line.
}
\end{figure}

From this perspective, 2-associahedra are natural enough that the authors feel they could have been discovered independently of considerations of Fukaya categories, although without this strong motivation it may have been difficult to anticipate that 2-associahedra are so well-behaved.

The second author conjectured that 2-associahedra could be realized as face posets of convex polytopes and made progress towards resolving this conjecture by proving that they are abstract polytopes \cite{b:2-associahedra}, that they are Eulerian (\cite{bm}, with Mavrides), and that they are homeomorphic to closed balls (\cite{abouzaid_bottman}, with Abouzaid, relying on \cite{bo}, with Oblomkov).
We take a significant next step towards resolving this conjecture by producing the first complete fan realization of 2-associahedra, which we call \emph{the velocity fan}.
This allows us to recover all of the aforementioned partial results.
Furthermore, we introduce a natural extension of 2-associahedra and classical associahedra, which we call \emph{categorical $n$-associahedra}, and prove that our velocity fan provides a realization for this larger class of posets.
In future work, we will investigate projectivity for velocity fans, i.e.\ the existence of polytopes whose normal fans are the velocity fans.

While the precise definition of a 2-associahedron is not short, the essential idea of this construction can be conveyed very intuitively via compactified real moduli spaces as indicated above.
This perspective in turn motivates the introduction of categorical $n$-associahedra in \S\ref{s:n-associahedra}.
We begin by recalling the description of the associahedron as the poset of strata of a compactified real moduli space.
Take the space of $n$ ordered points on a line, considered up to translation and positive dilation, where points are allowed to move but cannot collide with or move past their neighboring points.
We compactify this space by allowing points to collide, and then further allowing iterated collisions of points in infinitesimal neighborhoods of other collisions.
This idea is captured combinatorially by the parenthesization of a word encoding the points, hence the compactified space obtained in this way has the associahedron as its poset of strata (see Figure \ref{fig:K4_models}).

Suppose that one has an arrangement $X$ of vertical lines in $\mathbb{R}^2$ with a possibly empty set of points on each line, although at least one line should have a point.
For describing the 2-associahedron determined by $X$, we are similarly interested in this space of points and lines, considered up to translation and dilation, where points are allowed to move vertically but they cannot collide with or move past neighboring points, and vertical lines are allowed to move horizontally but they cannot collide with or move past neighboring vertical lines.
The second author proved that 2-associahedra can be realized as the poset of strata of a Gromov compactification of this space of marked vertical lines in $\mathbb{R}^2$ \cite{bottman2017moduli}.
Because a 2-associahedron is determined by the combinatorial type of $X$, we can index each 2-associahedron $W_{\bf{n}}$ by a sequence of nonnegative integers $\bf{n}$ encoding the number of points on each line.
The elements of a $2$-associahedron are called 2-bracketings and should be understood as a 2-version of the parenthesization of a word.
In the same way that a parenthesization of a word can be considered as a union of compatible pairs of parentheses, a 2-bracketing in a 2-associahedron can always be considered as a union of compatible collisions of points and lines.
For a bracketing in the associahedron, this expression is unique, whereas a 2-bracketing can often be expressed as the union of collisions in multiple ways.
This implies that associahedra are simple, but 2-associahedra are not.

We introduce categorical $n$-associahedra as an extension of this viewpoint to $\mathbb{R}^n$.
In this work we give a combinatorial definition of categorical $n$-associahedra as a model for the poset of strata of a compactification of the real moduli space of a tree of affine initial coordinate subspaces. The formal definition of categorical $n$-associahedra takes as input a single rooted plane tree $\cT$ of depth $n$ (see Figure \ref{fig:arrangement_in_R3_and_rootedtreeexample}).
In this way, our categorical $n$-associahedra specialize to the 2-associahedra where the indexing sequence of nonnegative integers is encoded by a rooted plane tree of depth two, and to the classical associahedra where the indexing positive integer is encoded by a rooted plane tree of depth one.\footnote{It is well-known that rooted plane trees can be identified with the faces of associahedra.
Thus we have a bijection between categorical $n$-associahedra and faces of associahedra --- at the time of writing this is simply a curious observation, but it would be wonderful to have some deeper justification for this correspondence.}

\begin{figure}[ht]
\centering
\includegraphics[height=2in]{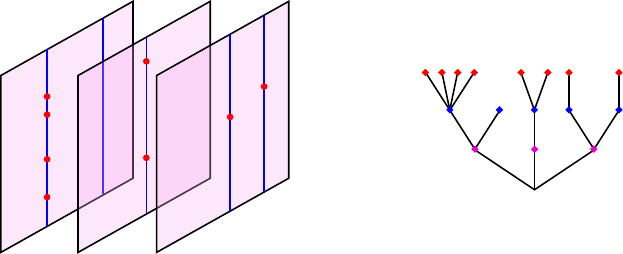}
\caption{On the left: an arrangement of affine spaces for producing a categorical 3-associahedron.
On the right: A rooted plane tree $\cT$ encoding the combinatorial type of the arrangement on the left.}
\label{fig:arrangement_in_R3_and_rootedtreeexample}
\end{figure}
\begin{figure}[ht]
\centering
\includegraphics[width=0.85\textwidth]{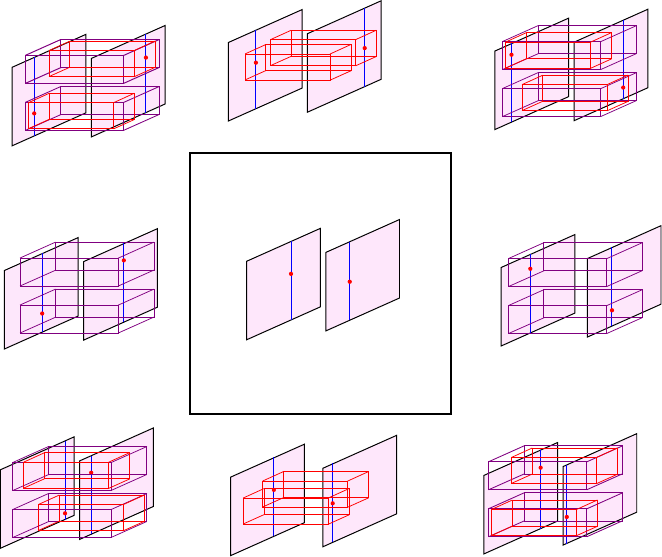}
\caption{
A 2-dimensional example of a $3$-associahedron.}
\end{figure}

Given a rooted plane tree $\cT$ and its corresponding categorical $n$-associahedron $\cK(\cT)$, we introduce a collection of polyhedral cones called the velocity fan $\cF(\cT)$.
The ray generators for the velocity fan are associated to collisions, i.e.\ the atoms of $\cK(\cT)$, and have coordinates encoding the relative velocities of consecutive pairs of affine coordinate subspaces in an arrangement during a collision.
The cones of the velocity fan are then determined by these rays generators: for a particular element of $\sB \in \cK(\cT)$ the associated cone $\tau(\sB)$ is the convex hull of all ray generators $\rho(\sC)$ for collisions $\sC \leq \sB$, plus the linear span of the all ones vector.
The following is the main result of this article:

\begin{theorem}
\label{mainthm}
(\ref{mainthmagain})
The velocity fan $\cF(\cT)$ is a complete fan whose face poset is $\cK(\cT)$.
\end{theorem}

\begin{figure}[ht]
\centering
\begin{tabular}{cc}
\includegraphics[height=4.5in]{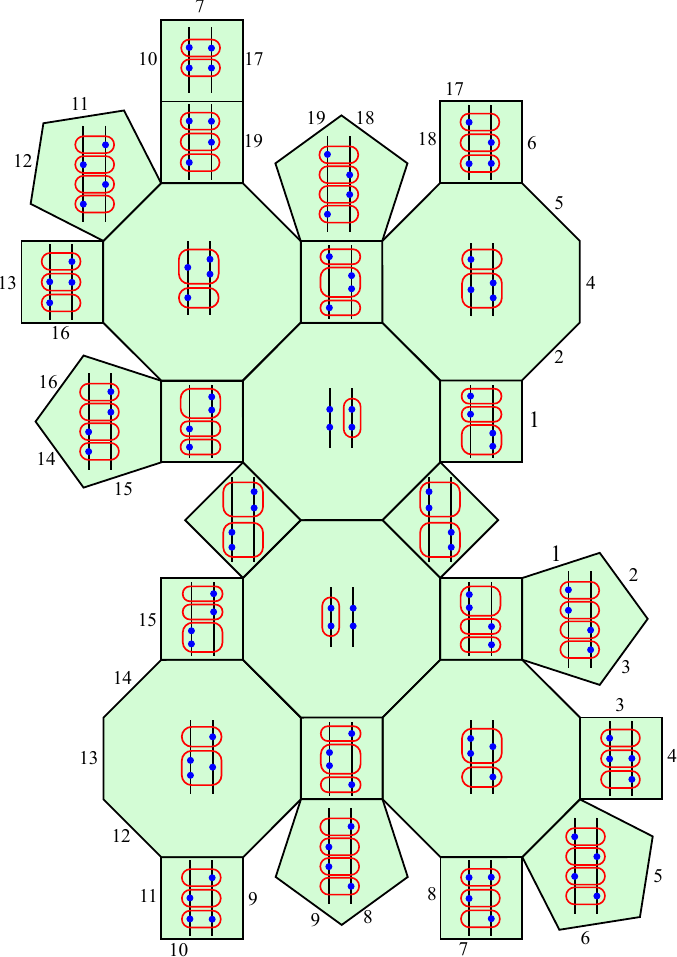}
&
\includegraphics[height=4.5in]{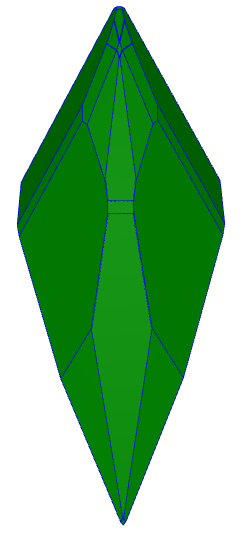}
\end{tabular}
\caption{On the left: a net depicting the 2-associahedron $W_{2,2}$.
On the right: a polytopal realization of the 2-associahedron $W_{2,2}$ with normal fan given by the velocity fan.
}
\label{W22}
\end{figure}

It requires significant work to prove that $\cF(\cT)$ is indeed a fan.
For doing so, we utilize the theory of cone complexes, which abstract and generalize the notion of a fan.
Cone complexes play a central role in modern logarithmic geometry where they have been applied in the study of moduli spaces, tropical geometry, and Berkovich spaces.
We construct a cone complex $\cK^{\met}(\cT)$ of metric $n$-bracketings and prove that this cone complex has a face poset equal to that of $\cK(\cT)$.
Theorem \ref{mainthm} (without completeness) is a direct consequence of the following result.

\begin{proposition}\label{metricisomorphismprop}
(\ref{mainvelocitythm})
There exists a piecewise-linear isomorphism
\begin{align}
\Gamma:\cK^{\met}(\cT)\rightarrow \cF(\cT)\,.
\end{align}
\end{proposition} 
See Figure \ref{fig:W_21} for the 2-associahedron $W_{2,1}$, and Figure \ref{fig:W_21_fans} for its abstract fan of metric 2-bracketings and its velocity fan.
We feel that this concrete combinatorial application of cone complexes is a curious feature of this article and may be of general interest to polyhedral and logarithmic geometers.

\begin{figure}[ht]
\centering
\begin{tabular}{ccc}
\includegraphics[width=0.425\textwidth]{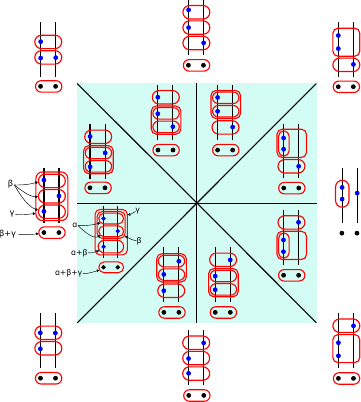}
&
\hspace{0.075\textwidth}
&
\includegraphics[width=0.425\textwidth]{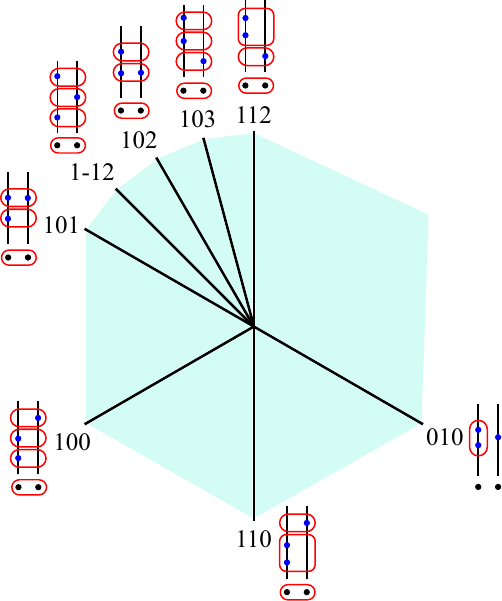}
\end{tabular}
\caption{
\label{fig:W_21_fans} 
On the left is a depiction of the cone complex $\cK^{\met}(\cT)$ of metric 2-bracketings for $W_{2,1}$; this figure should not be taken literally as $\cK^{\met}(\cT)$ is not embedded in Euclidean space.
On the right is the velocity fan $\cF(\cT)$ for $W_{2,1}$ viewed along the vector $(1,1,1)$.
See also Figure \ref{fig:maps_to_braid_example}.}
\end{figure}

During a collision, the affine spaces in a tree arrangement can be permuted --- this is an essential feature of $n$-associahedra.
For the 2-associahedra associated to a pair of lines, the permutations of points which can occur are precisely the riffle shuffle permutations, thus we refer to these permutations for general $n$-associahedra as $\cT$-shuffles.
Let $P_{\sigma}$ be the standard permutation transformation written in the basis of the positive simple roots. The adjoint permutation transformation $P^T_{\sigma}$ allows us to describe how the shuffle $\sigma$ is manifested geometrically in the velocity fan.
(See Lemmas \ref{permutationtransformationlemma1} and \ref{contractionlemma} for precise statements.)
We note that the lineality space $\langle {\bf 1}\rangle_{\mathbb{R}}$ of the velocity fan $\cF(\cT)$ is not contained in the eigenspace of $P^T_{\sigma}$, and for this reason we cannot quotient out $\cF(\cT)$ by its lineality space until argumentation involving $P^T_{\sigma}$ has been concluded.

In the case of the classical associahedra, our velocity fan specializes to the normal fan of Loday's realization of the associahedron $\cF_n$ \cite{Lodayassociahedron}.
It is well-known that $\cF_n$ is a coarsening of the braid arrangement and thus Loday's realization is a generalized permutahedron, equivalently a polymatroid base polytope.
In fact, Loday's associahedron was one of the primary motivations for Postnikov's development of the theory of generalized permutahedra where he expressed Loday's associahedron as a Minkowski sum of standard simplices \cite{postnikovgp}.\footnote{We refer the curious reader to the excellent recent survey of Loday's associahederon by Pilaud--Santos--Ziegler \cite{pilaud2023celebrating}.}
It is impossible to realize all categorical $n$-associahedra as generalized permutahedra for the simple fact that their face posets are too large.\footnote{
\label{permufootnote}
Each 2-dimensional generalized permutahedron has at most 6 sides, but the 2-associahedron $W_{2,1}$ is an octagon.
There are a few different extensions of generalized permutahedra, perhaps the most prominent constructions being the generalized bipermutahedra of Ardila--Denham--Huh \cite{lagrangian}, the type B/C-generalized permutahedra of Gelfand et al.\ \cite{ardila2020coxeter,borovik2003coxeter, gelfand1987combinatorial}, and the generalized nested permutahedra of Castillo--Liu \cite{nestedbraid}.
It is impossible to realize categorical $n$-associahedra as generalized bipermutahedra for the same reason as for the generalized permutahedra; the former cannot realize octagons.
The 2-dimensional type B/C generalized permutahedron is an octagon, but the 3-dimensional type B/C permutahedron has 26 facets, whereas the 2-associahedron $W_{2,2}$ has 27 facets (see Figure \ref{W22}), thus it is also impossible to realize $n$-associahedra in this setting.
It is not so easy to exclude the generalized nested permutahedra based on $f$-vector considerations, although we note that the rays of the velocity fan are not rays of the nested braid fan.}
On the other hand, because our velocity fan specializes to the normal fan of Loday's associahedron, as well as the normal fan of Forcey's multiplihedron, which is also a generalized permutahedron \cite{Forcey, Ardilla-Doker}, it is reasonable to ask whether there is a direct connection between general velocity fans and the braid arrangement.
We provide an affirmative answer to this question in several ways outlined below.

While the velocity fan is not simplicial, we produce a simplicial fan $\cF(\ccD)$ and establish the following result.

\begin{proposition}
(\ref{triangulationtheorem})
The fan $\cF(\ccD)$ is a canonical smooth flag triangulation of $\cF(\cT)$ on the same set of rays.\footnote{Technically, we must first quotient out these fans by their lineality space $\langle {\bf 1}\rangle_{\mathbb{R}}$ for this statement to be accurate.}
\end{proposition}
The chambers of this triangulated velocity fan $\cF(\ccD)$ admit a description in terms of certain rooted binary trees whose nodes are labeled by collisions.
This generalizes a classical description of the associahedron, and we make an explicit connection with Postnikov's $B$-trees and nestohedra, an important family of generalized permutahedra arising from the theory of wonderful compactifications.
We also describe a second, finer triangulation $\cF(\mathscr{O}(\cT))$ which specializes to the triangulation of the normal fan of Loday's associahedron by the braid arrangement, and thus offers a new generalization of the braid arrangement associated to a rooted plane tree.

Generalized permutahedra can be characterized as those polytopes whose normal fans have walls which are orthogonal to the type $A$ roots $e_i-e_j$ and thus Loday's associahedron, being a generalized permutahedron, has a normal fan with this property.
We apply the triangulated velocity fan to give a recursive calculation of the normal vectors for the walls of the velocity fan.
In particular, our calculation demonstrates that
the normal vectors of the walls of $\cF(\cT)$ have unbounded support.

As a second application of the triangulated velocity fan, we prove that the map $\Gamma$ from Proposition \ref{metricisomorphismprop} restricts to a monoid isomorphism from the integral metric $n$-bracketings in a conical set of $\cK^{\met}(\cT)$ to the integer points in the corresponding cone of the velocity fan $\cF(\cT)$.

We then produce a function $\Gamma^{\rho}_{\zeta}$ on the velocity fan and prove the following result. 

\begin{theorem}
(\ref{permutahedroidthm})
The function $\Gamma^{\rho}_{\zeta}:\cF(\cT)\rightarrow \mathbb{R}^m$ is a piecewise-unimodular map which is nondegenerate on each cone, and the image of each cone is union of cones in the braid arrangement $\cA_m$.
\end{theorem}
This connection between general velocity fans and the braid arrangement suggests a new extension of generalized permutahedra, which we call permutahedroids.
We then produce a different family of maps $\{\Gamma^\rho_{\sigma}\}$ defined on certain shuffle charts $\{\Gamma(U_{\sigma})\}$ of the velocity fan which are determined by a fixed shuffle order on the underlying affine spaces.
We establish the following result.
\begin{theorem}
(\ref{localPLmapsprop})
For a fixed shuffle $\sigma$, the map $\Gamma^\rho_{\sigma}: \Gamma(U_{\sigma})\rightarrow \mathbb{R}^m$ is a piecewise-unimodular isomorphism onto its image, and the image of each cone is a union of cones in the braid arrangement $\cA_m$.
Moreover, when we apply $\Gamma^\rho_{\sigma}$ to $\Gamma(U_{\sigma})$ in the triangulated velocity fan, the image is a full-dimensional subfan of a nestohedral fan.
\end{theorem}

This result, applied in reverse, shows that the triangulated velocity fan can be constructed by taking a collection of certain nestohedral fans, deleting the stars of some of the coordinate rays, acting on the remaining cones by piecewise unimodular isomorphisms, and then gluing the resulting collection of cones together.
See Figure \ref{fig:maps_to_braid_example} for an illustration of these results.

Each face of an associahedron factors as a product of smaller associahedra --- this observation is essential for understanding the associahedron as an operad, and is referred to by some authors as the Hopf algebra or Hopf monoid structure of the associahedron (see \cite{aguiar2023hopf}).
It was shown by the second author that 2-associahedra admit a similar but more intricate recursive structure: each facet of a 2-associahedron factors as a product of two smaller objects, where the first is a smaller 2-associahedron and the second is a fiber product of 2-associahedra over an associahedron --- the second author and Carmeli used this recursive structure to establish that 2-associahedra form a relative 2-operad \cite{bottman_carmeli}.
We extend this description to categorical $n$-associahedra: each facet of an $n$-associahedron factors as a product of a smaller $n$-associahedron with an iterated fiber product of $n$-associahedra, and we leave open the problem of verifying that $n$-associahedra determine a relative $n$-operad.
The normal fan of Loday's associahedron realizes the recursive structure of the associahedron geometrically: the star of each ray factors as a product of two smaller such fans.
The following result characterizes how this result extends to the velocity fan.

\begin{theorem}
(\ref{localvelocity})
Let ${\rho}(\sC)$ be a ray of the velocity fan $\cF(\cT)$.
There exists a piecewise-unimodular isomorphism 
\begin{align}
\Theta_{\sC}:\,\,\, \cF(\cT)_{{\rho}(\sC)} \rightarrow \cF(\cT/_{\sC}) \times \cF(\cT|_{\sC})
\end{align}
such that $\Theta_{\sC}^{-1}$ restricted to $\cF(\cT/_{\sC}) \times {\bf 0}$ is linear and agrees with $(P_{\sigma}^T)^{-1}\times {\bf 0}$ restricted to $\cF(\cT/_{\sC}) \times {\bf 0}$, with $\sigma$ a compatible $\cT$-shuffle for $\sC$, and with coordinates ordered according to $\sigma$.
\end{theorem}

We conclude our article by investigating a natural class of $n$-associahedra, containing the associahedron and the multiplihedron, which we call \emph{concentrated $n$-associahedra}.
These are the $n$-associahedra determined by rooted plane trees such that each pair of vertices at depth $k$ are siblings.
For this class, there can be no shuffling of subspaces during a collision in the corresponding arrangement.
This manifests geometrically in an interesting way: the velocity fans of concentrated $n$-associahedra are coarsenings of the braid arrangement.
As an advertisement for our future work, where we will investigate projectivity for general velocity fans, we establish the following result.

\begin{theorem}
(\ref{polytopes})
 Each concentrated $n$-associahedron $\cK(\cT)$ can be realized as the face poset of an integral generalized permutahedron whose normal fan is the velocity fan $\cF(\cT)$.
\end{theorem}

We give explicit facet and vertex descriptions for these polytopes generalizing Loday's realization of the associahedron and Forcey's realization of Stasheff's multiplihedron.
Additionally, we describe these polytopes as positive Minkowski sums of standard simplices generalizing a description of Loday's associahedron due to Postnikov \cite{postnikovgp} and a description of Forcey's multiplihedron dues to Ardila-Doker \cite{Ardilla-Doker}.
We also demonstrate how to recover the constrainahedra of the second and third authors: each constrainahedron is the Minkowski sum of all concentrated $n$-associahedra with a fixed profile.

\subsection{Compactified real moduli spaces}
\label{ss:compactified_moduli_spaces}

\

In this subsection, we provide some further details concerning the topological realization of $n$-associahedra as the poset of strata of compactified moduli spaces of configurations of subspaces in $\bR^n$.
(This subsection is entirely motivational.
In principle it can be skipped, but we recommend against this.)

We begin by recalling how the associahedra $K_r$ can be realized as a compactified moduli space of configurations of $r$ numbered marked points on $\bR$ \cite{kapranov1993permutoassociahedron, drinfeld1989quasi, drinfeld1990quantum, lambrechts_turchin_volic}.
(It is more conventional to consider $r+1$ marked points on $S^1$, with one marked point distinguished; this is an entirely equivalent perspective.) This realization is a necessary part of the construction of the Fukaya category of a symplectic manifold \cite[\S9f]{seidel}, because the pseudoholomorphic maps involved in the definition of the composition operation have domains that are disks with boundary marked points, one of them distinguished.

Define $\cM_r$ to be the moduli space
\begin{align}
\cM_r
\coloneqq
\bigl\{
(x_1,\ldots,x_r)
\:|\:
x_1 < \cdots < x_r
\bigr\}/\sim,
\end{align}
where the quotient indicates that we identify two configurations that differ by an overall translation and dilation $x \mapsto ax + b$, for $a \in \bR_{>0}$, $b \in \bR$.

\begin{example}
Consider the case of $\cM_4$.
The locus of points of the form $(0,1,x_3,x_4)$ with $1 < x_3 < x_4$ defines a global slice for the action of translations and dilations, so we can identify $\cM_4$ with an open simplex.
\null\hfill$\triangle$
\end{example}

We can form a Gromov-compactification of $\cM_r$ by defining $\ol\cM_r$ to be the moduli space of \emph{stable trees of configurations of points on $\bR$}.
We will not define $\ol\cM_r$ precisely, but we invite the interested reader to consult \cite[\S2]{bottman2017moduli}.\footnote{\label{bubbleofffootnote}A standard technique in symplectic geometry is to enlarge spaces of pseudoholomorphic maps by including nodal maps, to form ``Gromov-compactified moduli spaces''.
The principle is that when the gradient blows up in a sequence of maps, one rescales at the blowup point to form a ``bubble''.
This is treated in great detail in \cite[\S\S4--5, and \S4.2 in particular]{mcduff-salamon:big}.
The Gromov-compactification of $\cM_r$ that we describe here is an analogous procedure for configuration spaces, rather than spaces of maps.
In the symplectic context, these configuration spaces should be thought of as moduli spaces of domains.}
$\ol\cM_r$ is constructed by including limits of Cauchy sequences in $\cM_r$, following the instruction that when points collide, we ``zoom in'' on the collision site and remember the fashion in which the points collided by adding in an auxiliary copy of $\bR$ that carries these points. 
 The space $\ol\cM_r$ naturally has the structure of a CW complex.
In fact, it is a CW realization of $\cM_r$: the poset of cells is isomorphic to the associahedron $K_r$.
See Figure \ref{fig:K4_models} for a depiction of the $r=4$ case.

We motivate the $n$-associahedra in terms of a generalization of $\ol\cM_r$.
Specifically, we will consider moduli spaces of certain configurations of affine subspaces of $\bR^n$.
While we will not prove a relationship between these spaces and the $n$-associahedra, we expect that they are a topological, and even a smooth, realization of the $n$-associahedra, just as $\ol\cM_r$ is a topological realization of $K_r$.
(See Expectation \ref{exp:top_version}.)

We begin by defining \emph{trees of spaces}.

\begin{definition}
Let $n$ and $k$ be positive integers with $k \leq n$.
We say that $V \subset \bR^n$ is an \emph{initial coordinate subspace of dimension $k$} if there exists a vector $\ba \in \mathbb{R}^{n-k}$ such that
\begin{align}
V
=
\bigl\{
\bx \in \bR^n
\:|\:
\:
x_j = a_j
\:\forall\:
j \in [1, n-k]
\bigr\}.
\end{align}

A \emph{tree of spaces in $\bR^n$} is a nonempty finite collection $\cC = \bigcup_{k=0}^n \cC_k$ of initial coordinate subspaces in $\bR^n$, such that the following conditions hold:
\begin{enumerate}
\item
Every $A \in \cC_k$ has dimension $k$.

\item
For every $k\leq n-1$ and $A \in \cC_k$, there exists $A' \in \cC_{k+1}$ such that the containment $A \subset A'$ holds.
\end{enumerate}

\noindent
The \emph{type} of $\cC$ is the directed rooted plane tree defined like so:
\begin{enumerate}
\item
We define the vertex set by $V(\cT(\cC)) \coloneqq \cC$.
For clarity, we denote by $p_A$ the vertex corresponding to $A \in \cC$.

\item 
For $A, A' \in \cC$, there is an edge from $p_A$ to $p_{A'}$ if and only if $\dim A' = \dim A + 1$, and $A$ is contained in $A'$.

\item 
The root of $\cT(\cC)$ is $p_{\bR^n}$.

\item 
For any $p_A \in \cC_k$, we equip the set of incoming vertices of $p_A$ with the linear order induced by considering the $(n-k+1)$-st coordinate of the elements of $\cC_{k-1}$ contained in $A$.
\end{enumerate}

Suppose that $\cC$ is a tree of spaces in $\bR^n$.
We say that $\cC$ is \emph{stable} if $\cC_0$ is nonempty, and $\cT(\cC)$ has at least $n+1$ edges.
\null\hfill$\triangle$
\end{definition}

Now that we have introduced trees of spaces, we can explain the topological version of the $n$-associahedra.
Suppose that $\cT$ is a depth-$n$ rooted plane tree with at least $n+1$ edges.
We then define $\wh\cM(\cT)$ to be the collection of all stable trees of spaces in $\bR^n$ of type $\cT$.
The group $\bR^n \rtimes \bR_{>0}$ acts on $\bR^n$ by translations and dilations, which induces an action of the same group on $\wh\cM(\cT)$.
The stability condition implies that this action is free, and we now define $\cM(\cT)$ to be the quotient
\begin{align}
\cM(\cT)
\coloneqq
\wh\cM(\cT)/(\bR^n\rtimes\bR_{>0}).
\end{align}
In the following expectation, we describe a compactification of $\cM(\cT)$ that is motivated by the Gromov-compactification of the $n=2$ case of $\cM(\cT)$ performed in \cite{bottman2017moduli} and Bottman--Oblomkov's work in \cite{bo}.
This compactification is the topological version of the $n$-associahedra.
We intend to verify the following expectation in forthcoming work.
This expectation shows how $\cM(\cT)$ informs our definition of categorical $n$-associahedra in \S \ref{s:n-associahedra}.

\begin{expectation}
\label{exp:top_version}
For any depth-$n$ rooted plane tree with at least $n+1$ edges, there is a compact, contractible, and metrizable space $\ol\cM(\cT)$ that contains $\cM(\cT)$ as a dense subset.
This will reduce to the space $\ol{2\cM}_\bn$ constructed in \cite{bottman2017moduli} in the case $n=2$.
Convergence in $\ol\cM(\cT)$ will be defined in terms of the Gromov-convergence explained in \cite[\S1]{bottman2017moduli}.
Given a sequence $(\cC^{(i)})_{i=1}^\infty$ of trees of spaces in $\cM(\cT)$, here is the approximate procedure for determining the limit:
\begin{enumerate}
\item 
Identify the set of points $p \in \bR^n$ with the property that as $i\to\infty$, a collision takes place at $p$ and involves at least one marked point.

\smallskip

\item 
For each such $p$, choose a sequence of real numbers $\eps_1,\eps_2,\ldots \to 0^+$ such that if we recenter $\cC^{(i)}$ at $p$ and rescale by a factor of $\tfrac1{\eps_i}$, then we have separated the objects that are the slowest to collide at $p$ to a positive finite distance.

\smallskip

\item 
Iteratively continue this procedure until every collision has been resolved.
Record the results in a tree of trees of spaces.
\end{enumerate}
We encourage the reader to consult \cite[\S1.1]{bottman2017moduli} for a detailed and explicit example of a Gromov-convergent sequence.
The Gromov-compactification $\cM(\cT)$ will have a stratification (in fact, a CW decomposition) by combinatorial type of the trees of trees of spaces, and the poset of strata will coincide with the categorical $n$-associahedron $\cK(\cT)$.

The construction of $\ol\cM(\cT)$ is purely topological, but we intend to go further and equip it with a smooth structure --- specifically, with the structure of a manifold with generalized corners \cite{joyce}.
These are generalizations of manifolds with corners, which can be thought of as a positive-real version of toroidal varieties.
In the $n=2$ case, the first author constructed with Oblomkov in \cite{bo} a complex version of the 2-associahedra as toroidal complex varieties.
The ideas in Bottman--Oblomkov's work can be used to endow the $n=2$ case of $\ol\cM(\cT)$ with the structure of a manifold with g-corners.
We plan to generalize Bottman--Oblomkov's work to the arbitrary-$n$ case, and to show that in a precise sense, the positive-real part of the resulting varieties can be identified with $\ol\cM(\cT)$.
\null\hfill$\triangle$
\end{expectation}

\noindent

\subsection{The wonderful associahedral fan and the braid arrangement}
\label{Lodayfansubsection}

\

We give a brief self-contained presentation of the associahedron and the normal fan of Loday's realization of the associahedron, which we refer to in this article as the \emph{wonderful associahedral fan}.\footnote{We have chosen this name because the toric variety of this fan corresponds to a wonderful compactification of the torus, in the sense of De Concini-Procesi \cite{de1995wonderful}, with respect to a building set which is neither maximal nor minimal. 
 Pilaud \cite{pilaud2022pebble} calls this fan the sylvester fan --- ``sylvester'' is an old English word meaning forest.}
A \emph{bracket} of $[n]$ is a nonempty consecutive subset $A\subseteq [n]$.
\begin{definition}
A bracketing is a collection $\sB$ of brackets from $[n]$ such that \begin{enumerate}
\item $[n]\in \sB$,

\item for each $i \in [n]$, we have $\{i\}\in \sB$,

\item if $A_i, A_j \in \sB$ and 
$A_i \cap A_j \neq \emptyset$, then $A_i\subseteq A_j$ or $A_j \subseteq A_i$.
\end{enumerate}
The associahedron $\cK_n$ is the collection of all bracketings of $[n] $ ordered by containment.\footnote{The brackets in (1) and (2) are not strictly necessary.
We have chosen to include them as this definition of a bracketing is compatible with our definition of $n$-bracketings (where the maximum and singleton brackets are more useful).} 
\null\hfill$\triangle$
\end{definition}

Loday's polytopal realization of the associahedron is one of the most important realizations of the associahedron.
It is celebrated for being the simplest realization as well as for making appearances in many different parts of mathematics \cite{pilaud2023celebrating}.
Here we do not fully recall the construction of Loday's polytopal realization (it is recoverable as a special case of the results in \S \ref{s:concentrated_realization} on concentrated $n$-associahedra).
We instead give a description of the normal fan for Loday's realization of the associahedron, \emph{the wonderful associahedral fan}.
We will later observe that our velocity fan realization of categorical $n$-associahedra specializes to the wonderful associahedral fan for categorical $1$-associahedra, i.e.\ associahedra.

\begin{definition}\label{deflodayfan}
Let $n \in \mathbb{Z}_{\geq 2}$, and let $A\subseteq[n]$ 
be a bracket. 
We define the vector $\rho(A) \subset \mathbb{R}^{n-1}$ with $i$th entry
\begin{equation}
\rho(A)_i=
\begin{cases}
\,\, 1 & \text{if} \,\, i, i+1 \in A \\
\,\,0 & \text{otherwise}. \\
\end{cases}
\end{equation}

Given a bracketing $\sB$ we define the associated cone
\begin{align}
\tau(\sB)= cone\{{\rho}(A): A\in \sB \}+\langle {\bf 1}\rangle_{\mathbb{R}}.
\end{align}
The \emph{wonderful associahedral fan} is the collection of polyhedral cones
\begin{align}
\cF_n=\{\tau(\sB):\sB \in \cK_n\}.
\end{align}
\null\hfill$\triangle$
\end{definition}

\begin{definition}
\label{defbraidarrangement}
Let $n \in \mathbb{N}$.
The \emph{braid arrangement} $\mathcal{A}_n$ is the following collection of hyperplanes in $\mathbb{R}^n$:
\begin{align}
\mathcal{A}_n = \{H_{i,j}= \{{\bf x} \in \mathbb{R}^n:x_i=x_j\}: 1\leq i\neq j \leq n\}.
\end{align}

One may identify a hyperplane arrangement with the fan structure it induces on $\mathbb{R}^n$.
Thus, as an abuse of terminology and notation, we will also call the following fan, the braid arrangement.
Let ${\bf F} =\{ \emptyset = F_0 \subsetneq \ldots \subsetneq F_k = [n]\}$ denote a flag of subsets.
Let $\mathscr{F}_n$ denote the set of all such flags.
Given a set $S \subseteq [n]$, let $\chi_S$ denote the indicator vector for $S$.
Given ${\bf F} \in\mathscr{F}_n$ we associate the cone
\begin{align}
\tau_{\,{\bf F}} = \bigl\{\sum_{F_i \in {\bf F}} \lambda_i \chi_{F_i}:\lambda_i \geq 0\bigr\}+\langle {\bf 1} \rangle_{\mathbb{R}},
\end{align}
and we define \emph{the braid arrangement} to be
\begin{align}
\mathcal{A}_n = \{\tau_{\,\bf F}:{\bf F} \in \mathscr{F}_n\}.
\end{align}
\null\hfill$\triangle$
\end{definition}

\begin{definition}\label{defgeneralizedperm}
A \emph{generalized permutahedron} \cite{postnikovgp}, equivalently a \emph{polymatroid base polytope} \cite{edmonds1970submodular}, is a polytope whose normal fan coarsens the braid arrangement $\mathcal{A}_n$.
A generalized permutahedron, or polymatroid base polytope, is \emph{integral} if its vertices have integer coordinates.
\null\hfill$\triangle$
\end{definition}

\begin{theorem}\cite{ shnider1993quantum,Lodayassociahedron,postnikovgp} \text{(see}
\cite{pilaud2023celebrating}\text{)}\label{lodayfanisafan}
The wonderful associahedral fan $\cF_n$ is a complete fan whose face poset is $\cK_n$.
Moreover, $\cF_n$ is a coarsening of the braid arrangement, and $\cF_n$ is the normal fan of a polytope (an integral generalized permutahedron).
\end{theorem}

\

\subsection{Glossary of essential notation}

\

\label{ss:notation_and_conventions}

\bgroup
\def\arraystretch{1.5}
\begin{center}
\begin{longtable}{|m{0.24\textwidth}|m{0.57\textwidth}|m{0.09\textwidth}|}
\hline
notation & interpretation & page first defined \\
\hline
\,$\cT$ & a rooted plane tree & p.\ \pageref{rootedplanetreedef}
\\
\hline
\,$V^k(\cT)$ & the depth $k$ vertices of $\cT$ & p.\ \pageref{def:depth-k_vertices}
\\
\hline
\,$\pi^k(\cT)$ & The $k$-th truncation of $\cT$ & p.\ \pageref{def:truncation_of_tree}
\\
\hline
\,$C(u)$ & the children of $u$ & p.\ \pageref{def:children_of_vertex}
\\
\hline
\,$D(u)$ ($\overline{D}(u)$) & the (weak) descendants of $u$ & p.\ \pageref{def:decendants_of_vertex}
\\
\hline
\,$u^k_i$ & the $i$-th depth-$k$ vertex $\cT$ with respect to the plane order & p.\ \pageref{def:vertex_notation}
\\
\hline
\,$A = (A^0,\ldots,A^k)$ & a $k$-bracket & p.\ \pageref{def:k-bracket}
\\
\hline
\,$\sB = (\sB^0,\ldots,\sB^n)$ & an $n$-bracketing & p.\ \pageref{def:n-bracketing}
\\
\hline
\,$\sB_\min$ & the minimum $n$-bracketing & p.\ \pageref{def:minimal_n-bracketing}
\\
\hline
\,$\cK(\cT)$ & a categorical $n$-associahedron & p.\ \pageref{def:n-associahedron}
\\
\hline
$\cK(A)$ & the restriction of $\cK(\cT)$ to an $n$-bracket $A$ & p.\ \pageref{restrictbracket}
\\
\hline
\,$\wh\cK(\cT)$ & $\cK(\cT) \sqcup \{\star\}$, the completion of $\cK(\cT)$ & p.\ \pageref{def:poset_completion}
\\ 
\hline
\,$\sC$ & a collision & p.\ \pageref{def:collision}
\\
\hline
\,$\fX(\cT)$ & the collection of collisions & p.\ \pageref{def:set_of_collisions}
\\
\hline
\,$\wh \fX(\cT)$ & $\fX(\cT) \sqcup \{\sB_\min\}$ & p.\ \pageref{def:completed_set_of_collisions}
\\
\hline
$\psi^{\cT}_{\cT/\sC}$ & the collision map associated to $\sC$ & p.\ \pageref{def:collision_maps}
\\
\hline
$\cT/\sC$ & the quotient of a rooted plane tree by a collision & p.\ \pageref{def:quotient_of_tree_by_collision}
\\ 
\hline
$\cF, \cG$ & cone complexes, e.g.\ fans & p.\ \pageref{def:cone_complex}
\\
\hline
$\cP(\cF)$ & the face poset of a cone complex & p.\ \pageref{def:face_poset}
\\
\hline
$(\sB, \ell_\sB)$ & a metric $n$-bracketing with underlying $n$-bracketing $\sB$ & p.\ \pageref{def:metric_stable_tree-pair}
\\
\hline
$\ell(\sC)$ & the standard metric $n$-bracketing associated to a collision & p.\ \pageref{def:metric_n-bracketing_of_collision}
\\
\hline
$\cK^{\met}(\cT),\,({\overline \cK}^{\met}(\cT))$ & the (reduced) metric $n$-bracketing complex associated to $\cT$ & p.\ \pageref{def:reduced_metric_n-bracketing_complex}
\\
\hline
{\bf v} & vector notation & p.\ \pageref{def:vec_notation}
\\
\hline
\,$\bf{1}$ & the all-ones vector & p.\ \pageref{def:all-ones}
\\
\hline
${\rho}(\sC)$ & the standard ray generator for the collision $\sC$ & p.\ \pageref{def:ray_generator}
\\
\hline
$\tau(\sB)$ & the cone in the velocity fan associated to the $n$-bracketing $\sB$ & p.\ \pageref{def:cone_associated_to_n-bracketing}
\\
\hline
$\cF(\cT),\,(\overline{\cF}(\cT))$ & the (reduced) velocity fan associated to $\cT$ & p.\ \pageref{def:velocity_fan}, (p.\ \pageref{reducedvelocitydef})
\\
\hline
\,$\pi(\cK(\cT))$ & $\cK(\pi(\cT))$ & p.\ \pageref{def:proj_of_n-associahedron}
\\
\hline
$P^T_{\sigma}$ & the adjoint permutation transformation associated to $\sigma$ & p.\ \pageref{permtransdef}
\\
\hline
$\cF(\ccD),\,(\overline{\cF}(\ccD))$ & the (reduced) triangulated velocity fan associated to $\cT$ & p.\ \pageref{def:triangulated_velocity_fan}, (p.\ \pageref{def:reduced_triangulated_velocity_fan})
\\
\hline
$\cK(\cT|_\sC)$ & the fiber product associated to a collision & p.\ \pageref{fiberproductnassociahedra}
\\
\hline
\end{longtable}
\end{center}
\egroup

\subsection{Acknowledgments}

\

We thank Federico Castillo, Sergei Elizalde, Gaku Liu, Sam Molcho, Arnau Padrol, Sam Payne, Vincent Pilaud, Vic Reiner, Andrew Sack, David Speyer, Eric Stucky, Jim Stasheff, Hugh Thomas, Martin Ulirsch, Greg Warrington, and Geva Yashfe for helpful conversations.

S.B.\ was supported by a Simons Collaboration Gift \#854037 and an NSF Grant (DMS-2246967).
N.B.\ was supported by an NSF Grant (DMS-1906220).
D.P. was supported by Danish National Research Foundation grant (DNRF157).
S.B., N.B., and D.P. are grateful to the Max Planck Institute for Mathematics in Bonn for its hospitality and financial support.
S.B.\ and D.P are thankful to the Simons Center for Geometry and Physics for their hospitality and financial support during the workshop ``Combinatorics and Geometry of Convex Polyhedra".

\section{\texorpdfstring{$n$}{n}-bracketings and categorical \texorpdfstring{$n$}{n}-associahedra} 

\label{s:n-associahedra}

In this section, we will define the categorical $n$-associahedra and explore their basic structure.

\subsection{The definition of categorical \texorpdfstring{$n$}{n}-associahedra}

\

Our input for defining an $n$-associahedron will be a rooted plane tree of depth $n$.
The rooted plane trees of depth 1 are naturally in bijection with the natural numbers and index 1-associahedra, i.e.\ classical associahedra.

\begin{definition}
\label{def:depth-k_vertices}
A \emph{rooted tree} $\cT$ is a tree equipped with a distinguished vertex $u_0$, called the \emph{root}.
The of \emph{depth} of a vertex $u$ in $\cT$ is the distance from $u_0$ to $u$.
The depth of $\cT$ is the maximum depth of a vertex in $\cT$.
We denote the set of depth $k$ vertices in $\cT$ by $V^k(\cT)$.

Given a rooted tree $\cT$ with root $u_0$ and distinct $u, w \in V(\cT)$, we say that $w$ is a \emph{descendant} of $u$ if $u$ is on the unique path from $u_0$ to $w$.
\label{def:children_of_vertex}
\label{def:decendants_of_vertex}
If $w$ is a descendant of $u$ and, in addition, $(u,w)$ is an edge of $\cT$, we say that $w$ is a \emph{child} of $u$.
We denote the set of descendants of $u$ by $D(u)$, and the set of children of $u$ by $C(u)$. The set of \emph{weak descendants of $u$} is $\overline{D}(u)\coloneqq D(u) \sqcup \{u\}$.

\label{def:truncation_of_tree}
Given rooted tree $\cT$ of depth $n$, we define the \emph{k-level truncation of $\cT$} to be the tree $\pi^{n-k}(\cT)$ obtained by deleting all of the vertices of depth greater than $k$.
For simplicity, we refer to the \emph{$(n-1)$-level truncation} of a rooted plane tree of depth $n$ as simply the \emph{truncation of $\cT$}, and denote it as $\pi(\cT)$.
\null\hfill$\triangle$
\end{definition}

\begin{definition}
\label{rootedplanetreedef}
A \emph{rooted plane tree} $\cT$ will be a rooted tree $\cT$ with a total order $<_{\cT}$ on the children $C(u)$ of each vertex $u \in V(\cT)$.
\null\hfill$\triangle$
\end{definition}

\begin{definition}\label{vertexorderdef}
We can naturally encode the vertices of a rooted plane tree by positive integer vectors: suppose that $w$ is the $i$-th child of $u$, and $\textbf{v}(u)$ is the positive integer vector which encodes $u$.
The positive integer vector $\textbf{v}(w)$ encoding $w$ is obtained from ${\bf{v}}(u)$ by appending an entry equal to $i$.
The lexicographic ordering of $V^k(\cT)$ induces a total order on the vertices of $V^k(\cT)$, which we denote by by $<_\cT$ (as it extends the total order on the children of each vertex).
\label{def:vertex_notation}
We let $u^k_i \in V(\cT)$ denote the $i$th vertex of $V^k(\cT)$ with respect to $<_\cT$.  See Figure \ref{3brackex}.
\null\hfill$\triangle$
\end{definition}

\begin{definition}[$k$-brackets]
Fix $n\geq0$ and a rooted plane tree $\cT$ of depth $n$.
Fix $k$ with $0 \leq k \leq n$.
\label{def:k-bracket}
A \emph{$k$-bracket of $\cT$} is a tuple $A = (A^0,\ldots,A^k)$ where for each $1\leq i \leq k$, $A^i \subseteq V^i(\cT)$, and the following conditions hold.
\begin{enumerate}
\item
If $u \in A^i$ then
the parent of $u$ is in $A^{i-1}$.

\smallskip

\item
If $u \in A^{i-1}$, then $A^i \cap C(u)$ is a consecutive subset of $C(u)$ with respect to $<_{\cT}$.
\null\hfill$\triangle$
\end{enumerate}
\end{definition}

\begin{figure}[ht]
\centering
\def\svgwidth{0.6\columnwidth}
{\tiny
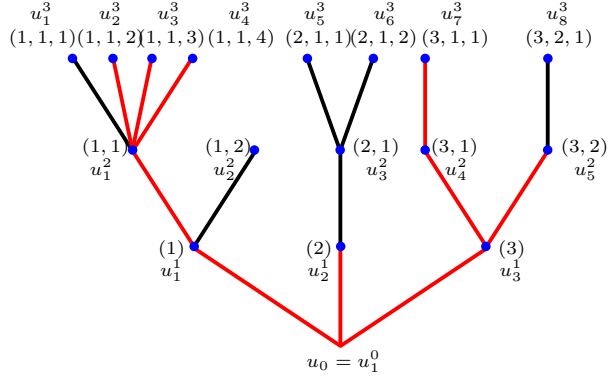}
\caption{A rooted plane tree previously depicted in Figure \ref{fig:arrangement_in_R3_and_rootedtreeexample}, and a subtree whose vertices determine a 3-bracket.}
\label{3brackex}
\end{figure}

\begin{definition}[truncation of $k$-brackets]
\label{def:truncation_of_bracket}
Suppose that $A = (A^0,\ldots,A^k)$ is a $k$-bracket of a rooted plane tree $\cT$, and suppose that $k$ is at least $1$.
Then the \emph{truncation $\pi(A)$ of $A$} is the $(k-1)$-bracket $(A^0,\ldots,A^{k-1})$.
For any $0 \leq i \leq n$, we denote by $\pi^i(A)$ the result of applying $\pi$ to $A$ $i$ times.
\null\hfill$\triangle$
\end{definition}

\begin{remark}
We will use the notation $\pi$ for several different, but compatible, notions of truncations in this paper. 
We may alternately refer to truncations as \emph{projections}. We say that an object $Y$ is a \emph{weak projection} of an object $X$, or that \emph{$X$ weakly projects to $Y$}, if there exists some $i\geq 0$ such that $\pi^i(X) = Y$.
\null\hfill$\triangle$
\end{remark}

\begin{remark}
Throughout this paper, we will use superscripts to indicate the different parts of a bracket (and later, a bracketing); we will use subscripts to index different objects.
\null\hfill$\triangle$
\end{remark}

\begin{definition}[singleton, maximum, and nontrivial $k$-brackets]
Suppose that $u$ is a depth-$k$ vertex of a rooted plane tree $\cT$, for some $k \geq 0$.
Then the \emph{singleton $k$-bracket associated to $u$}, denoted $A(u)$, is the $k$-bracket where $A(u)^\ell$ consists of the unique depth-$\ell$ vertex of $\cT$ that lies on the path between $u$ and the root of $\cT$.

For any depth-$n$ rooted plane tree $\cT$ and $k$ with $0 \leq k \leq n$, the \emph{maximum $k$-bracket of $\cT$} is the $k$-bracket $A$ such that $A^\ell = V^{\ell}(\cT)$ for $\ell\leq k$.
We say that a $k$-bracket $A$ of $\cT$ is \emph{nontrivial} if $A$ is neither a singleton $k$-bracket nor the maximum $k$-bracket.
\null\hfill$\triangle$
\end{definition}

\begin{definition}[containment of $k$-brackets]
Suppose that $A_1$ and $A_2$ are $k$-brackets of $\cT$.
We write $A_1 \subseteq A_2$ if, for every $\ell$ with $0 \leq \ell \leq k$, we have $A_1^\ell \subseteq A_2^\ell$.
\null\hfill$\triangle$
\end{definition}

We encourage the reader to refer to Figures \ref{fig:ex_2-bracketing} and \ref{fig:2-bracketing_non-examples} while parsing the following definition.
In those figures, and all future figures, we depict a $k$-bracket $A$ by drawing a bubble in a tree arrangement which encloses the elements in $A^k$, and projects to the bubble depicting $\pi(A)$.

\begin{definition}[$n$-bracketings]
\label{def:n-bracketings}
Fix $n\geq0$ and a rooted plane tree $\cT$ of depth $n$ that has at least $n+1$ edges, i.e.\ $\cT$ is not a path.
\label{def:n-bracketing}
An \emph{$n$-bracketing of $\cT$} is a pair 
\begin{align}
\sB = ((\sB^0, \ldots, \sB^n),<_{\sB}),
\end{align}
where for each $k$, $\sB^k$ is a collection of $k$-brackets of $\cT$, and $<_{\sB}$ is a partial order on $\bigsqcup_{k=0}^n\sB^k$, called \emph{the height partial order}, such that $\sB$ satisfies the following conditions.

\begin{itemize}

\item[]
{\sc (trivial $k$-brackets)}
Each $\sB^k$ contains the maximum $k$-bracket and every singleton $k$-bracket of $\cT$.

\medskip

\item[]
{\sc ($k$-bracketing projection)}
If $A \in \sB^k$ with $1\leq k\leq n$, then $\pi(A)\in \sB^{k-1}$.

\medskip

\item[]\label{nestedprop}{\sc (nested)}
If $A_1, A_2 \in \sB^k$ and $A_1^k \cap A_2^k \neq \emptyset$, then $A_1 \subseteq A_2$ or $A_2 \subseteq A_1$.

\medskip

\end{itemize}

\noindent
For any $k$ with $1 \leq k \leq n$ and any $A \in \sB^{k-1}$, write $\sB^k_{A} \coloneqq \left\{\wt A \in \sB^k \:|\: \pi(\wt A) = A\right\}$.

\begin{itemize}
\item[]\label{partition}
{\sc (partition)}
\begin{itemize}
\item
For any $k$ with $1 \leq k \leq n$ and $A \in \sB^{k-1}$, we have $\bigcup_{\wt A \in \sB_{ A}} \wt A^k = \bigcup_{u \in A} C(u)$.

\item
Fix $A_1, A_2 \in \sB^k_{A}$ with $A_2 \subsetneq A_1$.
For every $u \in A^{k-1}$ and $u' \in C(u) \cap A_1^k$, there exists $A_3 \in \sB^k_{A}$ with $A_3 \subsetneq A_1$ and $u' \in A_3^k$.

\end{itemize}

\medskip

\item[]
{\sc (height partial orders)}
The height partial order $<_\sB$ has the following properties:
\begin{itemize}
\item
Distinct $A_1, A_2 \in \sB^k$ are comparable with respect to $<_\sB$ if and only if $\pi(A_1) = \pi(A_2)$ and $A_1^k \cap A_2^k = \emptyset$.

\item
For any $u \in V^{k-1}(\cT)$ and for any $u_1,u_2 \in C(u)$ with $u_1<_{\cT}u_2$, we have
\begin{align}
A(u_1)
<_\sB
A(u_2).
\end{align}

\item
Fix $\wt A, A \in \sB^{k-1}$ with $A \subseteq \wt A$.
Fix distinct $A_1, A_2 \in \sB_{A}$ and distinct $\wt A_1, \wt A_2 \in \sB_{\wt A}$ such that $\wt A_1 \subseteq A_1$ and $\wt A_2 \subseteq A_2$, and such that $A_1$ and $A_2$, respectively \ $\wt A_1$ and $\wt A_2$, are comparable.
Then we have the equivalence
\begin{align}
A_1 <_{\sB} A_2
\iff
\wt A_1 <_{\sB} \wt A_2.
\end{align}
\end{itemize}
\end{itemize}

\label{def:minimal_n-bracketing}
We let $\sB_\min$ denote the containment-minimal $n$-bracketing, and we give an explicit description of $\sB_\min$: for any $k$ with $1 \leq k \leq n$, $\sB_\min^k$ contains the singleton $k$-brackets and maximum $k$-bracket, and for each $u \in V(\cT)$, the partial order $<_{\sB_{\min}}$ on the singleton brackets associated to $C(u)$ agrees with $<_{\cT}$ on $C(u)$.
\null\hfill$\triangle$
\end{definition}

\begin{figure}[ht]
\includegraphics[width=0.5\textwidth]{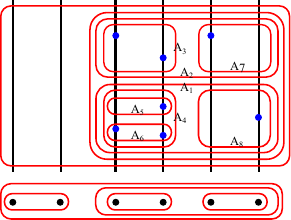}
\caption{
\label{fig:ex_2-bracketing}
A depiction of a 2-bracketing $\sB$ with bubbles representing nonsingleton brackets.
The conditions {\sc ($k$-bracketing projection)}, {\sc(nested)}, and {\sc(trivial $k$-brackets)} are easy to verify in the figure.
For the relevance of the property {\sc(partition)}, notice that the points in $A_5$ and $A_6$ partition the points in $A_4$.
Similarly, the points in $A_1$ and $A_2$ partition the points in the maximum $2$-bracket.
As the name suggests, the height partial order $<_\sB$ is compatible with the heights of the bubbles in the figure.
For the relevance of the property {\sc(height partial order)}, notice that $A_1 <_\sB A_2$ which implies that $A_4 <_\sB A_3$ and $A_8 <_\sB A_7$.
Note that there are no 2-brackets that project to the left-most 1-bracket.
Because the maximum brackets are always present in a bracketing, we may omit them in some later figures.
We will depict maximum brackets when we feel that they help illustrate some aspect of a construction or argument, e.g.\ when they appear as essential brackets in a collision, or when they are in the support of a metric $n$-bracketing.
}
\label{2bracketingex}
\end{figure}

\begin{figure}[ht]
\includegraphics[width=0.8\textwidth]{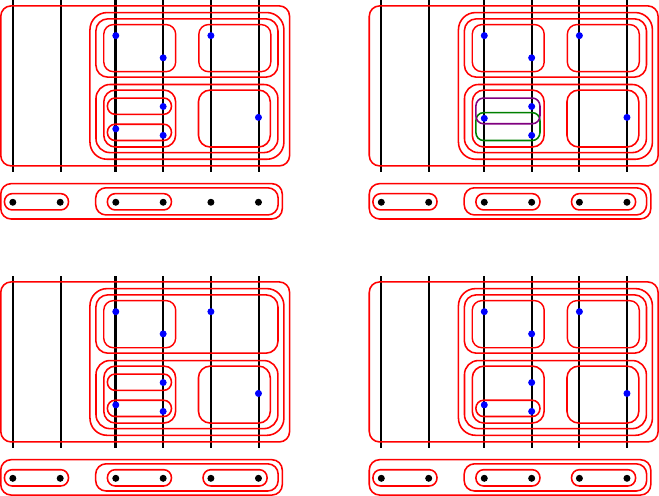}
\caption{
\label{fig:2-bracketing_non-examples}
In this figure, we illustrate four different nonexamples of 2-bracketings.
In the upper left, we have deleted the right-most nontrivial 1-bracket, so the result fails the \textsc{($k$-bracketing projection)} condition.
In the upper right, we have enlarged one of the smaller 2-brackets, so that the result fails the \textsc{(nested)} condition.
(We have colored the relevant 2-brackets purple and green, respectively, for clarity.)
In the bottom left and bottom right, we have deleted 2-brackets so that the results fail the first and second parts of the \textsc{(partition)} condition, respectively.
}
\end{figure}

\begin{definition}[categorical $n$-associahedra]
\label{def:n-associahedron}
Let $\cT$ be a rooted plane tree.
The categorical $n$-associahedron
$\cK(\cT)$ (or simply $n$-associahedron) is the collection of $n$-bracketings of $\cT$ equipped with the following partial order: $\sB_1 \leq \sB_2$ if for every $k$ with $0\leq k\leq n$ we have that $\sB_1^k \subseteq \sB_2^k$, and $<_{\sB_1}$ is a restriction of $<_{\sB_2}$.
\null\hfill$\triangle$
\end{definition}

\begin{remark}
There is a unique 0-associahedron, and it is the singleton poset.
\null\hfill$\triangle$
\end{remark}

\begin{figure}[ht]
\includegraphics[width=0.9\textwidth]{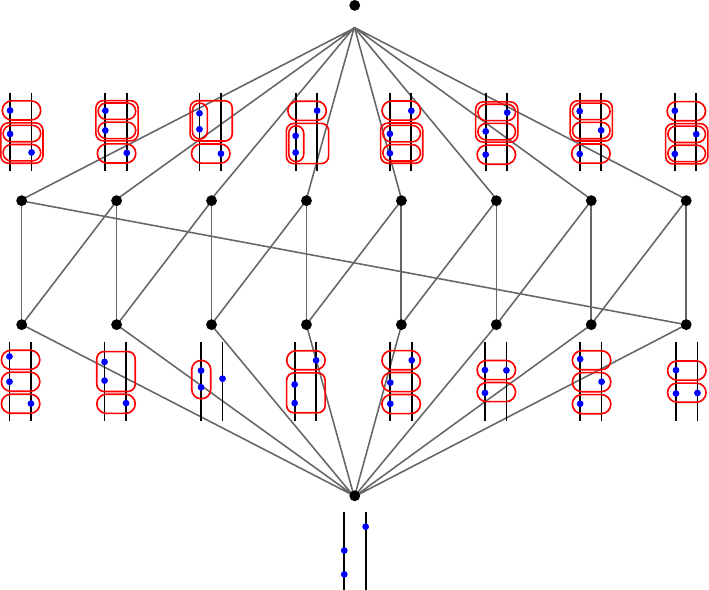}
\caption{
\label{2associahedraposet}
A 2-associahedron.
}
\end{figure}

\begin{remark}
In the case of $2$-associahedra, the poset $\cK(\cT)$ defined above is the opposite of the poset $W_\bn^\br$ defined by the second author.
We have chosen to work with $\cK(\cT)$ because it is compatible with the procedure of taking unions of $n$-bracketings, and with the facial structure of the velocity fan.
\null\hfill$\triangle$
\end{remark}

\begin{remark}
The (\textsc{trivial $k$-brackets}) condition may seem strange to some readers.
For one thing, we want to be able to apply the (\textsc{partition}) condition with $A$ the maximum $(k-1)$-bracket.
For another, as explained with the extended height partial order below, we want the partial orders to cohere with the orders on vertices, which is imposed by (\textsc{height partial orders}) in conjunction with the fact that all singleton brackets are included.
There are other reasons besides these; in general, the theory works better when we always include all trivial brackets.
\null\hfill$\triangle$
\end{remark}

We define a natural extension of the height partial order.

\begin{definition}\label{heightpartialorderremark}
Consider an $n$-bracketing $\sB$.
Fix distinct $A_1, A_2 \in \sB^k$ and  distinct $\wt A_1, \wt A_2 \in \sB^k$.  If \ $A_1<_{\sB}  A_2$, we set $A_1 \,{\wh <_{\sB}}\, A_2$. If $\wt A_1 \subseteq A_1$, $\wt A_2 \subseteq A_2$, and \ $\wt A_1<_{\sB} \wt A_2$, we set $A_1 \,{\wh <_{\sB}}\, A_2$.   If $A_1 \subseteq \wt  A_1$, $ A_2 \subseteq \wt A_2$, and \ $\wt A_1<_{\sB} \wt A_2$, we set $A_1 \,{\wh <_{\sB}}\, A_2$.
We refer to $\,{\wh <_{\sB}}\,$ as the \emph{extended height partial order}.
\null\hfill$\triangle$
\end{definition}

The {\sc (nested)} and {\sc (height partial order)} properties imply that ${\wh <_{\sB}}$ is well-defined.
Observe in Figure \ref{fig:ex_2-bracketing} that $A_8 \subseteq A_1$, $A_3 \subseteq A_2$, and $A_1 <_{\sB} A_2$, therefore $A_8 \,\,{\wh <_{\sB}}\,\, A_3$, and this agrees with the height of these brackets in the figure.

\begin{lemma}\label{heightpartialorderlemma}
The extended height partial order ${\wh <_{\sB}}$ on the singleton brackets in $\sB$ determines the height partial order $<_{\sB}$ on all brackets in $\sB$.
\end{lemma}

\begin{proof}
Suppose that $A_1, A_2 \in \sB^k$ are distinct and comparable.
Then $A_1 <_{\sB}A_2$ if and only if for any pair of vertices $u^k_i \in A_1^k$ and $u^k_j \in A_2^k$, we have $A(u^k_i) \,\, {\wh <_{\sB}} \,\,A(u^k_j)$.
\end{proof}

\begin{definition}[restriction of an $n$-associahedron to an $n$-bracket]
\label{restrictbracket}
Let $A$ be a $k$-bracket.
The $n$-bracket $A$ naturally defines a rooted plane subtree of $\cT$ depth $k$, which we denote $\cT |_{A}$.
We define $\cK(A) \coloneqq \cK(\cT|_{A})$.
\null\hfill$\triangle$
\end{definition}

\begin{definition}
Given an $n$-bracketing $\sB \in \cK(\cT)$ and a nontrivial $k$-bracket $A \in \sB^k$ for some $1\leq k \leq n$.
We say that $A$ is \emph{containment-minimal} in $\sB$ if there exists no nontrivial $A' \in \sB^k$ such that $(A')^k \subsetneq A^k$.
\null\hfill$\triangle$
\end{definition}

\begin{definition}
\label{def:compatible_n-bracketings}
Suppose that $\sB_1, \ldots, \sB_m$ are $n$-bracketings in $\cK(\cT)$.
We say that this collection is \emph{compatible} if the following conditions hold:
\begin{enumerate}
\item\label{compatibleconndition1}
$\sB_1\cup\ldots\cup\sB_m
\coloneqq
(\sB^0_1\cup\cdots\cup\sB^0_m,\ldots,\sB^n_1\cup\cdots\cup\sB^n_m)$ satisfies the \textsc{(nested)} condition in Def.\ \ref{def:n-bracketings}.

\item\label{compatibleconndition2}
There exists a partial order on $\sB_1\cup\cdots\cup\sB_m$ that satisfies the \textsc{(height partial orders)} condition in Def.\ \ref{def:n-bracketings}, and restricts to the height partial order on $\sB_i$ for every $i$.
\null\hfill$\triangle$
\end{enumerate}
\end{definition}

\begin{proposition}\label{n-bracketingunionlemma}
Suppose that $\sB_1, \ldots, \sB_m$ are a compatible collection of $n$-bracketings in $\cK(\cT)$.
Then the partial order in condition (\ref{compatibleconndition2}) of Definition \ref{def:compatible_n-bracketings} is uniquely-defined.
We may therefore interpret $\sB_1\cup\cdots\cup\sB_m$ as an $n$-bracketing in a unique way.
\null\hfill$\triangle$
\end{proposition}

\begin{proof}
 
The only part of this definition that we need to check is that if $\sB, \sB'$ are $n$-bracketings satisfying the two bulleted conditions, then the partial orders appearing in the second bullet are determined uniquely.
Fix $A \in \sB^k$ and $A' \in \sB'^{\,k}$ that lie over the same $(k-1)$-bracket $\wt A \in \sB^{k-1} \cup (\sB')^{\,k-1}$ and that have $A^k \cap A'^{\,k} = \emptyset$.
By the {\sc (nested)} and {\sc (partition)} axioms, there exists $A'' \in \sB_{\wt A}$ with $A''^{\,k} \cap A'^{\, k} \neq \emptyset$.
$A$ and $A''$ are comparable by $<_\sB$, which determines the relation between $A$ and $A'$ by $<_{\sB\cup\sB'}$.
\end{proof}

\begin{definition}
\label{def:poset_completion}
Given a poset $\cP$ without a maximum element, we define the completion of $\cP$, denoted $\wh{\cP}$, as the poset on set of elements $\cP \sqcup\{\star\}$ having the relations from $\cP$ and the added relations $\star \geq x$ in $\wh{\cP}$ for every $x \in \cP$.
\null\hfill$\triangle$
\end{definition}

\begin{definition}
A \emph{join-semilattice} is a partially ordered set $\cP$ such that for any $x,y \in \cP$, the least upper bound for $\{x,y\}$ exists in $\cP$.
In such a case we write $x \vee y$ for this least upper bound and call this object the \emph{join} of $x$ and $y$.
Similarly, a \emph{meet-semilattice} is a partially ordered set $\cP$ such that for any $x,y \in \cP$, the greatest lower bound for $\{x,y\}$ exists in $\cP$.
In such a case we write $x \wedge y$ for this least upper bound and call this object the \emph{meet} of $x$ and $y$.
A \emph{lattice} is a poset which is both a join-semilattice and a meet-semilattice.
\null\hfill$\triangle$
\end{definition}

It is a classical fact that join-semilattice with a minimum element is also a meet-semilattice, and thus a lattice.

\begin{proposition}
The poset $\wh \cK(\cT)$ is a lattice.
\end{proposition}

\begin{proof}
The poset $\wh \cK(\cT)$ has a minimum element $\sB_{\text{min}}$, hence it suffices to show that $\wh \cK(\cT)$ is a join-semilattice.
Given two $n$-bracketings $\sB$ and 
$\sB'$, $\sB \cup \sB'$ is the smallest $n$-bracketing such that $
\sB \leq \sB \cup \sB'$ and $\sB' \leq \sB \cup \sB'$ if one exists, in which case $\sB \vee \sB' = \sB \cup \sB'$.
If $\sB \cup \sB'$ does not exist, then $\sB \vee \sB' = \star$.
\end{proof}

Next we demonstrate that $n$-associahedra are flag.
While flagness is a condition typically discussed for (face posets of) simplicial complexes, and $\cK(\cT)$ is usually not simplicial, flagness exists in greater generality.
The following definition is equivalent to the one given in \cite{huang2024cycles}.

\begin{definition}
Let $\cP$ be a poset, then $\cP$ is \emph{flag} if the following condition holds:

\begin{itemize}
\item Let $x_1, \ldots, x_\ell \in \cP$ such that $x_i \vee x_j \in \cP$ for all $1\leq i,j \leq \ell$, then $\bigvee_{i=1}^\ell x_i\in \cP$.
\null\hfill$\triangle$
\end{itemize}
\end{definition}

\begin{lemma}\label{flagcondition}
The poset $\cK(\cT)$ is flag.
\end{lemma}

\begin{proof}
We prove the claim by induction on $\ell$ with base case $\ell=3$.
Consider $n$-bracketings $\sB_1, \sB_2, \sB_3$.
The only potential issue we must consider is whether there is a valid height partial order on $\sB \coloneqq \sB_1 \cup \sB_2 \cup \sB_3$.

Fix disjoint $A, A' \in \sB^k$ which lie over the same $(k-1)$-bracket $\wt A$.
We must exclude the following representative situation:
\begin{align}
A <_{\sB_1 \cup \sB_2} A',
\qquad
A >_{\sB_1 \cup \sB_3} A'.
\end{align}
Suppose that $A$ lies in $\sB_1$ and $A'$ lies in $\sB_2$ and $\sB_3$.
(The other possibilities may be treated in a similar way.)
By the {\sc (nesting)} and {\sc (partition)} properties, we may choose $A'' \in (\sB_1)_{\wt A}$ with $(A'')^k \cap (A')^k \neq \emptyset$, and $(A'')^k \cap (A)^k = \emptyset$.
Then $A$ and $A''$ are comparable via $<_{\sB_1}$, and this determines the relation between $A$ and $A'$ with respect to both $<_{\sB_1 \cup \sB_2}$ and $<_{\sB_1 \cup \sB_3}$.

Now take $\ell \geq 4$.
By the $\ell=3$ argument above, it is the case that for every $i$, $\sB_1 \cup \sB_2$ is compatible with $\sB_i$.
The collection $\bigl(\sB_1 \cup \sB_2\bigr) \cup \bigl(\sB_i\bigr)_{3 \leq i \leq \ell}$ therefore satisfies the hypothesis of the lemma, so we may apply induction to this collection.
\end{proof}

\subsection{Collisions as atoms}\label{collisionsasatoms}

\

In this subsection, we introduce \emph{collisions}, which are the minimal elements of $\cK(\cT)\setminus\{\sB_\min\}$.
We provide a combinatorial characterization of collisions and prove that the $n$-associahedra are atomic lattices.

\begin{definition}
Let $(P,\leq)$ be a partially ordered set with a minimum element denoted $\wh{0}$.
An atom in $P$ is an element $x \in P$ such that $x$ covers $\wh{0}$.
\null\hfill$\triangle$
\end{definition}

\begin{definition}
\label{def:collision}
An $n$-bracketing $\sB$ is a \emph{collision} if it is an atom in $\cK(\cT)$.
We will typically denote a collision by $\sC$ rather than $\sB$.
\label{def:set_of_collisions}
\label{def:completed_set_of_collisions}
We let $\fX(\cT)$ denote the set of all collisions in $\cK(\cT)$, and let $\wh \fX(\cT) \coloneqq \fX(\cT) \sqcup \{\sB_\min\}$ be the set of \emph{extended collisions}. 
\null\hfill$\triangle$
\end{definition}

\begin{remark}
The term ``collision'' is related to the geometric motivation for $n$-associahedra.
Indeed, as described in \S\ref{ss:compactified_moduli_spaces}, the $n$-associahedra will be the posets of strata for a certain compactification of a configuration space of coordinate subspaces in $\bR^n$.
A collision in $\cK(\cT)$ corresponds to an indecomposable collision of subspaces, i.e.\ a codimension-1 stratum in this topological realization.
\null\hfill$\triangle$
\end{remark}

\begin{remark}
As the notation suggests, $\wh \fX(\cT)$ is naturally the completion of $\fX(\cT)$ equipped with a certain partial order described in \S\ref{triangulationsection}.
\null\hfill$\triangle$
\end{remark}

We define the root bracket and fusion brackets of an $n$-bracketing, and utilize these notions for characterizing collisions.
These results will also be used in Proposition \ref{prop:alpha_is_bijective}.

\begin{defprop}
Fix an $n$-bracketing $\sB \in \cK(\cT)\setminus\{\sB_{\min}\}$.
Define
\begin{align}
S
\coloneqq
\Bigl\{
A(u) \in \sB
\text{ a singleton bracket}
\:\Big|\:
\forall
\text{ nontrivial }
A \in \sB,
\:\exists\:
\ell\geq0
:
\pi^\ell(A) = A(u)
\Bigr\}.
\end{align}
We define the \emph{root bracket of $\sB$} to be the $k$-bracket in $S$ with the largest $k$.
\null\hfill$\triangle$
\end{defprop}

\begin{proof}
We justify the well-definedness of the root bracket of a collision $\sC$.
The set $S$ is nonempty, because it contains the unique 0-bracket.
Suppose that $k$ is the largest integer such that $S$ contains at least one $k$-bracket, and suppose that $A(u_1), A(u_2)$ are $k$-brackets in $S$.
Let $A \in \sB^j$ be a nontrivial $j$-bracket, then by the definition of $S$, we must have $\pi^{j-k}(A) = A(u_1)$ and $\pi^{j-k}(A) = A(u_2)$, so $A(u_1) = A(u_2)$. 
\end{proof}

\begin{definition}
\label{def:fusion_bracket}
Fix an $n$-bracketing $\sB \in \cK(\cT)$, and denote by $A(u)$ the root bracket of $\sB$.
We say that $A \in \sB$ is a \emph{fusion bracket of $\sB$} if it is a nonsingleton bracket which satisfies the following conditions:
\begin{enumerate}
\item
There exists a nontrivial $\wt A \in \sB$ and $\ell\geq0$ such that $\pi^\ell\bigl(\wt A\bigr) = A$.

\item
The projection $\pi(A)$ is a singleton bracket.
\null\hfill$\triangle$
\end{enumerate}
\end{definition}

\begin{proposition}\label{collisioncharacterization}
Suppose that $\sB \in \cK(\cT)$ is an $n$-bracketing.
Then $\sB$ is a collision if and only if the following two conditions hold:
\begin{enumerate}
\item
$\sB$ has a unique fusion bracket.

\item
For each $k$, there exist no nontrivial $k$-brackets $A, A' \in \sB^k$ with $A \subsetneq A'$.
\end{enumerate}
\end{proposition}

\begin{proof}
Suppose that $\sB$ is a collision, but $A_1$ and $A_2$ are two distinct fusion brackets for $\sB$.
We can construct an $n$-bracketing $\sB_1$ with $\sB_1< \sB$ by taking the singleton and maximum brackets together with the brackets which project to $A_1$, and equipping $\sB_1$ with the restriction of the height partial order for $\sB$.
This contradicts the claim that $\sB$ is a collision.

Suppose that $A, A'$ are nontrivial brackets with $A \subsetneq A'$.
Suppose that $A, A' \in \pi(\sB)$, then by induction on $n$, $\pi(\sB)$ is not a collision, so let $\wt \sB$ is an $(n-1)$-bracketing with $\wt \sB<\pi(\sB)$.
We can lift $\wt \sB$ to an $n$-bracketing $\sB'<\sB$ by adding the singleton and maximum $n$-brackets to $\wt \sB$ together which the $n$-brackets in $\sB$ which project to $(n-1)$-brackets in $\wt \sB$, and taking $<_{\wt \sB}$ the restriction of $<_{\sB}$.
Thus we may assume that $A,A'$ are $n$-brackets and $\pi(A)=\pi(A')$, then we can obtain $\sB'$ from $\sB$ by deleting $A'$.
The \textsc{(nested)} and \textsc{(partition)} conditions are preserved by the deletion of $A'$, hence $\sB'$ is an $n$-bracketing with $\sB' <\sB$. 

Conversely, suppose that $\sB$ is not a collision, and $\sC$ is a collision such that $\sC <\sB$. Suppose that $\sB$ has a unique fusion bracket.
This must be the same as the fusion bracket for $\sC$.
It now follows from the \textsc{(nested)} and \textsc{(partition)} conditions that $\sB$ must have a pair of distinct nontrivial brackets $A,A'$ with $A\subsetneq A'$.

\end{proof}

\begin{lemma}
\label{lem:fusion_bracket_for_collision}
Let $\sC \in \fX(\cT)$ and take $A \in \sC$ the fusion bracket of $\sC$.
The projection $\pi(A)$ is equal to the root bracket $A(u)$ of $\sC$.
In particular, every nontrivial bracket in $\sC$ weakly projects to $A$.
\end{lemma}

\begin{proof}

Suppose that $\pi(A)$ is not the root bracket of $\sC$.
By the definition of the root bracket, there exists a nontrivial bracket $\wt A \in \sB$ and $\ell \geq 0$ such that for every $\ell \geq 0$, $\pi^\ell(\wt A) \neq \pi(A)$.
It must be that $\wt A$ projects to some fusion bracket, but this contradicts the assumption that $A$ is the unique fusion bracket for $\sC$.
\end{proof}

\begin{definition}
\label{codime1char}
Let $\sC$ be a collision.
We say that:
\begin{itemize}
\item
$\sC$ is \emph{type 1} if $\sC$ contains a single nontrivial $n$-bracket $A$, and $\pi(A)$ is a singleton $(n-1)$-bracket.

\smallskip

\item
$\sC$ is \emph{type 2} if $\pi(\sC)$ is an $(n-1)$-collision.

\smallskip

\item
$\sC$ is \emph{type 3} if $\pi(\sC)$ is the minimal $(n-1)$-bracketing and $\sC$ is not type 1.
\null\hfill$\triangle$
\end{itemize}
\end{definition}

It is a direct consequence of Lemma \ref{collisioncharacterization} that each collision is either type 1, 2, or 3.

\begin{definition}\label{essentialbrackets}
Let $\sC \in \fX(\cT)$, and $A \in \sC^k$.
We say that $A$ is an \emph{essential bracket} of $\sC$ if
\begin{enumerate}
\item there exists a nontrivial bracket $A' \in \sC^l$ with $l\geq k$ such that $\pi^{l-k}(A') = A$, and
\item there exists some $j\geq 0$ such that $\pi^j(A)$ is equal to the fusion bracket of $\sC$. 
\null\hfill$\triangle$
\end{enumerate}
\end{definition}

The set of essential brackets of a collision $\sC$ is a slight enlargement of the set of nontrivial brackets in $\sC$; if the fusion bracket for $\sC$ is a maximum bracket (necessarily the maximum 1-bracket), then the maximum $k$-brackets in $\sC$ which are projections of nontrivial brackets are essential.
This subtle definition is useful, and sometimes necessary, for investigating certain constructions built from collisions.

\begin{remark}
We make the convention that the extended collision  $\sB_{\min}$ has the maximum 1-bracket as its fusion bracket, the unique 0-bracket as its root bracket, and the collection of all maximum $k$-brackets as its essential brackets.
\end{remark}

\begin{figure}[ht]
\includegraphics[width=0.7\textwidth]{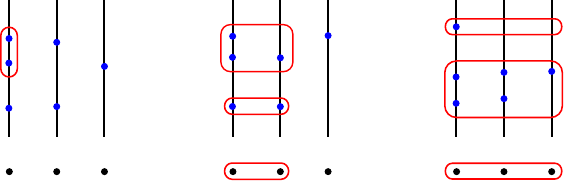}
\caption{On the left, we depict a type 1 collision.
As with any type 1 collision, its only nontrivial bracket is also its fusion bracket.
The singleton bracket corresponding to the left line is its root bracket.
In the middle is a type 2 collision.
Its nontrivial 1-bracket is its fusion bracket.
On the right is a type 3 collision.
The maximum 1-bracket is its fusion bracket. The maximum 1-bracket is an essential bracket in this collision.
This is the only essential bracket present in these three collisions which is trivial.
For the middle and right collisions, the root bracket is the singleton 0-bracket, which corresponds to the plane.}
\label{collisiontypes}
\end{figure}

\begin{figure}[ht]
\includegraphics[width=0.5\columnwidth]{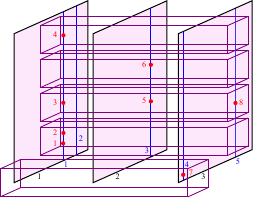}
\caption{A collision in the 3-associahedron depicted in Figures \ref{fig:arrangement_in_R3_and_rootedtreeexample} and \ref{3brackex}.}
\label{3collisionfig}
\end{figure}

\begin{definition}
A lattice is atomic if every element $x \in P$ is expressable as the join of a collection of atoms.
\null\hfill$\triangle$
\end{definition}

\begin{lemma}\label{collisioncontainingspecificbracket}
Let $A$ be a nontrivial $k$-bracket in $\sB^k$.
There exists a collision $\sC \leq \sB$ with $A \in \sC^k$.
\end{lemma}

\begin{proof}
We prove this claim by induction on $n$.
In what follows, we will discuss construction of collisions.
To avoid repetitious language, we will implicitly always be adding the necessary singleton and maximum brackets, and necessary height partial orders will assumed to be the restrictions of the height partial order $<_\sB$.
For $n=1$, we can take $\sC$ to be the 1-bracketing determined by $A$.
Next, suppose $n \geq 2$.
First suppose that $k<n$, then by induction, there exists some ${\wt \sC}$ an $(n-1)$-collision such that ${\wt \sC} \leq \pi(\sB)$ and $A \in {\wt \sC}^k$.
We can lift ${\wt \sC}$ to a type 2 collision $\sC$ with $\sC\leq \sB$ by taking, for each nontrivial $(n-1)$-bracket $A \in \sC^{n-1}$, the collection of containment-minimal $n$-brackets from $\sB$ which project to $A$. 

We will now assume that $A \in \sB^n$.
First, suppose that $\pi(A)$ is a singleton $(n-1)$-bracket.
Then we can take $\sC$ to be the type 1 collision whose fusion bracket is $A$.

Next, suppose that $\pi(A)$ is the maximum $(n-1)$-bracket.
We can define $\sC$ to be the type 3 collision whose nontrivial $n$-brackets consists of $A$ along with all containment-minimal nontrivial $n$-brackets $A' \in \sB^n$ with $A^n \cap (A')^n = \emptyset$ and $\pi(A')$ is the maximum $(n-1)$-bracket.
It follows from the (\textsc{partition}) condition that $\sC$ is a valid $n$-bracketing.

Finally, suppose that $\pi(A)$ is a nontrivial $(n-1)$-bracket.
By induction, there exists an $(n-1)$-collision $\wt\sC \leq \pi(\sB)$ such that $\bigl(\wt\sC\bigr)^{n-1}$ contains $\pi(A)$.
We will define a type 2 collision $\sC \leq \sB$ with the properties that $\pi(\sC) = \wt\sC$, and $\sC^n$ contains $A$. 
For each such $(n-1)$-bracket $\wt A \in {\wt \sC}^{n-1}$ not equal to $\pi(A)$, we add to $\sC^n$ the set of containment-minimal nontrivial $n$-brackets in $\sB_{\wt A}$.
For $\pi(A)$, we add $A$ to $\sC^n$ together with the set of all containment-minimal nontrivial $n$-brackets $A' \in \sB_{\pi(A)}$ that have $A^n \cap (A')^n = \emptyset$.
\end{proof}

\begin{proposition}
\label{lem:atomic_and_atomic}
The $n$-associahedron $\wh\cK(\cT)$ is atomic.
\end{proposition}

\begin{proof}

By Lemma \ref{collisioncontainingspecificbracket}, for each nontrivial bracket $A \in \sB^k$, we can construct a collision $\sC$ with $A \in \sC^k$ and $\sC\leq \sB$.
The collection of such collisions, varying over all $A$ is compatible, and thus we can take their join.
By Proposition \ref{n-bracketingunionlemma} this join is equal to $\sB$. 
\end{proof}

\subsection{Collisions as tree maps}

\

In this subsection we provide an alternate characterization of collisions via certain maps of rooted plane trees.
This characterization will later be used for defining the velocity fan.
We begin with some preparatory definitions.

\begin{definition}[morphisms of rooted plane trees]
Suppose that $\cT, \cT'$ are rooted plane trees.
A \emph{morphism of rooted plane trees $f\colon \cT \to \cT'$} is a map induced by its restriction to the vertices $f\colon V(\cT) \to V(\cT')$, which must satisfy the following conditions:
\begin{enumerate}
\item
$f$ sends the root of $\cT$ to the root of $\cT'$.

\smallskip

\item If $u,w \in V(\cT)$ and $(u,w) \in E(\cT)$, then $(f(u),f(w)) \in E(\cT')$.

\smallskip

\item
For any vertex $u$ of $\cT$, $f$ restricts to a map $C(u) \to C(f(u))$ which is weakly-order-preserving: for $w_1, w_2 \in C(u)$ with $w_1 < w_2$, we have $f(w_1) \leq f(w_2)$.
\null\hfill$\triangle$
\end{enumerate}
\end{definition}

\begin{remark}
If $f\colon \cT \to \cT'$ is a morphism of rooted plane trees, then $f$ is a graph homomorphism.
Furthermore, if we consider $\cT$ and $\cT'$ as directed graphs with all edges directed toward the roots, then $f$ is a directed graph homomorphism.
If $u$ is a depth-$k$ vertex of $\cT$, then $f(u)$ is a depth-$k$ vertex of $\cT'$.
\null\hfill$\triangle$
\end{remark}

\begin{definition}
Let $\cT$ be a rooted plane tree, $u$ a vertex of $\cT$, and $W = (w_1,\ldots,w_k)$ a tuple of consecutive elements of $C(u)$.
We let $\cT(u;W)$ denote to the rooted plane subtree of $\cT$ induced by the vertices $u, w_1, \ldots, w_k$, and the descendants of the $w_i$.
Here
$\cT(u;W)$ inherits the structure of a rooted plane tree with root $u$ from $\cT$.
\null\hfill$\triangle$
\end{definition}

\begin{definition}[collision maps]
\label{def:collision_maps}
Suppose that $\cT$ is a depth-$n$ rooted plane tree.
A \emph{local collision map $f$} consists of the following data.
\begin{itemize}
\item
A vertex $u$ of $\cT$.

\item
A collection $W$ of consecutive elements of $C(u)$, with $|W|\geq 2$.

\item
A rooted plane tree $\cT'$ with one depth-1 vertex, and a surjective morphism $f\colon \cT(u; W) \to \cT'$ of rooted plane trees.
We exclude the possibility that $\cT'$ is the path with $n+1$ vertices.
\end{itemize}

\noindent
Denote by $\wt\cT$ the result of replacing $\cT(u;W)$ by $\cT'$ in $\cT$.
That is, we delete all the nonroot vertices of $\cT(u;W)$ in $\cT$, and then identify the root of $\cT'$ with the root of $\cT(u;W)$.

The map $\cT(u;W) \to \cT'$ extends by the identity to a map $\cT \to \wt\cT$, which we call a \emph{global collision map}, or simply a \emph{collision map}.
We will denote a collision map as $\psi^{\cT}_{\wt \cT}$ or simply $\psi$ if ${\cT}$ and ${\wt \cT}$ are clear from the context.
Observe that given a collision map $\psi^{\cT}_{\wt \cT} = \psi$ we can naturally pullback brackets along this map: let $A$ be a $k$-bracket on $\wt \cT$, then we let $\psi^{-1}(A)$ denote the $k$-bracket on $\cT$ with 
$(\psi^{-1}(A))^{\ell}= \psi^{-1}(A^{\ell}).$
\null\hfill$\triangle$
\end{definition}

\begin{remark}
A local collision map can be recovered uniquely from the global collision map that it gives rise to, so these two notions of collision maps are equivalent.
\null\hfill$\triangle$
\end{remark}

\begin{figure}
\centering
\includegraphics[width=0.7\columnwidth]{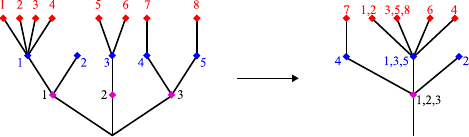}
\caption{A collision map associated to the collision depicted in Figure \ref{3collisionfig}.}
\label{collisionmapfig}
\end{figure}

We now explain how collision maps are naturally in bijection with collisions.

\begin{proposition}
\label{treemap}
Let
$\psi^{\cT}_{\wt \cT} = \psi$
be a collision map.
Then we may define a collision $\sC \eqqcolon \alpha(\psi)$ of $\cT$ in the following way:
\begin{itemize}
\item
$\sC$ contains all singleton and maximum brackets.

\smallskip

\item
Let $u' \in V^k(\wt \cT)$, and let $A(u')$ be the corresponding singleton $k$-bracket in $\wt\cT$. Then $\sC$ contains the $k$-bracket $\psi^{-1}(A(u'))$.

\smallskip

\item
For $u' \in V(\wt\cT)$ and $u_1, u_2 \in C(u')$.
Suppose that $u_1 <_{\wt\cT} u_2$, then we take $\psi^{-1}(A(u_1)) <_{\sC} \psi^{-1}(A(u_2))$.
\end{itemize}

\end{proposition}

\begin{proof}
Observe that if $\sC$ is an $n$-bracketing then it clearly satisfies the conditions of Proposition \ref{collisioncharacterization} with unique fusion bracket given by $\psi^{-1}(W)$. We now verify that $\sC$ is an $n$-bracketing.

The conditions
(\textsc{trivial $k$-brackets}) and 
(\textsc{$k$-bracketing projection}) are immediate. For verifying the
(\textsc{nested}) condition, suppose that $A_1, A_2$ are $k$-brackets corresponding to vertices $u_1, u_2$ in $\wt \cT$.
If $u_1 = u_2$, then $A_1$ and $A_2$ are the same $k$-brackets.
If $u_1 \neq u_2$, then $A_1^k$ and $A_2^k$ are disjoint.
The first part of the (\textsc{partition}) condition holds as $\psi$ is a surjective map of rooted plane trees.
The second part of the (\textsc{partition}) condition holds vacuously as $\sC$ does not contain any nested nontrivial brackets.

For verifying the first part of the
(\textsc{height partial orders}) condition, let $A_1, A_2$ be $k$-brackets in $\sC$ such that $\pi(A_1) = \pi(A_2)$.
This implies $A_1 = \psi^{-1}(u_1)$ and $A_2 = \psi^{-1}(u_2)$ for $u_1, u_2 \in C(u')$ for some $u' \in \wt \cT$.
The condition that $A_1, A_2$ are moreover disjoint is equivalent to $u_1 \neq u_2$.
It follows from our definition of $\sC$ that $A_1$ and $A_2$ are comparable. The second and third parts of the (\textsc{height partial orders}) condition hold as a consequence of the fact that $\psi$ is a map of rooted plane trees and thus is weakly order preserving.
\end{proof}

\begin{proposition}
\label{prop:alpha_is_bijective}
The map $\alpha$ sending collision maps on $\cT$ to collisions on $\cT$, is bijective.
\end{proposition}

\begin{proof}
\label{def:quotient_of_tree_by_collision}
Given a collision $\sC$, we construct $\alpha^{-1}(\sC)=\psi^{\cT}_{\cT/\sC}$.
For describing the tree $\cT/\sC = {\wt \cT}$, we must define the tree $\cT'$.
We do so by taking the root bracket $A(u)$ of $\sC$ to be associated to the root vertex of $\cT'$.
Let the vertices of $\cT'$ correspond to the essential brackets of $\sC$.
We order the vertices in $\cT'$ according to the height partial order of $\sC$.
Let $W$ be the vertices in the fusion bracket of $\sC$.
We define the local collision map $f:\cT(u;W)\rightarrow \cT'$ to send $u^k_i \in \cT(u;W)$ to the vertex of $\cT'$ corresponding to the bracket $A \in \sC^k$ such that $u^k_i \in A^k$.
Given this local collision map $f$, we associate the global collision map $\psi^{\cT}_{\wt \cT}$. 
 It is straightforward to verify that $\psi^{\cT}_{\wt \cT}$ is indeed a collision map with $\alpha(\psi^{\cT}_{\wt \cT})=\sC$.
\end{proof}

\begin{definition}[contraction of an $n$-associahedron along a collision]
\label{contractbracketing}
Fix a collision $\sC \in \fX(\cT)$.
Let $\psi\colon \cT \to \wt \cT$ be the corresponding collision map.
We define the \emph{contraction of $\cT$ along $\sC$} to be $\wt\cT$, and we define, the \emph{contraction of $\cK(\cT)$ along $\sC$} denoted $\cK(\cT/\sC)$, to be $\cK(\wt\cT)$.
\null\hfill$\triangle$
\end{definition}

\begin{remark}
Geometrically, we can think about $\cK(\cT/\sC)$ in the following way.
Consider a collision of coordinate subspaces in $\bR^n$.
To take such a limit, we let subspaces fuse, and bubble off additional copies of $\bR^n$ wherever appropriate (see footnote \ref{bubbleofffootnote}).
If we consider only the ``base'' copy of $\bR^n$, then the combinatorial type of the arrangement of coordinate subspaces there is exactly $\wt\cT$.
\null\hfill$\triangle$
\end{remark}

\begin{definition}[contraction of an $n$-bracketing along a collision]\label{quotientbracketingdef}
Let $\sC \in \fX(\cT)$ with corresponding collision map $\psi$. 
Let $\sB \in \cK(\cT)$ which is compatible with $\sC$.
We define the quotient of $\sB$ by $\sC$, denoted $\sB/\sC$, to be the $n$-bracketing with $(\sB/\sC)^k = \{ \psi(A): A \in \sC\}$ and the height partial order inherited from $\sB \cup \sC$.
\null\hfill$\triangle$
\end{definition}

\noindent
We leave it to the reader to
verify from the definitions that $\sB/\sC$ is indeed an $n$-bracketing.

\subsection{Collision definitions and lemmas}

\

We continue to develop the the theory of collisions introducing several definitions and useful lemmas which will be employed in this article.  The reader should feel welcome to skim (or even skip) this subsection, and return to it as it is needed.

\begin{definition}\label{collisionrelations}
Fix $\sC_1,\sC_2 \in \fX(\cT)$.
We say that $\sC_1$ is \emph{contained} in $\sC_2$, and write $\sC_1 \rightarrow \sC_2$, if $\sC_1$ and $\sC_2$ are compatible and for each essential bracket $A_1 \in \sC_1^k$, there exists an essential bracket $A_2 \in \sC_2$ such that $A_1 \subseteq A_2$.
If it is not true that $\sC_1 \rightarrow \sC_2$, we may write $\sC_1 \nrightarrow \sC_2$.
\null\hfill$\triangle$
\end{definition}

\begin{remark}
One may alternately define $\sC_1 \rightarrow \sC_2$ as follows: $\sC_1$ and $\sC_2$ are compatible, and for every $k\geq 0$ such that $\pi^{k}(\sC_2)\neq \pi^{k}(\sB_\min)$,
each nontrivial $(n-k)$-bracket $A_1 \in \sC_1^{n-k}$ is contained in a nontrivial $(n-k)$-bracket $A_2 \in \sC_2^{n-k}$.
We will use this characterization in arguments below.
\null\hfill$\triangle$
\end{remark}

\begin{definition}\label{collisionrelationsdisjoint}
Let $\sC_1,\sC_2 \in \fX(\cT)$ and let $A_1$ and $A_2$ be the fusion brackets for $\sC_1$ and $\sC_2$, respectively.
Let and ${\overline D}(A_1)$ and ${\overline D}(A_2)$ denote the weak descendants of the vertices in $A_1$ and $A_2$, respectively, then $\sC_1$ and $\sC_2$ are \emph{disjoint}, written $\sC_1 \sim \sC_2$, if ${\overline D}(A_1) \cap {\overline D}(A_2) = \emptyset$.
Equivalently for any nontrivial brackets $A_1 \in \sC_1^k$ and $A_2 \in \sC_2^k$ we have $A^k_1 \cap A^k_2 = \emptyset$.
If $\sC_1$ and $\sC_2$ are not disjoint, we may write $\sC_1\nsim \sC_2$.
\null\hfill$\triangle$
\end{definition}

Note that if $\sC_1 \sim \sC_2$, then $\sC_1$ and $\sC_2$ are compatible.

\begin{remark}
Although we are employing $\sim$ notation, this binary relation is not an equivalence relation: it is reflexive and symmetric, but not transitive.
This choice is instead motivated by graph adjacency notation as is justified in \S\ref{triangulationsection}.
\null\hfill$\triangle$
\end{remark} 

We extend the above relations to $\wh \fX(\cT)$ by insisting that $\sC \rightarrow \sB_\min$ and $\sC \nsim \sB_\min$ for all $\sC \in \fX(\cT)$.
The following is immediate from the definitions. 
\begin{lemma}\label{projnestedisnested}
Let $\sC, \sC' \in \fX(\cT)$ which are both not type 1.
\begin{itemize}
\item
If $\sC \sim \sC'$, then $\pi(\sC) \sim \pi(\sC')$.

\item
If $\sC\rightarrow \sC'$ then $\pi(\sC) \rightarrow \pi(\sC')$.
\end{itemize}
\end{lemma}

\begin{definition}
Fix $\sB\in \cK(\cT)$ and $\sC \in \fX(\cT)$ with $\sC \leq \sB$.
We say that $\sC$ is \emph{containment-minimal} in $\sB$ if each nontrivial bracket $A \in \sC^k$ is containment-minimal in $\sB$.
\null\hfill$\triangle$
\end{definition}

\begin{lemma}
\label{containmentmintocollisionminlem}
Fix $\sB\in \cK(\cT)$ and $\sC \in \fX(\cT)$.
Then $\sC$ is \emph{containment-minimal} in $\sB$ if and only if the following conditions hold:
\begin{enumerate}
\item $\sC \leq \sB$.

\smallskip

\item for every $\sC' \in \fX(\cT)$ with $\sC' \leq \sB$ either $\sC \rightarrow \sC'$ or $\sC \sim \sC'$.
\end{enumerate}
\end{lemma}
\begin{proof}
Suppose that $\sC_1$ is containment-minimal in $\sB$ and $\sC_2$ is another collision with $\sC_2 \leq \sB$.
If $\sC_1 \sim \sC_2$ there is nothing to prove.
So suppose that there exists some nontrivial brackets $A_1\in \sC_1^k$ and $A_2 \in \sC_2^k$ such that $A_1^k\cap A_2^k \neq \emptyset$.
Then $A_1 \subseteq A_2$ as $A_1$ is containment-minimal in $\sB$.
Let $A_1' \in \sC_1^j$ be another nontrivial bracket and suppose $\pi^{n-j}(\sC_2) \neq \pi^{n-j}(\sB_{\min})$.
We wish to prove that there exists some nontrivial $A_2' \in \sC_2^j$ such that $A_1'\subseteq A_2'$.
Because projection of brackets respects containment, we know that the fusion bracket for $\sC_1$ weakly projects into the fusion bracket for $\sC_2$.
There is some $i\geq 0$ such that $\pi^i(A_1')$ is the fusion bracket of $\sC_1$.
For an element $u \in (A_1')^j$, repeated application of the {\sc (partition)} property applied to the fusion bracket for $\sC_2$ gives some nontrivial bracket of $A_2' \in \sC_2^j$ such that $u \in (A_2')^{j}$. The desired containment $A_1'\subseteq A_2'$ then follows by containment-minimality of $A_1'$.

Conversely, suppose that $\sC_1$ satisfies the conditions of the lemma, but is not containment-minimal, then there exists some $A \in \sC_1^k$ and some $A' \in \sB$ such that $A' \subsetneq A$.
By Lemma \ref{collisioncontainingspecificbracket}, we know that there is some collision $\sC_2 \leq \sB$ such that $A' \in \sC_2$, but then it is clear that $\sC_1 \nrightarrow \sC_2$, a contradiction.
\end{proof}

From the above argument that we have the following scholium.\footnote{We use the word \emph{scholium} to indicate a statement which follows as a consequence of a proof.  This is in contrast to a corollary, which should be a consequence of a statement of a proposition.
We have found this to be a useful term, and we hope the reader does not consider its application too pretentious.}

\begin{scholium}
\label{simplifiedcontainment-minimal condition}
Fix $\sB\in \cK(\cT)$ and $\sC \in \fX(\cT)$.
Let $\sC_1, \ldots, \sC_k \in\fX(\cT)$ such that $\bigvee_{i=1}^k\sC_i = \sB$.
Then $\sC$ is \emph{containment-minimal} in $\sB$ if and only if the following conditions hold:
\begin{enumerate}
\item $\sC \leq \sB$.

\smallskip

\item for $1\leq i\leq k$, either $\sC \rightarrow \sC_i$ or $\sC \sim \sC_i$.
\end{enumerate}   
\end{scholium}

\begin{lemma}
Let $\sB \in \cK(\cT)$.
There exists a collision $\sC \in \fX(\cT)$ such that $\sC$ is containment-minimal in $\sB$.
\end{lemma}

\noindent

\begin{proof}
If there exists a type 1 collision $\sC' \leq \sB$ with fusion bracket $A$, then there exists a type 1 collision $\sC$ with $\sC \rightarrow \sC'$ which is containment-minimal in $\sB$.
Therefore we may assume that there is no such $\sC'$.
Let ${\overline \sC} \in {\wh{\fX}}(\pi(\cT))$ such that ${\overline \sC}$ is containment-minimal in $\pi(\sB)$.
We can lift ${\overline \sC}$ to a collision $\sC \in \fX(\cT)$, i.e.\ $\pi(\sC)={\overline \sC}$ by defining the nontrivial $n$-brackets in $\sC^n$ to be the containment-minimal $n$-brackets in $\sB^n$ which project to some essential $(n-1)$-bracket in ${\overline \sC}^{n-1}$.
By the assumption that there is no type 1 collision $\sC'$ with $\sC' \leq \sB$, it follows that the collision $\sC$ is containment-minimal in $\sB$.
\end{proof}

As a consequence of the previous argument we have the following scholium.

\begin{scholium}
\label{lem:collisionlift}
For every $\sC' \in {\wh \fX}(\pi(\cT))$ such that $\sC' \leq \pi(\sB)$ there exists $\sC \in \wh \fX(\cT)$ with $\sC \leq \sB$ and $\pi(\sC) =\sC'$.
\end{scholium}

\begin{lemma}\label{bracketinglift}
For any $\sB' \in \cK(\pi(\cT))$ there exists some $\sB\in \cK(\cT)$ such that $\pi(\sB) = \sB'$. 
\end{lemma}

\begin{proof}
We can lift $\sB'$ to $\sB$ by taking $\sB^n$ to consist of the singleton $n$-brackets, the maximum $n$-bracket, and the following nontrivial $n$-brackets.
For each $A' \in \sB'$ and $u_i^n \in V^n(\cT)$ we construct an $n$-bracket $A$ such that $\pi(A) =A'$ and $A^n =\{u_i^n\}$.
Finally we take the height partial order on the $n$-brackets which is induced by the extended height partial order on singleton brackets, as in Lemma \ref{heightpartialorderlemma}, which agrees with $<_{\cT}$.
\end{proof}

\begin{lemma}\label{projcontain}
Let $\sC \in \wh \fX(\cT)$ such that $\sC$ is not type 1, and let $\sB \in \cK(\cT)$ such that $\sC$ is containment-minimal in $\sB$, then $\pi(\sC)$ is containment-minimal in $\pi(\sB)$.
\end{lemma}

\begin{proof}
This is a direct consequence of the definition of a containment-minimal collision.
\end{proof}

\begin{lemma}\label{restrictnbracketingtocollision}
Let $\sB \in \cK(\cT)$ and $\sC \in \fX(\cT)$ with $\sC \leq \sB$.
Let $\sB'$ obtained from $\sB$ by removing all nontrivial $A\in \sB^k$ for which there exists no nontrivial $A' \in \sC^k$ such that $A\subseteq A'$, and taking $<_{\sB'}$ to be the restriction of $<_{\sB}$.
Then $\sB'$ is an $n$-bracketing.
\end{lemma}

\begin{proof}
One can check that the (\textsc{nested}) and (\textsc{partition}) properties are preserved for $\sB'$.
\end{proof}

\begin{lemma}\label{containmentlemma}
If $\sC \in \wh \fX(\cT)$ and $\sB \in \cK(\cT)$ with $\sC \leq \sB$, then there exists some $\sC' \in \wh \fX(\cT)$ which is containment-minimal in $\sB$ such that $\sC' \rightarrow \sC$.
\end{lemma}

\begin{proof}
By Lemma \ref{restrictnbracketingtocollision}, the restriction of $\sB$ to $\sC'$ defines an $n$-bracketing and we can take $\sC$ to be containment-minimal in this $n$-bracketing.
\end{proof}

\begin{definition}\label{minimalcollisiondef}

We say that $\sC \in \fX(\cT)$ is a \emph{minimal collision} if there exists no $\sC' \in \fX(\cT)$ such that $\sC' \rightarrow \sC$.
We say that $\sC \in \wh \fX(\cT)$ is a \emph{minimal extended collision} if $\sC$ is a minimal collision or $\sC = \sB_\min$.
\null\hfill$\triangle$
\end{definition}

The following lemma characterizes minimal collisions. 

\begin{lemma}\label{collisionminimalchar}
Suppose that $\sC \in \fX(\cT)$.
Then $\sC$ is a minimal collision if and only if 
\begin{itemize}
\item
the fusion bracket $A$ of $\sC$ is a $k$-bracket with $|A^k|=2$ and

\item
for each $j>k$, every essential $j$-bracket in ${\wt A} \in \sC^j$ has $|{\wt A}^j|=1$.
\end{itemize}
Moreover, if $\sC$ is not a minimal collision then there exists distinct $\sC_1, \sC_2 \in \wh \fX(\cT)$ such that $\sC_1\rightarrow \sC$ and $\sC_2\rightarrow \sC$.
\end{lemma}

\begin{proof}
Suppose that $\sC \in \wh \fX(\cT)$ satisfies the conditions of Lemma \ref{collisionminimalchar}.
Suppose for contradiction that $\sC$ is not a minimal collision.
Then there exists some ${\wt \sC} \in \fX(\cT)$ with ${\wt \sC} \neq \sC$ such that ${\wt \sC} \rightarrow \sC$. 
We may assume without loss of generality that ${\wt \sC}$ is a minimal collision.
As ${\wt \sC}\neq \sC$ there must exist some essential $i$-bracket ${\wt A} \in {\wt \sC}^i$ and a nontrivial $i$-bracket $A' \in \sC$ such that ${\wt A} \subsetneq A'$, but this is clearly impossible.
Conversely, let $\sC \in \wh \fX(\cT)$ be a minimal collision with fusion bracket $A$, and suppose for contradiction that the conditions of Lemma \ref{collisionminimalchar} are not satisfied.
We can construct two collisions $\sC_1, \sC_2$ such that $\sC_1\rightarrow \sC$, $\sC_2\rightarrow \sC$.
Moreover $\sC_1$ and $\sC_2$ will both satisfy the conditions of \ref{collisionminimalchar} and hence be minimal.
There are two cases:

\begin{itemize}
\item $|A^k| \geq 3$ or

\item there is some $j$-bracket in ${\wt A} \in \sC$ with $j>k$, which is an essential bracket in $\sC$ and $|{\wt A}^j|\geq 2$.
\end{itemize}

In either case it is easy to find the desired extended collisions $\sC_1$ and $\sC_2$ (see Figure \ref{nonminimalcollisions}.)
\end{proof}

\begin{figure}[ht]
\includegraphics[width=0.55\textwidth]{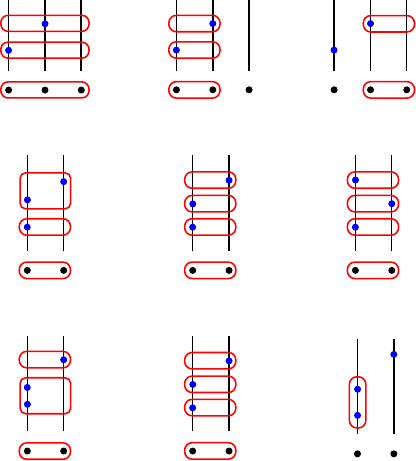}
\caption{In each line, the collision on the left violates the conditions of a minimal collision described in Lemma \ref{collisionminimalchar}, and the middle and right collisions in each row give a pair of minimal collisions $\sC_1$ and $\sC_2$ such that $\sC_1 \rightarrow \sC$ and $\sC_2 \rightarrow \sC$.
}
\label{nonminimalcollisions}
\end{figure}

For the following lemmas, recall Definition \ref{quotientbracketingdef}.

\begin{lemma}\label{uniquequotientlemma}
Let $\sC, \sC_1, \sC_2 \in \fX(\cT)$ such that for each $i \in \{1,2\}$ either $\sC \rightarrow \sC_i$ or $\sC \sim \sC_i$.
Then $\sC_1/\sC \neq \sC_2/\sC$.
\end{lemma}

\begin{proof}
If $A \in \sC_1^k$ and $A \notin \sC_2^k$, then $A/\sC \notin (\sC_2/\sC)^k$.
\end{proof}

\begin{lemma}\label{preimagelemma}
Let $\sC \in \wh \fX(\cT)$, and let $\overline{\sC}\in {\wh \fX}(\cT/\sC)$, then there is a unique preimage $\sC'$ of $\overline{\sC}$ in $\wh \fX(\cT)$ so that $\sC'/\sC = \overline{\sC}$.
Moreover, $\sC'$ is such that either $\sC\rightarrow {\sC}'$ or $\sC\sim {\sC}'$.
\end{lemma}

\begin{proof}
For each $k$-bracket $A \in \sC$, the image of $A$ in $\cT/\sC$ is a singleton bracket by definition.
Given a $k$-bracket $A' \in \cT/\sC$, it has a well-defined preimage $k$-bracket in $\cT$.
We address two cases separately. 
In both cases $\sC'$ will be assumed to have all singleton and maximum brackets, as it must, and $\sC'$ will be equipped with a height partial order inherited from ${\overline{\sC}}$.
Let $A(u)$ be the image of the fusion bracket of $\sC$ in $\cT/\sC$.
Let $A(u') \in \overline{\sC}$ be the root bracket of $\overline{\sC}$ in $\cT/\sC$.
 
The first case is where there does not exist some $i\geq0$ such that $\pi^i(A(u'))=A(u)$.
In this case we simply take the preimage of each nontrivial $k$-bracket in $\overline{\sC}$.
This produces $\sC'$ such that $\sC \sim \sC'$

The second case is that there does exist some $i\geq0$ such that $\pi^i(A(u'))=A(u)$.
We construct $\sC'$ by taking the collection of preimages of nontrivial $k$-brackets in $\overline{\sC}$ together with the preimages of each singleton $k$-bracket $A(u'')$ in $\cT/\sC$ such that $A(u'')$ is not contained in a nontrivial $k$-bracket in $\overline{\sC}$, and for which there exists some $j\geq 0$ such that $\pi^j(A(u''))=A(u)$.
This produces a collision $\sC' \in \fX(\cT)$ such that $\sC \rightarrow \sC'$. 
Uniqueness follows from Lemma \ref{uniquequotientlemma}.
\end{proof}

\begin{lemma}
\label{whenthequotientofacollisionisacollision}
If $\sC, \sC' \in \fX(\cT)$, $\sC\rightarrow \sC'$, and $\sC, \sC'$ have different fusion brackets, then the quotient $\sC'/\sC$ is a collision.
\end{lemma}

\begin{proof}
We know that the quotient $\sC'/\sC$ is an $n$-bracketing.
To verify that it is a collision, we should check the two conditions of Lemma \ref{collisioncharacterization}.
That there are no nontrivial nested brackets is clear.
Because the fusion bracket $A$ for $\sC'$ properly contains the fusion bracket for $\sC$, we can see that the image of $A$ is the unique fusion bracket for $\sC'$.
\end{proof}

See Figure \ref{collisionquotient} above for an example demonstrating how the conclusion of the lemma can fail if the fusion bracket for $\sC'$ is equal to the fusion bracket for $\sC$.
The following Lemma refines the above observation.

\begin{lemma}\label{quotientofcollisionbycollision}
Suppose that $\sC, \sC' \in \fX(\cT)$ and let $\sC\rightarrow \sC'$.
Then ${\overline \sC'}$, the image of $\sC'$ in $\cK(\cT/\sC)$, is an $n$-bracketing which is uniquely expressable as
\begin{align}
{\overline \sC'}= \bigcup_{r=1}^t \sC_r,
\end{align}
where $\sC_r \in \fX(\cT/\sC)$ with $1\leq r\leq t$.
Moreover, for $1\leq r,s\leq t$ with $r \neq s$ we have that $\sC_r \sim \sC_s$. 
\end{lemma}

\begin{proof}
Let $A$ and $A'$ be the fusion brackets for $\sC$ and $\sC'$, respectively.
If $A'\neq A$, then ${\overline \sC'}$ is a collision as described above in Lemma \ref{whenthequotientofacollisionisacollision}.
On the other hand, suppose that $A=A'$.
Let $\{A_r:1\leq r \leq t\}$ be the set of essential brackets in $\sC'$ which properly contain some nontrivial bracket in $\sC$, but $\pi(A_r)$ is equal to some bracket in $\sC$.
Let $X_r$ be the set of essential brackets in $\sC'$ which project to $A_r$.
Then the images of the $X_r$ in $\cT/\sC$ form the essential $k$-brackets of $\sC_r$ in the statement of the Lemma, and the images of the $A_r$ are their fusion brackets.
\end{proof}

\begin{lemma}\label{liftdisjointcollisions}
Suppose that $\sC \in \fX(\cT)$ and let $\{\sC_r:1\leq r\leq t\} \subseteq \fX(\cT/\sC)$ such that for $1\leq r,s\leq t$ with $r \neq s$ we have that $\sC_r \sim \sC_s$, and the root vertex of $\sC_r$ weakly projects to the image of the fusion bracket of $\sC$ in $\cT/\sC$.
Then there exists a unique collision $\sC' \in \fX(\cT)$ such that $\sC\rightarrow \sC'$ and ${\overline \sC'}$, the image of $\sC'$ in $\cT/\sC$ decomposes as
\begin{align}
{\overline \sC'}= \bigvee_{r=1}^t \sC_r.
\end{align}
\end{lemma}

\begin{proof}
Let $S$ be the collection of essential brackets which appear in some $\sC_r$.
Let $T$ be the collection of singleton brackets in $\cT/\sC$ which project to the image of the fusion bracket of $\sC$, i.e.\ the singletons which are the images of the essential brackets in $\sC$.
We can construct $\sC'$ by taking, in addition to the necessary singleton and maximum brackets, the preimages of the brackets in $S$ and the preimages of the singleton brackets in $T$ which are not contained in any bracket in $S$.
We take the height partial order on these nontrivial brackets inherited from the height partial orders on $\sC$ and the $\{\sC_r\}$. The resulting $\sC'$ is an $n$-bracketing which satisfies the desired conditions.
Uniqueness follows from Lemma \ref{uniquequotientlemma}.
\end{proof}

\begin{figure}[ht]
\includegraphics[width=0.4\textwidth]{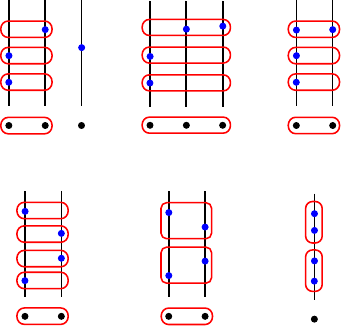}
\caption{In each line, the first and second figures give a pair of collisions $\sC$ and $\sC'$ with $\sC\rightarrow \sC'$ and the third figure is the quotient $\sC'/\sC$.
In the first line, the quotient is a collision, and the second line the quotient $\sC'/\sC$ is the union of two disjoint collisions.}
\label{collisionquotient}
\end{figure}

\section{The metric \texorpdfstring{$n$}{n}-bracketing cone complex}

In this section, we begin by recalling the theory of cone complexes.
We then introduce the collection of metric $n$-bracketings $\cK^{\met}(\cT)$, and prove that $\cK^{\met}(\cT)$ forms a cone complex whose face poset is $\cK(\cT)$.
The image of a finite cone complex under a globally injective piecewise-linear map to $\mathbb{R}^m$ is a fan, and in \S \ref{velocityfansection} we apply this fact to demonstrate that the velocity fan is indeed a fan.
In later sections, $\cK^{\met}(\cT)$ will continue to serve as our central tool for investigating the structure of the velocity fan.

\subsection{Cone complexes and piecewise-linear transformations}

\

A cone complex is a generalization of a (rational) fan which arises naturally in the context of logarithmic geometry.
Our presentation of cone complexes is slightly different from what others have given and is motivated by our applications as well as a desire to provide an elementary description of these objects which is accessible to researchers who may not be fluent in toric geometry.
In Remark \ref{compareconecomplexdefs}, we outline how our definition compares with definitions in the existing literature.
We note that our definition of a cone complex suppresses ambient lattices and thus admits irrational fans as a special case.

We begin by introducing the notion of a conical set complex, which may be viewed as a ``precone complex".
Recall that, informally, a monoid is a group without inverses and a semiring is a ring without additive inverses.  The reader may consult \cite{golan_semirings} for precise definitions if they are unfamiliar with these notions.

\begin{definition}[\S7, \cite{golan_semirings}]
Suppose that $R$ is a semiring.
A \emph{(left) $R$-semimodule} is an additively-written commutative monoid $M$ with identity $0_M$, together with a function $R\times M \to M$, $(r,m) \mapsto r\, m$, which is required to satisfy the following conditions for all $r, s \in R$ and $m, n \in M$:
\begin{enumerate}
\item
$(rs)m = r(sm)$,

\item
$r(m+n) = rm + rn$,

\item
$(r+s)m = rm + sm$,

\item
$1_Rm = m$,

\item
$r0_M = 0_M = 0_Rm$.
\null\hfill$\triangle$
\end{enumerate}
\end{definition}

\begin{definition}
A \emph{conical set} is a semimodule over $\mathbb{R}_{\geq 0}$.
\null\hfill$\triangle$
\end{definition}

\begin{definition}
    Let $\tau_1$ be a conical set and $\tau_2 \subseteq \tau_1$.  We say that $\tau_2$ is a \emph{conical subset} of $\tau_1$ if $\tau_2$ is a conical set and its $\mathbb{R}_{\geq 0}$-semimodule structure is the restriction of the $\mathbb{R}_{\geq 0}$-semimodule structure on $\tau_1$.
\end{definition}

\begin{definition}\label{facedefinition}
Given a conical set $\tau_1$, we define a \emph{face} of $\tau_1$ to be a conical subset $\tau_2\subseteq \tau_1$ such that given any $\lambda \in \bR_{>0}$ and $x, y \in \tau_1$, the following implications hold:
\begin{enumerate}\label{eq:face_requirements}
\item\label{facecond1} $\lambda x \in \tau_2 \Rightarrow x \in \tau_2$,
\qquad
\item\label{facecond2} $x+y \in \tau_2 \Rightarrow x, y \in \tau_2$.
\null\hfill$\triangle$
\end{enumerate}
\end{definition}

\begin{definition}\label{conecomdef}
Let $X$ be a set and $\cF$ be a collection of conical sets with $\tau \subseteq X$ for every $\tau \in \cF$.
We say that $\cF$ is a \emph{conical set complex} if
\begin{enumerate}
\item
\label{conecom1}
For any $\tau_1 \in \cF$ and for any conical set $\tau_2$ which is a face of $\tau_1$, we have $\tau_2\in \cF$.

\item
\label{conecomdef2}
For any $\tau_1, \tau_2 \in \cF$, $\tau_1 \cap \tau_2$ is a union of faces of both $\tau_1$ and $\tau_2$.
\null\hfill$\triangle$
\end{enumerate}
\end{definition}

\begin{definition}
Let $\tau_1,\tau_2$ be conical sets and $f$ be a map $f\colon \tau_1\rightarrow \tau_2$, then $f$ is \emph{linear} if for every $\lambda \in \bR_{\geq0}$ and $x, y \in \tau_1$ we have
\begin{align}
f(\lambda x) = \lambda f(x),
\qquad
f(x+y) = f(x) + f(y).
\end{align}
A linear map $f$ is an \emph{isomorphism} if it is bijective.
\null\hfill$\triangle$
\end{definition}

\begin{definition}
Let $\tau \subseteq \mathbb{R}^n$ be a conical set.
We say that $\tau$ is a \emph{closed polyhedral cone} if it is a conical set in $\mathbb{R}^n$ which is closed  and has finitely many faces.  
We say that $\tau$ is an \emph{open polyhedral cone} if it is the relative interior of some closed polyhedral cone.
\null\hfill$\triangle$
\end{definition}

When no confusion may arise, we will simply refer to a closed polyhedral cone as a \emph{cone}.

\begin{definition}
An \emph{abstract cone} is a conical set which admits a linear isomorphism to a closed polyhedral cone.
\null\hfill$\triangle$
\end{definition}

\begin{definition}
\label{def:cone_complex}
A conical set complex $\cF$ is a \emph{cone complex} if each conical set in $\cF$ is an abstract cone.
\null\hfill$\triangle$
\end{definition}

\begin{remark}\label{compareconecomplexdefs}
There are some slight variations in the definition of a cone complex in the literature (see \cite{kempftoroidal,kato1994toric, abramovich1997weak,payne2009toric,ulirsch2017functorial}). 
Often one defines a cone complex as a topological space and then acquires the algebraic structure via homeomorphisms to polyhedral cones.
Instead we begin with the algebraic structure of a conical set complex and observe that the relevant topology on a cone complex can be acquired as the weak-$*$ topology induced by the linear isomorphisms from the conical sets to polyhedral cones.
Thus our definition essentially agrees with, for example, the one given in \cite{ulirsch2017functorial}.
Almost all other authors keep careful track of ambient lattices in their definition, which we have not done, and observe that this allows irrational fans as a special case (although all fans considered in this article are rational).
We note that there are ambient lattices, often implicit, in our work, and all maps presented respect these lattice structures (see Corollary \ref{conecomplexlatticepropertyvelocityfan}).
Finally, we note that we do not insist that the conical sets or cones under consideration are pointed, i.e.\ strongly convex.
Indeed, as is justified by Remark \ref{linealityspacermk}, we will work with cone complexes with lineality spaces.
\null\hfill$\triangle$
\end{remark}

\begin{definition}\label{abstractfandef}
Let $X$ be a set and $\cF$ be a finite collection of abstract cones with $\tau \subseteq X$ for every $\tau \in \cF$.
We say that $\cF$ is an \emph{abstract fan}
if 
\begin{enumerate}
\item
\label{absfandef1}
For any $\tau_1 \in \cF$ and for any conical set $\tau_2$ which is a face of $\tau_1$, we have $\tau_2\in \cF$.

\item
\label{absfandef2}
For any $\tau_1, \tau_2 \in \cF$, $\tau_1 \cap \tau_2$ is a face of both $\tau_1$ and $\tau_2$.
\null\hfill$\triangle$
\end{enumerate}
\end{definition}

We emphasize that abstract fans are cone complexes where we have restricted to finite collections of conical sets and strengthened condition (2).\footnote{The authors were tempted to define an abstract fan as a conical set complex satisfying conditions (1) and (2) of Definition \ref{abstractfandef}, i.e.\ to not insist that abstract fans be cone complexes.
Other researchers may be interested in considering such objects.}

\begin{definition}
A \emph{fan} is an abstract fan whose ambient space is some Euclidean space $\mathbb{R}^m$.
\null\hfill$\triangle$
\end{definition}

\begin{definition}
\label{def:face_poset}
Let $\cF$ be a conical set complex, then the \emph{face poset} of $\cF$ is the partially ordered set $\cP(\cF)$ whose elements are the conical sets belonging to $\cF$, and for $\tau_1,\tau_2 \in \cF$, $\tau_1\leq \tau_2$ when $\tau_1$ is a face of $\tau_2$. 
\null\hfill$\triangle$
\end{definition}

\begin{definition}
\label{supportdefinition}
Let $\cF$ be a conical set complex with ambient set $X$.
We define the \emph{support of $\cF$} to the be the set of points $supp(\cF)=\{ {\bf v}\in X: \exists \, \tau \in \cF, {\bf v} \in \tau\}$.
\null\hfill$\triangle$
\end{definition}

\begin{definition}
\label{completedefinition}
Let $\cF$ be a fan in $\mathbb{R}^m$.
We say that $\cF$ is \emph{complete} if the support of $\cF$ is a linear space.
\null\hfill$\triangle$
\end{definition}

Typically, authors define a fan to be complete if its support is equal to all of $\mathbb{R}^n$, so this definition is a bit more general, but one can always project onto a coordinate subspace so that the support becomes all of a Euclidean space while preserving the face poset of the fan.
Alternately, one can add the orthogonal complement of the support of $\cF$ as the lineality space.
Our choice to work with a more general definition is only relevant in \S\ref{localsection} on fiber products of fans.
More precisely, for Corollary \ref{fibercomplete} to be true, we must use this more general definition of completeness for fans.

\begin{lemma}\label{generatorsconsetlem}
Let $\tau_1$ be a conical set for which there exists $R = \{x_1, \ldots, x_k\} \subseteq \tau_1$ such that $\tau_1 = \{\sum_{i=1}^k\lambda_i x_i:\lambda_i \in \mathbb{R}_{\geq 0}\}$.  For $\tau_2$ a face of $\tau_1$, there exists some $S \subseteq R$ such that $\tau_2 = \{\sum_{i=1}^k\lambda_i x_i:\lambda_i \in \mathbb{R}_{\geq 0}, x_i \in S\}.$
\end{lemma}

\begin{lemma}\label{abstractsubconelem}
Let $\tau_1$ be an abstract cone and let $\tau_2$ be a conical subset of $\tau_1$, then $\tau_2$ is an abstract cone.
\end{lemma}

\begin{definition}
Let $\cF$ be a conical set complex, and let $\cG$ be a collection of conical sets such that the elements of $\cG$ are subsets of a common ambient space $Y$.
Let $f$ be a collection of maps $\{f_\tau:\tau \in \cF\}$ such that the domain of $f_{\tau}$ is $\tau$ and the codomain of $f_{\tau}$ is some $\tau' \in \cG$.
Then we say that $f$ is \emph{piecewise-linear} if each $f_{\tau}$ is linear on $\tau$ and for any two conical sets $\tau_1, \tau_2 \in \cF$, we have
\begin{align}
f_{\tau_1}|_{\tau_1\cap \tau_2} = f_{\tau_2}|_{\tau_1\cap \tau_2}.
\end{align}
Such an $f$ is a \emph{morphism} of conical set complexes.
\null\hfill$\triangle$
\end{definition}

\noindent

The following notion is due to Gross \cite{gross2018intersection}, and was motivated by the desire to develop tropical intersection theory for cone complexes.

\begin{definition}\label{weaklyembeddeddef}
Let $\cF$ be a cone complex and $\cG$ be a collection of cones in $\mathbb{R}^n$.
Let $f:\cF \rightarrow \cG$ be a piecewise-linear map such that each $f_{\tau}$ with $\tau \in \cF$ is a linear isomorphism.
Then we say that the triple $(\cF,\cG, f)$ is a \emph{weakly embedded cone complex} and $\cF$ is a \emph{weakly embeddable cone complex}.
\null\hfill$\triangle$
\end{definition}

Note that in Definition \ref{weaklyembeddeddef}, the cones in $\cG$ need not be internally disjoint.

\begin{definition}
Let $\cF_1$ and $\cF_2$ be conical set complexes and $f$ be a piecewise-linear map from $\cF_1$ to $\cF_2$.
We say that $f$ is a piecewise-linear isomorphism if there exists some piecewise-linear $g:\cF_2 \rightarrow \cF_1$ such that $g$ is a two-sided inverse for $f$.
\null\hfill$\triangle$
\end{definition}

The proof of the following lemma is straightforward.

\begin{lemma}
\label{faceposetisolem}
Let $\cF$ be an conical set complex with ambient set $X$, and let $\cG$ be a collection of conical sets with ambient set $Y$.
Suppose that $f$ is piecewise-linear function from $\cF$ to $\cG$ such that $f$ induces a bijection from $supp(\cF)$ to $supp(\cG)$, and for each $\tau$, the map $f_\tau$ is an isomorphism onto a conical set in $\cG$.
Then $\cG$ is a conical set complex, $f$ is a piecewise-linear isomorphism, and $\cP(\cF) \cong \cP(\cG)$.
\end{lemma}

We now discuss subdivisions of conical set complexes.

\begin{definition}
Let $\tau$ be a conical set and $\cF$ be a conical set complex.
We say that $\tau$ is \emph{simplicial} if there exists no nontrivial relation
\begin{align}
\sum_{x_i \in Y}\lambda_i x_i = \sum_{x_i \in Y}\gamma_i x_i \, ,
\end{align}
for some finite set $Y \subseteq \tau$, and $\lambda_i,\gamma_i\geq 0$.  We say that $\cF$ is simplicial if each conical set in $\cF$ is simplicial.
\null\hfill$\triangle$
\end{definition}

\noindent

\begin{definition}
Let $\cF_1, \cF_2$ be conical set complexes living in a common ambient set $X$.
We say that $\cF_2$ is a \emph{refinement} of $\cF_1$ if $supp(\cF_1)=supp(\cF_2)$, and each conical set in $\cF_1$ is a union of conical sets in $\cF_2$.
We say that $\cF_2$ is a \emph{triangulation} of $\cF_1$ if $\cF_2$ is a refinement of $\cF_1$ and each conical set in $\cF_2$ is simplicial.
\null\hfill$\triangle$
\end{definition}

The following is immediate.

\begin{lemma}\label{imagetriang}
Let $\cF_1$ and $\cG_1$ be conical set complexes and suppose that $f:\cF_1 \rightarrow \cG_1$ is a piecewise-linear isomorphism.
Suppose that $\cF_2$ is a conical set complex which refines $\cF_1$, then the image $f(\cF_2)$ refines $\cG_1$. 
\end{lemma}

\begin{definition}
Let $\tau$ be an abstract cone.
We say that $V \subseteq \tau$ is a \emph{vector subspace} if $V$ is a conical subset of $\tau$ such that for each $x \in \tau$ there exists a $y \in \tau$ such that $x+y =0_M$.
For $x, y \in \tau$, let $x\sim_V y$ if $x=y+v$ for some $v \in V$.
We define $\tau/V$ as the quotient of $\tau$ by the relation $\sim_V$.
Suppose that $\cF$ is a cone complex living in an ambient space $X$.
We say that $V \subseteq X$ is a \emph{vector subspace of $\cF$} if, for every conical set $\tau$ of $\cF$, $V$ is a vector subspace of $\tau$.
If $V$ is a vector subspace of $\cF$, we define the quotient $\cF/V = \{\tau/V: \tau \in \cF\}$.
We define the \emph{lineality space of $\cF$} to be the maximum vector subspace of $\cF$.
If $L$ is the lineality space of $\cF$, we may denote $\cF/L$ by $\overline{\cF}$.
\null\hfill$\triangle$
\end{definition}

Note that if $V$ is a vector subspace of a conical set $\tau$, then the semimodule structure of $V$ over $\bR_{\geq0}$ extends to a module structure over $\bR$, and $\tau/V$ is a conical set.

\begin{lemma}
\label{lem:abstract_fan_equivalence_relation}
Suppose that $\cF$ is a cone complex, and that $V$ is a vector subspace of $\cF$.
Then the quotient $\cF/V$ is an cone complex, and $\cP(\cF/V) = \cP(\cF)$.
\end{lemma}

\begin{lemma}\label{reducedisomorphismlemma}
Suppose that $f:\cF\rightarrow \cG$ is a piecewise-linear isomorphism of cone complexes, then the induced map ${\overline f}:{\overline \cF}\rightarrow{\overline \cG}$ is piecewise-linear isomorphism of cone complexes.
\end{lemma}

\begin{definition}
If $\tau$ is a conical set whose lineality space is $0_M$, we say that $\tau$ is \emph{pointed}.
If $\cF$ is a cone complex whose lineality space is $0_M$, we say that $\cF$ is \emph{pointed}.
\null\hfill$\triangle$
\end{definition}

\begin{definition}\label{sigmaRdef}
Let $\sigma$ be an abstract cone.
We define $\langle \sigma \rangle_{\mathbb{R}} = \{x-y : x,y \in \sigma\}$ with $x-y \coloneqq \{(x,y) \in \sigma \times \sigma\}/\sim$ where $\sim$ is the equivalence relation $(x_1, y_1) \sim (x_2, y_2)$ if $x_1 + y_2 = x_2 + y_1$.
\null\hfill$\triangle$
\end{definition}

\begin{definition}\label{localizationforabtractfandef}
Suppose that $\cF$ is a cone complex, and $\sigma \in \cF$.
We define the \emph{star of $\sigma$} to be $star(\sigma) = \{ \tau \in \cF: \sigma \subseteq \tau\}$.
For $\tau \in star(\sigma)$, we define the \emph{localization of $\tau$ at $\sigma$} 
\begin{align}
\tau_\sigma
= \tau +\langle \sigma \rangle_{\mathbb{R}}\coloneqq
\{x - y \:|\: x \in \tau,  y \in \sigma\},
\end{align}
with $x - y$ defined as in Definition \ref{sigmaRdef}.
We then define \emph{localization of $\cF$ at $\sigma$} to be the union of all of these conical sets
\begin{align}
\cF_{\sigma}
\coloneqq
\{\tau_\sigma \:|\: \tau \in star(\sigma)\}.
\end{align}
\null\hfill$\triangle$
\end{definition}

We will not immediately pass to $\cF_{\sigma}/ \langle \sigma \rangle_{\mathbb{R}}$ for reasons alluded to in Remark \ref{linealityspacermk}, which are manifested in Theorem \ref{localvelocity}.

\begin{lemma}
Suppose that $\cF$ is a cone complex and $\sigma \in \cF$, then $\cF_{\sigma}$ is a cone complex.
\end{lemma}

\begin{lemma}
\label{quotientrefinement}
Let $\cF$ and $\cG$ be cone complexes such that $\cF$ refines $\cG$ and both $\cF$ and $\cG$ have the same lineality space, then the quotient $\overline{\cF}$ refines the quotient $\overline{\cG}$.
\end{lemma}

\begin{definition}
Let $\cF$ be a cone complex with lineality space $V$.
We define a \emph{ray} to be a face $\tau$ of $\cF$ for which there exists some $\rho \in \tau$ such that
\begin{align}
{\tau} = \{ \lambda{\rho} + V:\lambda \in \mathbb{R}_{\geq 0}\}.
\end{align}
If $\tau$ is a ray, we call any such $\rho$ a \emph{ray generator} for $\tau$. 
\null\hfill$\triangle$
\end{definition}

\begin{remark} 
Usually a ray of a cone $\tau$ or fan $\cF$ is defined to be a 1-dimensional face.
The definition above is reasonable; if we project an abstract cone $\tau$ (or a cone complex $\cF$) along its lineality space $L$, we obtain a pointed conical set (or cone complex), and the rays project to the one dimensional conical sets of the image.
This will arise naturally in our setting as the metric $n$-bracketing complex and the velocity fan both have 1-dimensional lineality spaces.
The velocity fan, like the braid arrangement, has a lineality space generated by the all-ones vector.
Anticipating this construction we offer the following definition.
\null\hfill$\triangle$
\end{remark}

\begin{definition}
We define \emph{tropical projective space} to be the quotient vector space
\begin{align}
\mathbb{T}^{m-1} = \mathbb{R}^m/\langle {\bf 1} \rangle_{\mathbb{R}}.
\end{align}
\null\hfill$\triangle$
\end{definition}
For analyzing the local structure of fans and cone complexes we will need the following lemma.

\begin{lemma}\label{raysoflocalizationlemma}
Let $\cF$ be a cone complex with lineality space $L$, and let $\tau$ be a ray of $\cF$.  The rays of $\cF_{\tau}$ are of the form $\tau' +\tau$, where $\tau'$ is a ray in $\cF$ and $\tau' +\tau \in \cF$.
\end{lemma}

\subsection{Metric \texorpdfstring{$n$}{n}-bracketings}

\

In \cite{b:simplicial}, the second author introduced a simplicial complex $W_\bn^W$, which is composed of the \emph{type-$\bn$ metric 2-bracketings} and is expected to be homeomorphic to the realization of $W_\bn$ as a stratified topological space (see \cite{bottman2017moduli}).
In this section, we introduce a generalization of a slight variation on the notion of a metric 2-bracketing.
We abuse terminology and call our version by the same name.

\begin{remark}
The terminology ``metric $n$-bracketings'' is inspired by the second author's notion of metric tree-pairs in \cite{b:simplicial}, which in turn was inspired by the metric rooted plane trees that appear in the Boardman--Vogt's W-construction (see e.g.\ \cite{markl_shnider}).
\null\hfill$\triangle$
\end{remark}

\begin{remark}
In symplectic geometry, a metric 2-bracketing can be thought of as a witch curve (in the sense of \cite{bottman2017moduli}), together with instructions on how much, and whether, to smooth the nodes.
One can think of a metric $n$-bracketing in a similar way.
This is not of obvious relevance to symplectic geometry, but nevertheless it provides some geometric intuition.
\null\hfill$\triangle$
\end{remark}

\begin{definition}[metric $n$-bracketings]
\label{def:metric_stable_tree-pair}
For $n \geq 1$, and $\cT$ a rooted plane tree of depth $n$, a \emph{metric $n$-bracketing} is a pair $(\sB, \ell_\sB)$ consisting of an $n$-bracketing $\sB \in \cK(\cT)$ and a function $\ell_\sB$ from the brackets of $\cT$ to $\bR$ satisfying the following requirements: 
\begin{enumerate}

\smallskip

\item\label{metricdefcond1}
For any $k$-bracket $A$ not in $\sB^k$, we have $\ell_\sB(A) = 0$.

\smallskip

\item\label{metricdefcond2}
For any singleton $k$-bracket $A(u) \in \sB^k$, we have $\ell_\sB(A(u)) = 0$.
 
\smallskip

\item\label{metricdefcond3}
For $A \in \sB^k$ a nontrivial bracket, we have $\ell_\sB(A) > 0$.

\smallskip

\item\label{metricdefcond4}
Fix a $k$-bracket $\wt A$ that is not a singleton.
Fix $(k+1)$-brackets $A, A' \in \sB_{\wt A}$ with $A$ containment-maximal in $\sB_{\wt A}$ and $A'$ containment-minimal in $\sB_{\wt A}$.
We require that the following equality holds:
\begin{align}
\label{eq:coherences}
\ell_\sB\bigl(\wt A\bigr)
=
\sum_{
\substack{
A'' \in \sB_{\wt A},
\\
A' \subseteq A'' \subseteq A
}
}
\ell_\sB(A'').
\end{align}
\end{enumerate}

When there is no ambiguity, we will abuse notation and denote a metric $n$-bracketing $(\sB,\ell_\sB)$ by $\ell_\sB$.  By convention, we declare that there is a unique metric $0$-bracketing consisting of the  unique $0$-bracketing and the function from this bracketing to $0$.
\null\hfill$\triangle$
\end{definition}

\begin{remark}
Note that by the (\textsc{nested}) condition required for $n$-bracketings, the brackets $A''$ appearing in \eqref{eq:coherences} form a chain.
That is, for any two such brackets, one is contained in the other.
\null\hfill$\triangle$
\end{remark}

\begin{definition}
Fix a rooted plane tree $\cT$ and an $n$-bracketing $\sB \in \cK(\cT)$.
We define  $\tau^\met(\sB)$ to be the set of all metric $n$-bracketings of the form $\ell_{\sB'}$ with $\sB' \leq \sB$.
\null\hfill$\triangle$
\end{definition}

\begin{lemma}\label{def:ops_on_metric_n-bracketings}
For $\sB \in \cK(\cT)$, the set $\tau^{\met}(\sB)$ is naturally equipped with the structure of a conical set.
\end{lemma}

\begin{proof}
We explain how the metric $n$-bracketings in $\tau^{\met}(\sB)$ form an additive monoid with a multiplication by nonnegative real numbers making $\tau^\met(\sB)$ an $\mathbb{R}_{\geq 0}$ semimodule, i.e.\ a conical set.  Suppose that $(\sB_1, \ell_{\sB_1})$ and $(\sB_2, \ell_{\sB_2})$ are metric $n$-bracketings such that $\sB_1, \sB_2 \leq \sB$, then we can define the metric $n$-bracketing
\begin{align}
(\sB_1, \ell_{\sB_1})+(\sB_2, \ell_{\sB_2}) \coloneqq (\sB_1 \vee \sB_2, \ell_{\sB_1}+\ell_{\sB_2}).
\end{align}
Let $(\sB,\ell_\sB)$ be a metric $n$-bracketing and $t$ a nonnegative real,
then $t(\sB,\ell_\sB)\coloneqq(\sB,t\ell_\sB)$ is a metric $n$-bracketing.
\end{proof}

\begin{definition}
\label{defprop:n-bracketing_complex}
Let
\begin{align}
X^\met(\cT)
\coloneqq
\{(\sB,\ell_{\sB}):\sB \in \cK(\cT)\}
\end{align}

The \emph{metric $n$-bracketing complex is}
\begin{align}
\cK^\met(\cT)=\{\tau^\met(\sB):\sB \in \cK(\cT) \}
\end{align}
 in the ambient space $X^\met(\cT)$.
\null\hfill$\triangle$
\end{definition}

\begin{figure}[ht]
\centering
\includegraphics[width=0.7\textwidth]{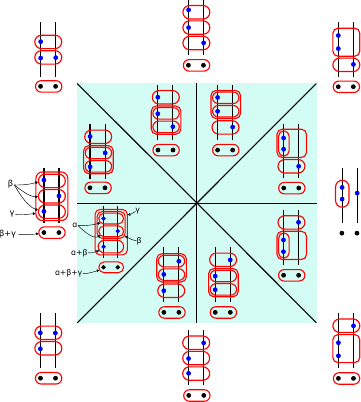}
\caption{
\label{fig:metric_2-bracketings_ex}
Here we depict the metric 2-bracketings of type $(2,1)$, which has 17 conical sets.
For two of these conical sets, we indicate the numbers associated to the brackets and the relations they must satisfy.
For the remaining 15 conical sets, we indicate only the underlying 2-bracketing.  This figure should not be taken too literally as this metric $n$-bracketing complex is not embedded in $\mathbb{R}^2$. 
}
\end{figure}

\subsection{Metric \texorpdfstring{$n$}{n}-bracketings as conical combinations of collisions}

\

There is a canonical map which associates a metric $n$-bracketing to each extended collision.

\begin{defprop}
\label{defprop:metric_bracketing_associated_to_collision}
Let $\sC \in {\wh\fX}(\cT)$.
We define the \emph{associated metric $n$-bracketing}
\begin{align}
\ell_\sC(A)
\coloneqq
\begin{cases}
1 &
\text{if } A \text{ is an essential bracket in $\sC$}, 
\\
0 &
\text{otherwise}.
\end{cases}
\end{align}
\null\hfill$\triangle$
\end{defprop}

It is straightforward to verify that $\ell_\sC$ is a metric $n$-bracketing.
\label{def:metric_n-bracketing_of_collision}
In the future we will typically denote $\ell_\sC$ by $\ell(\sC)$.
We now characterize the metric $n$-bracketings in $\tau^{\met}(\sB)$.

\begin{proposition}\label{conicalcomb}
Let $\sB \in \cK(\cT)$ and let  $\{\sC_0, \sC_1,\ldots,\sC_k\}\subseteq {\wh \fX}(\cT)$ such that $\sC_i \leq \sB$, and let $\sC_0 = \sB_{\min}$.
Then
\begin{align}
\tau^{\met}({\sB}) = \biggl\{ \,\sum_{i=0}^k \lambda_i\ell(\sC_i): \lambda_i \in \mathbb{R}, \, \text{and}\, \lambda_i\geq 0\, \text{for i}\,\neq 0 \,\biggr\}.
\end{align}
\end{proposition}

\begin{proof}
It is clear that $\tau^{\met}({\sB})$ contains the right hand side, so it remains to establish that $\tau^{\met}({\sB})$ is contained in the right hand side.
Fix a metric $n$-bracketing $(\sB',\ell_{\sB'})\in \tau^{\met}({\sB})$.
We must show that there exist $\lambda_0\ldots, \lambda_k$ as in the statement of the lemma so that $ \ell_{\sB'}=\sum_{i=0}^k \lambda_i\ell(\sC_i)$.
We prove this by induction on the number of nontrivial brackets in $\sB'$.
For $\sB'$ a collision or the minimum bracketing, this is trivial.
Next, suppose that $\sB'$ is neither a collision nor the minimum bracketing.
Choose a collision $\sC\leq \sB'$.
Set $t=\min\,\ell_{\sB'}(A)$ taken over all nontrivial brackets in the support of $\ell_\sC$.  Then we claim that there is a metric $n$-bracketing $\ell_{\sB''}$ with $\sB''<\sB'$ such that $\ell_{\sB''}+ t\,\ell_\sC=\ell_{
\sB'}$.  While conical sets do not admit subtraction, it does make sense to subtract the functions $\ell_{
\sB'}-t\,\ell_\sC$.  We can check that this function does satisfy the conditions of Definition \ref{def:metric_stable_tree-pair} (ignoring for a moment that the underlying object should be an $n$-bracketing).  It is then straightforward to check that the support $\ell_{
\sB'}-t\,\ell_\sC$ (together with the singleton and maximum brackets) determines a collection of brackets satisfying the conditions of an $n$-bracketing $\sB''$.
The {\sc (partition)} condition follows as a consequence of the fact that property (\ref{metricdefcond4}) from the definition of a metric $n$-bracketing is preserved.
We can equip $\sB''$ with the height partial order obtained by the restriction of the height partial order for $\sB'$.
\end{proof}

\begin{scholium}\label{inetermetricbracketings}
Define $\tau_{\mathbb{Z}}^{\met}(\sB)$ to be the metric $n$-bracketings in $\tau^{\met}(\sB)$ such that for each $A \in \sB^k$, we have $\ell_{\sB}(A) \in \mathbb{Z}$; these are the integer points of $\tau^{\met}(\sB)$.
It follows from the above argumentation that
\begin{align}
\tau_{\mathbb{Z}}^{\met}({\sB}) = \biggl\{ \sum_{i=0}^k \lambda_i\ell(\sC_i): \lambda_i \in {\mathbb{Z}}, \, \text{and}\, \lambda_i\geq 0\, \text{for i}\,\neq 0 \biggr\}.
\end{align}
\null\hfill$\triangle$
\end{scholium}
 
Scholium \ref{inetermetricbracketings} is extended by Scholium \ref{2ndinetermetricbracketings} and Corollary \ref{conecomplexlatticepropertyvelocityfan}, which provide a justification for our claim that while we ignore the standard lattice theoretic aspects of the definition of a cone complex, our cone complexes have implicit lattices and our maps respect these lattices. 

We observe another useful scholium of the proof of Proposition \ref{conicalcomb}.

\begin{scholium}
\label{sch:subtract}
Let $\ell_\sB$ be a metric $n$-bracketing with $\sB \neq \sB_{\min}$.
Let $\sC \in \fX(\cT)$ with $\sC \leq \sB$ and set $t_0=\min\,\ell_{\sB'}(A)$ taken over all nontrivial brackets in the support of $\ell_\sC$.  If $t \in \mathbb{R}_\geq 0$ with $t \leq t_0$, then $(\sB,\ell_\sB) - t(\sC,\ell_\sC)$  is a metric $n$-bracketing, and if $t=t_0$, the support of this metric $n$-bracketing is strictly bounded above by the support of $\ell_{\sB}$.
\null\hfill$\square$
\end{scholium}

\begin{figure}[ht]
\centering
\includegraphics[width=0.7\textwidth]{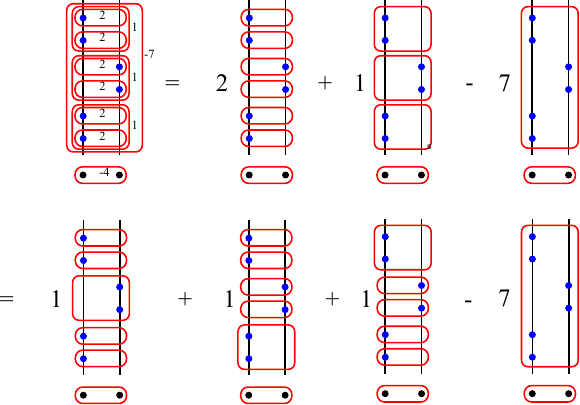}
\caption{
\label{sumofcollisionsfig}
A metric $n$-bracketing and two different expressions for this metric $n$-bracketing as a conical combination of metric $n$-bracketings associated to collisions.} 
\end{figure}

\begin{lemma}\label{fusionmetricbracketingdecomp}
Let $\sB \in \cK(\cT)$ and $\ell_{\sB} \in \cK^{\met}(\cT)$.
Let $\sC_0, \ldots, \sC_k \in \wh{\fX}(\cT)$ with $\sC_i\leq \sB$ for all $i$, and $\sC_0 = \sB_{\min}$.
Suppose that \begin{align}
\ell_{\sB} = \sum_{i=0}^k \lambda_i (\sC_i)
\end{align}
with $\lambda_i \in \mathbb{R}$ and $\lambda_i\geq 0$ for $i>0$.
Let $A$ be a fusion bracket for $\sB$.
Let $\ell_{\sB|_A}$ be the restriction of $\ell_{\sB}$ to the brackets which weakly project to $A$, and extended by 0 elsewhere.
Let $R(A) \subseteq \{\sC_i: 0\leq i\leq k\}$ be the subset of extended collisions which have $A$ as their fusion bracket, then 
\begin{align}
\ell_{\sB|_A} = \sum_{\sC_i \in R(A)}\lambda_i \ell(\sC_i).
\end{align}
In particular, $\ell_{\sB|_A}$ is a metric $n$-bracketing.  
\end{lemma}

\begin{proof}
First assume that $A$ is nontrivial.
Let $A' \in \sB^k$ be a (necessarily nontrivial) bracket which weakly projects to $A$.  Observe that for a particular $i$, we have $A' \in \sC_i$ if and only if $A$ is the fusion bracket for $\sC_i$, i.e. $\sC_i \in R(A)$.  The statement when $A$ is the maximum 1-bracket (the only possible trivial fusion bracket for a collision) and $A'$ is some maximum $k$-bracket is similar.  Note that in this case, $\sB_{\min} \in R(A)$.
\end{proof}

\begin{lemma}\label{metricpartitionlemma}
Let $\sB \in \cK(\cT)$ and let $U(\sB)$ be the collection of fusion brackets for $\sB$.
Let $\ell_{\sB}$ be a metric $n$-bracketing, then 
\begin{align}
\ell_{\sB} = \sum_{A \in U(\sB)} \ell_{\sB|_A}
\end{align}
\end{lemma}

\begin{proof}
This is a direct consequence of Proposition \ref{conicalcomb} and Lemma \ref{fusionmetricbracketingdecomp}.
\end{proof}
 
\begin{lemma}\label{equalityofcoefficientsums}
Let $\sB \in \cK(\cT)$ and let  $\{\sC_0, \sC_1,\ldots,\sC_k\}\subseteq {\wh \fX}(\cT)$ be the set of extended collisions such that $\sC_i \leq \sB$, and let $\sC_0 = \sB_{\min}$.  Let $\lambda_i, \gamma_i \in \mathbb{R}$ with $\lambda_i, \gamma_i \geq 0$ if $i>0$.  

Suppose that
\begin{align}
\ell_\sB = \sum_{i=0}^k \lambda_i \ell(\sC_i) = \sum_{i=0}^k \gamma_i \ell(\sC_i).
\end{align}
Then 
\begin{align}
\sum_{i=0}^k \lambda_i = \sum_{i=0}^k \gamma_i.
\end{align}
\end{lemma}

\begin{proof}
By Lemma \ref{metricpartitionlemma}, we can group the terms in both expressions for $\ell_{\sB}$ according to the fusion brackets of the collisions, and each sum must agree.
Therefore the verification of this lemma reduces to the case where $\sB$ has a single fusion bracket $A$.
If $A$ is an $n$-bracket then there is only a single  (type 1) collision having $A$ as its fusion bracket, and the statement is trivial.
Otherwise, we can project the sums, i.e.\ replace each term $\ell(\sC)$ with $\ell(\pi(\sC))$, and apply induction.
\end{proof}

\subsection{The metric \texorpdfstring{$n$}{n}-bracketings form an abstract fan}

\

\begin{lemma}\label{metricabstractfanlemma}
The metric $n$-bracketing complex $\cK^{\met}(\cT)$ is a conical set complex whose face poset is $\cK(\cT)$.
\end{lemma}

\begin{proof}
We must verify the conditions of Definition \ref{conecomdef}.  For establishing condition $(1)$, let $\sB \in \cK(\cT)$ and  
let $\tau$ be a face of $\tau^{\met}(\sB)$.  Let $\sC_1, \ldots, \sC_k \in \fX(\cT)$ be the collisions such that $\ell(\sC_i) \in \tau$ for all $i$.  We claim that $\tau = \tau(\sB')$ with $\sB' = \vee_{i=1}^k \sC_i$.  As $\tau$ is a conical set, we know that $\tau(\sB')\subseteq \tau$.  On the other hand if $\ell_{\sB''} \in \tau$ with ${\sB''} \nleq \sB'$, then by Lemma \ref{conicalcomb} there is an expression for $\ell_{\sB''}$ as a  conical combination of metric collisions.  There must be some $\ell(\sC)$ in this expression with $\sC \leq \sB''$ such that $\sC \nleq \sB'$.  By Definition \ref{facedefinition} for a face, $\ell(\sC) \in \tau$ contradicting the assumption that $\sC \neq \sC_i$ for any $i$.

For verifying condition $(2)$, we show the stronger statement that if $\tau^\met(\sB_1)$ and $\tau^\met(\sB_2)$ are elements of $\cK^\met(\cT)$, then $\tau^\met(\sB_1) \,\cap\, \tau^\met(\sB_2) = \tau^\met(\sB_1\,\wedge\,\sB_2)$.
Both inclusions $ \tau^\met(\sB_1 \,\wedge\, \sB_2) \subseteq \tau^\met(\sB_1) \,\cap\, \tau^\met(\sB_2)$ and $\tau^\met(\sB_1) \,\cap\, \tau^\met(\sB_2) \subseteq \tau^\met(\sB_1 \,\wedge\, \sB_2)$ are straightforward.
The face poset of $\cK^{\met}(\cT)$ is $\cK(\cT)$ by the definition of $\cK^{\met}(\cT)$.
\end{proof}

\begin{lemma}\label{metricweaklyembedded}
 The metric $n$-bracketing complex $\cK^{\met}(\cT)$ is a weakly embeddable cone complex.
\end{lemma}

\begin{proof}
We define a piecewise-linear map  $\eta:\cK^{\met}(\cT)\rightarrow \mathbb{R}^w$, where $w$ is the total number of nonsingleton brackets for $\cT$, such that the image of each conical set is a cone.  
Send $\ell_{\sB} \mapsto {\bf{v}_{\ell_{\sB}}}$ where the coordinate of ${\bf{v}_{\ell_{\sB}}}$ associated to the bracket $A$ has entry $\ell_{\sB}(A)$.  
It is clear that this map is piecewise-linear and nondegenerate on each $\tau^{\met}(\sB)$.
\end{proof}

\begin{remark}
The map $\eta$ forgets the height partial order, thus it may not be globally injective, and the image of $\eta$ may not be a fan.
We note that image is a fan if $\cK(\cT)$ is a concentrated $n$-associahedron.
It is clear from considerations of dimension that this fan is not complete.
For the associahedron, this provides a seemingly canonical fan realization which we are not aware of authors having explicitly considered previously, perhaps because it is incomplete.
\null\hfill$\triangle$
\end{remark}

\begin{proposition}\label{metrinbracketingabstractfanprop}
The metric $n$-bracketing complex $\cK^{\met}(\cT)$ is an abstract fan whose face poset is $\cK(\cT)$.
\end{proposition}

\begin{proof}
This follows from combining Lemma \ref{metricabstractfanlemma} with Lemma \ref{metricweaklyembedded}.
\end{proof}

\begin{remark}
Lemma \ref{metricweaklyembedded} is not an ingredient in establishing Theorem \ref{mainthmagain}; we only need Lemma \ref{metricabstractfanlemma}.
We have chosen to define an abstract fan as being a cone complex so that we may freely discuss quotients and localizations.
\end{remark}

\begin{definition}
\label{def:reduced_metric_n-bracketing_complex}
The \emph{reduced metric $n$-bracketing complex} is the quotient
\begin{align}
{\overline \cK}^\met(\cT)\coloneqq\cK^\met(\cT)/\langle\ell_{\sB_\min}\rangle.
\end{align}
\null\hfill$\triangle$
\end{definition}

\begin{definition}\label{c-metricdef}
Fix a collision $\sC$ in an $n$-associahedron $\cK(\cT)$.
We define the \emph{$\sC$-metric $n$-bracketing complex} to be the localization $\cK^{\met}(\cT)_{\tau^{\met}(\sC)}$, which we denote  $\cK^{\met}(\cT)_{\sC}$, and
we define the \emph{$\sC$-metric $n$-bracketings} to be the elements of $\cK^{\met}(\cT)_{\sC}$.
\null\hfill$\triangle$
\end{definition}

\begin{lemma}
\label{C-metriclemma}
The $\sC$-metric $n$-bracketings are naturally identified with pairs $(\sB, \ell_\sB)$ for $\sC \leq \sB$ with $\ell_{\sB}$ as described in Definition \ref{def:metric_stable_tree-pair} with condition (\ref{metricdefcond4}) relaxed to allow $\ell_{\sB}(A)<0$ if $A$ is an essential bracket in $\sC$. \end{lemma}

\begin{proof}
Let $(\sB,\ell_{\sB})$ be an  $\sC$-metric $n$-bracketing, then it is clear that $\ell_\sB$ is a function as described in Definition \ref{def:metric_stable_tree-pair} with the prescribed relaxation of condition (\ref{metricdefcond4}) because we have subtracted some multiple of $\ell_\sC$ from a metric $n$-bracketing.  Conversely, suppose that $\ell_\sB$ is a function as described in Definition \ref{def:metric_stable_tree-pair} with the prescribed relaxation of condition (\ref{metricdefcond4}).  Then we can add $\lambda\,\ell(\sC)$ for $\lambda >0$ large enough to obtain the new function $ \ell'_{\sB}$ which is a metric $n$-bracketing.
It follows that $\ell'_\sB \in \tau^{\met}(\sB)$ and thus $\ell_{\sB}$ is a $\sC$-metric $n$-bracketing.
\end{proof}

\section{The velocity fan}\label{velocityfansection}

In this section we will introduce a collection of polyhedral cones called the \emph{velocity fan} associated to an $n$-associahedron.
We establish Theorem \ref{mainthmagain}, the main result of this paper: the velocity fan associated to $\cK(\cT)$ is a complete polyhedral fan whose face poset is $\cK(\cT)$.

\subsection{Definition of the velocity fan}

\

We are now ready to define the velocity fan associated to a categorical $n$-associahedron.
Our formal definition is given using the input of a rooted plane tree.
Afterward, we will explain how the velocity fan can be understood very naturally via relative velocities of colliding affine coordinate spaces.
We emphasize that, despite its name, it is not at all clear that the velocity fan is in fact a fan, this is the content of Theorem \ref{mainvelocitystatement}.

Let $\cT$ be a rooted plane tree of depth $n\geq 1$ with $\cK(\cT)$ the corresponding $n$-associahedron.
We begin by recalling that there is a natural linear order $<_{\cT}$ on $V^k(\cT)$ induced by the rooted plane structure of $\cT$ (see Definitions \ref{rootedplanetreedef} and \ref{vertexorderdef}). 
We define the function $d_{\cT,k}:V^k(\cT) \rightarrow \mathbb{N}$ where for $u \in V^k(\cT)$, we have that $u$ is the $d_{\cT,k}(u)$-th vertex in $V^k(\cT)$ with respect to $<_{\cT}$ so that $d_{\cT,k}(u^k_i)=i$.  We simply write $d_{\cT}$ for $d_{\cT,k}$ if $k$ is clear from the context.
Let $m =|V(\cT)|-n$, and for $1\leq k \leq n$ set $t_k = |V^k(\cT)|-1$.\footnote{We note that our construction ignores each vertex in $\cT$ which is the unique vertex at its depth.}
Given ${\bf v} \in \mathbb{R}^m$, we denote its coordinates as
\label{def:vec_notation}
\begin{align}
{\bf v}
=
(v^1_1,\dots, v^1_{t_1},\dots,  v^n_1,\dots,v^n_{t_n}).
\end{align}
For defining the velocity fan, we first introduce a distinguished collection of ray generators associated to the collisions in $\cK(\cT)$.
Given $\sC \in \fX(\cT)$,
let
\begin{align}
\psi^\cT_{\cT/\sC}:V(\cT)\rightarrow V(\cT/\sC)
\end{align}
be the associated collision map as described in Definition \ref{def:collision_maps}, Propositions \ref{treemap} and \ref{prop:alpha_is_bijective}.
We may write $\psi_{\cT/\sC}$ if $\cT$ is clear from the context, or simply $\psi$ if both $\cT$ and $\cT/\sC$ are clear from the context.
\label{def:ray_generator}
We define a ray generator ${\rho}(\sC) \in \mathbb{R}^{m}$ having entries 
\begin{align}
{\rho}(\sC)^k_i
=
d_{\cT/\sC}(\psi(u^k_i))-d_{\cT/\sC}(\psi(u^k_{i+1}))+1.
\end{align}

Furthermore, we define
\label{def:all-ones}
\begin{align}
\rho(\sB_\min)={\bf 1}, 
\end{align}
where ${\bf 1}$ is the all-ones vector.
\label{def:cone_associated_to_n-bracketing}
Given $\sB \in \cK(\cT)$, we define the corresponding cone
\begin{align}
\tau(\sB) = \biggl\{ \,\,\sum_{\sC\leq \sB} \lambda_{\sC}\rho(\sC) +\lambda \rho(\sB_\min): \sC \in \fX(\cT) , \lambda_{\sC} \in \mathbb{R}_{\geq 0}, \lambda \in \mathbb{R}\,\biggr\}.
\end{align}

For $n=0$ and $\sB_{\min}$ the unique $0$-bracketing, we define $\tau(\sB_{\min})=\mathbb{R}^0$.

\begin{definition}
\label{def:velocity_fan}
Given a rooted plane tree $\cT$ of depth $n\geq 0$, we define the \emph{velocity fan} to be the collection of polyhedral cones
\begin{align}
\cF(\cT) = \{\tau(\sB):\sB \in \cK(\cT)\}.
\end{align}
\null\hfill$\triangle$
\end{definition}

\begin{remark}
Observe that that $\tau(\sB_\min) = \langle {\bf 1} \rangle_{\mathbb{R}}$.
Once $\cF(\cT)$ has been shown to be fan, it follows immediately that the lineality space of $\cF(\cT)$ is $\langle {\bf 1} \rangle_{\mathbb{R}}$.
Additionally, observe that
\begin{align}
{\tau}(\sC) = \{ \lambda {\rho}(\sC)+ \gamma \rho(\sB_\min): \lambda, \gamma \in \mathbb{R}, \lambda \geq 0\}.
\end{align}
\null\hfill$\triangle$
\end{remark}

\begin{figure}[ht]
\centering
\def\svgwidth{0.45\columnwidth}
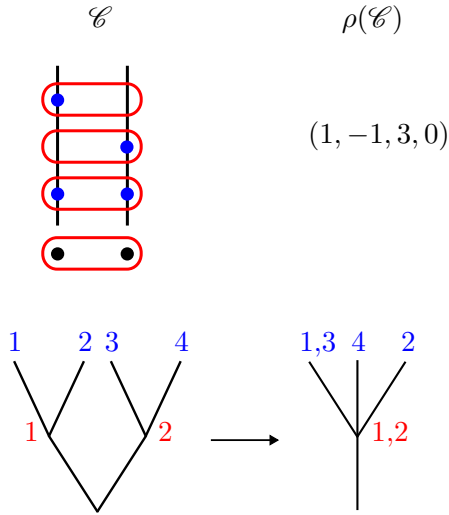
\caption{
\label{velocityrayex}
A collision $\sC$ and the corresponding ray $\rho(\sC)$.
}
\end{figure}

\begin{figure}[ht]
\centering
\includegraphics[width=0.75\columnwidth]{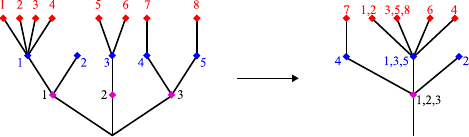}

\

\begin{align*}
\rho(\sC)=(1,1,0,2,2,0,1,0,-1,3,0,4,-2)
\end{align*}
\caption{The collision (map) from Figures \ref{3collisionfig} and \ref{collisionmapfig}, and the corresponding ray generator.}
\label{raygenfor3collfig}
\end{figure}

The following is the main result of this paper.

\begin{theorem}
\label{mainthmagain}
The velocity fan $\cF(\cT)$ is a complete fan whose face poset is $\cK(\cT)$.
\end{theorem}

\begin{proof}
We will establish Theorem \ref{mainthmagain} as a combination of two results: Theorem \ref{mainvelocitystatement} (\emph{the velocity fan is a fan}) and Theorem \ref {complete2} (\emph{the velocity fan is complete}).
Theorem \ref{mainvelocitystatement} will follow as a direct consequence of Proposition \ref{mainvelocitythm} (\emph{the velocity fan is piecewise-linearly isomorphic to the metric $n$-bracketing complex $\cK^{\met}(\cT)$}), combined with Lemma \ref{faceposetisolem} (\emph{The bijective piecewise-linear image of conical set complex is a conical set complex with the same face poset}) and Lemma \ref{metricabstractfanlemma} (\emph{$\cK^{\met}(\cT)$ is a conical set complex with face poset $\cK(\cT)$}).
\end{proof}

Thus Theorem \ref{mainthmagain} reduces to verifying Proposition \ref{mainvelocitythm} and Theorem \ref {complete2}.

\begin{definition}
\label{reducedvelocitydef}
We define the \emph{reduced velocity fan} to be the image of the velocity fan in tropical projective space
\begin{align}
\overline{\cF}(\cT)
=
\{\tau(\sB)/\langle {\bf 1} \rangle_{\mathbb{R}}:\sB \in \cK(\cT)\}.
\end{align} 
\null\hfill$\triangle$
\end{definition}

\begin{remark}\label{linealityspacermk}
As noted above, the velocity fan has lineality space $\langle {\bf 1} \rangle_{\mathbb{R}}$, hence it will follow as an immediate consequence of Theorem \ref{mainthmagain} that the reduced velocity fan is a pointed complete polyhedral fan.
Experts will note that the relationship between these two fans is a common and sometimes subtle distinction which arises in the literature on tropical geometry and generalized permutahedra.
Our choice to work with the velocity fan for much of the paper, rather than the reduced velocity fan, is due to two considerations.
The first, more superficial, reason is that some arguments are quite technical and we find it easier to work with explicit coordinates.
The second, more substantive, reason is that we will be interested in acting on subsets of the velocity fan by certain linear transformations (the adjoint permutation transformations $P^T_{\sigma}$\,) which do not have ${\bf 1}$ as an eigenvector, and we currently do not have a satisfactory way of discussing such maps in tropical projective space (see subsection \ref{permtuationssubsection}).
We feel that this second reason highlights a deeper aspect of our construction.
\null\hfill$\triangle$
\end{remark}

For considering the following proposition, we invite the reader to revisit Subsection \ref{Lodayfansubsection}.
\begin{proposition}\label{Lodayspecial}
The velocity fan of a categorical $1$-associahedron, i.e.\ an associahedron, is the wonderful associahedral fan, i.e. the normal fan of Loday's polytopal realization of the associahedron.
\end{proposition}

\begin{proof}
By Theorem \ref{lodayfanisafan}, know that the wonderful associahedral fan $\cF(\cK_n)$ has face poset given by the associahedron $\cK_n$, and $\langle {\bf 1} \rangle_{\mathbb{R}}$ as its lineality space.
Thus for verifying this statement it suffices, by the definition of the velocity fan, to observe that our ray generators specialize to the standard 0-1 ray generators for the wonderful associahedral fan.
\end{proof}

\begin{remark}
In our definition of the ray generator associated to a collision,
\begin{align}
{\rho}(\sC)^k_i
=
d_{\cT/\sC}(\psi(u^k_i))-d_{\cT/\sC}(\psi(u^k_{i+1}))+1,
\end{align}
the $1$ added at the end can be removed and this does not change the definition of the velocity fan.
This is because ${\bf 1}$ is in the lineality space of each cone and thus we could replace each ray generator with ${\rho}(\sC)-{\bf 1}$.  We have chosen not to do this for two reasons.
The first is that in Proposition \ref{Lodayspecial}, our choice of ${\rho}(\sC)$ specializes to the standard 0-1 ray generators for the wonderful associahedral fan.
More generally, as will be seen in \S \ref{s:concentrated_realization}, our velocity fan specializes to certain coarsenings of the braid arrangement where our choice of ray generators agrees with the standard 0-1 ray generators for the braid arrangement.
The second reason, explained in the following Remark \ref{whyremark}, is that there is a natural way of understanding our ray generators in terms of relative velocities of colliding affine coordinate spaces and our choice to leave the 1 is made to agree with this framework.
This all being said, there are some results for which ${\rho}(\sC)-{\bf 1}$ appears to be slightly nicer to work with, e.g.\ Lemma \ref{permutationtransformationlemma1}.
\null\hfill$\triangle$
\end{remark}

\begin{remark}\label{whyremark}{\bf The reason the velocity fan is called the velocity fan.}

Take $u \in V(\cT)$ and let ${\bf v}(u)$ be the vector encoding $u$ (see Definition \ref{vertexorderdef} and Figure \ref{nestedprop}).
We define a geometric arrangement of affine coordinate spaces.
Let $L(u) \subseteq \mathbb{R}^n$ be the affine coordinate space cut out by the equations $x_i = {\bf v}(u)_i$.
We define the arrangement $Y_\cT = \{ L(u): u \in V(\cT)\}$.
Let $\sC \in \fX(\cT)$.
Imagine the arrangement $Y_\cT$ moving to the arrangement $Y_{\cT/\sC}$ over a single unit of time, then ${\rho}(\sC)^k_i$ is equal to the relative velocity of $L(u^k_i)$ and $L(u^k_{i+1})$ in the direction $x_k$.  See Figure \ref{explicitcollisionfigure}.
\null\hfill$\triangle$
\end{remark}

\begin{figure}[ht]
\includegraphics[height=2.25in]{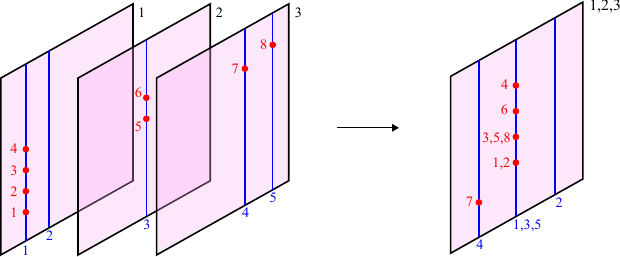}
\caption{A collision in a tree arrangement of spaces previously depicted in Figures  \ref{3collisionfig} and \ref{collisionmapfig}.  The ray generator $\rho(\sC)$ from Figure \ref{raygenfor3collfig} agrees with the relative velocities of the consecutive spaces as described in Remark \ref{whyremark}.}
\label{explicitcollisionfigure}
\end{figure}

We conclude this subsection with some fundamental lemmas for the velocity fan.
Let 
\begin{align}
{\bf v}
=
(v^1_1,\dots, v^1_{t_1},\dots, v^n_1,\dots,v^n_{t_n}) \in \mathbb{R}^m
\end{align}
and take
\begin{align}
\Delta^n_{j,k}({\bf v} ) \coloneqq \sum_{h=j}^{k-1} v^n_h.
\end{align}

\begin{lemma}
\label{fundvellemma}
Let $\sC, \sC' \in \wh \fX(\cT)$ and $1\leq j<k\leq t_n$, then

\begin{itemize}
\setlength{\itemindent}{-2.5em}
\item $\Delta^n_{j,k}({\rho}({\sC}))=k-j$ if and only if $u^n_j,u^n_k \in A$ for some essential $A \in {\sC}^n$,
\item $\Delta^n_{j,k}({\rho}({\sC}))>k-j$ if and only if $A(u^n_k)\,\,{\wh<_{\sC}}\,\,A(u^n_j)$,\footnote{Recall here, ${\wh<_{\sC}}$ is the extended height partial order for ${\sC}$.}
\item $\Delta^n_{j,k}({\rho}({\sC}))=0$ if $u^n_j,u^n_k$ are not descended from the root vertex of $\sC$.
\end{itemize}
\end{lemma}

\begin{proof}
We calculate that
\begin{align}
\Delta^n_{j,k}({\rho}({\sC}))
\end{align}
\begin{align}
=\sum_{i=j}^{k-1}d_{\cT/\sC}(\psi(u^n_i))-d_{\cT/\sC}(\psi(u^n_{i+1}))+1 \nonumber
\end{align}
\begin{align}
= d_{\cT/\sC}(\psi(u^n_j))-d_{\cT/\sC}(\psi(u^n_{k}))+(k-j).
\nonumber
\end{align}

Thus $\Delta^n_{j,k}({\rho}({\sC}))=k-j$ if and only if $d_{\cT/\sC}(\psi(u^n_j))=d_{\cT/\sC}(\psi(u^n_{k}))$, and $\Delta^n_{j,k}({\rho}({\sC}))>k-j$ if and only if $d_{\cT/\sC}(\psi(u^n_{k}))<d_{\cT/\sC}(\psi(u^n_j))$, which occurs if and only if $A(u^n_k)\,\,{\wh<_{\sC}}\,\,A(u^n_j)$ in the extended height partial order for $\sC$.
Finally if $u^n_j,u^n_k$ are not descended from the root vertex of $\sC$, then $d_{\cT/\sC}(\psi(u^n_j))-d_{\cT/\sC}(\psi(u^n_{k}))=j-k$.
\end{proof}

As remarked previously, we will use the notation $\pi$ to indicate several different projection maps. Given a rooted plane tree $\cT$ of depth $n$, we let $\pi(\cT)$ be the rooted plane tree obtained from $\cT$ by deleting the vertices $V^n(\cT)$.
\label{def:proj_of_n-associahedron}
Let $\pi(\cK(\cT))$ be the $(n-1)$-associahedron $\cK(\pi(\cT)$.
Given a point ${\bf v} = (v^1_1,\dots, v^1_{t_1},\dots, v^n_1,\dots,v^n_{t_n})$, we let $\pi({\bf v}) = ( v^1_1,\dots, v^1_{t_1},\dots, v^{n-1}_1,\dots,v^{n-1}_{t_{n-1}}).$
Given $Y\subseteq \mathbb{R}^n$ we set $\pi(Y) =\{ \pi{\bf v}: {\bf v} \in Y\}$.
Given a family $\cF$ of subsets of $\mathbb{R}^m$, we define the projection $\pi(\cF) = \{\pi(\tau):\tau \in \cF\}$. 

\begin{lemma}\label{rayprojection}
Suppose that $\sC \in \wh \fX(\cT)$.

\begin{itemize}
\item
If $\sC$ is type 2 or 3, then $\pi(\rho(\sC))= \rho(\pi(\sC))$.

\item
If $\sC$ is type 1, then 
$\pi(\rho(\sC))= {\bf 0}$ while $\rho(\pi(\sC)) = {\bf 1}.$
\end{itemize}

In particular, if $\{\sC_i:1\leq i\leq k\} \subseteq \wh \fX(\cT)$ is a collection of compatible extended collisions which are all not type 1, and ${\bf v} = \sum _{i =1}^k \lambda_i {\rho}(\sC_i)$ with $\lambda_i \in \mathbb{R}$ for $1\leq i \leq k$, then $\pi({\bf v}) = \sum _{i =1}^k \lambda_i {\rho}(\pi(\sC_i))$.
\end{lemma}

\begin{proof}
This is clear from the definition of $\rho(\sC)$.
\end{proof}

\begin{proposition}\label{projfanthing}
For any rooted plane tree $\cT$, we have $\pi(\cF(\cT))=\cF(\pi(\cT)). $
\end{proposition}

\begin{proof}
By definition
\begin{align}
\pi(\cF(\cT))= \{\pi(\tau(\sB)): \sB \in \cK(\cT)\},
\end{align}
and
\begin{align}\cF(\pi(\cT))=\{\tau(\sB):\sB \in \cK(\pi(\cT))\}.
\end{align}

Let $\sB \in \cK(\cT)$, then $\pi(\sB) \in \cK(\pi(\cT))$.
Applying Lemma \ref{rayprojection}, we observe that for any $\sC \in \wh \fX(\cT)$ with $\sC \leq \sB$

\begin{itemize}
\item
if $\sC$ is type 1, then $\pi(\rho(\sC))={\bf 0}$

\item
if $\sC$ is type 2, then $\pi(\rho(\sC))$ is a ray generator for $\tau(\pi(\sB))$, and

\item
if $\sC$ is type 3, then $\pi(\rho(\sC)) = {\bf 1}$ lies in the lineality space of $\tau(\pi(\sB))$.
\end{itemize}
By Scholium \ref{lem:collisionlift}, for every $\sC$ an extended collision with  $\sC\leq \pi(\sB)$ there exists some extended collision $\sC' \leq \sB$ such that $\pi(\sC') =\sC$.
So, the ray and lineality space generators of $\tau(\sB)$ surject onto the ray and lineality space generators for $\tau(\pi(\sB))$.
By linearity of the projection map, $\pi(\tau(\sB))=\tau(\pi(\sB))$.
Moreover, by Lemma \ref{bracketinglift}, for any $\sB \in \cK(\pi(\cT))$ there exists some $\sB' \in \cK(\cT)$ such that $\pi(\sB') = \sB$, and the desired statement follows.
\end{proof}

\subsection{Metric \texorpdfstring{$n$}{n}-bracketings and the velocity fan}

\

Our proof that the velocity fan $\cF(\cT)$ associated to an $n$-associahedron $\cK(\cT)$ is a  fan will proceed by induction on $n$.
Thus we will assume in what follows that every statement being proven for $n$-associahedra holds for all $(n-1)$-associahedra.
The base case of this induction is when $n=0$ and the only $0$-associahedron is the singleton poset with velocity fan a point.\footnote{The base case $n=0$ was chosen for aesthetic reasons: we would prefer to recover the special case of $n=1$ from our construction and arguments rather than build on it.}  
  
\begin{definition}\label{Gammamap}
We define a map $\Gamma: \cK^{\met}(\cT) \rightarrow \cF(\cT)$.  Let $\sC_0, \dots, \sC_k$ be extended collisions with $\sC_0 = \sB_\min$ such that $\sC_i \leq \sB$ for all $i$.
Let $\lambda_i \in \mathbb{R}$ with $\lambda_i\geq 0$ for $i>0$.
Given a metric $n$-bracketing $\ell_{\sB}$ such that $\ell_{\sB} = \sum _{i =0}^k \lambda_i \ell(\sC_i)$, we define
\begin{align}
\Gamma(\ell_{\sB}) \coloneqq \sum _{i =0}^k \lambda_i \rho(\sC_i).
\end{align}
\null\hfill$\triangle$
\end{definition}

Equivalently, for $\sC \in \wh \fX(\cT)$, take $\Gamma(\ell(\sC))=\rho(\sC)$,
then $\Gamma$ is obtained by extending this map  conically on the conical sets in $\cK^{\met}(\cT)$.  

\begin{proposition}\label{mainvelocitythm}
The map $\Gamma$ is a piecewise-linear isomorphism. 
\end{proposition}

\begin{remark}
Because the velocity fan is typically nonsimplicial, the fact that $\Gamma$ is an isomorphism when restricted to a single conical set in $\cK^{\met}$ is already a nontrivial statement.
\null\hfill$\triangle$
\end{remark} 

\begin{theorem}\label{mainvelocitystatement}
Given a rooted plane tree $\cT$, the velocity fan $\cF(\cT)$ is a fan whose face poset if $\cK(\cT)$.
\end{theorem}

\begin{proof}
Combine Lemma \ref{faceposetisolem}, Lemma \ref{metricabstractfanlemma}, and Proposition \ref{mainvelocitythm}.
\end{proof}

Thus for proving Theorem \ref{mainvelocitystatement}, it suffices to establish Proposition \ref{mainvelocitythm}. 
 Proposition \ref{mainvelocitythm} is a direct consequence of the following lemma.

\begin{lemma}\label{metriclemma}
Let $\sB$ and $\sB'$ be $n$-bracketings in $\cK(\cT)$.
Let $\sC_1, \dots, \sC_k$ and $\sC'_1 \dots \sC'_l$ be collisions less than $\sB$ and $\sB'$, respectively.  Let $\sC_0 = \sC'_0 = \sB_\min$.
There exist $\lambda_i, \gamma_i \in \mathbb{R}_{\geq 0}$ for $i>0$, and $\lambda_0, \gamma_0 \in \bR$ such that
\begin{align}
\label{vecequal}
\sum _{i =0}^k \lambda_i \ell(\sC_i) = \sum _{i =0}^l \gamma_i \ell(\sC'_i)
\end{align}
if and only if
\begin{align}\label{brackequal}
\sum _{i=0}^k \lambda_i {\rho}(\sC_i) = \sum _{i =0}^l \gamma_i \rho(\sC'_i).
\end{align}
\end{lemma}

\begin{proof}
(\ref{mainvelocitythm})
It is clear that the map $\Gamma$, if it is well-defined, is surjective.
Furthermore, by Lemma \ref{metricabstractfanlemma} we know that $\cK^{\met}(\cT)$ is a conical set complex whose face poset is $\cK(\cT)$, hence $\Gamma$ agrees on the intersection of conical sets.  What remains to be seen is that $\Gamma$ is well-defined and injective.
We observe that the well-definedness of $\Gamma$ follows from $(\ref{vecequal}) \Rightarrow (\ref{brackequal})$ in Lemma \ref{metriclemma}.
The injectivity of $\Gamma$ follows from $(\ref{brackequal}) \Rightarrow (\ref{vecequal})$ in Lemma \ref{metriclemma}.
\end{proof}

\begin{corollary}
\label{reducedgammaiso}
Let ${\overline \Gamma}: {\overline \cK}^{\met}(\cT) \rightarrow {\overline \cF}(\cT)$ be the map induced by $\Gamma$.  It follows from Lemma \ref{reducedisomorphismlemma} and Proposition \ref{mainvelocitythm} that ${\overline \Gamma}$ is an isomorphism.
\end{corollary}

\subsection{\texorpdfstring{$\cT$}{T}-shuffles and adjoint permutation transformations}\label{permtuationssubsection}

\

The vertices of $\cT$ can be permuted during a collision.
For a 2-associahedron associated to a pair of lines, the permutations of the points which can occur during a collision are precisely the riffle shuffle permutations.
Motivated by this special case, we define a \emph{$\cT$-shuffle} to a be a permutation of the of vertices of $\cT$ which can occur during a collision.  More precisely,

\begin{definition}
Let $\cT$ be a rooted plane tree, and let 
\begin{align}
S_{\cT} \coloneqq \bigoplus_{i=1}^n S_{|V^k(\cT)|}
\end{align}
be a direct sum of permutation groups. 
Let $\sigma \in S_{\cT}$ with $\sigma = (\sigma^1, \ldots, \sigma^n)$.  
We say that $\sigma$ is a \emph{$\cT$-shuffle} if $\sigma^k(j)<\sigma^k(i)$ and $i<j$, then $\pi(u^k_{\sigma(j)}) \neq \pi(u^k_{\sigma(i)})$.  For simplicity, we may refer to $\cT$-shuffles as shuffles.
We refer to the process of shuffles acting on labels of vertices of $\cT$ as \emph{shuffling}.
\null\hfill$\triangle$
\end{definition}

\begin{remark}
The above definition is geometrically reasonable: if we have initial affine spaces $L^k_i$ and $L^k_j$ in a tree arrangement which are being permuted during a collision, they should not lie on a common affine space $L$ of dimension 1 greater, e.g.\ points on a line should not be permuted.
\null\hfill$\triangle$
\end{remark}

In this subsection, we will define and investigate certain linear transformations encoding permutations for understanding how shuffling is geometrically manifested in the velocity fan.

\begin{definition}
Let $n\in \mathbb{N}$.
Let $e_i$ for $0\leq i\leq n$ be the standard basis vectors for $\mathbb{R}^{n+1}$.
The \emph{type A-root system} is the set of vectors $A_n=\{e_i-e_j:0\leq i,j \leq n\}$\footnote{This notation conflicts with that of a bracket with subscript $n$, but the root system $A_n$ will be discussed sparingly, and no confusion should arise.}
The \emph{positive simple roots} are the elements of $A_n$ of the form $e_{i+1}-e_i$ for $0\leq i\leq n-1$.
The \emph{type-$A_{n}$ root lattice} $\mathbb{A}_n$ is the integer span of the type-$A_n$ roots, equivalently $\mathbb{A}_n= \{x \in \mathbb{Z}^{n+1}: \sum_{i=0}^{n}
x_i = 0\}$.
Let $\mathbb{A}^{\mathbb{R}}_n= \{x \in \mathbb{R}^{n+1}: \sum_{i=0}^{n}
x_i = 0\}=\mathbb{A}_n \otimes \mathbb{R}$.
\null\hfill$\triangle$
\end{definition}

\begin{definition}
\label{permtransdef}
Let $S_n$ be the $n$th permutation group.
Let $\sigma \in S_{n+1}$ be a permutation.
Let $R_{\sigma}:\mathbb{R}^{n+1} \rightarrow \mathbb{R}^{n+1}$ be the linear transformation determined by $R_{\sigma}(e_{i})=e_{\sigma(i)} $.
We define $P_\sigma:\mathbb{R}^{n} \rightarrow \mathbb{R}^{n}$ to be the restriction of $R_{\sigma}$ to $\mathbb{A}^{\mathbb{R}}_n$ written in the basis of the positive simple roots.
We call $P_{\sigma}$ a \emph{permutation transformation}.
We define the transpose $P_{\sigma}^T$ to be an \emph{adjoint permutation transformation.}\footnote{Explicit descriptions of permutation transformation matrices (not only for transpositions) are provided in the thesis of Stucky \cite{stucky2021cyclic}.}
\null\hfill$\triangle$
\end{definition}

We note that $R_\sigma^T = R_\sigma^{-1}=R_{\sigma^{-1}}$. As a result, there is a lack of agreement in the literature about the definition of $R_\sigma$, which is sometimes taken to mean the transpose of what we describe here.
In this paper, we will be primarily interested in the adjoint permutation transformations.
We emphasize that $P_\sigma^T \neq P_\sigma^{-1}$.
Moreover, $P_\sigma^T$ is not itself a permutation transformation.

Observe that 

\begin{align}
R_{\sigma}(e_{i+1}-e_i)=e_{\sigma(i+1)}-e_{\sigma(i)}.
\end{align}

Utilizing the elementary fact 

\begin{equation}
e_{\sigma(i+1)}-e_{\sigma(i)}=
\begin{cases}
\,\, \sum_{j=\sigma(i)}^{\sigma(i+1)-1}e_j-e_{j+1} & \text{if} \,\, \sigma(i)<\sigma(i+1) \\
\,\,\sum_{j=\sigma(i+1)}^{\sigma(i)-1}e_{j+1}-e_{j} & \text{if} \,\, \sigma(i)>  \sigma(i+1),
\end{cases}
\end{equation}
we may write the map $R_{\sigma}$ in the basis of the positive simple roots to obtain
\begin{align}
(P_{\sigma})_{i,j}=
\begin{cases}
\,\,\,\, 1 & \text{if} \,\, \sigma(j)\leq i<\sigma(j+1) \\
-1 & \text{if} \,\, \sigma(j)> i\geq \sigma(j+1) \\
\,\,\,\, 0 & \text{otherwise.}
\end{cases}
\end{align}

This implies that the transpose can be written as
\begin{align}
(P^T_{\sigma})_{i,j}=
\begin{cases}
\,\,\,\, 1 & \text{if} \,\, \sigma(i)\leq j<\sigma(i+1) \\
-1 & \text{if} \,\, \sigma(i)> j\geq \sigma(i+1) \\
\,\,\,\, 0 & \text{otherwise. } \\
\end{cases}
\end{align}

\begin{example}
An adjoint permutation transformation associated to $\sigma = 3142$:
\begin{align}
P_{3142}^T
=
\left(
\begin{array}{rrr}
-1 & -1 & 0
\\
1 & 1 & 1
\\
0 & -1 & -1
\end{array}
\right).
\end{align}
\null\hfill$\triangle$
\end{example}

The following lemma describes the essential relevance of the adjoint permutation transformations in this article.

\begin{lemma}\label{relevantadjoint}
Suppose that ${\bf v} \in \mathbb{R}^{n+1}$ and let ${\overline{\bf v}} \in \mathbb{R}^n$ be the vector such that ${\overline{\bf v}}_i=v_{i}-v_{i+1}$.
Then for each $\sigma \in S_{n+1}$,
\begin{align}
\label{eq:action_of_P^T_on_differences}
P^T_{\sigma}({\overline{\bf v}})_i=v_{\sigma(i)}-v_{\sigma(i+1)}.
\end{align}

Additionally, 
\begin{align}
P^T_{\sigma}({\bf 1})_i=\sigma(i+1)-\sigma(i).
\end{align}
\end{lemma}

\begin{proof}
The image of a vector ${\overline{\bf v}} \in \mathbb{R}^n$ under $P^T_{\sigma}$ is as follows:
\begin{equation}
P_{\sigma}^T({\overline{\bf v}})_i=
\begin{cases}
\,\, \sum_{j=\sigma(i)}^{\sigma(i+1)-1}{\overline{\bf v}}_j & \text{if} \,\, \sigma(i)<\sigma(i+1) \\
\,\,\sum_{j=\sigma(i+1)}^{\sigma(i)-1}-{\overline{\bf v}}_j & \text{if} \,\, \sigma(i)>\sigma(i+1) \\ 
\end{cases}
\end{equation}

\begin{align}
P_{\sigma}^T({\overline{\bf v}})_i=
\begin{cases}
\,\, \sum_{j=\sigma(i)}^{\sigma(i+1)-1}v_j-v_{j+1} & \text{if} \,\, \sigma(i)<\sigma(i+1) \\
\,\,\sum_{j=\sigma(i+1)}^{\sigma(i)-1}v_{j+1}-v_j & \text{if} \,\, \sigma(i)>  \sigma(i+1) \\
\end{cases}
\end{align}
\begin{align}
=v_{\sigma(i)}-v_{\sigma(i+1)}.
\end{align}

The verification of $P^T_{\sigma}({\bf 1})_i=\sigma(i+1)-\sigma(i)$ is similar.
\end{proof}

Observe that given a collision $\sC \in \fX(\cT)$, the vector $\rho(\sC)-{\bf 1}$ is the concatenation of $n$ vectors of the form ${\overline{\bf v}}$ described in Lemma \ref{relevantadjoint}.  Various unimodularity, i.e.\ toric smoothness, results in this paper are built on the following fact.
\begin{lemma}\label{unimodperm}
For each $\sigma\in S_{n+1}$, $P_\sigma \in GL_n(\mathbb{Z})$.

\end{lemma}

\begin{proof}
Let $\sigma_1, \sigma_2 \in S_{n+1}$, then one can check that $P_{\sigma_1}\circ P_{\sigma_2} = P_{\sigma_1\,\circ \, \sigma_2}$.  
Therefore $P_{\sigma}\circ P_{\sigma^{-1}}=I_{n}$, so $P_{\sigma}^{-1} = P_{\sigma^{-1}}$, and $P_{\sigma} \in GL_n(\mathbb{Z})$.  
\end{proof}

\begin{remark}
The \emph{consecutive-ones property} for these matrices (ignoring signs) implies the stronger statement that $P_{\sigma}$ is totally unimodular.
\null\hfill$\triangle$
\end{remark}

\begin{definition}
Let $\sC$ be a collision and $\sigma=(\sigma^1,\ldots, \sigma^n)$ be a $\cT$-shuffle.  We say that $\sigma$ is a \emph{compatible $\cT$-shuffle for $\sC$}, or simply a \emph{compatible shuffle}, if the following  holds: let 
$\psi^{\cT}_{\cT/\sC}$ be the collision map from Def.\ \ref{def:collision_maps} (see also Propositions \ref{treemap} and \ref{prop:alpha_is_bijective}).
Then either
\begin{itemize}
\item $\psi^{\cT}_{\cT/\sC}(u^k_{\sigma^k(i+1)})=\psi^{\cT}_{\cT/\sC}(u^k_{\sigma^k(i)})$ or

\smallskip

\item $\psi^{\cT}_{\cT/\sC}(u^k_{\sigma^k(i+1)})$ is the successor of $\psi^{\cT}_{\cT/\sC}(u^k_{\sigma^k(i)})$ in $\cT/\sC$.
\end{itemize}
We adopt the convention that any $\cT$-shuffle is a compatible shuffle for $\sB_{\min}$.
\end{definition}

In the future, if $k$ is clear from the context, we may denote $\sigma^k(i)$ by $\sigma(i)$.

It is easy to construct a compatible shuffle $\sigma$ for a collision $\sC$: take the order $<_{\cT/\sC}$ on $V^k(\cT/\sC)$ and extend this to a total order $\sigma_k$ of $V^k(\cT)$ by ordering the elements in the fibers of the map $\psi^{\cT}_{\cT/\sC}:\mathcal{T}\rightarrow \mathcal{T}/\sC$ so that the shuffle condition is preserved.

\begin{definition}
Given $\sigma \in S_{\cT}$, we define $P_{\sigma}\coloneqq\bigoplus_{i=0}^n P_{\sigma^k}$.
\null\hfill$\triangle$
\end{definition}

\begin{example}
\label{adjointpermtransformationex2}
An adjoint permutation transformation associated to a compatible shuffle for the collision in Figure \ref{velocityrayex}:
\begin{align}
P_{(12,3142)}^T
=
\left(
\begin{array}{rrrr}
1 & 0 & 0 & 0
\\
0 & -1 & -1 & 0
\\
0 & 1 & 1 & 1
\\
0 & 0 & -1 & -1
\end{array}
\right).
\end{align}
\null\hfill$\triangle$
\end{example}

The following Lemmas \ref{compatibleidentity} and \ref{inclusionpermutation} are clear from inspection.

\begin{lemma}\label{compatibleidentity}
Let $\sC \in \fX(\cT)$ and let $\sigma$ be a compatible shuffle for $\sC$.  Then for any vertex $u^k_i \in V^k(\cT)$ which is not descended from a vertex in the fusion bracket of $\sC$, we have that $u^k_{\sigma^k(i)} = u^k_i$.
\end{lemma}

\begin{lemma}\label{inclusionpermutation}
Let $\sC$ and $\sC'$ be collisions such that $\sC$ and $\sC'$ have the same fusion bracket and $\sC \rightarrow \sC'$.  Suppose $\sigma$ is a compatible shuffle for $\sC$, then $\sigma$ is a compatible shuffle for $\sC'$.
\end{lemma}

The following lemma will be utilized in the proof of Theorem \ref{mainvelocitystatement}.

\begin{lemma}\label{permutationtransformationlemma1}
Suppose that $\sC$ is an extended collision and $\sigma$ is a compatible shuffle for $\sC$.
Then $\sigma({\rho}(\sC))\coloneqq P_{\sigma}^T(({\rho}(\sC)-\bf{1})+\bf{1}$ is a 0-1 vector, and the entry $\sigma({\rho}(\sC))^k_i=1$ if and only if $u^k_{\sigma(i)}$ and $u^k_{\sigma(i+1)}$ belong to a common essential bracket in $\sC$.
\end{lemma}

\begin{proof}

We first give a formal proof and then outline how we are thinking about things in terms of collisions of affine coordinate spaces.
Recall that
\begin{align}
{\rho}(\sC)^k_i
=
d_{\cT/\sC}(\psi^{\cT}_{\cT/\sC}(u^k_i))-d_{\cT/\sC}(\psi^{\cT}_{\cT/\sC}(u^k_{i+1}))+1.
\end{align}

Therefore \begin{align}
({\rho}(\sC)-{\bf 1})^k_i
=
d_{\cT/\sC}(\psi^{\cT}_{\cT/\sC}(u^k_{i}))-d_{\cT/\sC}(\psi^{\cT}_{\cT/\sC}(u^k_{i+1})).
\end{align}
By Lemma \ref{relevantadjoint}, we know that
\begin{align}
P_\sigma^T({\rho}(\sC)-{\bf 1})^k_i=d_{\cT/\sC}(\psi^{\cT}_{\cT/\sC}(u^k_{\sigma(i)}))-d_{\cT/\sC}(\psi^{\cT}_{\cT/\sC}(u^k_{\sigma(i+1)})),
\end{align}

hence

\begin{align}(P_\sigma^T({\rho}(\sC)-{\bf 1})+{\bf1})^k_i=d_{\cT/\sC}(\psi^{\cT}_{\cT/\sC}(u^k_{\sigma(i)}))-d_{\cT/\sC}(\psi^{\cT}_{\cT/\sC}(u^k_{\sigma(i+1)}))+1.
\end{align}

\

Because $\sigma$ is a compatible shuffle for $\sC$, for every $1\leq i\leq t_k$ we have that 
\begin{align}d_{\cT/\sC}(\psi^{\cT}_{\cT/\sC}(u^k_{\sigma(i)}))=d_{\cT/\sC}(\psi^{\cT}_{\cT/\sC}(u^k_{\sigma(i+1)}))
\end{align}

or

\begin{align}d_{\cT/\sC}(\psi^{\cT}_{\cT/\sC}(u^k_{\sigma(i)}))=d_{\cT/\sC}(\psi^{\cT}_{\cT/\sC}(u^k_{\sigma(i+1)}))-1,
\end{align}

and the former case occurs if and only if $u^k_{\sigma(i)}$ and $u^k_{\sigma(i+1)}$ are contained in a common essential $k$-bracket in $\sC$. 
\end{proof}

Recall that an entry in ${\rho}(\sC)$ encodes the relative change of distance between a consecutive pair of initial coordinate spaces during a collision.
When we take the difference, ${\rho}(\sC)-\bf{1}$, this vector encodes the relative change in distance between a consecutive pair of initial coordinate spaces starting at the arrangement where all initial spaces have collapsed to a single flag of linear coordinate spaces, and ending at the configuration associated to  $\sC$.
When we apply the transformation $P_\sigma^T$ we find that an entry in $P_\sigma^T({\rho}(\sC)-\bf{1})$ similarly encodes the relative change in distance between a consecutive pair of initial coordinate spaces starting at the arrangement where all spaces have collapsed to a single flag and ending at the configuration associated to $\sC$, but now initial coordinate spaces are taken to be consecutive with respect to $\sigma$.
When we add the vector $\bf{1}$ we find that an entry in $P_\sigma^T({\rho}(\sC)-\bf{1})+\bf{1}$ encodes the relative change in distance between a consecutive pair of initial coordinate spaces starting at the arrangement where all spaces have been placed at positions according to $\sigma$ and pass to the configuration associated to  $\sC$, and consecutivity is again taken with respect to $\sigma$.
By our choice of $\sigma$, consecutive initial coordinate spaces are spaced one unit apart and either collide or remain one unit apart, establishing the lemma.  See Figure \ref{easytransformfig} for an illustration.

\begin{figure}%[H]
\centering
\def\svgwidth{0.75\textwidth}
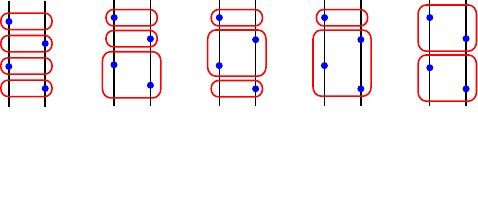
\vspace{22pt}
\begin{align}
P^T_{(12,3142)}\left(\left(\begin{array}{r}
1 \\
-1 \\
3 \\
0
\end{array}\right) - \left(\begin{array}{r}
1 \\
1 \\
1 \\
1
\end{array}\right)\right)
+
\left(\begin{array}{r}
1 \\
1 \\
1 \\
1
\end{array}\right)
\end{align}
\begin{align}
%\hspace{0.75in}
\,\,\,\,\,\,=\left(\begin{array}{rrrr}
1 & 0 & 0 & 0 \\
0 & -1 & -1 & 0 \\
0 & 1 & 1 & 1 \\
0 & 0 & -1 & -1
\end{array}\right)\left(\begin{array}{r}
0 \\
-2 \\
2 \\
-1
\end{array}\right)
+
\left(\begin{array}{r}
1 \\
1 \\
1 \\
1
\end{array}\right)
=
\left(\begin{array}{c}
1 \\
1 \\
0 \\
0
\end{array}\right).\nonumber
\end{align}
\caption{
We illustrate the map ${\rho} \mapsto P_{\sigma}^T({\rho}-{\bf 1})+{\bf 1}$ from Lemma \ref{permutationtransformationlemma1} using the transformation $P^T_{\sigma}$ from Example \ref{adjointpermtransformationex2}, and we explicitly calculate this map for the collision second from the left on the top row.
Observe that the points are placed at heights according to the permutation $3142$.
}\label{easytransformfig}
\end{figure}

The following technical Lemma \ref{contractionlemma} and its strengthening Lemma \ref{refinedcontractionlemma} are essential for understanding the recursive structure of the velocity fan, but will not be utilized until \S\ref{triangulationsection}, after the velocity fan has been proven to be a fan which is complete.
So, if the reader is primarily interested in understanding the proof of this main theorem, the following lemma can be skipped on an initial read.

\begin{lemma}
\label{contractionlemma}
Let $\sC \in \fX(\cT)$ and $\sC' \in \wh \fX(\cT)$. Take $\sigma$ a compatible shuffle for $\sC$.
Let $u^k_j \in V^k(\cT/\sC)$ and $u^k_{\sigma(i)} \in V^k(\cT)$ be such that $\psi^\cT_{\cT/\sC}(u^k_{\sigma(i)})=u^k_j$.  

If $\sC \rightarrow \sC'$ and $\sC/\sC' \in \fX(\cT/\sC)$, then
\begin{align}
(P_{\sigma}^T({\rho}(\sC')-{\rho}(\sC)))^k_{i}=
\begin{cases}
{\rho}(\sC'/\sC)^k_j & \text{if} \,\, i = {\rm max} \{h: \psi^\cT_{\cT/\sC}(u^k_{\sigma(h)}) = u^k_j\}  \\
\,\, 0& \text{otherwise,} \\
\end{cases}
\end{align}
and if $\sC \sim \sC'$ then
\begin{align}
\,\,\,\,\,\,\,\,\,\,\,\,\,\,\,\,\,\,\,\,\,(P_{\sigma}^T({\rho}(\sC')))^k_{i}=
\begin{cases}
{\rho}(\sC'/\sC)^k_j & \text{if} \,\, i = {\rm max} \{h: \psi^\cT_{\cT/\sC}(u^k_{\sigma(h)}) = u^k_j\}  \\
\,\, 0 & \text{otherwise.} \\
\end{cases}
\end{align}
\end{lemma}

\begin{proof}
Suppose that $\sC \rightarrow \sC'$ and $\sC/\sC' \in \fX(\cT/\sC)$.
We calculate $(P_{\sigma}^T({\rho}(\sC')-{\rho}(\sC)))^k_{i}.$  Recall that

\begin{align}
{\rho}(\sC)^k_i
=
d_{\cT/\sC}(\psi^\cT_{\cT/\sC}(u^k_i))-d_{\cT/\sC}(\psi^\cT_{\cT/\sC}(u^k_{i+1}))+1.
\end{align}
Therefore 

\begin{align}
({\rho}(\sC')-{\rho}(\sC))^k_i=
\end{align}

\begin{align}
 d_{\cT/\sC'}(\psi^\cT_{\cT/\sC'}(u^k_{i}))-d_{\cT/\sC'}(\psi^\cT_{\cT/\sC'}(u^k_{i+1}))-(d_{\cT/\sC}(\psi^\cT_{\cT/\sC}(u^k_{i}))-d_{\cT/\sC}(\psi^\cT_{\cT/\sC}(u^k_{i+1}))),  \nonumber
\end{align}

\begin{align}
{\rm \, hence\,\, by \,\, Lemma \,\, \ref{relevantadjoint}}\,\,(P_\sigma^T({\rho}(\sC')-{\rho}(\sC)))^k_i=
\end{align}

\begin{align}
d_{\cT/\sC'}(\psi^\cT_{\cT/\sC'}(u^k_{\sigma(i)}))-d_{\cT/\sC'}(\psi^\cT_{\cT/\sC'}(u^k_{\sigma(i+1)}))\\
-(d_{\cT/\sC}(\psi^\cT_{\cT/\sC}(u^k_{\sigma(i)}))-d_{\cT/\sC}(\psi^\cT_{\cT/\sC}(u^k_{\sigma(i+1)}))).\nonumber  
\end{align}

If $i \neq {\rm max} \{h: \psi^\cT_{\cT/\sC}(u^k_{\sigma(h)}) = u^k_j\}$, then $\psi^\cT_{\cT/\sC}(u^k_{\sigma(i)})=\psi^\cT_{\cT/\sC}(u^k_{\sigma(i+1)}))$ and $\psi^\cT_{\cT/\sC'}(u^k_{\sigma(i)})=\psi^\cT_{\cT/\sC'}(u^k_{\sigma(i+1)})$, thus $P_\sigma^T({\rho}(\sC')-{\rho}(\sC))^k_i=0$.
On the other hand, suppose that $i = {\rm max} \{h: \psi^\cT_{\cT/\sC}(u^k_{\sigma(h)}) = u^k_j\}$.
We would like to prove that 
\begin{align}
d_{\cT/\sC'}(\psi^\cT_{\cT/\sC'}(u^k_{\sigma(i)}))-d_{\cT/\sC'}(\psi^\cT_{\cT/\sC'}(u^k_{\sigma(i+1)}))\\
-(d_{\cT/\sC}(\psi^\cT_{\cT/\sC}(u^k_{\sigma(i)}))-d_{\cT/\sC}(\psi^\cT_{\cT/\sC}(u^k_{\sigma(i+1)})))\nonumber
\end{align}
\begin{align}
= {\rho}(\sC'/\sC)^k_j
\nonumber
\end{align}
\begin{align}
=
d_{\cT/\sC'}(\psi^{\cT/\sC}_{(\cT/\sC)/(\sC'/\sC)}(u^k_j))-d_{\cT/\sC'}((\psi^{\cT/\sC}_{(\cT/\sC)/(\sC'/\sC)}(u^k_{j+1}))+1.
\nonumber
\end{align}

By construction, 
\begin{align}
d_{\cT/\sC}(\psi^\cT_{\cT/\sC}(u^k_{\sigma(i)}))-d_{\cT/\sC}(\psi^\cT_{\cT/\sC}(u^k_{\sigma(i+1)}))=-1 
\end{align}

 and it remains to check that 

\begin{align}
d_{\cT/\sC'}(\psi^\cT_{\cT/\sC'}(u^k_{\sigma(i)}))-d_{\cT/\sC'}(\psi^\cT_{\cT/\sC'}(u^k_{\sigma(i+1)})) 
\end{align}

\begin{align}
=d_{\cT/\sC'}(\psi^{\cT/\sC}_{(\cT/\sC)/(\sC'/\sC)}(u^k_j))-d_{\cT/\sC'}(\psi^{\cT/\sC}_{(\cT/\sC)/(\sC'/\sC)}(u^k_{j+1})).  
\end{align}

Because $\sC \rightarrow \sC'$, the map of rooted plane trees $\psi^\cT_{\cT/\sC'}$ factors as 
\begin{align}
\psi^\cT_{\cT/\sC'}=\psi^{\cT/\sC}_{(\cT/\sC)/(\sC'/\sC)}\circ\psi^\cT_{\cT/\sC}  
\end{align}
and the desired equality follows.

We next assume that $\sC \sim \sC'$.
In verifying the corresponding part of the statement of the lemma we employ (admittedly annoying) case analysis.  
Recall 
\begin{align}
{\rho}(\sC')^k_i
=
d_{\cT/\sC'}(\psi^\cT_{\cT/\sC'}(u^k_i))-d_{\cT/\sC'}(\psi^\cT_{\cT/\sC'}(u^k_{i+1}))+1
\end{align}

\begin{align}
{\rm \, hence\,\, by \,\, Lemma \,\, \ref{relevantadjoint}},\,\,(P_\sigma^T({\rho}(\sC'))^k_i=  
\end{align}

\begin{align}
d_{\cT/\sC'}(\psi^\cT_{\cT/\sC'}(u^k_{\sigma(i)}))-d_{\cT/\sC'}(\psi^\cT_{\cT/\sC'}(u^k_{\sigma(i+1)}))+\sigma(i+1)-\sigma(i).
\end{align}

First, we note that $u^k_i$, respectively $u^k_{i+1}$, is descended from a vertex in the fusion bracket for $\sC$ if and only if $u^k_{\sigma(i)}$, respectively $u^k_{\sigma(i+1)}$, is descended from a vertex in the fusion bracket for $\sC$.

 Suppose $i \neq {\rm max} \{h: \psi^\cT_{\cT/\sC}(u^k_{\sigma(h)}) = u^k_j\}$, it must be that $u^k_i$ and $u^k_{i+1}$ belong to a common essential $k$-bracket in $\sC^k$, hence they are descended from vertices in the fusion bracket for $\sC$.  Then  

\begin{align}
 d_{\cT/\sC'}(\psi^\cT_{\cT/\sC'}(u^k_{\sigma(i)}))-\sigma(i)=d_{\cT/\sC'}(\psi^\cT_{\cT/\sC'}(u^k_{\sigma(i+1)}))-\sigma(i+1), \, \text{implying}
\end{align}

 \begin{align}
 d_{\cT/\sC'}(\psi^\cT_{\cT/\sC'}(u^k_{\sigma(i)}))-d_{\cT/\sC'}(\psi^\cT_{\cT/\sC'}(u^k_{\sigma(i+1)}))+\sigma(i+1)-\sigma(i)=0,
\end{align}

 so $(P_\sigma^T({\rho}(\sC'))^k_i=0$, as desired.

Next assume that $i = {\rm max} \{h: \psi^\cT_{\cT/\sC}(u^k_{\sigma(h)}) = u^k_j\}$.
Recall
\begin{align}
{\rho}(\sC'/\sC)^k_j=d_{({\cT/\sC})/\sC'}(\psi^{\cT/\sC}_{({\cT/\sC})/\sC'}(u^k_j))-d_{({\cT/\sC})/\sC'}(\psi^{\cT/\sC}_{({\cT/\sC})/\sC'}(u^k_{j+1}))+1.  
\end{align}

First suppose $u^k_i$ and $u^k_{i+1}$ are descended from vertices in the fusion bracket for $\sC$.  As observed above, $(P_\sigma^T({\rho}(\sC'))^k_i=0$, so we must show that ${\rho}(\sC'/\sC)^k_j=0$ as well. 
 We know that $\psi^\cT_{\cT/\sC}(u^k_{\sigma(i)})=u^k_j$ and $\psi^\cT_{\cT/\sC}(u^k_{\sigma(i+1)})=u^k_{j+1}$.  Moreover, $u^k_j$ and $u^k_{j+1}$ are both not descended from the fusion bracket for $\sC'/\sC$ in $\cT/\sC$, thus the image of $u^k_{j+1}$ in $\cT/\sC/\sC'$ is the successor of the image of $u^k_j$ in $\cT/\sC/\sC'$, i.e.\ $d_{({\cT/\sC})/\sC'}(\psi^{\cT/\sC}_{({\cT/\sC})/\sC'}(u^k_j))-d_{({\cT/\sC})/\sC'}(\psi^{\cT/\sC}_{({\cT/\sC})/\sC'}(u^k_{j+1}))=-1$ and so ${\rho}(\sC'/\sC)^k_j=0$, as desired.

For the remaining analysis we will assume that $u^k_i$ and $u^k_{i+1}$ are not both descended from vertices in the fusion bracket for $\sC$ (continuing to assume that $i = {\rm max} \{h: \psi^\cT_{\cT/\sC}(u^k_{\sigma(h)}) = u^k_j\}$).

  Next suppose that $u^k_i$ and $u^k_{i+1}$ are both not descended from vertices in the fusion bracket for $\sC$.
 Because $\sigma$ is a compatible shuffle for $\sC$ we can apply Lemma \ref{compatibleidentity} to deduce that 
\begin{align}
P_{\sigma}^T({\rho}(\sC'))^k_{i}=d_{\cT/\sC'}(\psi^\cT_{\cT/\sC'}(u^k_i))-d_{\cT/\sC'}(\psi^\cT_{\cT/\sC'}(u^k_{i+1}))+1.
 \end{align}

One may check that 

\begin{align}
d_{\cT/\sC'}(\psi^\cT_{\cT/\sC'}(u^k_i))-d_{({\cT/\sC})/\sC'}(\psi^{\cT/\sC}_{({\cT/\sC})/\sC'}(u^k_j))= 
\end{align}
\begin{align}
d_{\cT/\sC'}(\psi^\cT_{\cT/\sC'}(u^k_{i+1})) -d_{({\cT/\sC})/\sC'}(\psi^{\cT/\sC}_{({\cT/\sC})/\sC'}(u^k_{j+1})), \end{align}

 thus $P_{\sigma}^T({\rho}(\sC'))^k_{i}={\rho}(\sC'/\sC)^k_j$.

 Next suppose that  $u^k_{i}$ is not descended from a vertex in the fusion bracket for $\sC$, but $u^k_{i+1}$ is descended from a vertex in the fusion bracket for $\sC$. By Lemma \ref{compatibleidentity}, we have that $\sigma(i)=i=j$. 

Suppose further that $u^k_i$ is not descended from a vertex in the fusion bracket for $\sC'$.  In this case, because $\sC \sim \sC'$, $u^k_i$ and $u^k_{i+1}$ are both not descended from a vertex in the fusion bracket for $\sC'$, and $u^k_j$ and $u^k_{j+1}$ are both not descended from a vertex in the fusion bracket for $\sC'/\sC$.  Therefore

\begin{align}
(P_\sigma^T({\rho}(\sC'))^k_i=d_{\cT/\sC'}(\psi^\cT_{\cT/\sC'}(u^k_{\sigma(i)}))-d_{\cT/\sC'}(\psi^\cT_{\cT/\sC'}(u^k_{\sigma(i+1)}))+\sigma(i+1)-\sigma(i)=0,
\end{align}
and 
\begin{align}
 {\rho}(\sC'/\sC)^k_j=d_{({\cT/\sC})/\sC'}(\psi^{\cT/\sC}_{({\cT/\sC})/\sC'}(u^k_j))-d_{({\cT/\sC})/\sC'}(\psi^{\cT/\sC}_{({\cT/\sC})/\sC'}(u^k_{j+1}))+1=0.
\end{align}

Hence $(P_\sigma^T({\rho}(\sC'))^k_i={\rho}(\sC'/\sC)^k_j.$

Next, suppose that  $u^k_{i}$ is descended from a vertex in the fusion bracket for $\sC'$, and $u^k_{i+1}$ is descended from a vertex in the fusion bracket for $\sC$.

Because $\sigma(i)=i=j$, we see that 

\begin{align}
 d_{\cT/\sC'}(\psi^\cT_{\cT/\sC'}(u^k_{\sigma(i)}))
 = d_{({\cT/\sC})/\sC'}(\psi^{\cT/\sC}_{({\cT/\sC})/\sC'}(u^k_j)). 
\end{align}

So to prove that $(P_\sigma^T({\rho}(\sC'))^k_i={\rho}(\sC'/\sC)^k_j$, we must show that

\begin{align}
d_{({\cT/\sC})/\sC'}(\psi^{\cT/\sC}_{({\cT/\sC})/\sC'}(u^k_{j+1}))-(j+1)=d_{\cT/\sC'}(\psi^\cT_{\cT/\sC'}(u^k_{\sigma(i+1)}))- \sigma(i+1).
\end{align}

This equality follows from the observations that 
\begin{align}d_{\cT/\sC'}(\psi^\cT_{\cT/\sC'}(u^k_{\sigma(i+1)}))=\sigma(i+1)-(|V^k(\cT)|-|V^k(\cT/\sC')|)
\end{align}
and
\begin{align}
d_{({\cT/\sC})/\sC'}(\psi^{\cT/\sC}_{({\cT/\sC})/\sC'}(u^k_{j+1}))=j+1-(|V^k(\cT/\sC)|-|V^k((\cT/\sC)/\sC')|)
\end{align}
\begin{align}
=j+1-(|V^k(\cT)|-|V^k(\cT/\sC')|).  
\end{align}

The remaining case to check when $u^k_{i}$ is descended from a vertex in the fusion bracket for $\sC$, but $u^k_{i+1}$ is not descended from a vertex in the fusion bracket for $\sC$ (with subcases addressing whether or not $u^k_{i+1}$ is descended from a vertex in the fusion bracket for $\sC'$).
This case is very similar to the previous case above and we leave the details to the reader.  
\end{proof}

\begin{figure}
\centering
\def\svgwidth{0.6\textwidth}
%% Creator: Inkscape 1.2 (dc2aeda, 2022-05-15), www.inkscape.org
%% Accompanies image file 'illustratingsecondtransformation1.pdf' (pdf, eps, ps)
\begingroup%
  \makeatletter%
  \providecommand\color[2][]{%
    \errmessage{(Inkscape) Color is used for the text in Inkscape, but the package 'color.sty' is not loaded}%
    \renewcommand\color[2][]{}%
  }%
  \providecommand\transparent[1]{%
    \errmessage{(Inkscape) Transparency is used (non-zero) for the text in Inkscape, but the package 'transparent.sty' is not loaded}%
    \renewcommand\transparent[1]{}%
  }%
  \providecommand\rotatebox[2]{#2}%
  \newcommand*\fsize{\dimexpr\f@size pt\relax}%
  \newcommand*\lineheight[1]{\fontsize{\fsize}{#1\fsize}\selectfont}%
  \ifx\svgwidth\undefined%
    \setlength{\unitlength}{221.07983416bp}%
    \ifx\svgscale\undefined%
      \relax%
    \else%
      \setlength{\unitlength}{\unitlength * \real{\svgscale}}%
    \fi%
  \else%
    \setlength{\unitlength}{\svgwidth}%
  \fi%
  \global\let\svgwidth\undefined%
  \global\let\svgscale\undefined%
  \makeatother%
  \begin{picture}(1,0.46330513)%
    \lineheight{1}%
    \setlength\tabcolsep{0pt}%
    \put(-0.00136934,0.0036398){\makebox(0,0)[lt]{\lineheight{1.25}\smash{\begin{tabular}[t]{l}$(1,0,3,0,-1,0)$\end{tabular}}}}%
    \put(0.3860706,0.0036398){\makebox(0,0)[lt]{\lineheight{1.25}\smash{\begin{tabular}[t]{l}$(1,1,2,1,0,0)$\end{tabular}}}}%
    \put(0.770021,0.0036398){\makebox(0,0)[lt]{\lineheight{1.25}\smash{\begin{tabular}[t]{l}$(0,0,0,0,0,1)$\end{tabular}}}}%
    \put(0,0){\includegraphics[width=\unitlength,page=1]{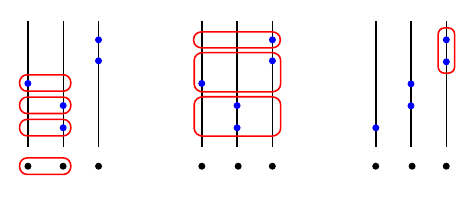}}%
    \put(0.12087941,0.44242042){\makebox(0,0)[lt]{\lineheight{1.25}\smash{\begin{tabular}[t]{l}$\rho_1$\end{tabular}}}}%
    \put(0.49703148,0.44242042){\makebox(0,0)[lt]{\lineheight{1.25}\smash{\begin{tabular}[t]{l}$\rho_2$\end{tabular}}}}%
    \put(0.87531459,0.44242042){\makebox(0,0)[lt]{\lineheight{1.25}\smash{\begin{tabular}[t]{l}$\rho_3$\end{tabular}}}}%
    \put(0,0){\includegraphics[width=\unitlength,page=2]{illustratingsecondtransformation1.pdf}}%
  \end{picture}%
\endgroup%

\vspace{11pt}
\begin{gather*}
\sigma = (12,23145)
\\
\\
P_\sigma^T
=
\left(\begin{array}{rrrrrr}
1 & 0 & 0 & 0 & 0 & 0 \\
0 & 1 & 0 & 0 & 0 & 0 \\
0 & 0 & 0 & 1 & 0 & 0 \\
0 & 0 & -1 & -1 & 0 & 0 \\
0 & 0 & 1 & 1 & 1 & 0 \\
0 & 0 & 0 & 0 & 0 & 1
\end{array}\right)
\\
\\
P_\sigma^T(\rho_2-\rho_1) = (\textcolor{red}{0},1,1,0,1,0),
\qquad
P_\sigma^T(\rho_3) = (\textcolor{red}{0},0,0,0,0,1)
\\
\end{gather*}
\def\svgwidth{0.275\textwidth}
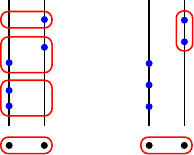
\caption{
We illustrate the maps ${\rho}(\sC') \rightarrow P^T({\rho}(\sC')-{\rho}(\sC))$ and ${\rho}(\sC') \rightarrow P^T({\rho}(\sC'))$ from Lemma \ref{contractionlemma}.
}\label{technicalshufflefig}
\end{figure}

We offer the following strengthening of the first part of Lemma \ref{contractionlemma}
.

\begin{lemma}\label{refinedcontractionlemma}
Let $\sC \in \fX(\cT)$ and $\sC' \in \wh \fX(\cT)$, and let $\sigma$ be a compatible shuffle for $\sC$. Let $u^k_{\sigma(i)} \in V^k(\cT)$ and $u^k_j \in V^k(\cT/\sC)$ such that $\psi^\cT_{\cT/\sC}(u^k_{\sigma(i)})=u^k_j$.
If $\sC \rightarrow \sC'$ and $\{(\sC'/\sC)_r:1\leq r\leq t\}$ are the collisions whose union is $\sC'/\sC$, as is guaranteed by Lemma \ref{quotientofcollisionbycollision}, then

\begin{equation}
(P_{\sigma}^T({\rho}(\sC')-{\rho}(\sC)))^k_{i}=
\begin{cases}
(\sum_{r=1}^t{\rho}(\sC'/\sC)_r)^k_j & \text{if} \,\, i = {\rm max} \{h: \psi^\cT_{\cT/\sC}(u^k_{\sigma(h)}) = u^k_j\} \\
\,\, 0& \text{otherwise. } \\
\end{cases}
\end{equation}
\end{lemma}

\begin{proof}
Recall that $r>1$ only if $\sC$ and $\sC'$ have the same fusion bracket.
Let $\sC'_r$ denote the preimage of $(\sC'/\sC)_r$ in $\fX(\cT)$ provided by Lemma \ref{preimagelemma}.
By Lemma \ref{contractionlemma} we know that
\begin{equation}
(P_{\sigma}^T({\rho}(\sC'_r)-{\rho}(\sC)))^k_{i}=
\begin{cases}
{\rho}(\sC'_r/\sC)^k_j & \text{if} \,\, i = {\rm max} \{h: \psi^\cT_{\cT/\sC}(u^k_{\sigma(h)}) = u^k_j\}  \\
\,\, 0& \text{otherwise } \\
\end{cases}
\end{equation}

Thus it suffices to prove that
\begin{equation}
P_{\sigma}^T({\rho}(\sC')-{\rho}(\sC))= \sum_{i=1}^tP_{\sigma}^T({\rho}(\sC'_r)-{\rho}(\sC))
\end{equation}

Because $P_{\sigma}^T$ is an invertible linear transformation this reduces to checking that 

\begin{equation}
{\rho}(\sC')+(t-1){\rho}(\sC)= \sum_{i=1}^t{\rho}(\sC'_r)
\end{equation}

At this point one can use the definition of $\rho$ and calculate directly that this equation indeed holds, but this is a little technical and so we have chosen to go a different route.
We invoke Proposition \ref{mainvelocitythm}, which states that $\Gamma$ is a piecewise-linear isomorphism.\footnote{Although we invoke the fact that $\Gamma$ is a piecewise-linear isomorphism, we emphasize that Lemma \ref{refinedcontractionlemma} will not be utilized until \S \ref{triangulationsection} after we have proven Proposition \ref{mainvelocitythm}.}
This allows us to apply $\Gamma^{-1}$ and reduce the above equation to the observation that \begin{equation}
\ell(\sC')+(t-1)\ell(\sC)= \sum_{i=1}^t\ell(\sC'_r).
\end{equation}

Let $\sB = \sC'\vee\sC = \vee_{i=1}^t \sC_r$. 
 Let $\ell_{\sB}$ and $\ell'_{\sB}$ be the metric $n$-bracketings associated to the left and right side of the above equation, respectively.  We wish to show that for each $k$-bracket $A$, we have $\ell_{\sB}(A) = \ell'_{\sB}(A)$.  By definition, if $A \notin \sB^k$, or $A$ is a singleton bracket, then $\ell_{\sB}(A) = \ell'_{\sB}(A)=0$.
Let $A \in \sB^k$ be a nonsingleton bracket.  If $A \in (\sC')^k\setminus \sC^k$, then $A \in \sC_r^k$ for a unique $r$, hence $\ell_{\sB}(A) = \ell'_{\sB}(A)=1$.  If $A \in  (\sC')^k$ and $A \in \sC^k$, then $A \in \sC_r$ for all $r$, and $\ell_{\sB}(A) = \ell'_{\sB}(A)=t$.
If $A \in \sC^k\setminus (\sC')^k$ then there exists a $k$-bracket $A' \in (\sC')^k$ such that $A\subsetneq A'$, and there exists a unique $r$ such that $A' \in \sC_r$.
It follows that $A \in \sC_s$ for all $s \neq r$, and $\ell_{\sB}(A) = \ell'_{\sB}(A)=t-1$.
\end{proof}

\subsection{The main lemma for the velocity fan}

\

In this section we establish  Lemma \ref{main} which will later be employed for inductively proving $(\ref{brackequal}) \Rightarrow (\ref{vecequal})$ in Lemma \ref{metriclemma}. 
The motivating idea of this section is rather simple, even if the resulting argument is not: given a vector ${\bf v}$ in the velocity fan, we aim to find a collision $\sC$, which is a function of ${\bf v}$, such that for any metric $n$-bracketing  $\ell_\sB$ with $\Gamma(\ell_\sB)=\bf v$, we have $\sC \leq \sB$.
Given such a $\sC$, we hope to subtract some multiple of $\rho(\sC)$ from  ${\bf v}$ to get a ``smaller" vector to which we can apply induction for demonstrating $(\ref{brackequal}) \Rightarrow (\ref{vecequal})$.

\begin{lemma}\label{main}
Let ${\bf v} \in \cF(\cT)$.
There exists an extended collision $\sC_{\bf v}$, which is a function of ${\bf v}$ alone, such that given any compatible collection $\{\sC_i \in \wh \fX(\cT): 0\leq i\leq k\}$ with $\sC_0 = \sB_\min$ and $\lambda_i \in \mathbb{R}$ with $\lambda_i > 0$ for $i>0$ such that ${\bf v} = \sum _{i =1}^k \lambda_i {\rho}(\sC_i)$, we have that $\sC_{\bf v}$ is containment-minimal in $\bigvee_{i =1}^k \sC_i$.
\end{lemma}

The proof of Lemma \ref{main} requires us to assume Proposition \ref{mainvelocitythm} holds for all $(n-1)$-associahedra.
We will establish Lemma \ref{main} as consequence of Lemma \ref{technicallemma}, which allows us to identify containment-minimal collisions.
Lemma \ref{technicallemma} requires the following technical Definition \ref{techdefinition}.

Recall that we denote a vector ${\bf v} \in \mathbb{R}^m$ as
${\bf v} = (v^1_1,\dots, v^1_{t_1},\dots,  v^n_1,\dots,v^n_{t_n})$,
and for $j<k$, we have the following notation:
\begin{align}
\Delta^i_{j,k}({\bf v} ) \coloneqq \sum_{h=j}^{k-1} v^i_h.
\end{align}

In particular, $\Delta^i_{j,j+1}({\bf v} )=v^i_j$.
\begin{definition}\label{techdefinition}
Given a rooted plane tree $\cT$ of depth $n$ and a vector in ${\bf v} \in \cF(\cT)$, we define a collection of extended collisions $\Xi({\bf v})$ associated to ${\bf v}$.
This set is the union of two other sets $\Xi_1({\bf v})$ and  $\Xi_2({\bf v})$.
Let $\sB$ be the $(n-1)$-bracketing in $\cK(\pi(\cT))$ such that $\ell_{\sB}=\Gamma^{-1}(\pi({\bf v}))$
\footnote{
\label{inductivefootnote}
The existence of $\ell_{\sB}$ is afforded by Proposition \ref{projfanthing} combined with our inductive hypothesis that \ref{mainvelocitythm} holds for all $(n-1)$-associahedra.}.
To simplify notation, let $\ell = \ell_{\sB}$.
Given an $(n-1)$-bracket $A \in \sB^{n-1}$, we define 

\begin{align}
\gamma({\ell},A)
\coloneqq
\sum_{\substack{A' \in \sB^{n-1}\\ A \subseteq A'}} \ell(A').
\end{align}

Let $\sC \in \wh \fX(\cT)$ then

\begin{enumerate}
\item
\label{1tech}
$\sC \in \Xi_1({\bf v})$ if

\begin{enumerate}
\item \label{1techa} $\sC$ is type 1, with unique nontrivial $n$-bracket $A$.

\item \label{1techb} There exists some real number $w({\bf v},\sC)$ with  $w({\bf v},\sC)> \gamma({\ell},\pi(A))$ such that if $u^n_j,u^n_k \in A^n $ with $j<k$, then $\Delta^n_{j,k}({\bf v}) =w({\bf v},\sC) (k-j)$.

\item \label{1techc} For any $n$-bracket $A'$ with $A \subsetneq A'$ and $\pi(A') = \pi(A)$, there exists  $u^n_{j} \in A'^{\,n}$ with $v^n_j<w({\bf v},\sC)$.
\end{enumerate}

\medskip

\item
\label{2tech}
$\sC \in \Xi_2({\bf v})$ if

\begin{enumerate}
\item
\label{2techa}
$\sC$ is not type 1.
 
\item
\label{2techb}
There exists no $\sC' \in \Xi_1({\bf v})$ such that  $\sC' \rightarrow \sC$.
 
\item
\label{2techc}
$\pi(\sC)$ is containment-minimal in $\sB$. 
\item
\label{2techd}

Let $A \in \sC^{n-1}$, and $u^n_j,u^n_k \in V^n(\cT)$ with $j<k$ and $\pi(u^n_j),\pi(u^n_k)\in A$. 

\begin{enumerate}
\item\label{2techd1} If $\Delta^n_{j,k}({\bf v}) = \gamma({\ell},A)(k-j)$, there exists ${\wt A} \in \sC^{n}$ such that $\pi({\wt A})=A$ with $u^n_j,u^n_k \in  {\wt A}^n$. 

\item \label{2techd2}
If $\Delta^n_{j,k}({\bf v}) >\gamma({\ell},A)(k-j)$, there exists $A_1, A_2 \in \sC^n$ such that $\pi(A_1)= \pi(A_2) = A$, $A_2<_{\sC}A_1$ in the height partial order for $\sC$, and $u^n_j \in A_1^n $, $u^n_k \in A_2^{n}$.

\item \label{2techd3} 
If $\Delta^n_{j,k}({\bf v}) <\gamma({\ell},A)(k-j)$, there exists $A_1, A_2 \in \sC^n$ such that $\pi(A_1)= \pi(A_2) = A$, $A_1<_{\sC}A_2$ in the height partial order for $\sC$, and $u^n_j \in A_1^n $, $u^n_k \in A_2^{n}$.
\end{enumerate}
\end{enumerate}
\end{enumerate}

We define
\begin{align}
\Xi({\bf v}) \coloneqq \Xi_1({\bf v}) \cup \Xi_2({\bf v}).
\end{align}
\null\hfill$\triangle$
\end{definition}

\begin{figure}[ht]
\centering
\def\svgwidth{0.65\textwidth}
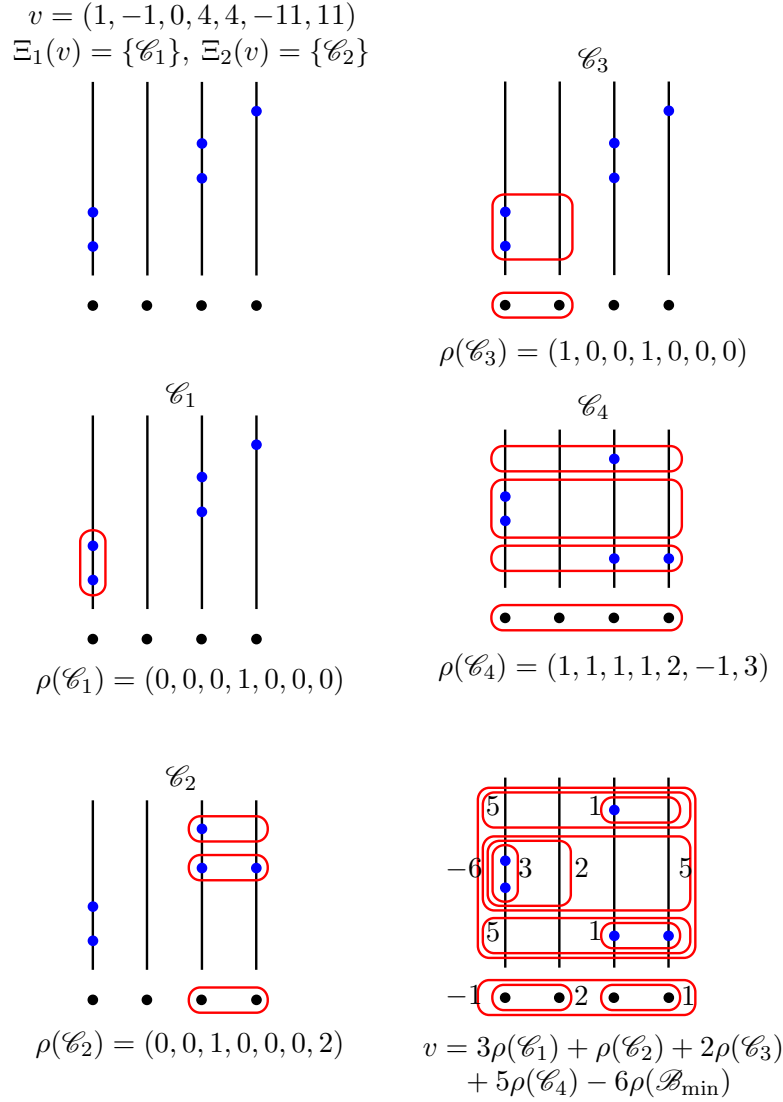
\caption{
\label{Xiexample}
In the left column, we have depicted an arrangement for a 2-associahedron, a vector ${\bf v} \in \mathbb{R}^7$, and the sets $\Xi_1({\bf v})$ and $\Xi_2({\bf v})$.
In the right column, we have depicted an expression of ${\bf{v}}$ as a point in $\cF(\cT)$ and the metric $n$-bracketing $\ell_\sB$ such that $\Gamma(\ell_\sB) ={\bf v}$.}
\end{figure}

\begin{lemma}\label{technicallemma}
Let $\sC_i \in \wh \fX(\cT)$ for $0\leq i\leq k$ with $\sC_0 = \sB_\min$.
Let $\lambda_i \in \mathbb{R}$ with $\lambda_i > 0$ for $i>0$.
Suppose that the $\sC_i$ are compatible and let
\begin{align}
\ell_{\wt \sB} = \sum_{i=1}^k \lambda_i\ell(\sC_i).
\end{align}

For ${\bf v} = \sum _{i =1}^k \lambda_i {\rho}(\sC_i)$ and $\sC \in \wh \fX(\cT)$, we have that $\sC$ is containment-minimal in ${\wt \sB}$ if and only if $\sC \in \Xi({\bf v}) $.
\end{lemma}

\begin{proof}(\ref{main})
By Lemma \ref{technicallemma} we may take $\sC_{\bf v}$ to be any collision in $\Xi({\bf v})$ and observe that $\Xi({\bf v}) = \emptyset$ if and only if ${\bf v} = 0$.
\end{proof}

\begin{remark}
Let $M(\wt \sB)$ denote the collection of extended collisions which are containment-minimal in $\wt \sB$.  For proving Lemma \ref{metriclemma} we will only need the ``if'' direction of the statement of Lemma \ref{technicallemma} ($\sC \in \Xi({\bf v}) \Rightarrow \sC \in M(\wt \sB)$), but this part of the argument naturally builds on the ``only if'' direction ($\sC \in M(\wt \sB) \Rightarrow \sC \in \Xi({\bf v})$).
\null\hfill$\triangle$
\end{remark}

\begin{proof}
(\ref{technicallemma})
Let $\sC \in \wh \fX(\cT)$.  Recall that Scholium \ref{simplifiedcontainment-minimal condition} clarifies that $\sC$ is containment-minimal in ${\wt \sB}$ if and only if for each $i$, either  $\sC\rightarrow \sC_i$ or $\sC \sim \sC_i$.
We let $\sB$ = $\pi({\wt \sB})$ and $\ell=\ell_{\sB}$  aligning our notation with that of Definition \ref{techdefinition}.

Suppose that $\sC$ is containment-minimal in ${\wt \sB}$ and that $\sC$ is type 1 with unique nontrivial bracket $A \in \sC^n$ (condition (\ref{1techa})).
We wish to show that $\sC$ satisfies conditions 
(\ref{1techb}) and (\ref{1techc}). 
Observe that because $\sC$ is containment-minimal in $\wt \sB$, and $\sC$ is type 1, we must have that $\sC = \sC_j$ for some $j$.  
Let $R \subseteq \{\sC_i\}$ be the set of extended collisions such that $\sC \rightarrow \sC_i$, and let $S \subseteq R$ be those $\sC_i$ in $R$ which are not type 1. We claim that $\gamma({\ell},\pi(A)) = \sum_{\sC_i \in S}\lambda_i $.
The $n$-bracketing condition {\sc (nested)} and Lemma \ref{collisioncharacterization} (2) combined imply that 
for each $\sC_i \in S$ there exists a unique nonsingleton $(n-1)$-bracket $A' \in \sC_i^{n-1}$ with $\pi(A)\subsetneq A'$.  In the other direction, if $A'$ is a nonsingleton $(n-1)$-bracket with $A' \in \sB^{n-1}$ and $\pi(A)\subsetneq A'$, it follows from Lemma \ref{collisioncontainingspecificbracket} and Scholium \ref{lem:collisionlift} that there exists some $\sC_i \in S$ such that $A' \in \sC_i^{n-1}$.  Because $\sC$ is containment-minimal in $\wt \sB$, and $\pi(\sC) \subseteq A'$, we know that $\sC$ and $\sC_i$ are not disjoint, hence it must be that $\sC \rightarrow \sC_i$.  The claim follows.

Let $u^n_j,u^n_k \in A$ with $j<k$, and let $w({\bf v},\sC) = \sum_{\sC_i \in R}\lambda_i$.  Because $\sC \in R\setminus S$, we have that $w({\bf v},\sC) > \gamma({\ell},\pi(A))$.
Lemma \ref{fundvellemma} states that if $\sC \rightarrow \sC_i$ then $\Delta^n_{j,k}({\rho}(\sC_i))=k-j$, and if $\sC_i\sim \sC$ then $\Delta^n_{j,k}({\rho}(\sC_i))=0$.
Thus $\Delta^n_{j,k}({\bf v}) =w({\bf v},\sC) (k-j)$ establishing condition (\ref{1techb}).
Let $A'$ be an $n$-bracket  with $A\subsetneq A'$ and $\pi(A) = \pi(A')$.
Suppose there exists $u^n_j,u^n_{j+1} \in A'^{ \,n}$ with $u^n_j \in A^n$, and $u^n_{j+1}\in A'^{\,n} \setminus A^n$.
Let $Q \subseteq \{\sC_i\}$ be the collection of extended collisions which are either equal to $\sB_{\min}$ or have a nontrivial $n$-bracket ${\wt A} \in \sC_i$ such that $u^n_{j},u^n_{j+1} \in {\wt A}$.
It follows by the $n$-bracketing condition {\sc (nested)} and the containment-minimality of $\sC$ that $Q \subseteq R$.
For each $\sC_i$ we have that $\rho(\sC)^n_j$ is an integer at most 1, and is equal to 1 if and only if $\sC_i \in Q$, hence $v^n_{j} \leq \sum_{\sC_i \in Q}\lambda_i$.
Because $\sC \in R\setminus Q$, we have $\sum_{\sC_i \in Q}\lambda_i<w({\bf v},\sC)$ implying $v^n_j<w({\bf v},\sC)$ (condition (\ref{1techc})).  The case where $u^n_j \in A^n$ and $u^n_{j-1} \in {\wt A} \setminus A$ is similar.

Conversely, suppose that $\sC$ satisfies conditions (\ref{1tech}), but  $\sC$ is not containment-minimal in $\wt \sB$.  Either $\sC$ is not compatible with $\wt \sB$, or $\sC$ is compatible with $\wt \sB$ but not containment-minimal in $\wt \sB$. 
Suppose first that $\sC$ is compatible with $\wt \sB$, but not containment-minimal in $\wt \sB$.  In what follows, for $\sC_i$ a type 1 collision, we will take $A_i$ to be the fusion bracket for $\sC_i$, i.e.\ the unique nontrivial $n$-bracket in $\sC_i$.
Let $\sC_i$ be a type 1 collision which is containment-minimal in $\wt \sB$ with $\sC_i \rightarrow \sC$.  
By the ``only if'' direction above, we know that $\sC_i$ satisfies the conditions in (\ref{1tech}).  Let $w({\bf v}, \sC_i)$ be the constant in condition (\ref{1techb}) for $\sC_i$.  Suppose there exist vertices $u^n_{j-1}, u^n_j, u^n_{j+1} \in A^n$ with $u^n_{j+1} \in A^n \setminus A_i^n$.  
Take $A'$ to be the $n$-bracket obtained from $A_i$ by adding $u^n_{j+1}$ to $A_i^n$.  Then condition (\ref{1techc}) for $\sC_i$ with $A'$ as described implies that $v^n_j < v^n_{j-1} = w({\bf v}, \sC_i)$. 
 On the other hand, condition (\ref{1techb}) for $\sC$ implies that $v^n_j=  v^n_{j-1} = w({\bf v}, \sC)$, a contradiction.  The case where $u^n_{j-1}, u^n_j, u^n_{j+1} \in A^n$ with $u^n_{j-1} \in A^n \setminus A_i^n$ is similar.  
 Therefore we may assume that $\sC$ is not compatible with $\wt \sB$.
We claim that there exists some type 1 collision $\sC_i$ such that $\sC$ and $\sC_i$ are not compatible. 
 Because there exist $u^n_j,u^n_k \in A^n$ such that $\Delta^n_{j,k}({\bf v}) =w({\bf v},\sC) (k-j)$ and $w({\bf v},\sC) > \gamma({\ell},A)$, the argumentation of the ``only if" direction above implies there must exist some type 1 collision $\sC_i$ with $u^n_j,u^n_k \in A_i^n$.  If $\sC_i$ is not compatible with $\sC$, the claim is verified.  Assume instead $\sC_i \rightarrow \sC$, then there exists some $\sC_j\rightarrow \sC_i$ with $\sC_j$ containment-minimal in $\wt \sB$.  We see that $\sC_j \rightarrow \sC$, but this situation was excluded previously.  

Suppose instead that $\sC \rightarrow \sC_i$.  If there exists some $\sC_j\rightarrow \sC_i$, such that $\sC_j$ and $\sC$ are not compatible, we have verified the claim, so we may assume that for each $\sC_j$ with $\sC_j \rightarrow \sC_i$, we have $\sC_j$ and $\sC$ are compatible.  But, this implies that $\sC$ and $\wt \sB$ are compatible  as any other collision $\sC_j$ which is not compatible with $\sC$ would also need to contain $\sC_i$, and this is impossible, hence the claim is also verified.

Now take $\sC_i$ type 1 such that $\sC$ and $\sC_i$ are not compatible.  Suppose there exist some vertices $u^n_{j-1},u^n_{j},u^n_{j+1}$ with $u^n_{j-1} \in A^n\setminus A_i^n$, $u^n_j \in A_i^n\cap A^n$ and $u^n_{j+1} \in A_i^n$.
Let $U \subseteq \{\sC_i\}$ be the extended collisions which are equal to $\sB_{\min}$ or contain a nontrivial $n$-bracket $\wt A$ with $\{u^n_{j-1},u^n_{j}\} \in {\wt A}^n$, and let $W \subseteq \{\sC_i\}$ be the extended collisions which are equal to $\sB_{\min}$ or contain a nontrivial $n$-bracket $\wt A$ with $\{u^n_{j},u^n_{j+1}\} \in {\wt A}^n$.  Because $u^n_{j},u^n_{j+1} \in A_i^n$ and $u^n_{j-1}\notin A_i^n$, the condition {\sc (nested)} for $n$-bracketings implies that $U\subsetneq W$. 
Again applying Lemma \ref{fundvellemma}, we calculate $v^n_{j-1} \leq  \sum_{\sC_i \in U}\lambda_i$ and $v^n_{j} = \sum_{\sC_i \in W}\lambda_i$.  It follows that $v^n_{j-1}<v^n_{j}$.  If $u^n_{j+1} \in A^n$, the inequality $v^n_{j-1}<v^n_{j}$ violates condition condition (\ref{1techb}) for $\sC$. If $u^n_{j+1} \notin A^n$, the inequality $v^n_{j-1}<v^n_{j}$ violates condition (\ref{1techc}) for $\sC$, with $A'$ taken to be the $n$-bracket obtained from $A$ by adding $u^n_{j+1}$ to $A^n$.  
The remaining case where there exist vertices $u^n_{j-1},u^n_{j},u^n_{j+1}$ with $u^n_{j+1} \in A^n\setminus A_i^n$, $u^n_j \in A_i^n\cap A^n$, and $u^n_{j-1} \in A_i^n$ is similar.  We conclude that $\sC$ is containment-minimal in $\wt \sB$.

Next, suppose that $\sC$ is containment-minimal in $\wt \sB$ and $\sC$ is not type 1 (condition (\ref{2techa})). 
We will show that $\sC$ satisfies conditions  (\ref{2techb}), (\ref{2techc}), and (\ref{2techd}).
By containment-minimality, there exists no $\sC' \in \Xi_1({\bf v})$ such that  $\sC' \rightarrow \sC$, which is condition (\ref{2techb}).
Condition (\ref{2techc}) follows from Lemma \ref{projcontain}.

Let $A \in \sC^{n-1}$ and $u^n_j,u^n_k \in V^n(\cT)$ with $j<k$ such that $\pi(u^n_j),\pi(u^n_k)\in A$.  Suppose that $\sC \rightarrow \sC_i$.  By containment-minimality of $\sC$, if there exists ${\wt A} \in \sC^{n}$ such that $\pi({\wt A})=A$ with $u^n_j,u^n_k \in  {\wt A}^n$, then for any $A_i \in \sC_i^n$ with $u_j^n \in A_i$, we have $u_k^n \in A_i$. On the other hand, suppose there exists $A,A' \in \sC^n$ such that $\pi(A) = \pi(A')$ with $u^n_j \in A^n$, $u^n_k \in (A')^n$, and $A'<_{\sC}A$.
Let $A_i, A_i' \in \sC^{n}_i$ such that $A \subseteq A_i$ and $A' \subseteq A'_i$.  By condition {\sc (height partial orders)} of an $n$-bracketing, either $A'_i<_{\sC_i} A_i$ or $A_i = A'_i$.  A symmetric implication holds: if $A'>_{\sC}A$, then $A'_i>_{\sC_i}A_i$ or $A_i = A'_i$.  Combining these observations with Lemma \ref{fundvellemma}, we have the following implications:
\begin{itemize}
\item If $\Delta^n_{j,k}({\rho}(\sC))=k-j$ then $\Delta^n_{j,k}({\rho}(\sC_i))= k-j$.
\item If $\Delta^n_{j,k}({\rho}(\sC))>k-j$ then $\Delta^n_{j,k}({\rho}(\sC_i))\geq k-j$.
\item If $\Delta^n_{j,k}({\rho}(\sC))<k-j$ then $\Delta^n_{j,k}({\rho}(\sC_i))\leq k-j$.
\end{itemize}

By Lemma \ref{fundvellemma}, $\Delta^n_{j,k}(\rho(\sC_i)) = 0$ if $\sC \sim \sC_i$.
Thus, letting $R\subseteq \{\sC_i\}$ be the set of extended collisions with $\sC \rightarrow \sC_i$, we know that $\Delta^n_{j,k}({\bf v})=\sum_{\sC_i \in R}\lambda_i\Delta^n_{j,k}({\rho}(\sC_i))$ and
$\gamma({\ell},A) = \sum_{\sC_i \in R}\lambda_i$.  Therefore

\begin{itemize}
\item If $\Delta^n_{j,k}({\bf v}) = \gamma({\ell},A)(k-j)$ then $\Delta^n_{j,k}({\rho}(\sC))=k-j$. 
\item If $\Delta^n_{j,k}({\bf v}) > \gamma({\ell},A)(k-j)$ then $\Delta^n_{j,k}({\rho}(\sC))>k-j$.
\item If $\Delta^n_{j,k}({\bf v}) < \gamma({\ell},A)(k-j)$ then $\Delta^n_{j,k}({\rho}(\sC))<k-j$.
\end{itemize}
Condition (\ref{2tech2d}) now follows by reinterpreting these implications using Lemma \ref{fundvellemma}.

Conversely, suppose that $\sC$ satisfies the conditions of (\ref{2tech}).
We wish to prove that $\sC$ is containment-minimal in $\wt \sB$.
By Scholium \ref{lem:collisionlift} combined with
Lemma \ref{containmentlemma}, there exists some $\sC' \in \wh \fX(\cT)$ which is containment-minimal in ${\wt \sB}$, and $\pi(\sC') \rightarrow \pi(\sC)$ or $\sC'$ is type 1.
We will demonstrate that $\sC = \sC'$.
By condition (\ref{2techb}) for $\sC$, we know that $\sC'$ is not type 1.  By Lemma \ref{projcontain}, $\pi(\sC')$ is containment-minimal in $\pi({\wt \sB})$.  Combined with condition (\ref{2techb}) for $\sC$ and the fact that $\pi(\sC')\rightarrow \pi(\sC)$, we conclude that $\pi(\sC')= \pi(\sC)$.  

By the ``only if'' direction verified above, we know that $\sC'$ also satisfies the conditions of (\ref{2tech}).  We claim that the conditions (\ref{2techd}) determine $\sC$ uniquely, implying $\sC=\sC'$.  Indeed, we can use condition (\ref{2techd1}) to determine the $n$-brackets in $\sC'$, and we can we can use conditions (\ref{2techd2}) and (\ref{2techd3}) to determine the height partial order on these $n$-brackets.
\end{proof}

\subsection{The velocity fan is a fan}\label{metricprfsect}

\ 

We are now ready to prove Theorem \ref{mainvelocitystatement}, which has been reduced to Proposition \ref{metriclemma}.

\begin{proof}(\ref{metriclemma})
 We will first prove that (\ref{vecequal}) $\Rightarrow$ (\ref{brackequal}).
Let
\begin{align}
\ell_{\sB}=\sum _{i=0}^k \lambda_i \ell(\sC_i)= \sum _{i=0}^l \gamma_i \ell({\sC}_i'),
\end{align}
and set ${\bf u} = \sum _{i=0}^k \lambda_i {\rho}(\sC_i)$ and ${\bf v} = \sum _{i=0}^l \gamma_i {\rho}({\sC}_i')$.  We will show that ${\bf u}={\bf v}$.

By Lemmas \ref{fusionmetricbracketingdecomp} and \ref{metricpartitionlemma}, it suffices to prove (\ref{vecequal}) $\Rightarrow$ (\ref{brackequal}) when $\ell_{\sB}$ has a single fusion bracket $A$.  Let $\sC$ be a containment minimal collision in $\sB$ and take $\sigma$ to be a compatible shuffle for $\sC$.  For each $i$, we know that $\sC \nsim \sC_i$, as they have the same fusion bracket, hence $\sC \rightarrow \sC_i$.  By Lemma \ref{inclusionpermutation}, $\sigma$ is a compatible shuffle for each $\sC_i$.
Let 
\begin{align}
\sigma({\bf u}) =\sum_{i=0}^l\gamma_i(P_{\sigma}^T(\rho(\sC_i) -{\bf 1})+{\bf 1}),
\end{align}
\begin{align}
\sigma({\bf v}) =\sum_{i=0}^k\lambda_i(P_{\sigma}^T(\rho(\sC_i') -{\bf 1})+{\bf 1}).
\end{align}
 
We remark that $\sigma({\bf u})$ is, \emph{a priori}, not determined by ${\bf u}$ alone; it is determined by the given expression for ${\bf u}$.
By Lemma \ref{permutationtransformationlemma1}, we have that for a particular index $i$, the coordinate $\sigma({\bf v})^k_i$ is equal to
\begin{align}\sum_{\substack {A' \in \sB^k \\ (A')^{n}\, \ni\, u^n_{\sigma(i)},u^n_{\sigma(i+1)}} } l_{\sB}(A').
\end{align}

This is an invariant of the metric $n$-bracketing $\ell_{\sB}$.  In particular, $\sigma({\bf v})=\sigma({\bf u})$. 
Furthermore,  
\begin{align}
\sigma({\bf v}) =P_{\sigma}^T(\sum_{i=0}^k\lambda_i\rho(\sC_i')) -P_{\sigma}^T(\sum_{i=0}^k\lambda_i{\bf 1})+\sum_{i=0}^k\lambda_i{\bf 1}.
\end{align}

By Lemma \ref{equalityofcoefficientsums}, we have that $ \sum_{i=0}^k\lambda_i = \sum_{i=0}^l\gamma_i=\ell_{\sB}(A)$, hence

\begin{align}
(P_{\sigma}^T)^{-1}(\sigma({\bf v}) +P_{\sigma}^T(\ell_{\sB}(A){\bf 1})+\ell_{\sB}(A){\bf 1})= {\bf v}.
\end{align}
\begin{align}(P_{\sigma}^T)^{-1}(\sigma({\bf u}) +P_{\sigma}^T(\ell_{\sB}(A){\bf 1})+\ell_{\sB}(A){\bf 1})= {\bf u}.
\end{align}

Because $\sigma({\bf v})=\sigma({\bf u})$, it follows that ${\bf v} = {\bf u}$ as desired.

\

Next we will prove  (\ref{brackequal}) $\Rightarrow$ (\ref{vecequal}).

\medskip

\noindent
Let
\begin{align}
{\bf v}=\sum _{i =1}^k \lambda_i \rho(\sC_i) = \sum _{i =1}^l \gamma_i \rho(\sC_i'),
\end{align}

\begin{align}
\ell_{\sB_1} = \sum _{i=0}^k \lambda_i \ell(\sC_i) \,\,\,\,\,{\rm and} \,\,\,\,\, \ell_{\sB_2} = \sum _{i =1}^l \gamma_i \ell(\sC_i').
\end{align}

Given $\sB \in \cK(\cT)$, let $|\sB|\coloneqq \sum_{i=1}^n |\sB^k|$.  We proceed by induction on $|\sB_1|+|\sB_2|$.
By Lemma \ref{main} we can always find an extended collision $\sC_{\bf v}$, which is a function of  ${\bf v}$ alone, such that $\sC_{\bf v}$ is containment-minimal in both $\sB_1$ and $\sB_2$.  Let $\epsilon >0$ be the minimum value of $\ell_{\sB_1}(A)$ and $\ell_{\sB_2}(A)$, where $A$ ranges over all nontrivial brackets in $\sC_{\bf v}$.  By Scholium\ \ref{sch:subtract} we know that there exist metric $n$-bracketings $\ell_{{\underline \sB}_1}$ and $\ell_{{\underline \sB}_2}$ with underlying $n$-bracketings ${{\underline \sB}_1}$ and ${{\underline \sB}_2}$, respectively, such that

\begin{gather}
\ell_{{\underline \sB}_1}
=
\ell_{\sB_1}-\epsilon \ell(\sC_{\bf v}),
\\
\ell_{{\underline \sB}_2}
=
\ell_{\sB_2}-\epsilon \ell(\sC_{\bf v}).
\nonumber
\end{gather}

By Lemma \ref{conicalcomb} there exist extended collisions ${\underline \sC_0}, {\underline \sC_1}, \ldots, {\underline \sC_{r}}$ and  ${\underline \sC_0'}, {\underline \sC_1'}, \ldots , {\underline \sC'_{s}}$  such that ${\underline \sC_0} = {\underline \sC_0'} = \sB_\min$, and ${\underline \lambda_i},{\underline \gamma_i}\in \mathbb{R}$ with ${\underline \lambda_i},{\underline \gamma_i}\geq 0$ for $i>0$ such that
\begin{gather}
\ell_{{\underline \sB}_1} =  \sum _{i =1}^{r} {\underline \lambda_i} \ell({\underline \sC_i}),
\\
\ell_{{\underline \sB}_2} = \sum _{i =1}^{s} {\underline \gamma_i} \ell({\underline \sC_i'}).
\nonumber
\end{gather}
Because 
$\ell_{\sB_1} = \ell_{{\underline \sB}_1}  + \epsilon \ell(\sC_{\bf v})$ and $\ell_{\sB_2} = \ell_{{\underline \sB}_2} + \epsilon \ell(\sC_{\bf v})$, we know by (\ref{vecequal})  $\Rightarrow$ (\ref{brackequal}), verified above, that
\begin{align}
{\bf v} = \sum _{i =0}^k \lambda_i \rho(\sC_i)= \sum _{i =0}^{r} {\underline \lambda_i} \rho({\underline {\sC}_i}) +\epsilon {\rho}(\sC_{\bf v}),
\end{align}
and 

\begin{align}
{\bf v}  = \sum _{i=0}^l \gamma_i \rho(\sC_i')=\sum _{i=0}^{s} {\underline \gamma_i} \rho({\underline {\sC}_i'}) +\epsilon {\rho}(\sC_{\bf v}).
\end{align}

Thus we can take
\begin{align}
{\bf v}' =\sum _{i=0}^{r} {\underline \lambda_i} \rho({\underline {\sC}_i}) = \sum _{i=0}^{s} {\underline \gamma_i} \rho({\underline {\sC}_i'}).
\end{align}

By our choices of $\epsilon$, we have that
\begin{align}
|{{\underline \sB}_1}|+ |{{\underline \sB}_2}| < |{\sB_1}|+ |{\sB_2}|. 
\end{align}
By induction, we can conclude that $\ell_{{\underline \sB}_1} = \ell_{{\underline \sB}_2}$, hence $\ell_{\sB_1} = \ell_{\sB_2}$ as desired.
\end{proof}

 \subsection{The velocity fan is complete}\label{velocityfaniscompletesection}

 \

In this subsection we prove that the velocity fan is complete\footnote{For this section, we take \emph{complete} to mean having support all of $\mathbb{R}^m$.}  by demonstrating that it has no boundary.
We proceed by  combinatorial analysis of $n$-bracketings.

\begin{lemma}\label{codimension1lemma}
Let $\cT$ be a rooted plane tree of depth $n$, and let $\sB \in \cK(\cT)$ be a nonmaximal $n$-bracketing.
There exist $\sB_1, \sB_2 \in \cK(\cT)$ with $\sB_1 \neq \sB_2$ such that $\sB <\sB_1$ and $\sB <\sB_2$.
\end{lemma}

\begin{proof}
We proceed by induction on $|V(\cT)|$.
Let $\sC \in \wh \fX(\cT)$ be containment-minimal in $\sB$.
If $\sC$ is a minimal collision, then there is an isomorphism of poset intervals $[\sB,\star_{\,\cT}] \subseteq {\wh \cK}(\cT)$ and $[\sB/\sC,\star_{\,\cT/\sC}]\subseteq {\wh \cK}(\cT/\sC)$, and we can apply induction.
If $\sC$ is not a minimal collision, by Lemma \ref{collisionminimalchar} we can find two distinct minimal collisions $\sC_1$ and $\sC_2$ such that $\sC_1 \rightarrow \sC$ and $\sC_2 \rightarrow \sC$.
Thus we can let $\sB\vee\sC_1 = \sB_1$ and $\sB \vee \sC_2 =\sB_2$.
\end{proof}

The following is standard.  

\begin{lemma}\label{completecondition}
Let $\cF$ be a fan in $\mathbb{R}^m$, then $\cF$ is complete if and only if $\cF$ has no boundary.
\end{lemma}

This is equivalent to the following statement.

\begin{lemma}\label{refinedcompletecondition}
Let $\cF$ be a fan in $\mathbb{R}^m$, then $\cF$ is complete if and only if $\cF$ is full-dimensional and for each codimension 
1 cone $\tau \in \cF$ there exist full-dimensional cones $\tau_1,\tau_2 \in \cF$ with $\tau_1 \neq\tau_2$ such that  $\tau <\tau_1$ and $\tau < \tau_2$.
\end{lemma}

\begin{definition}
A fan is \emph{pure} if each maximal cone has the same dimension.
\null\hfill$\triangle$
\end{definition}

\begin{lemma}\label{fulldimlemma}
The velocity fan $\cF(\cT)$ is a  full-dimensional pure fan in $\mathbb{R}^m$.

\end{lemma}
\begin{proof}
We must prove that each maximal cone in $\cF(\cT)$ has dimension $m$.
From Theorem \ref{mainvelocitystatement}, it follows that $\cK(\cT)$ is a ranked poset.
Let $\text{rk}(\sB)$ denote the rank of $\sB \in \cK(\cT)$.  We have that $\text{rk}(\sB) = \text{dim}(\tau(\sB))-1$, with the -$1$ accounting for the lineality space.
Thus, it suffices to prove that any maximal $n$-bracketing $\sB \in \cK(\cT)$ has $\text{rk}(\sB)=m-1$.  
We proceed by induction on $m$.
For $m=1$, the statement is clear, so we assume $m>1$.
Next, let $\sB \in \cK(\cT)$ be a maximal $n$-bracketing, and observe that because $m>1$ there must exist some $\sC \in \fX(\cT)$ with $\sC\leq \sB$.
Let $\sC \in \fX(\cT)$ be containment-minimal in $\sB$.
We know that $\sC$ is a minimal collision, otherwise we would have some $\sC'\in \fX(\cT)$ with $\sC' \neq \sC$, $\sC' \rightarrow \sC$, and $\sB\vee \sC' >\sB$, contradicting maximality of $\sB$.
As described in Definition \ref{quotientbracketingdef}, $\sB/\sC \in \cK(\cT/\sC)$ and it is clear that $[\sC, \sB] \cong [\sB_{\min},\sB/\sC]$.
Moreover, $\sB/\sC$ is maximal in $\cK(\cT/\sC)$, otherwise if we had some $\sB'>\sB/\sC$ we would be able to lift $\sB'$ to some $\wt \sB \in \cK(\cT)$ such that $\wt \sB >\sB$, e.g. by taking the joins of the preimages of the collisions less than $\sB'$ afforded by Lemma \ref{preimagelemma}.
As $\sC$ is a minimal collision, it satisfies the conditions of  Lemma \ref{collisionminimalchar}, therefore $|V(\cT/\sC)| = |V(\cT)|-1$.  It follows by induction that $\text{rk}(\sB/\sC)=m-2$, and $\text{rk}(\sC)=1$, hence $\text{rk}(\sB)=(m-2)+1 =m-1$.
\end{proof}

\begin{theorem}
\label{complete2}
The velocity fan $\cF(\cT)$ is complete.
\end{theorem}

\begin{proof}
We will verify the conditions of Lemma \ref{refinedcompletecondition}.  By Lemma \ref{fulldimlemma}, we know that $\cF(\cT)$ is full-dimensional, thus the set of all codimension 1 cones in $\cF(\cT)$ is nonempty; we can take a facet of a full-dimensional cone.  Let $\tau(\sB)$ be a codimension 1 cone in $\cF(\cT)$.  By Lemma \ref{fulldimlemma}, we know that $\sB$ is not maximal.  By Lemma \ref{codimension1lemma}, there exists maximal $n$-bracketings $\sB_1$ and $\sB_2$ such that $\sB < \sB_1$ and $\sB < \sB_2$.  Thus we can take $\tau_1 = \tau(\sB_1)$ and $\tau_2 = \tau(\sB_2)$.
\end{proof}

\

\section{Triangulating \texorpdfstring{$n$}{n}-associahedra}\label{triangulationsection}
The velocity fan is not smooth, it is not even simplicial.  
In this section we describe a canonical smooth flag triangulation of the reduced velocity fan on the same set of rays.
In the special case of the associahedron this triangulation is trivial; it agrees with the wonderful associahedral fan which is already 
 simplicial.
For concentrated $n$-associahedra such as the multiplihedron, this triangulation is a nestohedral fan (this will be elaborated upon in \S \ref{s:concentrated_realization}).  In \S \ref{nestohedralatlassection} we describe a general connection between our triangulation and nestohedra for all $n$-associahedra.  We apply the triangulated velocity fan as a tool for giving a recursive calculation of the normal vectors for the walls of the velocity fan, and for proving the unimodularity of the map $\Gamma$.  Additionally, we introduce a second, finer triangulation which generalizes the braid arrangement viewed as a triangulation of the wonderful associahedral fan.

\subsection{Nested collisions and triangulated \texorpdfstring{$n$}{n}-associahedra}

\

We begin by describing an abstract simplicial complex $\ccD$ on the set of collisions $\fX(\cT)$.  We will eventually show that this simplicial complex induces a smooth flag triangulation of the reduced velocity fan.
Recall Definitions \ref{collisionrelations} and \ref{collisionrelationsdisjoint}.

\begin{definition}\label{nestedcollisionsdef}
Let $\{\sC_i : 1\leq i \leq k\}$ be a collection of collisions.
Suppose that for all $1\leq i <j \leq k$ either 
\begin{enumerate}
\item \label{nestedcon1} $\sC_i\sim \sC_j$ or 
\item \label{nestedcon2} $\sC_i \rightarrow \sC_j$,
\end{enumerate} 

then we say that the collection $\{\sC_i\}$ is \emph{nested}.
\null\hfill$\triangle$
\end{definition}

\begin{lemma}
\label{growing}
Let $\{\sC_i :1\leq i \leq k\}$ be a nested collection of extended collisions, then they are a compatible collection of collisions. 
\end{lemma}

\begin{proof}
This is a special case of Lemma \ref{flagcondition} which says that if a collection of collisions are pairwise compatible, then they are compatible. 
\end{proof}

\begin{definition}
Let $V$ be a finite set. 
An abstract simplicial complex $\Delta$ on $V$ is a collection of subsets of $V$ which contains all singletons of $V$ and for any $Y \in \Delta$ and $Y'\subseteq Y$, then $Y' \in \Delta$.
\null\hfill$\triangle$
\end{definition}

\begin{definition}
Given a rooted plane tree $\cT$ of depth $n$, we define the \emph{triangulated $n$-associahedron} $\ccD$ to be the abstract simplicial complex whose faces are the nested collections of collisions.
\null\hfill$\triangle$
\end{definition}

\begin{remark}
Note that a nested collection of collisions may have more than one order for which it satisfies the conditions of Definition \ref{nestedcollisionsdef}, but this collection only determines a single face of $\ccD$.
\null\hfill$\triangle$
\end{remark}

We have a natural surjective poset morphism $\Omega: \ccD\rightarrow \cK(\cT)$.  For $Y \in \ccD$, we let
\begin{align}
\Omega(Y) = \bigvee_{\sC \in Y} \, \sC.
\end{align}

We now use the triangulated $n$-associahedra to define triangulations of the reduced $n$-bracketing complex and the reduced velocity fan.

\begin{definition}
Given $Y \in \ccD$, let
\begin{align}
\tau^{\met}(Y)= cone\{ \ell(\sC): \sC \in Y\}+\langle\ell(\sB_\min)\rangle_{\mathbb{R}}\,.
\end{align}

\noindent
The \emph{triangulated $n$-bracketing complex} is the collection of conical sets
\begin{align}
\cK^{\met}(\ccD) = \{ \tau^{\met}(Y):Y \in \ccD\}.
\end{align}

\noindent
The \emph{reduced triangulated $n$-bracketing complex} is the quotient
\begin{align}
{\overline \cK^{\met}}(\ccD) = \{ \tau^{\met}(Y)/\langle \ell(\sB_\min) \rangle_{\mathbb{R}}: Y \in \ccD\}.
\end{align}
\null\hfill$\triangle$
\end{definition}

\begin{definition}
Given $Y \in \ccD$, let 
\begin{align}
\tau(Y)=  cone\{ \rho(\sC): \sC \in Y\}+\langle 
{\bf 1}\rangle_{\mathbb{R}}\,.
\end{align}
\label{def:triangulated_velocity_fan}
The \emph{triangulated velocity fan}\footnote{Our use of the word ``triangulated'' is a bit dubious as this fan has a non trivial lineality space and thus is not simplicial.  On the other hand, the reduced triangulated velocity fan is simplicial.} is the collection of cones
\begin{align}
\cF(\ccD) = \{ \tau(Y):Y \in \ccD\}.
\end{align}
\label{def:reduced_triangulated_velocity_fan}
The \emph{reduced triangulated velocity fan} is the is the image of the triangulated velocity fan in tropical projective space
\begin{align}
{\overline \cF}(\ccD) = \{ \tau(Y)/\langle {\bf 1} \rangle_{\mathbb{R}}: Y \in \ccD\}.
\end{align}
\null\hfill$\triangle$
\end{definition}

\begin{figure}[ht]
\centering
\def\svgwidth{0.6\textwidth}
%% Creator: Inkscape 1.2 (dc2aeda, 2022-05-15), www.inkscape.org
%% Accompanies image file 'triangulatedW2,2.pdf' (pdf, eps, ps)
\begingroup%
  \makeatletter%
  \providecommand\color[2][]{%
    \errmessage{(Inkscape) Color is used for the text in Inkscape, but the package 'color.sty' is not loaded}%
    \renewcommand\color[2][]{}%
  }%
  \providecommand\transparent[1]{%
    \errmessage{(Inkscape) Transparency is used (non-zero) for the text in Inkscape, but the package 'transparent.sty' is not loaded}%
    \renewcommand\transparent[1]{}%
  }%
  \providecommand\rotatebox[2]{#2}%
  \newcommand*\fsize{\dimexpr\f@size pt\relax}%
  \newcommand*\lineheight[1]{\fontsize{\fsize}{#1\fsize}\selectfont}%
  \ifx\svgwidth\undefined%
    \setlength{\unitlength}{202.20366955bp}%
    \ifx\svgscale\undefined%
      \relax%
    \else%
      \setlength{\unitlength}{\unitlength * \real{\svgscale}}%
    \fi%
  \else%
    \setlength{\unitlength}{\svgwidth}%
  \fi%
  \global\let\svgwidth\undefined%
  \global\let\svgscale\undefined%
  \makeatother%
  \begin{picture}(1,1.39610166)%
    \lineheight{1}%
    \setlength\tabcolsep{0pt}%
    \put(0,0){\includegraphics[width=\unitlength,page=1]{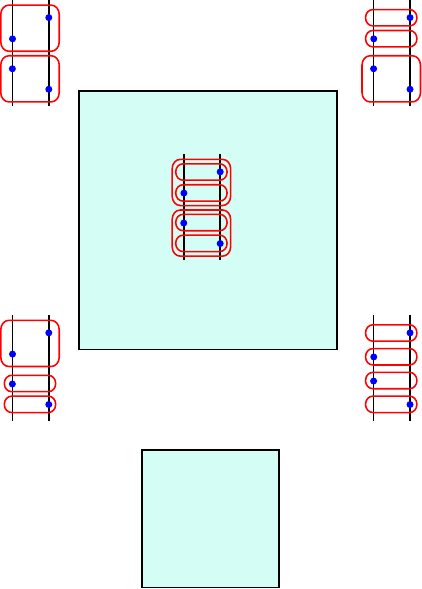}}%
    \put(0.19632419,1.14342437){\makebox(0,0)[lt]{\lineheight{1.25}\smash{\begin{tabular}[t]{l}$(1,0,2,0)$\end{tabular}}}}%
    \put(0.19632419,0.58800106){\makebox(0,0)[lt]{\lineheight{1.25}\smash{\begin{tabular}[t]{l}$(1,0,3,-1)$\end{tabular}}}}%
    \put(0.58894891,0.58800106){\makebox(0,0)[lt]{\lineheight{1.25}\smash{\begin{tabular}[t]{l}$(1,0,3,-2)$\end{tabular}}}}%
    \put(0.58894891,1.14342437){\makebox(0,0)[lt]{\lineheight{1.25}\smash{\begin{tabular}[t]{l}$(1,0,2,-1)$\end{tabular}}}}%
    \put(0,0){\includegraphics[width=\unitlength,page=2]{triangulatedW2,2.pdf}}%
  \end{picture}%
\endgroup%

\caption{
\label{triangulatedfigured}
An example of a nonsimplicial chamber in a reduced velocity fan ${\overline\cF}(\cT)$ (viewed from above) and its triangulation into two simplicial chambers in ${\overline\cF}(\ccD)$.
}
\end{figure}

\begin{definition}
Given a simplicial fan $\cF$ we can associate an abstract simplicial  complex $\Delta(\cF)$ with vertices corresponding to the rays of $\cF$ and a face $Y \in \Delta(\cF)$ if the corresponding rays generate a cone in $\cF$.
\null\hfill$\triangle$
\end{definition}

\begin{definition}\label{defflagcomplex}
An abstract simplicial complex $\Delta$ is a \emph{flag simplicial complex} if for every set of vertices $Y \subseteq V$ such that for any $u_1,u_2 \in Y$, $\{u_1,u_2\} \in \Delta$, we have $Y \in \Delta$.
\null\hfill$\triangle$
\end{definition}

\begin{definition}
A simplicial fan $\cF$ is flag if its associated simplicial complex is flag.
\null\hfill$\triangle$
\end{definition}

\begin{definition}
Let $N$ be a $\mathbb{Z}$-lattice.  Let $\tau$ be a simplicial cone in $N \otimes \mathbb{R}$.
We say that $\tau$ is \emph{ smooth} if there exist ray generators ${\bf v}^1, \ldots, {\bf v}^k$ for $\tau$ such that ${\bf v}^1, \ldots, {\bf v}^k$ can be extended to a basis ${\bf v}^1, \ldots, {\bf v}^m$ for $N$.
A fan $\cF$ in $N \otimes \mathbb{R}$ is {\emph smooth} if each cone in $\cF$ is smooth.
\null\hfill$\triangle$
\end{definition}

The following is the main result of this section.

\begin{theorem}\label{triangulationtheorem}
The reduced triangulated velocity fan ${\overline \cF}(\ccD)$ is a canonical smooth flag triangulation of the reduced velocity fan ${\overline \cF}(\cT)$ on the same set of rays.
\end{theorem}

Our choice of the word \emph{canonical} in Theorem \ref{triangulationtheorem} is intended to communicate that the definition of the triangulated velocity fan requires no input data beyond $\cT$.
We emphasize that  ${\overline \cF}(\ccD)$ and ${\overline \cF}(\cT)$ have the same set of rays which, as we know, correspond to the collisions.
At the level of toric varieties, this says that the toric variety of ${\overline \cF}(\ccD)$ defines a crepant resolution of singularities for the toric variety of ${\overline \cF}(\cT)$.  We note that a fan can only be considered simplicial, furthermore flag or smooth, if it is pointed, thus Theorem \ref{triangulationtheorem} would not make sense if stated for the nonreduced velocity and triangulated velocity fans.

\subsection{The collision mixed graph}

\

In this section we establish that the the triangulated $n$-associahedron is flag by reinterpreting this simplicial complex as the clique complex of a graph.
We note that the flagness of $\ccD$ does not yet immediately follow from Lemma \ref{flagcondition} because of the order required in Definition \ref{nestedcollisionsdef}.

\begin{definition}
A \emph{mixed graph} is a graph in which every edge is either directed or undirected.
\null\hfill$\triangle$
\end{definition}

We will introduce a mixed graph on $\fX(\cT)$.

\begin{definition}
Given a rooted plane tree $\cT$, we define the \emph{collision mixed graph} $G_\cT$ with vertex set $V(G_\cT)=\fX(\cT)$.
Let $\sC_1,\sC_2 \in \fX(\cT)$ with $\sC_1 \neq \sC_2$, then we have a directed edge $(\sC_1,\sC_2)  \in E(G_\cT)$ if $\sC_1 \rightarrow \sC_2$ and an undirected edge $\{\sC_1,\sC_2\} \in E(G_\cT)$ if $\sC_1 \sim \sC_2$.
\null\hfill$\triangle$
\end{definition}

\begin{definition}
Let $G_\cT$ be the collision mixed graph associated to the rooted plane tree $\cT$.
We define the \emph{undirected collision graph} to be the undirected graph, written $G^{\circ}_\cT$, obtained from $G_\cT$ by keeping the undirected edges, and replacing the directed edges with undirected edges.
\null\hfill$\triangle$
\end{definition}

\begin{definition}
Let $G$ be an undirected simple graph.
The \emph{clique complex of $G$}, denoted $\Delta(G)$, is the abstract simplicial complex with vertices indexed by vertices of $G$ and a face for every complete subgraph, i.e.\ clique, in $G$.
\null\hfill$\triangle$
\end{definition}

The following lemma is standard and straightforward to verify.

\begin{lemma}\label{flagclique}
An abstract simplicial complex $\Delta$ is a flag simplicial complex if and only if there exists a graph $G$ such that $\Delta$ is the clique complex of $G$.
\end{lemma}

We will now investigate $\Delta(G^\circ_{\cT})$.

\begin{lemma}\label{transitive}
The directed part of the collision mixed graph is transitive: if $\sC_1, \sC_2, \sC_3 \in \wh \fX(\cT)$ such that $\sC_1 \rightarrow \sC_2$ and $\sC_2 \rightarrow \sC_3$, then $\sC_1 \rightarrow \sC_3$.
\end{lemma}

\begin{proof}
This is clear from Definition \ref{collisionrelations}.
\end{proof}

\begin{lemma}\label{acyclic}
The directed part of the collision mixed graph is acyclic.
\end{lemma}

\begin{proof}
This is a corollary of Lemma \ref{transitive} and the fact that it is impossible to have a pair of distinct collisions $\sC,\sC'$ such that $\sC \rightarrow \sC'$ and $\sC'\rightarrow \sC$.
\end{proof}

\begin{proposition}\label{flagnessresult}
For a rooted plane tree $\cT$, we have
\begin{align}
\ccD=\Delta(G^\circ_{\cT}),
\end{align}
hence $\ccD$ is a flag simplicial complex.
\end{proposition}

\begin{proof}
For $Y \in \Delta_\cT$, it is clear that $Y \in \Delta(G^\circ_{\cT})$.  For the reverse inclusion, let $Y \in \Delta(G^\circ_{\cT})$.
We wish to prove that there exists a total order on the collisions in $Y$ so that they satisfy the conditions of being nested.
By Lemmas \ref{transitive} and \ref{acyclic} we may view the directed part of $Y$ as the set of relations in a poset for which we can take a linear extension, and this linear extension is the desired total order.  
\end{proof}

\subsection{Tree combinatorics of \texorpdfstring{$\ccD$}{D}}

\

In this section we will describe how the faces of $\ccD$ (hence the conical sets in ${\overline{ \cK}^{\met}}(\ccD)$, and the cones in ${\overline \cF}(\ccD)$) can be identified with certain rooted trees.
We remark that the main result of this section, Theorem \ref{treecombtheorem}, also follows from the results of \S \ref{nestohedralatlassection} when combined with Postnikov's theory of $\mathbb{B}$-trees, although we find this approach awkwardly indirect, and unnecessarily local, as it must first pass through the $\omega_{\sigma}$ maps with $\sigma$ varying over all $\cT$-shuffles.
Indeed the results of this subsection were obtained before the results of \S \ref{nestohedralatlassection}, and inspired the discovery of those results.
We first define the extended collision mixed graph.

\begin{definition}\label{extended}
The \emph{extended collision mixed graph} $\wh{G}_\cT$ is obtained from the collision mixed graph by adding a vertex corresponding to the minimum bracketing $\sB_\min$ with a directed edge $(\sC,\sB_\min)$ for each $\sC \in \fX(\cT)$.
The \emph{extended undirected collision graph} $\wh{G}^{\circ}_\cT$ is the undirected graph obtained from the extended collision mixed graph by undirecting the directed edges.
The \emph{extended triangulated $n$-associahedron} is defined as $\wh{\Delta}_\cT=\Delta(\wh{G}^{\circ}_\cT)$.
\end{definition}

Note that the above definition is consistent with our earlier definition that for each $\sC \in \fX(\cT)$, $\sC \rightarrow \sB_\min$ and $\sC \nsim \sB_\min$.  The purpose of Definition \ref{extended} is that $\sB_\min$ will conveniently play the role of the root vertex for certain trees, which otherwise would be forests.
This paradigm will be familiar to experts in associahedra, nestohedra, or  nested set complexes.

\begin{lemma}\label{outdeg1}
Let $Y\in \wh{\Delta}_\cT$, and let $\sC_1, \sC_2, \sC_3 \in Y$ such that $\sC_1 \rightarrow \sC_2$ and $\sC_1 \rightarrow \sC_3$, then either $\sC_2 \rightarrow \sC_3$ or $\sC_3 \rightarrow \sC_2$
\end{lemma}

\begin{proof}
Because $\sC_2, \sC_3 \in Y$, it must be that $\sC_2\sim \sC_3$, $\sC_2 \rightarrow \sC_3$, or $\sC_3 \rightarrow \sC_2$.
Because $\sC_1 \rightarrow \sC_2$ and $\sC_1 \rightarrow \sC_3$, it is not possible that $\sC_2\sim \sC_3$, and the statement follows. 
\end{proof}

We will now discuss cover relations in faces of an extended triangulated $n$-associahedron.

\begin{definition}
\label{coverdef}
Let $Y \in \wh{\Delta}_\cT$, and let $\sC_1,\sC_2 \in Y$ such that $(\sC_1,\sC_2)$ is a directed edge in the collision mixed graph.
We say that \emph{$\sC_2$ covers $\sC_1$ in $Y$} if there exists no $\sC_3 \in Y$ for which there is a directed path $\sC_1\rightarrow\sC_3\rightarrow\sC_2$ in $Y$.
\null\hfill$\triangle$
\end{definition}

\begin{lemma}\label{outdegreelem}
Let $Y$ be a maximal simplex in $\wh{\Delta}_\cT$, then for any collision $\sC \in Y$, there are at most two collisions in $Y$ which are covered by $\sC$ in $Y$.
\end{lemma}

\begin{proof}
Let $\{\sC_i: 1\leq i \leq s\}$ be an enumeration of all of the collisions in $Y$ which are covered by $\sC$ in $Y$. Suppose that $s\geq 2$.
It is not possible to have $\sC_i$ and  $\sC_j$ such that, without loss of generality, $\sC_i\rightarrow \sC_j$, as $\sC$ would not cover $\sC_i$, so it must be that $\sC_i\sim \sC_j$ for all $1\leq i,j \leq s$.
Because  $\sC \neq  \sC_i$ for all $1\leq i \leq s$, there must exist a minimum integer $k$, which is a function of $i$, and a pair of $k$-brackets $A_i \in \sC_i^k$ and $A_i' \in \sC^k$ such that $A_i \subsetneq A_i'$.
Because $s\geq 2$ and the $\sC_i$ are disjoint, we know that the fusion bracket for $\sC$ is distinct from the fusion brackets of the $\sC_i$. 
Thus, by Lemma \ref{whenthequotientofacollisionisacollision}, we may simultaneously contract all of the collisions $\{\sC_i: 1\leq i \leq s\}$ and the image of $\sC$ will be a collision which we will denote $\overline{\sC}$.
After this contraction, for each $i$, we find that 
 ${\overline A_i}'$, the image of each $A'_i$, will be a $k$-bracket in $\overline{\sC}^k$ containing at least two vertices of depth $k$.
We claim that the collision $\overline{\sC}$, is a minimal collision in the resulting contracted $n$-associahedron.
Suppose not and let $\overline{\sC'}$ be a collision with $\overline{\sC'}\rightarrow \overline{\sC}$, then the preimage $\sC'$ of $\overline{\sC'}$, as guaranteed by iterated application of Lemma \ref{liftdisjointcollisions}, is a collision such that $\{\sC'\} \cup Y \in \Delta$ contradicting the maximality of $Y$.
Thus by Lemma \ref{collisionminimalchar} the fusion $k$-bracket $A$ for $\overline{\sC}$ contains two singleton $k$-brackets.
It must be then that ${\overline A_i}'=A$ for all $i$, we have $s=2$, and the two  singleton $k$-brackets associated to the two vertices in $A^k$ are the images of the $A_i$ with $i \in \{1,2\}$.
\end{proof}

\begin{definition}
An abstract simplicial complex $\Delta$ is \emph{pure} if each maximal simplex (facet) has the same number of vertices.
\null\hfill$\triangle$
\end{definition}

\begin{lemma}
\label{maxnumbernested}
The extended triangulated $n$-associahedron $\wh{\Delta}_\cT$ is a pure complex of dimension $m-1$.
\end{lemma}

Let $Y=\{\sC_i :1\leq i \leq l\}$ be a maximal nested collection of extended collisions.
There are several ways to prove that $|Y|=m$.  
One natural way to attempt to prove this statement is to take a containment-minimal collision $\sC  \in Y$, contract $\sC$ and apply induction to the image of $Y$.  The immediate obstacle to this approach is that the image of $Y$ may not be a nested collection.
This subtlety is addressed in Lemma \ref{connstructingnestedcollisionsinthequotient}, which is later utilized for proving smoothness of the triangulated velocity fan.
One may use Lemma \ref{connstructingnestedcollisionsinthequotient} to canonically replace the image of $Y$ with a maximal nested set with the same cardinality, which by induction is equal to $|Y|-1$ thus proving the desired statement.
One may also derive Lemma \ref{maxnumbernested} from Proposition \ref{metrictriangulation} as we know that the dimension of $\overline{\cK}^{\met}(\cT)$ as an abstract fan is $m-1$.
We give another proof here which highlights some combinatorics not discussed elsewhere in this paper (and avoids forward referencing). 

\begin{proof}(\ref{maxnumbernested})
By the proof of Lemma \ref{outdegreelem}, for each $\sC_i \in Y$, there exists a unique $k$-bracket ${\wt A}_i \in \sC_i^k$ such that for every $k$-bracket $A \in \sC_j^k$ for  $j<i$, ${\wt A}_i^k\neq A^k$.
Let $\sB$ be the join of the collisions in $Y$.
The set of all ${\wt A}_i^k$ can be interpreted as forming a ``1-bracketing'' of the set $V^k(\cT)$, with vertices reordered to be compatible with the extended height partial for $\sB$.
If $Y$ is maximal then each of these 1-bracketings is maximal.
We know that a maximal 1-bracketing on a set with $t_k+1$ elements must have $t_k$ nonsingleton brackets, thus the total number of ${\wt A}_i$ is equal to $\sum_{k=1}^n t_k = m$, which is the cardinality of $Y$.
\end{proof}

\begin{remark}
\label{binarytreecharacterization}
In the case of $n=2$, the perspective offered in the previous proof is closely related to the \emph{stable tree pairs} of the second author \cite{b:2-associahedra}.
\null\hfill$\triangle$
\end{remark}

\begin{definition}\label{qrootedtreedef}
Let $D$ be a directed graph and $q$ a vertex in $D$.
Then $D$ is a \emph{$q$-rooted tree} if $D$ is acyclic, each vertex other than $q$ has exactly one outgoing edge, and $q$ has no outgoing edges.
A $q$-rooted tree is \emph{binary} if each vertex has at most two incoming edges.
\null\hfill$\triangle$
\end{definition}

\begin{theorem}\label{treecombtheorem}
Let $Y$ be a simplex in $\ccD$ with $\sB_{\min} \in Y$, then the set of directed cover relations in $Y$ form a $q$-rooted tree $\cT_Y$ where $q = \sB_\min$.
Moreover, if $Y$ is maximal then $\cT_Y$ is a binary tree with precisely $m$ vertices.
\end{theorem}

\begin{proof}
In order to prove that the cover relations $\cT_Y$ define a rooted tree, we observe that Lemma \ref{acyclic} implies the direct part of the the (extended) collision mixed graph is acyclic, and Lemma \ref{outdeg1} implies each vertex of $\cT_Y$ has out degree exactly one.
Each vertex in $\cT_Y$ other than $\sB_\min$ is incident to at least one outgoing edge, and $\sB_\min$ has no outgoing edges, hence $\sB_\min$ is the root of $\cT_Y$. 
 Assuming further that $Y$ is maximal, we see the statement that $\cT_Y$ is binary follows from Lemma \ref{outdegreelem}, and the statement that $\cT_Y$ has $m$ vertices is Lemma \ref{maxnumbernested}.
\end{proof}

\begin{figure}[ht]
\includegraphics[width=0.5\textwidth]{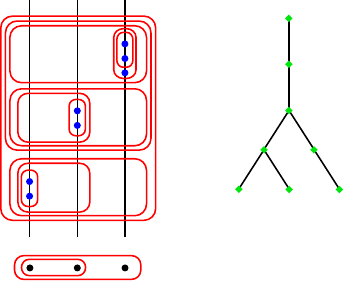}
\caption{
A maximal 2-bracketing, whose associated chamber in the velocity fan is already simplicial, and the associated rooted binary tree.
}
\label{binarytreeforbracketing}
\end{figure}

\subsection{\texorpdfstring{${\overline{ \cK}^{\met}}(\ccD)$ is a triangulation of $\overline{\cK}^{\met}(\cT)$}{KmetD is a triangulation of KmetT}}

\

In this subsection we prove Proposition \ref{metrictriangulation}, which states that ${\overline{ \cK}^{\met}}(\ccD)$ is a triangulation of $\overline{\cK}^{\met}(\cT)$.

\begin{lemma}\label{uniqueexpress}
Let $\ell_{\sB}$ be a metric $n$-bracketing.
There exists a unique expression
\begin{align}
\ell_{\sB} = \sum_{i=0}^s \lambda_i\ell(\sC_i)
\end{align}
with $\lambda_i \in \mathbb{R}$ and $\lambda_i > 0$ for $i>0$, where $\{\sC_i :1\leq i \leq s\}$ is a collection of nested collisions, and $\sC_0=\sB_\min$. 
\end{lemma}

\begin{proof}
The existence of such an expression is demonstrated by a slightly refinement of the proof of Lemma \ref{conicalcomb}: instead of picking a collision $\sC$ in the support of $\ell_{\sB}$ arbitrarily, take $\sC$ containment-minimal in the support of $\ell_{\sB}$ and again take $t$ the minimum value of $\ell_{\sB}(A)$ for $A$ a nontrivial bracket in $\sC$.
By induction on the number of brackets in the support of a metric $n$-bracketing, we have that $\ell_{\sB}-t\ell(\sC)$ admits an expression of the desired form, and by our choice of $\sC$, it follows that $\ell_{\sB}$ does as well.

Suppose that there exists two distinct expressions
\begin{align}
\ell_{\sB} = \sum_{i=1}^s \lambda_i\ell(\sC_i)=\sum_{i=1}^t \gamma_i\ell({\underline \sC}_i)
\end{align}
with $\lambda_i, \gamma_i \in \mathbb{R}$, and $\lambda_i, \gamma_i> 0$ for $i>0$, such that $\{\sC_i :1\leq i \leq s\}$ and $\{{\underline\sC}_i :1\leq i \leq t\}$ are collections of nested collisions, and ${\sC}_0={\underline\sC}_0=\sB_\min$.
Let $\sC$ be an extended collision which is containment-minimal in $\sB$.
We claim that there exists some $1\leq q\leq s$ and $1\leq r\leq t$ such that $\sC_q ={\underline\sC}_r= \sC$.
Suppose for contradiction that there exists no $q$ such that $\sC_q=\sC$.
For each $k$ with $1\leq k\leq n$, and $A \in \sC^k$ essential, there exists some $i$ such that $A \in \sC^k_i$.
Let $\sC_i$ be a collision which maximizes the number of common brackets with $\sC$.
Because $\sC_i\neq \sC$ we know that there exists some $k$ with $1\leq k\leq n$, and $A_1 \in \sC^k$ such that $A_1 \notin \sC_i^k$.
There exists some  $\sC_j$ such that $A_1 \in \sC_j^k$.
By the maximality assumption on $\sC_i$, we know that there exists some $l$ with $1\leq l\leq n$ and $A_2\in \sC^l$ such that $A_2\in \sC_i^l$, but $A_2\notin \sC_j^l$, otherwise $\sC_j$ would have more brackets in common with $\sC$ than $\sC_i$.  

Because $\sC$ is containment-minimal in $\sB$, and $\sC \nsim \sC_i$, $\sC \nsim \sC_j$, we must have $\sC \rightarrow \sC_i$, $\sC \rightarrow \sC_j$.  Hence there must exist some $A_1' \in \sC_i^k$ and $A_2' \in \sC_j^l$ such that $A_1\subsetneq A_1'$ and $A_2\subsetneq A_2'$.
This implies that $\sC_i \nrightarrow \sC_j$ and $\sC_j \nrightarrow \sC_i$.
We observe that $A_1$ and $A_2$ project to the fusion bracket for $\sC$, hence $\sC_i \nsim \sC_j$.
This contradicts our assumption that the pair $\sC_i$ and $\sC_j$ are nested.
It follows that there must exist indices $q$ and $r$ such that $\sC_q ={\underline\sC}_r= \sC$.
We may assume without loss of generality that $\lambda_q\leq \gamma_r$.
We have by Scholium \ref{sch:subtract} that $\ell_{\sB}-\lambda_q\ell(\sC)=\ell_{\sB'}$  is a metric $n$-bracketing, and we have the expression 
\begin{align}
\ell_{\sB'} = \sum_{i=0,i\neq q}^s \lambda_i\ell(\sC_i)= \biggl(\sum_{i=0, i\neq r}^t \gamma_i\ell({\underline \sC}_i)\biggr)+(\gamma_r-\lambda_q)\ell({\underline\sC}_r).
\end{align}

The total number of summands appearing in the above expression for $\ell_{\sB'}$ is at least 1 less than the total number of summands appearing in the expression for $\ell_{\sB}$.
It follows by induction on $s+t$ that the two expressions for $\ell_{\sB'}$ above are equal, hence the two expressions for $\ell_{\sB}$ above are also equal.
\end{proof}

\begin{lemma}\label{triangulatedmetriccomplexresult}
The collection of conical sets ${{ \cK}^{\met}}(\ccD)$ forms an abstract fan.
\end{lemma}

\begin{proof}
For proving that ${{ \cK}^{\met}}(\ccD)$ is an abstract fan, we first verify conditions (\ref{absfandef1}) and (\ref{absfandef2}) of Definition \ref{abstractfandef}.
Given a conical set $\tau$ in  ${\cK^{\met}}(\ccD)$, it follows from Lemma \ref{generatorsconsetlem} that each face of $\tau$ is generated by a subset of the ray generators of $\tau$, together with the lineality space generated by $\ell(\sB_\min)$.
By Lemma \ref{uniqueexpress}, each subset of the rays together with $\ell(\sB_\min)$ generate a face of $\tau$.
Because $\ccD$ is a simplicial complex, any subset of the rays of $\tau$ together with $\ell(\sB_\min)$ generate a conical set in ${\cK^{\met}}(\ccD)$, thus condition (\ref{absfandef1}) is satisfied. 

We verify condition (\ref{absfandef2}).
Let $Y_1, Y_2 \in \ccD$ with $Y_1 = \{ \sC_1, \ldots, \sC_k\}$ and $Y_2 = \{ {\underline \sC}_1, \ldots, {\underline \sC}_l\}$.
Let $\tau_1, \tau_2 \in { {\cK}^{\met}}(\ccD)$ with $\tau_1 = \tau^{\met}(Y_1)$ and $\tau_2 = \tau^{\met}(Y_2)$.
We claim that $\tau_1 \cap \tau_2 = \tau^{\met}(Y_1 \cap Y_2)$ implying $\tau_1\cap\tau_2 \in{ {\cK}^{\met}}(\ccD)$.
Clearly, $\tau_1 \cap \tau_2 \supseteq \tau^{\met}(Y_1 \cap Y_2)$.
For establishing opposite inclusion, let $\ell_\sB \in \tau_1 \cap \tau_2$.
There are expressions 
\begin{align}
\ell_{\sB} = \sum_{i=0}^k \lambda_i\ell(\sC_i)=\sum_{i=0}^l \gamma_i\ell({\underline \sC}_i)
\end{align}
with $\lambda_i, \gamma_i\geq 0$ for $i>0$, and $\sC_0={\underline \sC}_0=\sB_\min$.
Restricting to the support of these sums, we find ourselves in the setting of Lemma \ref{uniqueexpress}, which implies that the two expressions are the same.
Hence, the ray generators which appear in the sums must belong to $\tau^{\met}(Y_1 \cap Y_2)$, so $\ell_\sB \in \tau^{\met}(Y_1 \cap Y_2)$.

For verifying that each conical set in ${{ \cK}^{\met}}(\ccD)$ is an abstract cone, we may apply Lemma \ref{abstractsubconelem}.  
\end{proof}

\begin{scholium}\label{2ndinetermetricbracketings}
Let $Y \in \ccD$ with $Y = \{\sC_1, \ldots \sC_k\}$.  Define $\tau_{\mathbb{Z}}^{\met}(Y)$ to be the metric $n$-bracketings $\ell_{\sB} \in \tau^{\met}(Y)$ such that for each $A \in \sB^k$, we have $\ell_{\sB}(A) \in \mathbb{Z}$; these are the integer points of $\tau^{\met}(Y)$.
It follows from the above argumentation that each $\ell_{\sB} \in \tau_{\mathbb{Z}}^{\met}(Y)$ is uniquely expressible as
\begin{align}
\ell_{\sB} = \sum_{i=0}^k\lambda_i\ell(\sC_i)
\end{align}
with $\sC_0 =\sB_{\min}$, $\lambda_i \in {\mathbb{Z}}$, and $\lambda_i\geq 0$ for $i\neq 0$. 
\null\hfill$\triangle$
\end{scholium}

\begin{lemma}\label{refinementofmetriccomplex}
The conical set complex ${{ \cK}^{\met}}(\ccD)$ is a refinement of ${{ \cK}^{\met}}(\cT)$.
\end{lemma}

\begin{proof}
Let $\sB \in \cK(\cT)$ so that $\tau^{\met}(\sB) \in {{ \cK}^{\met}}(\cT)$.
We must show that $\tau^{\met}(\sB)$ is a union of conical sets in ${ \cK}^{\met}(\ccD)$.
Let $\ell_{\sB'} \in \tau^{\met}(\sB)$.
We know by Lemma \ref{uniqueexpress} that  we can express  
\begin{align}
\ell_{\sB'} =\sum_{i=0}^k \lambda_i\ell(\sC_i),
\end{align}
where the set of collisions $Y=\{{\sC}_i:1\leq i \leq k\}$ is nested and ${\sC}_0=\sB_\min$, hence $\ell_{\sB'} \in \tau^{\met}(Y)$.
For each $i$ with $1\leq i \leq k$, we have that ${\sC}_i \leq \sB$, thus $\ell ({\sC}_i) \in \tau^{\met}(\sB)$ and $\tau^{\met}(Y) \subseteq \tau^{\met}(\sB)$.
Thus $\tau^{\met}(\sB)$ is a union of conical sets in ${ \cK}^{\met}(\ccD)$.
\end{proof}

\begin{proposition}\label{metrictriangulation}
The conical set complex ${\overline{ \cK}^{\met}}(\ccD)$ is a triangulation of ${\overline{ \cK}^{\met}}(\cT)$.  
\end{proposition}

\begin{proof}
First observe that by Lemma \ref{triangulatedmetriccomplexresult} combined with Lemma \ref{lem:abstract_fan_equivalence_relation}, we have that ${\overline{ \cK}^{\met}}(\ccD)$ is an abstract fan.
Moreover, by Lemma \ref{refinementofmetriccomplex} and Lemma \ref{quotientrefinement},  ${\overline{ \cK}^{\met}}(\ccD)$ refines ${\overline{ \cK}^{\met}}(\cT)$.
It remains to check that every conical set in ${\overline{ \cK}^{\met}}(\ccD)$ is simplicial.  Let $\tau \in {\overline{ \cK}^{\met}}(\ccD)$ and $\ell \in \tau$.
By Lemma \ref{uniqueexpress}, we have that $\ell$ can be expressed uniquely as a conical combination of the rays of $\tau$, thus $\tau$ is simplicial. 
\end{proof}

Building on the ideas of this subsection, we briefly describe an algorithm for expressing a vector ${\bf v}$ as a conical combination of the rays of a cone in the velocity fan, moreover, the rays of a cone in the triangulated velocity fan.
This algorithm can alternately be interpreted as providing a method for computing $\Gamma^{-1}$.

Let ${\bf v}^0 = {\bf v}$.
At the $k$th step of the algorithm we have constructed a tuple $M^k=({\bf v}^k,\sC^k,\lambda_k)$ such that
\begin{align}
{\bf v}^k+\lambda_k{\rho}(\sC_k) = {\bf v}^{k-1}
\end{align}
so that
\begin{align}
{\bf v} = {\bf v}^k +\sum_{i=1}^k \lambda_i{\rho}(\sC_i)
\end{align}
and $\{\sC_i:1\leq i \leq k\}$ is a nested collection of collisions.

We stop if ${\bf v}^k = {\bf 0}$.  Otherwise, take $\sC^{k+1} \in \Xi({\bf v}^k )$ and let
\begin{align}
\lambda_{k+1} \coloneqq sup\{t \in \mathbb{R}_{\geq 0}: \sC^{k+1} \in \Xi({\bf v}^k-t {\rho}(\sC^{k+1}) \}.
\end{align}

To prove that this algorithm works, one can apply $\Gamma^{-1}$ to ${\bf v}$ and check that the algorithm agrees with the one outlined at the beginning of Lemma \ref{uniqueexpress}.
This assumes that the velocity fan is complete so that $\Gamma$ surjects onto $\mathbb{R}^m$. 
It is, perhaps, surprisingly difficult to prove directly that this algorithm succeeds without first knowing that the velocity fan is complete.

\subsection{\texorpdfstring{${\overline{\cF}}(\Delta_\cT)$}{FD} is a smooth triangulation of \texorpdfstring{${\overline{\cF}}(\cT)$}{FT}}

\

In this subsection, we establish Theorem \ref{triangulationtheorem} which is the main result of \S \ref{triangulationsection}.
If $\{\sC_i:1\leq i\leq k\}$ is a collection of nested collisions then it is tempting to expect that their image $\{\sC_i/\sC_1:1\leq i\leq k\}$ in $\cT/\sC_1$ forms a collection of nested collisions.
This is nearly true, except for the important fact that they may not form a collection of collisions!
In particular, if $\sC_i$ and $\sC_1$ have the same fusion bracket, then $\sC_i/\sC_1$ may not be a collision.
The following lemma effectively addresses this subtlety in a way which will allow us to utilize induction for proving that the triangulated velocity fan is smooth.

\begin{lemma}\label{connstructingnestedcollisionsinthequotient}
Let $\{\sC_i:1\leq i\leq k\}$ be a collection of nested collisions which all have the same fusion bracket and are maximal with this property.
For each $\sC_i$, we let $\{\sC_{i,r}:1\leq r\leq s_i\}$ denote the collection of pairwise disjoint collisions in $\fX(\cT/\sC_1)$ such that
\begin{align}
\sC_i/\sC_1 = \bigvee_{r=1}^{s_i} \sC_{i,r}\,,
\end{align}
as is guaranteed by Lemma \ref{quotientofcollisionbycollision}.
Then there exists a unique $r_i$ with $1\leq r_i\leq s_i$ such that $\sC_{i,r_i} \notin \{\sC_{i-1,r}: 1\leq r\leq s_{i-1}\}$, and  $\{\sC_{i,r_i}:1\leq i\leq k\}$ is a nested collection of collisions with
\begin{align}
\bigvee_{i=1}^k\sC_i/\sC_1 = \bigvee_{i=1}^k \sC_{i,r_i}\,.
\end{align}
\end{lemma}

\begin{proof}
Because $\{\sC_i:1\leq i\leq k\}$ is a collection of nested collisions with the same fusion bracket, it follow that no pair are disjoint, hence $\sC_i \rightarrow \sC_j$ for all $1\leq i<j\leq k$.
Because $\sC_{i-1} \neq \sC_i$ there must exist some $r_i$ such that $\sC_{i,r_i} \notin \{\sC_{i-1,r}: 1\leq r\leq s_{i-1}\}$.
Suppose that there exist $1\leq r_i < r_j \leq s_i$ such that $\sC_{i,r_i}, \sC_{i,r_j} \notin \{\sC_{i-1,r}: 1\leq r\leq s_{i-1}\}$.
Let $\overline \sC$ be the collection $\{\sC_{i,r_i}\} \cup \{\sC_{{i-1},r}: \sC_{i,r_i} \sim \sC_{{i-1},r}, 1\leq r\leq s_{i-1}\}$.
Then we can apply Lemma \ref{liftdisjointcollisions} to conclude that $\overline \sC$ is the image in $\fX(\cT/\sC_1)$ of a collision $\sC \in \fX(\cT)$.
Moreover, we have that $\sC \neq \sC_{i-1}$, $\sC \neq \sC_{i}$, and $\sC_{i-1} \rightarrow \sC\rightarrow \sC_i$.
Thus $\{\sC\}\cup \{\sC_i:1\leq i\leq k\}$ is a nested collection of collisions with the same fusion bracket which properly contains $\{\sC_i:1\leq i\leq k\}$, but this contradicts the maximality of the nested collection $\{\sC_i:1\leq i\leq k\}$.
\end{proof}

The following two lemmas are by definition.

\begin{lemma}\label{smoothchamber}
If a cone $\tau$ is smooth, then all of the faces of $\tau$ are smooth. 
\end{lemma}

\begin{lemma}\label{smoothtropical}
Let $N = \mathbb{Z}^m/\langle {\bf 1} \rangle_{\mathbb{Z}}$ so that $\mathbb{T}^{m-1} = N \otimes \mathbb{R}$.
Let $\tau$ be an $m-1$ dimensional cone in $\mathbb{T}^{m-1}$, then $\tau$ is smooth if and only if the primitive ray generators ${\rho}_1, \ldots, {\rho}_{m-1}$ for $\tau$ are such that the $m\times m$ matrix $M_{\tau}$ with columns ${\rho}_1, \ldots, {\rho}_{m-1}$ and additional column ${\bf 1}$ satisfies  $M_{\tau} \in GL_{m}(\mathbb{Z})$.
\end{lemma}

\begin{lemma}\label{smooth}
Each cone in ${\overline \cF}(\ccD)$ is smooth.
\end{lemma}

\begin{proof}(\ref{smooth})
Let $Y = \{\sC_i :1\leq i \leq k\}$ be a maximal nested collection of collisions.
For establishing the statement of the lemma, it suffices by Lemma \ref{smoothchamber} to prove that the cone $\tau =\tau(Y)\in {\overline \cF}(\ccD)$ is smooth.  By Lemma \ref{maxnumbernested}, we know that $k=m-1$.
Construct the matrix $M_{\tau}$ described in Lemma \ref{smoothtropical} with $\rho_i = \rho(\sC_i)$.
For completing the proof, it suffices by Lemma \ref{smoothtropical} to prove that $M_{\tau} \in GL_m(\mathbb{Z})$, i.e.\ $det(M_{\tau})=1$.
The collision $\sC_1$ is a minimal collision, otherwise we would be able to enlarge $Y$ by adding a collision $\sC$ with $\sC \rightarrow \sC_1$.
Subtract the column ${\rho}(\sC_1)$ from the column ${\rho}(\sC_i)$ for each $i$ such that $\sC_1\rightarrow \sC_i$.
Additionally, subtract the column $\rho(\sC_1)$ from to last column ${\bf 1}$ and let the resulting matrix be $M'_\tau$.
Let $\sigma$ be a compatible shuffle for $\sC_1$.
Multiply $M'_\tau$ by $P_{\sigma}^T$ on the left to obtain a matrix $M''_\tau$.
Then we may use Lemma \ref{refinedcontractionlemma}, to calculate the column of $M''_\tau$ corresponding to each $\sC_i$.
Let $Y' = \{\sC'_i:1\leq i\leq l\} \subsetneq \{\sC_i:1\leq i\leq k\}$ be the collection of collisions such that for each $1\leq i\leq l$ we have $\sC_1\rightarrow \sC'_i$, $\sC_1\neq \sC'_i $, and $\sC'_i$ has the same fusion bracket as $\sC_1$.
Beginning with $i=1$ and proceeding to $i=l$, subtract the columns of $M''_\tau$ corresponding to $\sC'_j$ from the column of $M''_\tau$ corresponding to $\sC'_{i+1}$ if $j\leq i$ and $\sC'_{i,r_i}=\sC'_{j,r}$ for some $1\leq r \leq s_i$.
Let the resulting matrix be $M'''_\tau$.
We have performed only column subtractions to go from from $M_{\tau}$ to $M'_\tau$ and from $M''_\tau$ to $M'''_\tau$.
Furthermore, by Lemma \ref{unimodperm} we know that $P_{\sigma}^T \in GL_m(\mathbb{Z})$, thus we have the equality of determinants: $det(M_{\tau})=det(M'_\tau)=det(M''_\tau)=det(M'''_\tau)$.

Let $i$ be the row which corresponds to the pair of depth $k$ vertices contained in the fusion of $\sC_1$ as guaranteed by Lemma \ref{collisionminimalchar}.
By Lemmas \ref{contractionlemma} and \ref{refinedcontractionlemma}, the matrix $M'''_\tau$ is such that the $(i,j)$th entry is 1 for $j=1$ and $0$ for $j>1$.
Using the Laplace expansion calculation of $det(M'''_\tau)$ along the $i$th row, we find that $det(M'''_\tau)$ is equal to the determinant of the principal submatrix obtained by deleting the $i$th row and the first column.
Let ${\Tilde{\tau}}$ be the cone in $\cF(\Delta_{\cT/\sC_1})$ which is generated by the rays associated to $\{\sC_i/\sC_1:\sC_i \in Y\setminus Y', i\neq 1\} \cup \{\sC'_{i,r_i}:\sC'_i \in Y'\}$ with $\sC'_{i,r_i}$ associated to $\sC'_i$ as in Lemma \ref{connstructingnestedcollisionsinthequotient}.  
By the arguments above, the principal submatrix in question is equal to $M_{{\Tilde{\tau}}}$.
By induction on the size the matrix, with base case $m=1$, we have that $1 = det(M_{{\Tilde{\tau}}}) = det(M'''_\tau) = det(M_{\tau})$.
\end{proof}

We are now ready to prove the main theorem of this section.

\begin{proof}(\ref{triangulationtheorem})
To see that ${{\overline \cF}(\ccD)}$ is a triangulation of ${{\overline \cF}(\cT)}$, we utilize Corollary \ref{reducedgammaiso} for Proposition \ref{metrictriangulation} 
combined with Lemma \ref{imagetriang}.
Flagness is established in Proposition \ref{flagnessresult}.
It is clear from the definitions that these fans use the same set of rays.  Smoothness follows from Lemma \ref{smooth}. 
\end{proof}

\begin{remark}
In our future work on projectivity for the velocity fan, we will investigate the triangulated velocity fan through the lens of smooth toric blowups, e.g.\ central stellar subdivisions.
\null\hfill$\triangle$
\end{remark}

\begin{corollary}\label{conecomplexlatticepropertyvelocityfan}
For $\sB \in \cK(\cT)$, the map $\Gamma$ induces a monoid isomorphism between the integer metric $n$-bracketings in a conical set $\tau^{\met}({\sB}) \in \cK^{\met}(\cT)$ and the integer points of the corresponding cone $\tau(\sB) \in \cF(\cT)$.
\end{corollary}

\begin{proof}
It suffices to prove this statement for $\Gamma$ viewed as a map $\Gamma:{{ \cK}^{\met}}(\ccD)\rightarrow \cF(\ccD)$.
By Lemma \ref{smooth}, we see that each integer point in $\tau(Y)$ is an integer combination of the ray generators and the statement follow by combining this observation with Scholium \ref{2ndinetermetricbracketings}, Proposition \ref{mainvelocitythm}, and the fact that $\Gamma$ maps the ray generators of ${{ \cK}^{\met}}(\ccD)$ to the ray generators of $\cF(\ccD)$.
\end{proof}

\subsection{The permutahedral \texorpdfstring{$n$}{n}-associahedron}\label{permutahedralassociahedronsubsection}

\

In this subsection, we describe a finer triangulation of the the triangulated velocity fan, which we call the permutahedral velocity fan.
For the case of $n=1$, the velocity fan and the triangulated velocity fan coincide and specialize to the wonderful associahedral fan, while the permutahedral velocity fan specializes to the braid arrangement.
Thus, this triangulation can be viewed as providing a new family of generalizations of the braid arrangement associated to rooted plane trees.
We note that several rooted plane trees can give the same permutahedral velocity fan, e.g.\ all of the concentrated $n$-associahedra presented in \S \ref{s:concentrated_realization} have permutahedral velocity fan equal to the braid arrangement.

We begin by describing a certain simplicial complex which will form the underlying combinatorial object for our triangulation.

\begin{definition}
\label{generalized_collision}
Let $\cT$ be a rooted plane tree and $\{\sC_1,\ldots, \sC_k\}$ be a collection of collisions such that for every $1\leq i<j \leq k$, we have $\sC_i \sim \sC_j$, then we define $\sD=\{\sC_1,\ldots, \sC_k\}$ to be a \emph{generalized collision}. We denote the collection of all generalized collisions by $\mathfrak{P}(\cT)$.
\null\hfill$\triangle$
\end{definition}

\begin{remark}
   We may identify a generalized collision $\sD =\{\sC_1,\ldots, \sC_k\}$ with the associated $n$-bracketing $\vee_{i=1}^k \sC_i$. 
\end{remark}

We now introduce a partial order on $\mathfrak{P}(\cT)$.

\begin{definition}
\label{perm_partial}
Given $\sD_1, \sD_2 \in \mathfrak{P}(\cT)$.
We say that $\sD_1 \leq_{\mathfrak{P}} \sD_2$ if for every $\sC_1 \in \sD_1$ there exists some $\sC_2 \in \sD_2$ such that $\sC_1 \rightarrow \sC_2$.
\null\hfill$\triangle$
\end{definition}

We recall the following construction from poset theory.

\begin{definition}
Given a partially ordered set $\cP$, the \emph{order complex} of $\cP$, which we denote $\mathscr{O}(\cP)$, is the abstract simplicial complex whose vertices are the elements of $\cP$ and whose faces are flags, i.e. chains, in $\cP$.
\end{definition}

\begin{definition}
Let $\cT$ be a rooted plane tree of depth $n$. 
The \emph{permutahedral $n$-associahedron} is $\mathscr{O}(\cT) \coloneqq \mathscr{O}(\mathfrak{P}(\cT))$.  
\null\hfill$\triangle$
\end{definition}

Given some $\sD \in \mathfrak{P}(\cT)$ we assign a metric $n$-bracketing and a vector:
\begin{align}
\ell(\sD) = \sum_{\sC \in \sD} \ell(\sC) \,\,\, , \,\,\, \rho(\sD)=\sum_{\sC \in \sD} \rho(\sC).
\end{align}

Given a face $Z \in \mathscr{O}(\cT)$, we assign a conical set and a cone:
\begin{gather}
\tau^{\met}(Z) = cone\{\ell(\sD):\sD \in Z\}+\langle {\ell(\sB_\min)}\rangle_{\mathbb{R}},
\\
\tau(Z) = cone\{\rho(\sD):\sD \in Z\}+\langle {\bf 1}\rangle_{\mathbb{R}}.
\nonumber
\end{gather}

\begin{definition}
The \emph{permutahedral metric $n$-bracketing complex} is
\begin{align}
\cK^{\met}(\mathscr{O}(\cT))=\{\tau^{\met}(Z):Z \in \mathscr{O}(\cT)\},
\end{align}
and the \emph{reduced permutahedral metric $n$-bracketing complex} is
\begin{align}
{\overline \cK}^{\met}(\mathscr{O}(\cT))=\{\tau^{\met}(Z)/\langle {\ell(\sB_\min)}\rangle_{\mathbb{R}} :Z \in \mathscr{O}(\cT)\}.
\end{align}
\null\hfill$\triangle$
\end{definition}

\begin{definition}
\label{permutahedral_assoc}
The \emph{permutahedral velocity fan} is
\begin{align}
\cF(\mathscr{O}(\cT))=\{\tau(Z):Z \in \mathscr{O}(\cT)\},
\end{align}
and the \emph{reduced permutahedral velocity fan} is
\begin{align}
{\overline \cF}(\mathscr{O}(\cT))=\{\tau(Z)/\langle {\bf 1}\rangle_{\mathbb{R}}:Z \in \mathscr{O}(\cT)\}.
\end{align}
\null\hfill$\triangle$
\end{definition}

The following is the main result of this subsection. 

\begin{figure}[ht]
\def\svgwidth{0.85\textwidth}
{\small
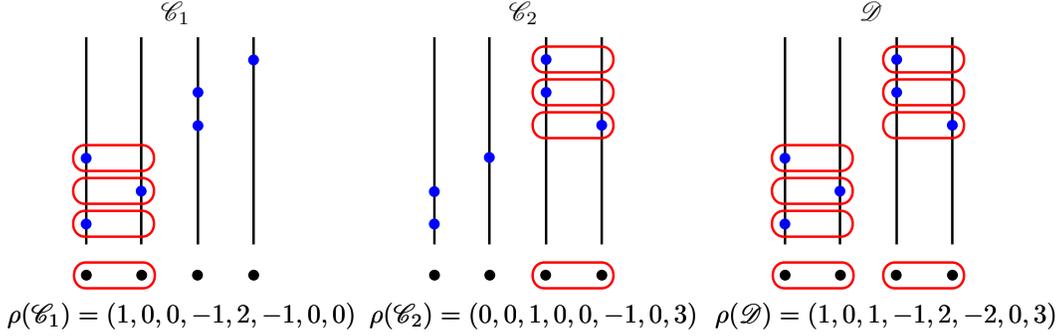
}
\caption{A pair of disjoint collisions, the generalized collision associated to the pair, and the corresponding ray generators.}
\label{generalizedcollisions}
\end{figure}

\begin{proposition}\label{permutahedralvelocitymainprop}
The reduced permutahedral velocity fan ${\overline \cF}(\mathscr{O}(\cT))$ is a smooth flag triangulation of  ${\overline \cF}(\ccD)$.
\end{proposition}

The proof of Proposition \ref{permutahedralvelocitymainprop} is very similar to, but easier than, the proof of Theorem \ref{triangulationtheorem}.
For this reason, we err on the side of concision.

We begin with a combinatorial observation about $\mathscr{O}(\cT)$ which, in the case $n=1$, specializes to a well-know relationship between the wonderful associahedral fan and the braid arrangement (see \cite{postnikov2008faces}).
Let $Y\in \ccD$ and let $\cT_Y$ be the corresponding rooted tree from Theorem \ref{treecombtheorem}.
Let $\cP(\cT_Y)$ be the poset whose as Hasse diagram is $\cT_Y$.
Let  $L(Y)$ denote the collection of linear extensions of $\cP(\cT_Y)$. For a fixed linear extension $\prec \, \in L(Y)$ and $\sC \in Y$, let
\begin{align}
max_{\prec}(\sC)
\coloneqq
\{\sC' \in Y: \sC' \preceq \sC,\, \nexists \, \, \sC'' \in Y \, \text{s.th.} \,\,  \sC''\preceq \sC, \sC' \neq \sC'' \,\text{and}\,\, \sC' \rightarrow \sC'' \}.
\end{align}
Let $\sD_{\sC,\prec} \in \mathfrak{P}(\cT)$ be
\begin{align}
\sD_{\sC,\prec} =\{ \sC':\sC'\in max_{\prec}(\sC)\}.
\end{align}
Let
\begin{align}
\Lambda:\bigsqcup_{Y \in \ccD} L(Y) \,\rightarrow \,\mathscr{O}(\cT)
\end{align}
be defined by
\begin{align}
\Lambda(\prec) = \{\sD_{\sC,\prec} : \sC \in Y\}.
\end{align}

\begin{proposition}
\label{linearextensionbijection}
The map $\Lambda$ is a bijection.
\end{proposition}
\begin{proof}
First observe that the linear extensions of $Y$ are those orders which allow $Y$ to be described as in Definition \ref{nestedcollisionsdef}.
With this observation in hand, the proposition follows readily from the definitions.
\end{proof}

The following Lemma is analogous to Lemma \ref{uniqueexpress}.

\begin{lemma}
\label{permuassuniqueexpress}
Let $\ell_{\sB}$ be a metric $n$-bracketing.
There exists a unique expression
\begin{align}
\ell_{\sB} = \sum_{i=0}^k \lambda_i\ell(\sD_i)
\end{align}
with $\sC_0=\sB_\min$, $\lambda_i \in \mathbb{R}$, and $\lambda_i \geq 0$ for $i>0$, where $\{\sD_i :1\leq i \leq k\}$ is a flag of generalized collisions. 
\end{lemma}

\begin{proof}
It is easy to  extend the argument at the beginning of Lemma \ref{uniqueexpress} to prove that  an expression of the desired form exists for $\ell_{\sB}$.
To show uniqueness, suppose we have two expressions for $\ell_\sB$ of the form in the statement of the lemma.
Look at the minimal generalized collisions in the flags for each sum.
Lemma \ref{uniqueexpress} gives that each collision appears in both expressions with the same coefficient.
We may apply this observation to show that the two minimal generalized collisions are the equal.
We can then subtract a positive scalar multiple of the corresponding metric generalized collision from both sides of the sum so that the number of terms drops by one in at least one of the sums and apply induction.
\end{proof}

The following propositions and their proofs are similar to Lemmas \ref{triangulatedmetriccomplexresult} and \ref{refinementofmetriccomplex} and their proofs, respectively.

\begin{proposition}\label{permutahedralmetricfan}
The permutahedral metric $n$-bracketing complex $\cK^{\met}(\mathscr{O}(\cT))$ is an abstract fan.
\end{proposition}

\begin{proposition}\label{permutahedralmetricrefinement}
The abstract fan ${{ \cK}^{\met}}(\mathscr{O}(\cT))$ is an refinement of ${{ \cK}^{\met}}(\ccD)$.
\end{proposition}

We now complete the proof of Proposition \ref{permutahedralvelocitymainprop}.

\begin{proof}(\ref{permutahedralvelocitymainprop})
By Proposition \ref{permutahedralmetricrefinement} and Lemma \ref{quotientrefinement}, ${{\overline{ \cK}}^{\met}}(\mathscr{O}(\cT))$ is a refinement of ${{\overline{ \cK}}^{\met}}(\ccD)$.
By Corollary \ref{reducedgammaiso} and Lemma \ref{imagetriang} we see that ${\overline \cF}(\mathscr{O}(\cT))$ is a fan which refines ${\overline \cF}(\ccD)$, and ${\overline \cF}(\mathscr{O}(\cT))$ is simplicial by construction.
The complex $\mathscr{O}(\cT)$ is flag as every order complex is flag.
For smoothness, given $Z \in \mathscr{O}(\cT)$, one can construct the matrix $M_{\tau(Z)}$ of primitive ray generators together with ${\bf 1}$ as in the statement of Lemma \ref{smoothtropical}.
Let $\prec \,=\Lambda^{-1}(Z) \in L(Y)$ as provided by Proposition \ref{linearextensionbijection}.
Take the matrix $M_{\tau(Y)}$ with its columns ordered according to the linear extension $\prec$.
To prove that $M_{\tau(Z)}$ has determinant 1, we can observe that it is related to $M_{\tau(Y)}$  by multiplication by an integer matrix $A$, and Proposition \ref{linearextensionbijection} describes $A$ as being uppertriangular with 1's on the diagonal, hence $det(M_{\tau(Z)})=det(M_{\tau(Y)})=1.$
\end{proof}

\subsection{1-dimensional \texorpdfstring{$n$}{n}-associahedra and walls of the velocity fan}
\label{wallssection}

\

In this section we provide a recursive calculation of the normal vectors for the walls of the velocity fan.
Generalized permutahedra are characterized as those polytopes whose edge directions are type $A$ roots, i.e.\ vectors of the form $e_i-e_j$.
This is equivalent to saying the vectors normal to the walls of the normal fan of a generalized permutahedron are type $A$ roots, which is in turn equivalent to saying that the normal fan should be a coarsening of the type $A$ Coxeter arrangement (the braid arrangement).
We refer the reader to \cite{ardila2020coxeter} for a concise treatment of this story.

The complexity of a family of polytopes which in some way extends generalized permutahedra can very roughly be measured by the complexity of the edge directions, i.e.\ the vectors orthogonal to the walls of their normal fan.
In this subsection, we will see  that the normal vectors for the walls of the velocity fan have unbounded support.  

For calculating these normal vectors we interpret them as normal vectors for walls of the triangulated velocity fan sitting inside walls of the nontriangulated velocity fan. We then utilize Lemma \ref{contractionlemma} and Lemma \ref{connstructingnestedcollisionsinthequotient} for giving a recursive reduction to a simpler case where the normal vectors of walls can be read of directly from the combinatorics of the corresponding nested collisions utilizing a characterization of the 1-dimensional $n$-associahedra.
This approach mimics the way in which one can understand the normal vectors for the walls of the wonderful associahedral fan by first characterizing these walls via the combinatorics of the unique 1-dimensional associahedron.
For fixed $n$ there is no longer a unique $1$-dimensional $n$-associahedron, there are quadratically many combinatorial types, nonetheless they fall into two distinct categories.

\begin{definition}
Let $\cT$ be a rooted plane tree.  We say that $\cK(\cT)$ is \emph{1-dimensional} if $\cK(\cT)$ has height 1 as a ranked poset, equivalently $\overline{\cF}(\cT)$ is a 1-dimensional fan.
\end{definition}

\begin{lemma}\label{1dimtrees}
Let $\cT$ be a rooted plane tree of depth $n$, and fix a path $\cP$ in $\cT$ of depth $n$.
Suppose that $\cK(\cT)$ is 1-dimensional, then either 
\begin{enumerate}
\item $\cT$ is obtained by simultaneously adding two leaves to a pair of (not necessarily distinct) vertices of depth at most $n-1$ in $\cP$, so that $\cT$ has three leaves, or
\item $\cT$ is obtained by adding a path of length two to a vertex at depth at most $n-2$ in $\cP$, so that $\cT$ has two leaves.
\end{enumerate}
\end{lemma}

\begin{proof}
If the reduced velocity fan is 1-dimensional, them $m=2$, hence $\cT$ has exactly two pairs of consecutive vertices at the same depth.
The only rooted plane trees of depth $n$ with this property are those described in the statement of the lemma.
\end{proof}

\begin{figure}[ht]
\includegraphics[height=1.5in]{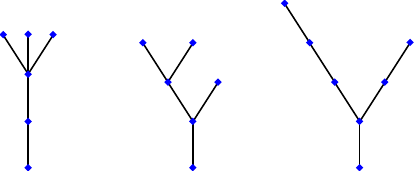}
\caption{Three different rooted planes trees with 1-dimensional $n$-associahedra.
The first two are of the form (1) in Lemma \ref{1dimtrees} and the third is of the form (2) in Lemma \ref{1dimtrees}.}
\label{1Dplanetrees}
\end{figure}

\begin{proposition}\label{braidonedimfan}
The velocity fan of every 1-dimensional $n$-associahedron is the braid arrangement $\mathcal{A}_2$.
\end{proposition}

\begin{proof}
Let $\cT$ be a rooted plane tree of depth $n$ whose corresponding $n$-associahedron is 1-dimensional.
We utilize the characterization given in Lemma \ref{1dimtrees}.
If $\cT$ has three leaves, then the two corresponding collisions have ray generators   $(1,0)$ and $(0,1)$, and lineality space generated by $(1,1)$ as desired.
On the other hand, if $\cT$ has two leaves, then the ray generators are $(1,0)$ and $(1,2)$ with lineality space generated by $(1,1)$.
We can subtract the lineality space generator $(1,1)$ from the ray generator $(1,2)$ to obtain the ray generator $(0,1)$.  
\end{proof}

\begin{remark}
We  highlight the curious fact that for 1-dimensional categorical $n$-associahedra associated to trees with two leaves, while the velocity fan is indeed the 1-dimensional braid arrangement, the velocity fan does not immediately give us the standard ray generators.
This observation is related to the fact that these types of 1-dimensional $n$-associahedra are the reason why general velocity fans do not live in the braid arrangement.
That is, suppose that $\cK(\cT)$ is an $n$-associahedron with the rank of $\cK(\cT)$ at least 2.
Then $\cF(\cT)$ is not a coarsening of the braid arrangement if and only if there exists a wall in $\cF(\cT)$ such that the identification from Lemma \ref{kbracketinreducedwall} gives a $k$-bracket $A$ associated to a  tree with two leaves.
\null\hfill$\triangle$
\end{remark}

In order to give a calculation of the  normal vectors for the walls of the velocity fan, it suffices to  calculate the normal vectors for certain walls of the triangulated velocity fan which are contained in walls of the velocity fan.
This simplifying observation is the reason why this calculation is presented in this section.

\begin{definition}\label{reducedwalldef}
Let $W \in \ccD$ with $W = \{\sC_i:1\leq i \leq m-2\}$, and let $\tau(W)$ be the corresponding wall of the triangulated velocity fan.
Suppose that $\vee_{i=1}^{m-2}\,\sC_i=\sB$ and suppose for each $\sC_i$ which is containment-minimal in $\sB$, $\sC_i$ is not a minimal collision, then we say that $\tau(W)$ is \emph{reduced}.
\null\hfill$\triangle$
\end{definition}

\begin{lemma}\label{pathtreereduced}
If $\tau(W)$ is a reduced wall of the triangulated velocity fan, then the tree $\cT_W$ corresponding to $W$, as described in Theorem \ref{treecombtheorem}, is a path. 
\end{lemma}

\begin{proof}
If $\cT_W$ is not a path, then it has at least two leaves $\sC_1$ and $\sC_2$, and if neither of these are minimal collisions, then the set of collisions $\{\sC\}$ such that $\sC \rightarrow \sC_1$ or $\sC \rightarrow \sC_2$ has size greater than 2.  Because $\cF(\cT)$ is a complete fan, this implies the codimension of $\tau(W)$ is greater than 1, a contradiction.
\end{proof}

\begin{lemma}
Suppose that $\tau(W)$ is a reduced wall of the triangulated velocity fan, then $\tau(W)$ is contained in a wall of the velocity fan.
\end{lemma}

\begin{proof}
Let $\sC_1$ be a minimal collision in $W$, then because $\tau(W)$ is reduced, we know that $\sC_1$ is not a minimal collision in $\cK(\cT)$.
Therefore the $n$-bracketing $\sB = \vee_{\sC \in W}\sC$ is not a maximal $n$-bracketing.  Because $\tau(W)$ is a wall of the  triangulated velocity fan, $\tau(\sB)$ must be a wall of the velocity fan, and we know $\tau(W) \subseteq \tau(\sB)$.
\end{proof}

For $W \in \ccD$ with codimension 1, let ${\bf v}(W)$ denote a primitive vector orthogonal to $\tau(W)$ and for $\sB \in \cK(\cT)$ with codimension 1, let ${\bf v}(\sB)$ denote a primitive vector orthogonal to $\tau(\sB)$.  We note that ${\bf v}(W)$ and ${\bf v}(\sB)$ are only determined up to sign, but this will be inconsequential for us. 
Let $\tau(\sB)$ be a wall of the velocity fan and let $\tau(W)$ be a wall of the triangulated velocity fan contained in $\tau(\sB)$.
For calculating ${\bf v}(\sB)={\bf v}(W)$, we provide a reduction to the case where $\tau(W)$ is reduced. We will need the following basic lemma from linear algebra.

\begin{lemma}\label{fundlinalg}
Suppose that ${\bf u},{\bf v_1,\ldots, v_k}\in \mathbb{R}^m$ such that ${\bf u}\perp {\bf v_i}$ for $1\leq i\leq k$.
Let $A:\mathbb{R}^m \rightarrow \mathbb{R}^m$ be an invertible linear transformation.
Then $A^{-1}({\bf u})\perp A^T({\bf v_i})$. 
\end{lemma}

Let $\tau(W)$ be a nonreduced wall of the triangulated velocity fan associated to a nested collection of collisions $\{\sC_i:1\leq i \leq m-2\}$.
Suppose that $\tau(W)$ is contained in a wall of the velocity fan associated to the $n$-bracketing $\sB$.
Suppose that $\sC_1$ is a minimal collision which is containment-minimal in $\sB$.

We begin by taking the following transformation of the ray generators for $\tau(W)$:
\begin{align}
{\overline {\rho}}(\sC_i)=
\begin{cases}
\,\, P_{\sigma}^T({\rho}(\sC_i)-{\rho}(\sC_1)) & \text{if} \,\, \sC_1 \rightarrow \sC_i \\
\,\, P_{\sigma}^T{\rho}(\sC_i) & \text{ if } \sC_1 \sim \sC_i .
\end{cases}
\end{align}

Lemmas \ref{contractionlemma} and \ref{refinedcontractionlemma} allow us to reinterpret these vectors.
We would like to invoke Lemma \ref{connstructingnestedcollisionsinthequotient} to reorganize the image of the quotients of the collisions as a collection of nested collisions, but we should careful because the nested collection is not maximal.
We observe that if the $\sC_{i,r_i} \notin \{\sC_{i-1,r}: 1\leq r\leq s_{i-1}\}$ from Lemma \ref{connstructingnestedcollisionsinthequotient} is not unique, then $\bigvee_{i=1}^{m-2}\,\sC_i$ is a maximal bracketing, and this would contradict the assumption that $\tau(W)$ is contained in a wall of the  velocity fan.
Reorganize the image of the quotients as a collection of collisions which form a wall $\tau(W')$ for $W' \in \Delta_{\cT/\sC_1}$.
Note that if we delete the 0 entry in ${\overline {\rho}}(\sC_i)$ corresponding to the pair of $k$-brackets which collided in $\sC_1$, then the span of the resulting vectors is equal to the span of $\tau(W')$.
 
By induction on the size of $\cT$, we can calculate ${\bf v}(W')$.
Extend ${\bf v}(W')$ by adding one extra entry $x$ in the place corresponding to the pair of $k$-brackets which collided in $\sC_1$ and let this new vector be ${\bf v}_x$.
To complete the calculation of ${\bf v}(W)$ we apply Lemma \ref{fundlinalg} and see that $(P_\sigma^T)^T({\bf v}_x)=P_\sigma({\bf v}_x)={\bf v}(W)$.
We only need to solve for the value of $x$ and to do this we can use the fact that ${\bf v}(W)\perp {\bf 1}$, i.e.\ the sum of the coordinates of ${\bf v}(W)$ is zero.
This is particularly easy: suppose the fusion bracket of $\sC_1$ is a $k$-bracket, then $\sigma$ is the identity permutation for all entries corresponding to pairs of consecutive $j$-brackets for $j\leq k$.
This implies that there is a single entry in $P_\sigma({\bf v}_x)$ equal to $x$ and this variable is not present anywhere else in the vector $P_\sigma({\bf v}_x)$, so we take $x$ equal to the negative of the sum of the remaining entries of $P_\sigma({\bf v}_x)$.

\begin{figure}[ht]
\includegraphics[width=0.2\textwidth]{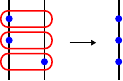}
\

\begin{gather*}
\rho(\sC) = (1,0,3),
\quad
\sigma = (1,312) \\
\vspace{1em}
\\
P_\sigma^T = \left(\begin{array}{rrr}
1 & 0 & 0 \\
0 & -1 & -1 \\
0 & 1 & 0
\end{array}\right),
\quad
(1,1) \perp (1,-1) \\
\vspace{1em}
\\
\left(\begin{array}{rrr}
1 & 0 & 0 \\
0 & -1 & 1 \\
0 & -1 & 0
\end{array}\right)\left(\begin{array}{r}
x \\
1 \\
-1
\end{array}\right)
=
\left(\begin{array}{r}
x \\
-2 \\
-1
\end{array}\right),
\quad
\left(\begin{array}{r}
x \\
-2 \\
-1
\end{array}\right)
\perp
(1,1,1)
\Rightarrow
x = 3.
\\
\vspace{1em}
\\
\text{Observe }
(3,-2,-1) \perp (1,0,3).
\end{gather*}
\caption{A recursive calculation of the normal vector for a wall of $W_{2,1}$.
\label{wallnormal}
}
\end{figure}

It remains to calculate ${\bf v}(W)$ when $\tau(W)$ is reduced, which forms the base case of the inductive calculations described above. 

\begin{lemma}
\label{kbracketinreducedwall}
Let $W =\{\sC_1, \ldots, \sC_{m-2}\}$ be a collection of nested collisions such that $\tau(W)$ is a reduced wall of the triangulated velocity fan, and $\sC_i \rightarrow \sC_j$ for all $1\leq i<j\leq m-2$.
Then there exists a unique $k$-bracket $A \in \sC_1^k$, with $k$ taken minimum, such that $\cK(A)$ is a 1-dimensional $k$-associahedron.
\end{lemma}

Our argumentation here is closely related to that of Lemma \ref{collisionminimalchar}.

\begin{proof}
Let $\sC$ be a minimal collision such that $\sC \rightarrow \sC_1$ and let $A \in \sC^k$ such that there exists an $A' \in \sC_1^k$ such that $A \subsetneq A'$, and we take $k$ minimum by assuming that $\pi(A)=\pi(A')$.
Then $\cK(A)$ has dimension at least 1 and $\sC|_A$ is a collision in $\cK(A)$.
Thus we can find some other collision in $\cK(A)$ and extend it, using $\sC_1$ outside of $A$, to obtain a collision $\sC' \in \fX(\cT)$ such that $\sC' \rightarrow \sC_1$, $\sC' \neq \sC$ and  $\sC' \neq \sC_1$.  
Therefore $\tau(W \cup \{\sC\})$ and $\tau(W \cup \{{\sC'}\})$ are two chambers of the triangulated velocity fan which contain $\tau(W)$, and thus are separated by $\tau(W)$ as a wall.
Because $\cF(\ccD)$ is a complete fan, it must be that if $\sC'' \in \fX(\cT)$ and $\sC'' \rightarrow \sC_1$ then $\sC'' \in \{\sC,\sC'\}$.
It follows that $\cK(A)$ must be 1-dimensional with exactly two collisions corresponding to the restrictions of $\sC$ and $\sC'$.  Moreover, there is no other $k$-bracket in $\sC_1$ with $k$ minimum such that the restriction is positive dimensional.
\end{proof}

\begin{lemma}
Let $\tau(W)$ be a reduced wall of the triangulated velocity fan, and assume the notation of Definition \ref{reducedwalldef}. 
Let $\sC_1$ be the unique collision which is minimal in $W$, and let $A$ be the unique $k$-bracket in $\sC_1$, with $k$ minimum, such that $\cK(A)$ is a 1-dimensional $n$-associahedron.  

\begin{enumerate}
\item\label{treedim1collision2} Suppose the tree $\cT(A)$ has three leaves.
Let $u_1$ and $u_2$ be the vertices of degree at least three which are parents of these leaves (noting that $u_1$ and $u_2$ are potentially equal).
Suppose the leftmost children of $u_1,u_2$ in $\cT(A)$ correspond to $u^{k_1}_{i_1},u^{k_2}_{i_2}\in V(\cT)$, respectively. Then $e^{k_1}_{i_1}-e^{k_2}_{i_2} \perp W$.  

\item\label{treedim1collision3} Suppose that $\cT(A)$ has two leaves.
Let $u$ be the nearest common ancestor of these leaves in $\cT(A)$, and suppose that $u^{k-2}_i \in V^{k-2}(\cT)$ corresponds to $u$.
Let the $u^{k-1}_{r},u^{k-1}_{r+1} \in V^{k-1}(\cT)$ correspond to the left and right children of $u$ in $\cT(A)$, respectively.
Let $u^{k}_{s_1},u^{k}_{s_2} \in V^k(\cT)$ correspond to the left and right grandchildren of $u$ in $\cT(A)$, respectively.
Then
\begin{align}
((s_2-s_1)\, e^{k-1}_{r} - \sum_{i=s_1}^{s_2-1}e^k_i \,)\perp W.
\end{align}
\end{enumerate}
\end{lemma}

Before giving a proof, we note that in case (\ref{treedim1collision3}), it must be that the vertices corresponding to the left and right children of $u$ in $\cT(A)$ are consecutive as the projection of $\cK(\pi(A))$ must be be zero dimensional and therefore the fusion bracket for $\sC_1$.

\begin{proof}
Let $W =\{\sC_1, \ldots, \sC_{m-2}\}$ be a collection of nested collisions and assume that $\sC_i \rightarrow \sC_j$ for all $1\leq i<j\leq m-2$.
For verifying that statement we first observe that ${\bf v} \perp W$ if and only if ${\bf v} \perp \rho(\sC_j)$ for $1\leq j \leq m-2$.
By Lemma \ref{pathtreereduced}, we know that the tree associated to $W$ is a path.
Hence for each $j$ with $1\leq j\leq m-2$, we have that there exists an essential $k$-bracket $A_j \in \sC_j^k$ such that $A\subseteq A_j$.
In particular, $A_j$ contains all of the singleton brackets contained in $A$.

We employ Lemma \ref{fundvellemma} in what follows.
If $\cT(A)$ has three leaves as in case (\ref{treedim1collision2}), then $\rho(\sC_j)^{k_1}_{i_1}=\rho(\sC_j)^{k_2}_{i_2} = 1$ for $1\leq j \leq m-2$.
On the other hand, if $\cT(A)$ has two leaves as in case (\ref{treedim1collision3}), then  ${\rho}({\sC}_j))^{k-1}_r=1$ and $\Delta^{k}_{s_1,s_2}({\rho}({\sC_j}))=s_2-s_1$ for all $1 \leq j \leq m-2$.
It follows that in both cases ${\bf v}(W)$ is equal to the vector described in the statement of the lemma.
\end{proof}

\begin{figure}[ht]
\def\svgwidth{0.15\textwidth}
%% Creator: Inkscape 1.2 (dc2aeda, 2022-05-15), www.inkscape.org
%% Accompanies image file 'wallnormals2.pdf' (pdf, eps, ps)
\begingroup%
  \makeatletter%
  \providecommand\color[2][]{%
    \errmessage{(Inkscape) Color is used for the text in Inkscape, but the package 'color.sty' is not loaded}%
    \renewcommand\color[2][]{}%
  }%
  \providecommand\transparent[1]{%
    \errmessage{(Inkscape) Transparency is used (non-zero) for the text in Inkscape, but the package 'transparent.sty' is not loaded}%
    \renewcommand\transparent[1]{}%
  }%
  \providecommand\rotatebox[2]{#2}%
  \newcommand*\fsize{\dimexpr\f@size pt\relax}%
  \newcommand*\lineheight[1]{\fontsize{\fsize}{#1\fsize}\selectfont}%
  \ifx\svgwidth\undefined%
    \setlength{\unitlength}{46.10243225bp}%
    \ifx\svgscale\undefined%
      \relax%
    \else%
      \setlength{\unitlength}{\unitlength * \real{\svgscale}}%
    \fi%
  \else%
    \setlength{\unitlength}{\svgwidth}%
  \fi%
  \global\let\svgwidth\undefined%
  \global\let\svgscale\undefined%
  \makeatother%
  \begin{picture}(1,1.5399833)%
    \lineheight{1}%
    \setlength\tabcolsep{0pt}%
    \put(0,0){\includegraphics[width=\unitlength,page=1]{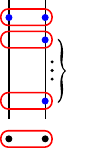}}%
    \put(0.72725476,0.77827682){\makebox(0,0)[lt]{\lineheight{1.25}\smash{\begin{tabular}[t]{l}$r-1$\end{tabular}}}}%
  \end{picture}%
\endgroup%

\caption{
A collision $\sC_1$ for a reduced wall in a triangulated velocity fan for a 2-associahedron with $A = (\{u^1_1,u^1_2\},\{u^2_1,u^2_{r+1}\})$ depicted at the top.
The associated normal vector is $r\, e^1_1 -\sum_{i=1}^{r}{e_i^2}$.
\label{wallnormal2}
}
\end{figure}

\section{Piecewise-unimodular maps to the braid arrangement}\label{piecewisesection}

Categorical $n$-associahedra cannot be realized by generalized permutahedra, nor do they appear to be realizable in any of their known extensions (see footnote \ref{permufootnote}).
On the other hand, our construction of the velocity fan generalizes the wonderful associahedral fan, and the results of \S\ref{triangulationsection} are reminiscent of nestohedra, so it is quite natural to ask whether there exists an explicit connection between general velocity fans and the braid arrangement.
In this section we give an affirmative answer to this question in two separate ways by constructing two different maps from the velocity fan to the braid arrangement.
For the wonderful associahedral fan, both of these maps specialize to the identity map.  For analyzing these maps, we view them as factoring through the isomorphism of the velocity fan with the metric $n$-bracketing complex.

The first map, addressed in subsection \ref{piecewisessubsection}, will be a piecewise unimodular map from the velocity fan to the braid arrangement.
This suggests an extension of the theory of generalized permutahedra which we call \emph{permutahedroids}.
The second map, addressed in subsection \ref{permutahedralatlassubsection}, will be a collection of  piecewise unimodular maps $\omega_{\sigma}$ defined on local shuffle charts $U_{\sigma}$ which cover the velocity fan.
While the maps $\omega_{\sigma}$ do not glue to give a global piecewise-linear function on the velocity fan, they have other nice properties:  each $\omega_{\sigma}$ is a piecewise unimodular isomorphism on the corresponding shuffle chart, and they provide a direct relationship to the theory of nestohedra via the triangulated velocity fan in subsection \ref{nestohedralatlassection}.

\subsection{Piecewise-unimodular transformations and permutahedroids}\label{piecewisessubsection}

\
  
We begin by defining a new extension of coarsenings of the braid arrangement and generalized permutahedra.
Recall the definition of the braid arrangement given in Definition \ref{defbraidarrangement}.

\begin{figure}
\centering
\includegraphics[width=0.6\textwidth]{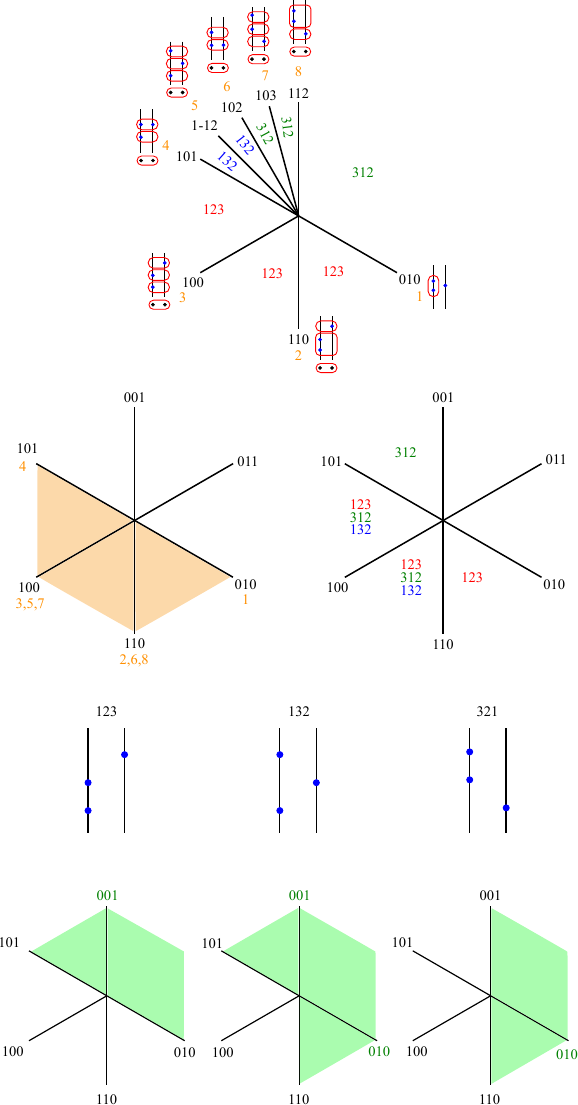}
\caption{
\label{fig:maps_to_braid_example}
The top figure is the velocity fan of $W_{2,1}$ with a clockwise ordering on the rays, and a labeling of chambers by $\cT$-shuffles.
On the left of the second row is the image of the velocity fan under the map $\Gamma^{\rho}_{\zeta}$; on the right is the union of the images of shuffle charts under the maps $\Gamma^{\rho}_{\omega_{\sigma}}$. 
\label{maps_to_braid_example}
On the third and fourth rows are the three nestohedral fans associated to the images of the shuffle charts in $W_{2,1}$, and the stars of rays which are deleted.
}
\end{figure}

\begin{definition}\label{permutahedroiddef}
Let $\mathcal{F}$ be a fan.
Suppose that there exists a piecewise-linear function $f$ defined on $\mathcal{F}$ such that for each cone $\tau \in \mathcal{F}$
\begin{enumerate}
\item\label{localpermfancond1} the restricted function $f|_{\tau}$ is unimodular and invertible,
\item\label{localpermfancond2} $f(\tau)$ is a union of cones in the braid arrangement.\footnote{Polyhedral cones which are unions of cones in the braid arrangement are sometimes called a \emph{preposet cones} as they are in bijection with, and encode much of the combinatorial structure of, preposets.
When the cones in question are full-dimensional, the corresponding preposets are honest posets, and by intersecting with standard cubes one recovers Stanley's order polytopes.
We refer the reader to the work of Postnikov-Reiner-Williams \cite{postnikov2008faces} for further details.}
\end{enumerate}

Then we say that $\mathcal{F}$ is a \emph{locally-braid fan}.
If $P$ is a polytope whose normal fan is a locally-braid fan, then we say that $P$ is a \emph{permutahedroid}.
\null\hfill$\triangle$
\end{definition}

Observe that when this $f$ is restricted to be the identity map, locally-braid fans specialized to subfans of coarsenings of the braid arrangement, and  permutahedroids specialize to generalized permutahedra.
We now state the main result of this subsection.

\begin{theorem}
\label{permutahedroidthm}
Each velocity fan is a locally-braid fan.
\end{theorem}
 
For a nontrivial example of a permutahedroid see Figure \ref{W22} depicting the 2-associahedron $W_{2,2}$.
Note that this polytope is not combinatorially equivalent to a generalized permutahedron.
In future work we aim to show that every velocity fan is projective, thus settling Bottman's conjecture in the affirmative, and producing a large family of examples of permutahedroids.
We are hopeful that the existence of these definitions may inspire researchers to discover other examples of such fans and polytopes.

In order to prove Theorem \ref{permutahedroidthm}, we will first define for each collision $\sC$ a 0-1 vector $\zeta(\sC)$.
We will then show that the map ${\ell}(\sC)\mapsto \zeta(\sC)$ extends to a piecewise-linear function on all of $\cK^{\met}(\cT)$.
By precomposing with $\Gamma^{-1}$ we obtain a piecewise-linear map $\Gamma^{\rho}_{\zeta}$ defined on the velocity fan which satisfied the conditions from Definition \ref{permutahedroiddef}.

\begin{definition}
Let $\sC \in \fX(\cT)$.
We take $\zeta(\sC)$ to be a vector of length $m$ such that $\zeta(\sC)^k_i = 1$ if the vertex $u^k_i$ belongs to an essential $k$-bracket $A \in\sC^k$ with some $u^k_j$ with $j>i$, and $\zeta(\sC)^k_i = 0$ otherwise.
\null\hfill$\triangle$
\end{definition}

\begin{definition}
Let $\{\sC_i:1\leq i \leq k\}$ be a collection of compatible collisions.
We define a cone associated to this collection:
\begin{align}
\tau_{\zeta}\{\sC_i\}
=
cone\{\zeta(\sC_1), \ldots, \zeta(\sC_k)\}+ \langle\bf{1}\rangle_{\mathbb{R}}.
\end{align}

Let $\sB$ be an $n$-bracketing and suppose $\{\sC_i\}$ is the set of all collisions contained in $\sB$.
We define the cone $\tau_{\zeta}(\sB) = \tau_{\zeta}\{\sC_i\}$.  Similarly, given a face $Y \in \Delta_\cT$, or $Y \in \mathscr{O}(\cT)$, we define $\tau_{\zeta}(Y) = \tau_{\zeta}\{\sC_i\}$ where $Y = \{\sC_i\}$.
\null\hfill$\triangle$
\end{definition}

\begin{definition}
Let $\sC_0, \ldots, \sC_k \in \wh \fX(\cT)$ be such that $\sC_0= \sB_\min$, and let $\lambda_i\in \mathbb{R}$ for $0\leq i\leq k$ with  $\lambda_i\geq 0$ for $i>0$.
Let $\ell_\sB\in{{\cK}}^{\met}(\cT)$ such that
\begin{align}
\ell_\sB = \sum_{i=0}^k \lambda_i \ell(\sC_i).
\end{align}
We define the map $\Gamma_{\zeta}:{{\cK}}^{\met}(\cT)\rightarrow \mathbb{R}^m$ as
\begin{align}
\Gamma_{\zeta}(\ell_{\sB})= \sum_{i=0}^k \lambda_i \zeta(\sC_i).
\end{align}
\end{definition}

\begin{lemma}\label{zetamap}
The map $\Gamma_{\zeta}$ defines a piecewise-linear map which restricts to a linear injection on each conical set in ${{\cK}}^{\met}(\cT)$.\footnote{Equivalently, $\Gamma_{\zeta}$ is a weak embedding of the cone complex ${{\cK}}^{\met}(\cT)$.}
\end{lemma}

We will prove Lemma \ref{zetamap} by an argument similar to the one used for proving Proposition \ref{mainvelocitythm}, although this version of the argument will be easier because, for example, we do not need to prove that $\Gamma_{\zeta}$ is globally injective (as it is not).  We will establish Lemma \ref{zetamap} by way of the following lemma.

\begin{lemma}
\label{zetametric}
Suppose that $\{\sC_0, \ldots, \sC_{l_1}\}$ and $\{\underline{\sC}_0, \ldots, \underline{\sC}_{l_2}\}$ are two sets of compatible extended collisions contained in $\sB$.
Let $ = \sC^0= \underline{\sC}_0=\sB_\min$ and take $\lambda_i,\gamma_i \in \mathbb{R}$ for $0\leq i\leq {l_1}$ and $0\leq i\leq {l_2}$ and $\lambda_i,\gamma_i \geq 0$ for  $i>0$.
Then
\begin{align}
\sum_{0\leq i\leq {l_1}}\lambda_i\ell(\sC_i)=\sum_{0\leq i\leq {l_2}}\gamma_i\ell(\underline{\sC}_i)
\end{align}
if and only if
\begin{align}
\sum_{0\leq i\leq {l_1}}\lambda_i\zeta(\sC_i)=\sum_{0\leq i\leq {l_2}}\gamma_i\zeta(\underline{\sC}_i).
\end{align}
\end{lemma}

\begin{proof}(\ref{zetamap})
It is clear that if $\Gamma_{\omega}$ is well-defined, it is linear on the conical sets of  ${{\cK}}^{\met}(\cT)$ and agrees on their intersections.
If follows by Lemma \ref{zetametric} that $\Gamma_{\omega}$ is well-defined and injective on each conical set of ${{\cK}}^{\met}(\cT)$.
\end{proof}

\begin{lemma}\label{projzeta}
Let $\sB\in{{\cK}}(\cT)$, then $\pi(\tau_{\zeta}(\sB)) = \tau_{\zeta}(\pi(\sB))$.
\end{lemma}

\begin{proof}
The statement is clear from the definition of $\zeta$ and $\tau_{\zeta}(\sB)$ together with the fact that taking conical combinations commutes with applying projection maps.
\end{proof}

\begin{lemma}
\label{zetacollisionminimal}
Let ${\bf v} \in \tau_{\zeta}(\sB)$.
There exists an extended collision $\sC_{\sB,{\bf v}}$, which is a function of $\sB$ and ${\bf v}$, such that for every set of compatible extended collisions 
$\{\sC_i :0\leq i\leq k\}$ with $\sC_0 = \sB_{\min}$ such that  $\sC_i \leq \sB$, and 
${\bf v} = \sum_{i=0}^k \lambda_i {\zeta}(\sC_i)$, with $\lambda_i > 0$ for $i>0$, we have that $\sC_{\sB,{\bf v}}$ is containment-minimal in $ \bigvee_{i =1}^k \sC_i$.
\end{lemma}

Lemma \ref{zetametric} is a direct consequence of Lemma \ref{zetatechnicallemma2}, which allows us to identify containment-minimal collisions.
Lemma \ref{zetatechnicallemma2} requires the following Definition \ref{zetatechdefinition2}.

\begin{definition}
\label{zetatechdefinition2}
Given ${\bf v} \in \tau_{\zeta}(\sB)$, we define a collection of collisions $\Xi(\sB,{\bf v})$ associated to ${\bf v}$ and ${\sB}$.
This set will be described as the union of two other sets $\Xi_1(\sB,{\bf v})$ and $\Xi_2(\sB,{\bf v})$. Let $\wt \sB$ be the $(n-1)$-bracketing in $\cK(\pi(\cT))$ with $\wt \sB\leq \pi(\sB)$ and $\ell_{\wt \sB}$ the unique metric $(n-1)$-bracketing supported on $\wt \sB$ with $\Gamma_{\zeta}(\ell_{\wt \sB})={\pi(\bf v)}$.\footnote{The existence and uniqueness of $\ell_{\wt \sB}$ is afforded by Lemma \ref{projzeta} together with induction on $n$ applied to Lemma \ref{zetamap}.}
Given an $(n-1)$-bracket $A \in \wt \sB^{n-1}$, recall the notation 
\begin{align}
\gamma({\ell_{\wt \sB}},A)
=
\sum_{\substack{A' \in {\wt \sB}^{n-1}\\ A \subseteq A'}}\ell_{\wt \sB}(A').
\end{align}

Let $\sC$ be a collision  such that $\sC \leq \sB$, then

\begin{enumerate}
\item
\label{zeta1tech2}
$\sC \in \Xi_1(\sB,{\bf v})$ if 

\begin{enumerate}
\item \label{zeta1tech2a} $\sC$ is type 1 with fusion $n$-bracket $A$.

\item \label{zeta1tech2b} There exists a real number $w_{\zeta}(\sB,{\bf v},\sC) > \gamma({\ell_{\wt \sB}},\pi(A)) $ such that $u_{j-1},u^n_j\in A^n$, precisely when $v^n_{j-1} =w_{\zeta}(\sB,{\bf v},\sC)$.

\item \label{zeta1tech2c} For any $n$-bracket ${\wt A} \in \sB^{n}$ with $A\subsetneq {\wt A}$ and $\pi(A) = \pi({\wt A})$, there exists  $u_{j-1}^n,u^n_j \in {\wt A}^n$ with $v^n_{j-1}<w_{\zeta}(\sB,{\bf v},\sC)$.
\end{enumerate}

\medskip

\item
\label{zeta2tech2}
$\sC \in \Xi_2(\sB,{\bf v})$ if

\begin{enumerate}
\item
\label{zeta2tech2a}
$\sC$ is not type 1,
 
\item
\label{zeta2tech2b}
There exists no $\sC' \in \Xi_1(\sB,{\bf v})$ such that  $\sC' \rightarrow \sC$.
 
\item
\label{zeta2tech2c}
$\pi(\sC)$ is containment-minimal in $\wt \sB$,
 
\item
\label{zeta2tech2d}
For each $A \in \sC^n$ and $u^n_{j},u_{k}^n \in A^n$, with $j<k$, we have $v^n_{j} = \gamma({\ell_{\wt \sB}},\pi(A))$.
\end{enumerate}

\medskip

\end{enumerate}

We define
\begin{align}
\Xi(\sB,{\bf v}) = \Xi_1(\sB,{\bf v})  \cup \Xi_2(\sB,{\bf v}).
\end{align}
\null\hfill$\triangle$
\end{definition}

\begin{remark}
Unlike Definition \ref{techdefinition}, we do not need a conditions like \ref{2techd2} or \ref{2techd3} above because the height partial order is inherited by from $\sB$.
\null\hfill$\triangle$
\end{remark}

\begin{lemma}\label{zetatechnicallemma2}
Let $\{\sC_i :1\leq i\leq k\}$ be a set of compatible collisions and let
\begin{align}
\sB = \bigvee_{i=1}^k \sC^i.
\end{align}

Let ${\bf v} = \sum _{i =0}^k \lambda_i {\zeta}(\sC_i)$ with $\sC_0 = \sB_\min$, $\lambda_i \in \mathbb{R}$, and $\lambda_i >0$ for $i>0$.
Then $\sC$ is containment-minimal in $\sB$ if and only if $\sC \in \Xi(\sB,{\bf v}) $.
\end{lemma}

\begin{proof}
Suppose that $\sC$ is type 1 and containment-minimal in $\sB$, then $\sC \in \Xi_1(\sB,{\bf v})$.
Conversely, if $\sC \in \Xi_1(\sB,{\bf v})$, it is straightforward to verify that $\sC$ is containment-minimal in $\sB$.
Suppose that $\sC$ is not type 1 and is containment-minimal in $\sB$, then it is clear that $\sC \in \Xi_2(\sB,{\bf v})$. The converse when $\sC$ is not type 1 is also straightforward.
\end{proof}

\begin{proof}
(\ref{zetametric})
Supposing the notational set  up of Lemma \ref{zetametric}, and
\begin{align}
\ell_{\sB} = \sum_{0\leq i\leq {l_1}}\lambda_i\ell(\sC_i)=\sum_{0\leq i\leq {l_2}}\gamma_i \ell(\underline{\sC}_i).
\end{align}

We claim that
\begin{align}
\sum_{0\leq i\leq {l_1}}\lambda_i\zeta(\sC_i)=\sum_{0\leq i\leq {l_2}}\gamma_i\zeta(\underline{\sC}_i).
\end{align}

Let $\sum_{0\leq i\leq {l_1}}\lambda_i\zeta(\sC_i)= {\bf w}$ and $\sum_{0\leq i\leq {l_2}}\gamma_i\zeta(\underline{\sC}_i)= {\bf v}$.
We observe that each coordinate $v^k_i$ 
is an invariant of the metric $n$-bracketing $\ell_{\sB}$ : $v^k_i$ is the sum of $\ell_{\sB}(A)$ where $A$ ranges over all essential $k$-brackets which contain $u^k_i$ and some other vertex $u^k_j$ with $i<j$, thus $w^k_i = v^k_i$, and ${\bf w}={\bf v}$.

For establishing the converse, suppose

\begin{align}
{\bf v} = \sum_{0\leq i\leq {l_1}}\lambda_i\zeta(\sC_i)=\sum_{0\leq i\leq {l_2}}\gamma_i\zeta(\underline{\sC}_i).
\end{align}

We wish to prove that
\begin{align}
\sum_{0\leq i\leq {l_1}}\lambda_i\ell(\sC_i)=\sum_{0\leq i\leq {l_2}}\gamma_i\ell(\underline{\sC}_i).
\end{align}

By Lemma \ref{technicallemma2} we can find a collision  which is containment-minimal in the support of both metric $n$-bracketings $\sum_{0\leq i\leq {l_1}}\lambda_i\ell(\sC_i)$ and $\sum_{0\leq i\leq {l_2}}\gamma_i\ell(\underline{\sC}_i)$.
The proof is completed by following the same inductive argument utilized in the proof of Lemma \ref{metriclemma} when  demonstrating that  (\ref{brackequal}) implies (\ref{vecequal}).  
\end{proof}

We conclude this section with a proof of the main Theorem \ref{permutahedroidthm}.

\begin{proof}
(\ref{permutahedroidthm})
We will verify the conditions of Definition \ref{permutahedroiddef} with $f = \Gamma_{\zeta}^{\rho}$.  
The statement that $\Gamma_{\zeta}^{\rho}$ is a well-defined piecewise-linear map is clear as $\Gamma$ and $\Gamma_{\zeta}$ are both piecewise-linear and $\Gamma$ is invertible by Proposition \ref{mainvelocitythm}.
We will now focus on verifying condition (\ref{localpermfancond1}).
The injectivity, equivalently invertibility, of $\Gamma_{\zeta}^{\rho}$  on each cone follows as $\Gamma_\zeta$ is injective on conical sets by Lemma \ref{zetamap}. 

Next we will prove that $\Gamma_{\zeta}^{\rho}$ is unimodular on each cone in the velocity fan.
It suffices to show that if we have a chamber $\tau(\sB) \in \cF(\cT)$, a linearly independent unimodular collection of vectors $B = \{v_1\ldots, v_{m-1},{\bf 1}\}$ contained in $\tau(\sB)$, then the image  $\Gamma_{\zeta}^{\rho}(B)$ in $\tau_{\zeta}(\sB)$ is unimodular.
To do so, we will utilize the permutahedral velocity fan introduced in subsection \ref{permutahedralassociahedronsubsection}.
Let $Z\in \mathscr{O}(\cT)$ with $Z = \{\sD_1, \ldots, \sD_{m-1}\}$ such that $\tau(Z) \subseteq \tau(\sB)$, and take $B = \{\rho(\sD_1), \ldots \rho(\sD_{m-1}),{\bf 1} \}$, the generators for $\tau(Z)$.
By Proposition \ref{permutahedralvelocitymainprop} combined with Theorem \ref{triangulationtheorem}, we have that $B$ is a linearly independent unimodular collection in $\tau(\sB)$. 
Hence it suffices to prove that $\{\Gamma_{\zeta}^{\rho}(\rho(\sD_1)), \ldots, \Gamma_{\zeta}^{\rho}(\rho(\sD_{m-1}),{\bf 1}\}$ form the columns of a unimodular matrix.
In fact we claim that $\{\Gamma_{\zeta}^{\rho}(\rho(\sD_1)), \ldots, \Gamma_{\zeta}^{\rho}(\rho(\sD_{m-1}),{\bf 1}\}$ form the generators for a chamber in the braid arrangement, i.e.\ they form indicator vectors for a maximal chain of subsets of $[m]$.
To see this simply observe that for $1\leq i\leq m-2$, when passing from $\Gamma_{\zeta}^{\rho}(\rho(\sD_i))$ to $\Gamma_{\zeta}^{\rho}(\rho(\sD_{i+1}))$ we increase a single entry from 0 to 1.  
 
Condition (\ref{localpermfancond2}) is now an immediate consequence of the above argument.
We know by the injectivity of $\Gamma_{\zeta}^{\rho}$ on $\tau(\sB)$ that $\tau_{\zeta}(\sB)$ is equal to the union of all $\tau_{\zeta}(Z)$ where $Z$ ranges over all flags of generalized collisions such that $\tau^{\met}(Z) \subseteq \tau^{\met}(\sB)$.
\end{proof}

\begin{remark}
It follows that $\tau_{\zeta}(\sB)$ is a poset cone with face poset equal to $[\sB_\min,\sB]$.
Moreover, in the above proof, we passed from the velocity fan to the permutahedral velocity fan, skipping over the triangulated velocity fan, but one can see that for $Y \in \ccD$, the corresponding cone $\tau_{\zeta}(Y)$ is a tree poset cone associated to the tree $\cT_Y$ from Theorem \ref{treecombtheorem}.
\null\hfill$\triangle$
\end{remark}

\subsection{The shuffle atlas for the velocity fan}\label{permutahedralatlassubsection}

\

In this section we describe a family $\omega=\{\omega_{\sigma}\}$ of piecewise-unimodular maps defined on subsets of the metric $n$-bracketing complex, called \emph{shuffle charts}, associated to linear extensions of the extended height partial order on the singleton brackets in an $n$-bracketing.
We will see that each $\omega_{\sigma}$ is a piecewise-linear isomorphism and that the image of each conical set is a union of cones in the braid arrangement.
Moreover, when we extend these maps to the velocity fan by precomposing with $\Gamma^{-1}$, we find that they are unimodular on each cone.
In subsection \ref{nestohedralatlassection} we show that for each $\sigma$, when $\omega_{\sigma}$ is applied to a shuffle chart in the triangulated metric $n$-bracketing complex, its image is a subfan of a nestohedral fan obtained by deleting the stars of certain coordinate rays.

\begin{definition}\label{permdef}
Given a $\sB \in \cK(\cT)$ we define the \emph{admissible $\cT$-shuffles} of $\sB$ to be the set $\mathfrak{S}(\sB)\subseteq S_{\cT}$ where $\sigma = (\sigma_1, \ldots, \sigma_n)$ satisfies $\sigma \in \mathfrak{S}(\sB)$ if each $\sigma_k$ linearly extends the extended 
height partial order $\wh <_{\sB}$ on $V^k(\cT)$, and if $A \in \sB^k$ is a nontrivial bracket, and $a<b<c$, we disallow $u_{\sigma(a)}^k, u_{\sigma(c)}^k \in A^k$ and $u_{\sigma(b)}^k \notin A^k$.
\null\hfill$\triangle$
\end{definition}

Note that for $\sB'\leq \sB$, we have $\mathfrak{S}(\sB')\subseteq \mathfrak{S}(\sB)$, and $\mathfrak{S}(\sB_\min) = S_{\cT}$.

\begin{lemma}
\label{uniqueshufflelemma}
Let $\sB \in \cK(\cT)$ be a maximal $n$-bracketing, then $|\mathfrak{S}(\sB)|=1$.
\end{lemma}

\begin{proof}
Suppose for contradiction that there exists $\sigma_1,\sigma_2 \in \mathfrak{S}(\sB)$ with $\sigma_1 \neq \sigma_2$.
Because $\pi(\sB)$ is a maximal $(n-1)$-bracketing in $\cK(\pi(\cT))$, we may apply induction on $n$ to obtain $\pi(\sigma_1)=\pi(\sigma_2)$.
Let $u_i^n, u_j^n \in V(\cT)^n$ such that 
$\sigma_1^n(i)<\sigma_1^n(j)$, but $\sigma_2^n(i) > \sigma_2^n(j)$.
Let $A \in \sB^n$ be such that $A^n$ is the smallest $n$-bracket which contains both $u_i^n$ and $u^n_j$, $A_1$ is the largest $n$-bracket which contains $u_i^n$, but not $u_j^n$, and $A_2$ is the largest $n$-bracket which contains $u_j^n$, but not $u^n_i$.
By the assumption that $u_i^n$ and $u_j^n$ are incomparable with respect to $\wh <_{\sB}$ we must have $\pi(A_1)\cap\pi(A_2)=\emptyset$. 
We claim that we can add a pair of additional $n$-bracket to $\sB$, contradicting its maximality.
Let ${\wt A_1}$ be the $n$-bracket with $\pi({\wt A_1})=\pi(A)$ and ${\wt A_1}^n= \{u \in V(\cT)^n : u \in A, \pi(u) \in \pi(A_1)^{n-1}\}$ and ${\wt A_2}$ be the $n$-bracket with $\pi({\wt A_2})=\pi(A)$ and ${\wt A_2}^n= \{u \in V(\cT)^n : u \in A, \pi(u) \in (\pi(A)^{n-1}\setminus \pi(A_1)^{n-1})\}$.
Then $\wt \sB$ is the $n$-bracketing with $\wt \sB^k = \sB^k$ for $k<n$ and $\wt \sB^n = \sB^n \cup \{{\wt A_1}, {\wt A_2}\}$, equipped with the additional relation ${\wt A_1}<_{\wt \sB}{\wt A_2}$.
It is straightforward to verify that $\wt \sB$ is an $n$-bracketing.
\end{proof}

\begin{remark}
For $\sB \in \cK(\cT)$, it can be shown that $\sigma \in \mathfrak{S}(\sB)$ if and only if there exists a maximal $n$-bracketing $\sB'$ with $\sB\leq \sB'$ such that $\{\sigma \} = \mathfrak{S}(\sB')$.
\end{remark}

\begin{definition}
Let $\sC$ be a collision such that $\sigma \in \mathfrak{S}(\sC)$.
We take $\omega_{\sigma}(\sC)$ to be a vector of length $m$ such that $\omega_{\sigma}(\sC)^k_i = 1$ if the vertices $u^k_{\sigma(i)}$ and $u^k_{\sigma(i+1)}$ belong to an essential $k$-bracket $A \in\sC^k$, and $\omega_{\sigma}(\sC)^k_i = 0$ otherwise.
\null\hfill$\triangle$
\end{definition}

\begin{definition}
Let $\{\sC_i:1\leq i \leq s\}$ be a collection of compatible collisions and $\sigma \in S_\cT$ such that $\sigma \in \mathfrak{S}(\sC_i)$ for each $i$ with $1\leq i \leq s$.
We define a cone associated to this collection as
\begin{align}
\tau_{\omega_{\sigma}}\{\sC_i\}
=
cone\{\omega_{\sigma}(\sC_1), \ldots, \omega_{\sigma}(\sC_k)\}+ \langle\bf{1}\rangle_{\mathbb{R}}.
\end{align}

Let $\sB$ be an $n$-bracketing and suppose $\{\sC_i\}$ is the set of all collisions contained in $\sB$.
We define the cone $\tau_{\omega_
\sigma}(\sB) = \tau_{\omega_
\sigma}\{\sC_i\}$.
Similarly, given a face $Y \in \Delta_\cT$, or $Y \in \mathscr{O}(\cT)$, we define $\tau_{\omega_
\sigma}(Y) = \tau_{\omega_
\sigma}\{\sC_i\}$ where $Y = \{\sC_i\}$.
\null\hfill$\triangle$
\end{definition}

\begin{definition}\label{permatlasdef}
Given $\sigma \in S_\cT$, we define the \emph{shuffle chart} of the metric $n$-bracketing complex associated to $\sigma$ to be the pair $(U_\sigma, \Gamma_{\omega_\sigma})$ where
\begin{align}
U_\sigma = \{\tau^{\met}(\sB) \in { \cK^{\met}}(\cT): \sigma \in \mathfrak{S}(\sB)\}
\end{align}
and $\Gamma_{\omega_\sigma}:U_{\sigma}\rightarrow \mathbb{R}^m$ is defined as follows.
Let $\sC_0, \ldots, \sC_k \in \wh \fX(\cT)$ with $\sC_0=\sB_{\min}$, and $\lambda_i\in \mathbb{R}$ for $0\leq i\leq k$ with $\lambda_i\geq 0$ for $i>0$.
Let $\ell_\sB \in U_{\sigma}$ be such that
\begin{align}\ell_\sB = \sum_{i=0}^k \lambda_i \ell(\sC_i).
\end{align}
We define the map $\Gamma_{\omega_{\sigma}}:U_{\sigma}\rightarrow \mathbb{R}^m$, as
\begin{align}
\Gamma_{\omega_{\sigma}}(\ell_{\sB})= \sum_{i=0}^k \lambda_i \omega_{\sigma}(\sC_i).
\end{align}

The \emph{shuffle atlas for the metric $n$-bracketing complex} is the collection of all shuffle charts $\{(U_\sigma, \Gamma_{\omega_\sigma})\}$ where $\sigma=(\sigma_1, \ldots, \sigma_n)$ ranges over all $\sigma \in S_\cT$.
We similarly define \emph{the shuffle atlas for the velocity fan} as the collection of \emph{shuffle charts} $\{\Gamma(U_\sigma),\Gamma^{\rho}_{\omega_{\sigma}}\}$ where $\Gamma^{\rho}_{\omega_{\sigma}} =\Gamma_{\omega_\sigma}\circ \Gamma^{-1}$.
\null\hfill$\triangle$
\end{definition}

We remark that at this stage, it is not clear that $\Gamma_{\omega_\sigma}$ is well-defined, but $U_{\sigma}$ clearly is.

\begin{proposition}
\label{localPLmapsprop}
For each $\sigma \in S_{\cT}$, the map $\Gamma^{\rho}_{\omega_\sigma}$ is an injective piecewise-unimodular map on $\Gamma(U_{\sigma})$, and the image of each cone is a union of cones in the braid arrangement.
\end{proposition}

\begin{proof}
The statement that $\Gamma^{\rho}_{\omega_\sigma}$ is an injective piecewise-linear map on $\Gamma(U_{\sigma})$ reduces to the corresponding statement for $\Gamma_{\omega_\sigma}$, which is a consequence of the following Lemma \ref{omegametric}.
That the image of each cone is a union of cones in the braid arrangement is established in Corollary \ref{unionbraicconeslemma}, and that $\Gamma^{\rho}_{\omega_\sigma}$ is unimodular on each cone is Corollary \ref{omegasigmaunimodular}.
\end{proof}

\begin{lemma}
\label{omegametric}
Suppose that $\{\sC_0, \ldots, \sC_{l_1}\}$ and $\{\underline{\sC}_0, \ldots, ,\underline{\sC}_{l_2}\}$ are two sets of compatible extended collisions such that $\sigma \in \mathfrak{S}(\sC_i)$ for $1\leq i\leq {l_1}$ and $\sigma \in \mathfrak{S}(\underline{\sC}_i)$ for $1\leq i\leq {l_2}$.
Let $ \sC_0= \underline{\sC}_0=\sB_\min$, take $\lambda_i,\gamma_i \in \mathbb{R}$ for $0\leq i\leq {l_1}$ and $0\leq i\leq {l_2}$, and $\lambda_i,\gamma_i \geq 0$ for $i>0$.
Then \begin{align}
 \sum_{0\leq i\leq {l_1}}\lambda_i\ell(\sC_i)=\sum_{0\leq i\leq {l_2}}\gamma_i\ell(\underline{\sC}_i)
\end{align}
if and only if
\begin{align}
\sum_{0\leq i\leq {l_1}}\lambda_i\omega_{\sigma}(\sC_i)=\sum_{0\leq i\leq {l_2}}\gamma_i\omega_{\sigma}(\underline{\sC}_i).
\end{align}
\end{lemma}

We will follow an argument similar to the one used for proving Lemmas \ref{zetametric} and \ref{metriclemma}.

\begin{lemma}
\label{projomega}
Let $\sB\in{{\cK}}(\cT)$ and $\sigma \in S_{\cT}$, then $\pi(\tau_{\omega_\sigma}(\sB)) = \tau_{\omega_{\pi(\sigma)}}(\pi(\sB))$.
\end{lemma}

\begin{proof}
The statement is clear from the definitions of $\omega_\sigma$ and $\tau_{\omega_\sigma}(\sB)$ together with the fact that taking conical combinations commutes with applying projection maps.
\end{proof}

\begin{lemma}\label{main2}
Let $\sigma \in S_{\cT}$ and ${\bf v} \in \omega_{\sigma}(U_{\sigma})$.
There exists a collision $\sC_{(\sigma,{\bf v})}$, which is a function of $\sigma$ and ${\bf v}$, such that for every set of compatible collisions $\{\sC_i :1\leq i\leq k\}$ such that $\sigma \in \mathfrak{S}(\sC_i)$, and ${\bf v} = \sum _{i =0}^k \lambda_i {\omega_{\sigma}}(\sC_i)$ with $\sC_0 = \sB_{\min}$, $\lambda_i \in\mathbb{R}$ and $\lambda_i>0$ for $i>0$, we have that $\sC_{(\sigma,{\bf v})}$ is containment-minimal in
$\bigvee _{i =1}^k \sC_i.$
\end{lemma}

We will establish Lemma \ref{main2} as consequence of Lemma \ref{technicallemma2}, which allows us to identify containment-minimal collisions. Lemma \ref{technicallemma2} requires the following Definition \ref{techdefinition2}.

\begin{definition}
\label{techdefinition2}
Given ${\bf v} \in \omega_{\sigma}(U_{\sigma})$, we define a collection of collisions $\Xi(\sigma,{\bf v})$ associated to ${\bf v}$ and ${\sigma}$.
This set will be described as the union of sets $\Xi_1(\sigma,{\bf v})$ and $\Xi_2(\sigma,{\bf v})$.
Let $\ell_{\wt \sB}$ be the unique metric $(n-1)$-bracketing in $\pi(U_{\sigma})$ such that $\Gamma_{\omega_{\pi(\sigma)}}(\ell_{\wt \sB})=\pi({\bf v})$.\footnote{The existence and uniqueness of $\ell_{\wt \sB}$ is afforded by Lemma \ref{projomega} together with induction on $n$ applied to Lemma \ref{omegametric}, and the fact that $\pi(U_\sigma) = U_{\pi(\sigma)}\subseteq \cK(\pi(\cT))$.}

Let $\sC$ be a collision such that $\sigma \in \mathfrak{S}(\sC)$, then

\begin{enumerate}
\item
\label{1tech2}
$\sC \in \Xi_1(\sigma,{\bf v})$ if 

\begin{enumerate}
\item \label{1tech2a} $\sC$ is type 1 with fusion $n$-bracket $A$,

\item \label{1tech2b} There exists a constant $w_{\omega}(\sigma,{\bf v},\sC) > \gamma({\ell_{\wt \sB}},A)$ such that for $u^n_{\sigma(j)}$ and $u^n_{\sigma(j+1)}$ in $A^n$, we have $v^n_{j} =w_{\omega}(\sigma,{\bf v},\sC)$.

\item \label{1tech2c} For any $n$-bracket $A'$ with $A \subsetneq A'$ and $\pi(A') = \pi(A)$, there exists $u^n_{\sigma(j)},u^n_{\sigma(j+1)} \in A'^{\, n}$ with $v^n_{j}<w_{\omega}(\sigma,{\bf v},\sC)$.
\end{enumerate}

\medskip

\item
\label{2tech2}
$\sC \in \Xi_2(\sigma,{\bf v})$ if

\begin{enumerate}
\item
\label{2tech2a}
$\sC$ is type 2

\item
\label{2tech2b}
There exists no $\sC' \in \Xi_1(\sigma,{\bf v})$ such that $\sC' \rightarrow \sC$.
 
\item
\label{2tech2c}
$\pi(\sC)$ is containment-minimal in $\wt \sB$,

\item
\label{2tech2d}
For each $A \in \sC^n$ and $u^n_{\sigma(j)}, u^n_{\sigma(j+1)}\in A^n$, we have $v^n_{j} = \gamma({\ell_{\wt \sB}},\pi(A))$.
\end{enumerate}
\end{enumerate}

\medskip

We define
\begin{align}
\Xi(\sigma,{\bf v}) = \Xi_1(\sigma,{\bf v}) \cup \Xi_2(\sigma,{\bf v}).
\end{align}
\null\hfill$\triangle$
\end{definition}

\begin{lemma}
\label{technicallemma2}
Let $\{\sC_i :1\leq i\leq k\}$ be a set of compatible collisions such that $\sigma \in \mathfrak{S}(\sC_i)$ for $1\leq i\leq k$, and let
$\sB = \bigvee_{i=1}^k \sC_i.$
Let ${\bf v} = \sum _{i =0}^k \lambda_i \omega_\sigma(\sC_i)$ with $\sC_0 = \sB_{\min}$, $\lambda_i \in \mathbb{R}$, and $\lambda_i > 0$ for $i>0$.
Then $\sC$ is containment-minimal in $\sB$ if and only if $\sC \in \Xi(\sigma,{\bf v}) $.
\end{lemma}

\begin{proof}
Suppose that $\sC$ is type 1 and $\sC$ is containment-minimal in $\sB$, then it is clear that $\sC \in \Xi_1(\sigma,{\bf v})$.
Conversely, if $\sC \in \Xi_1(\sigma,{\bf v})$, it is a type 1 collision such that $\sC$ is containment-minimal in $\sB$.
Suppose that $\sC$ is not type 1 and $\sC$ is containment-minimal in $\sB$, then it is clear that $\sC \in \Xi_2(\sigma,{\bf v})$.
The converse where $\sC \in \Xi_2(\sigma,{\bf v})$ is also straightforward. 
\end{proof}

\begin{proof}(\ref{omegametric})
Supposing the notational setup of Theorem \ref{omegametric}, let \begin{align}
\ell_{\sB} = \sum_{0\leq i\leq {l_1}}\lambda_i\ell(\sC_i)=\sum_{0\leq i\leq {l_2}}\gamma_i\ell(\underline{\sC}_i).
\end{align}

We claim that
\begin{align}
\sum_{0\leq i\leq {l_1}}\lambda_i\omega_{\sigma}(\sC_i)=\sum_{0\leq i\leq {l_2}}\gamma_i\omega_{\sigma}(\underline{\sC}_i).
\end{align}

Let $\sum_{0\leq i\leq {l_1}}\lambda_i\omega_{\sigma}(\sC_i)= {\bf w}$ and $\sum_{0\leq i\leq {l_2}}\gamma_i\omega_{\sigma}(\underline{\sC}_i)= {\bf v}$. We observe that a coordinate $v^k_j$ of ${\bf v}$ equal to the weighted sum of $\ell_{\sB}(A)$ over the $k$-brackets $A$ which contain $u^k_{\sigma(j)}$ and $u^k_{\sigma(j+1)}$.
This is an invariant of the metric $n$-bracketing $\ell_{\sB}$, thus $w^k_j = v^k_j$, and ${\bf w}={\bf v}$.

For establishing the converse, suppose
\begin{align}
{\bf v} = \sum_{0\leq i\leq {l_1}}\lambda_i\omega_{\sigma}(\sC_i)=\sum_{0\leq i\leq {l_2}}\gamma_i\omega_{\sigma}(\underline{\sC}_i).
\end{align}

We wish to prove that
\begin{align}
\sum_{0\leq i\leq {l_1}}\lambda_i\ell(\sC_i)=\sum_{0\leq i\leq {l_2}}\gamma_i\ell(\underline{\sC}_i).
\end{align}

By Lemma \ref{technicallemma2} we can find a collision which is containment-minimal in the support of both metric $n$-bracketings
$\sum_{0\leq i\leq {l_1}}\lambda_i\ell(\sC_i)$ and $\sum_{0\leq i\leq {l_2}}\gamma_i\ell(\underline{\sC}_i)$.
The proof is completed by following the same inductive argument utilized in the proof of Lemma \ref{metriclemma} when demonstrating that (\ref{brackequal}) implies (\ref{vecequal}).
\end{proof}

\begin{remark}
While $\Gamma^{\rho}_{\omega_\sigma}$ is injective and piecewise-linear on all of $\Gamma(U_{\sigma})$, it may not be linear on all of $\Gamma(U_{\sigma})$.
For example, in Figure \ref{maps_to_braid_example} we see that in the chart $\Gamma(U_{312})$\footnote{Here we have simplified notation and denote $\sigma = (\sigma^1, \sigma^2) = (12,312)$ by $312$.}, \, we have $(1,0,1)=(0,0,1)+(1,0,0)=\Gamma^{\rho}_{\omega_{312}}(0,1,0)+\Gamma^{\rho}_{\omega_{312}}(1,0,3) \neq \Gamma^{\rho}_{\omega_{312}}((0,1,0)+(1,0,3))$ as by injectivity, $\Gamma^{\rho}_{\omega_{312}}(1,1,3) \neq \Gamma^{\rho}_{\omega_{312}}(1,1,2) = (1,0,1)$.
\null\hfill$\triangle$
\end{remark}

\subsection{The nestohedral atlas for the triangulated velocity fan}
\label{nestohedralatlassection}

\

We now investigate the shuffle atlas for the triangulated $n$-bracketing complex.

\begin{definition}\label{nestatlasdef}
Given $Y \in \ccD$ and $\sB = \vee_{\sC \in Y} \sC$, let $\mathfrak{S}(Y)\coloneqq\{\sigma \in S_{\cT}: \sigma \in \mathfrak{S}(\sB)$.
We define the \emph{nestohedral chart of the triangulated $n$-bracketing complex associated to $\sigma$} to be the pair $(U_\sigma, \Gamma_{\omega_\sigma})$ where $U_\sigma = \{\tau^{\met}(Y) \in { \cK^{\met}}(\ccD): \sigma \in \mathfrak{S}(Y)\}$ and $\Gamma_{\omega_\sigma}$, is the piecewise-linear map obtained by extending ${\ell}(\sC)\mapsto \omega_{\sigma}(\sC)$ linearly on each conical set $\tau^{\met}(Y)$.
The \emph{nestohedral atlas for the triangulated metric $n$-bracketing complex} is the collection of all nestohedral charts $\{(U_\sigma, \Gamma_{\omega_\sigma})\}$ where $\sigma$ ranges over $S_{\cT}$.
We define the nestohedral charts and atlas for the triangulated velocity fan similarly by taking $\{\Gamma(U_\sigma),\Gamma^{\rho}_{\omega_{\sigma}}\}$ where $\Gamma^{\rho}_{\omega_{\sigma}} =\Gamma_{\omega_\sigma}\circ \Gamma^{-1}$.
\null\hfill$\triangle$
\end{definition}

The shuffle atlas is essentially the same as the nestohedral atlas except that the latter is defined on the triangulated velocity fan.
Our choice to introduce new terminology for this essentially equivalent object will become clear shortly.

We begin by recalling the definition of building sets and their corresponding nested set complexes for the Boolean lattice following Postnikov.
These definitions are derived from De Concini and Procesi's theory of wonderful compactifications of \cite{de1995wonderful} in the special case of the coordinate arrangement.\footnote{Although not explained in Backman-Danner \cite{backman2024convex}, one of the first author's primary motivations for that work was the desire to better understand the results of this subsection.  In future work, we will utilize tropical geometry for contextualizing the results of this subsection.}

\begin{definition}\label{builddef}
Let $\mathbb{B}$ be a collection of subsets of $[n]$.
We say that $\mathbb{B}$ is a \emph{building set} if 

\begin{enumerate}
\item\label{build1} For all $i \in [n]$, $\{i\} \in \mathbb{B}$

\item\label{build2} For all $B_1, B_2 \in \mathbb{B}$, if $B_1 \cap B_2 \neq \emptyset$ then $B_1 \cup B_2 \in \mathbb{B}$.
\null\hfill$\triangle$
\end{enumerate}
\end{definition}

\begin{definition}\label{nestedsetcomplexdef}
A set $N \subseteq \mathbb{B}$ is a \emph{nested set} if
\begin{enumerate}
\item
\label{nesteddef1} $B_1, B_2 \in N$, then $B_1, \subseteq B_2$, $B_2 \subseteq B_1$ or $B_1 \cap B_2 =\emptyset$, 

\item
\label{nesteddef2} when $B_1, \ldots B_k \in N$ are disjoint with $k\geq 2$, then
\begin{align}
\bigcup_{i=1}^k B_i \notin \mathbb{B}.
\end{align}
\end{enumerate}

The \emph{nested set complex} $\Delta_{\mathbb{B}}$ associated to $\mathbb{B}$ is the collection of all nested sets for $\mathbb{B}$.
It is clear that $\Delta_{\mathbb{B}}$ is an abstract simplicial complex.
\null\hfill$\triangle$
\end{definition}

Let $\fX(\cT)^{\sigma}=\{\sC \in \fX(\cT): \sigma \in \mathfrak{S}(\sC)\}$ and $\wh \fX(\cT)^{\sigma}=\{\sC \in \wh \fX(\cT): \sigma \in \mathfrak{S}(\sC)\}$.
In the remainder of this section we will identify the 0-1 vectors $\omega_{\sigma}(\sC)$ for $\sC \in \wh \fX(\cT)^{\sigma}$, with indicator vectors for sets and hence allow discussion of their unions and intersections.
Our ambient set will be $\bigsqcup_{k=1}^n[t_k]$, where $[t_k]$ indicates the set of the first $t_k$ positive integers.
We will let $e^k_i$ denote the $i$th element of $[t_k]$. 

\begin{definition}
\label{localbuildingsetdef}
Let $\sigma \in S_{\cT}$.
Let $W_{\sigma}=\{\omega_{\sigma}(\sC): \sC \in \wh \fX(\cT)^{\sigma}\}$.
Let $V_{\sigma} = \{ e^k_i: \pi(u^k_{\sigma(i)})\neq \pi(u^k_{\sigma(i+1)}) \}$.
The set $\mathbb{B}_{\sigma} = W_{\sigma} \sqcup V_{\sigma}$ is the \emph{local $\sigma$-collision building set}.
\null\hfill$\triangle$
\end{definition}

Our name for $\mathbb{B}_{\sigma}$ is justified by the following Proposition \ref{localbuildingsetprop}.

\begin{proposition}\label{localbuildingsetprop}
The set $\mathbb{B}_{\sigma}$ is a building set.
\end{proposition}

We begin with a few lemmas.

\begin{lemma}
\label{zetajoin}
Suppose that $\sC_1, \sC_2 \in \wh X_{\cT}^{\sigma}$ such that $\sC_1 \nsim \sC_2$, then there exists a collision $\sC_3\in \wh X_{\cT}^{\sigma}$ such that $\omega_{\sigma}(\sC_1) \cup \omega_{\sigma}(\sC_2) = \omega_{\sigma}(\sC_3)$.
\end{lemma}

\begin{proof}
For $k$ with $1\leq k \leq n$, we define an equivalence relation $\sim_{\sigma_k}$ as follows: if $A \in \sC_1^k$ and $A' \in \sC_2^k$ are essential brackets in $\sC_1$ and $\sC_2$, respectively, such that $A^k \cap A'^{\,k} \neq \emptyset$, then we let $A^k \sim_{\sigma_k} A'^{\,k}$, and we take $\sim_{\sigma_k}$ to be the transitive closure of this relation.
Let $x$ be an equivalence class of $\sim_{\sigma_k}$, and let $A^k_x = \{u^k_i \in V(\cT)^k: \exists A^k \in x, u^k_i \in A^k\}$.
Suppose we have already constructed $\sC_3^{k-1}$.
For each equivalence class $x$, we extend $A^k_x$ to an $k$-bracket in $\sC_3^{k}$ by appending it to the $(k-1)$-bracket $A \in {\underline \sC}_3^{k-1}$ such that $\pi(A^k_x)\subseteq A^{k-1}$.
It is straightforward to verify that such an $A \in \sC_3^{k-1}$ always exists.
We must also include the singleton and maximum $k$-brackets in $\sC_3^k$.
To complete our description of $\sC_3^k$, we give the height partial order on the $k$-brackets in $\sC_3^k$ induced by the total order $\sigma_k$: if $A_1, A_2 \in \sC_3^k$ such that $\pi(A_1)= \pi(A_2)$ then we take $A_1 <A_2$ if there exists $u^k_{\sigma(i)} \in A_1^k$ and $u^k_{\sigma(j)} \in A_2^k$ such that $i<j$.
It is straightforward to check that this partial order is well-defined.

We now verify that $\sC_3$ is an extended collision in $\fX({\cT})^{\sigma}$ such that $\omega_{\sigma}(\sC_1) \cup \omega_{\sigma}(\sC_2) = \omega_{\sigma}(\sC_3)$.
The depth $k$-vertices in the $k$-brackets in $\sC_1^k$ and $\sC_2^k$ can be viewed as consecutive subsets of the vertices in $V(\cT)^k$ ordered as in $\sigma_k$.
Thus, the equivalence classes defined above are also consecutive subsets of $V(\cT)^k$ ordered as in $\sigma_k$.
Using this observation, one may verify that $\sC_3$ is an extended collision in $\wh \fX({\cT})^{\sigma}$.

We check $\omega_{\sigma}(\sC_1) \cup \omega_{\sigma}(\sC_2) = \omega_{\sigma}(\sC_3)$.
Induction affords us that $\omega_{\pi(\sigma)}(\pi(\sC_1)) \cup \omega_{\pi(\sigma)}(\pi(\sC_2)) = \omega_{\pi(\sigma)}(\pi(\sC_3))$, thus we must simply check the desired statement for the last $t_n$ entries of $\omega_{\sigma}(\sC_3)$ encoding collisions of singleton $n$-brackets.
If $\omega_{\sigma}(\sC_1)^n_i=1$ or $\omega_{\sigma}(\sC_2)^n_i=1$, then $u^n_{\sigma(i)},u^n_{\sigma(i+1)} \in A^n$ for some $n$-bracket A belonging to either $\sC_1^n$ or $\sC_2^n$.
Let $x$ be the equivalence class of $\sim_{\sigma}$ containing $A$, then $u^n_{\sigma(i)},u^n_{\sigma(i+1)} \in A_x^n$ and so $\omega_{\sigma}(\sC_3)^n_i=1$.
Conversely, suppose that $\omega_{\sigma}(\sC_3)^n_i=1$, then there exists an equivalence class $x$ of $\sim_{\sigma}$ such that $u^n_{\sigma(i)},u^n_{\sigma(i+1)} \in A_x^n$.
Because $A_x^n$ is a union of collection of vertices $V^n(\cT)$ which are consecutive with respect to $\sigma_n$, it must be that there is some $n$-bracket $A \in \sC_1^n \cup \sC_2^n$ such that $u^n_{\sigma(i)},u^n_{\sigma(i+1)} \in A^n$, hence either $\omega_{\sigma}(\sC_1)^n_i=1$ or $\omega_{\sigma}(\sC_2)^n_i=1$.
\end{proof}

The following lemma is immediate.

\begin{lemma}
\label{intersectionimplication}
Let $\omega_{\sigma}(\sC_1),\omega_{\sigma}(\sC_2) \in \mathbb{B}_{\sigma}$ such that $\omega_{\sigma}(\sC_1)\cap \omega_{\sigma}(\sC_2) \neq \emptyset$, then $\sC_1 \nsim \sC_2$.
\end{lemma}

\begin{proof}
(\ref{localbuildingsetprop})
It is clear by construction that for all $e^k_i \in \bigsqcup_{k=0}^n[t_k]$, $\{e^k_i\} \in \mathbb{B}_{\sigma}$ thus condition (\ref{build1}) of Definition \ref{builddef} is satisfied.
We now verify condition (\ref{build2}) of Definition \ref{builddef}.
Let $\omega_{\sigma}(\sC_1),\omega_{\sigma}(\sC_2) \in \mathbb{B}_{\sigma}$ such that $\omega_{\sigma}(\sC_1)\cap \omega_{\sigma}(\sC_2) \neq \emptyset$.
By Lemma \ref{intersectionimplication} we know that $\sC_1 \nsim \sC_2$, and it follows from Lemma \ref{zetajoin} that $\omega_{\sigma}(\sC_1)\cup \omega_{\sigma}(\sC_2) \in \mathbb{B}_{\sigma}$.
\end{proof}

Given a building set $\mathbb{B}$ and its associated nested set complex $\Delta_{\mathbb{B}}$, Postnikov \cite{postnikovgp} explicitly described how one can realize the nested set complex by a smooth projective fan $\cF(\Delta_{\mathbb{B}})$ which coarsens the braid arrangement (see also Zelevinsky \cite{zelevinsky2005nested} and Feichtner-Sturmfels \cite{feichtner2004matroid}).
The toric variety associated to this fan is De Concini-Procesi's associated wonderful compactification of the torus with respect to $\mathbb{B}$ \cite{de1995wonderful}.
Let $N$ be a nested set for $\mathbb{B}$, then the corresponding cone in $\cF(\Delta_{\mathbb{B}})$ is 
\begin{align}
\tau(N) = cone\{\chi_B: B \in N\}+\langle {\bf 1} \rangle_{\mathbb{R}},
\end{align}
where $\chi_B$ is the indicator vector for $B$.

\begin{definition}
\label{stardef}
Give a simplicial complex $\Delta$, and a face $\tau \in \Delta$, the star of $\tau$ is 
\begin{align}
star(\tau)= \{\tau' \in \Delta:\tau \subseteq \tau'\}.
\end{align}
\null\hfill$\triangle$
\end{definition}

\begin{lemma}
\label{nestoisom}
Let $\Delta_{\cT_\sigma} = \{Y \in \Delta_\cT: \forall \, \sC \in Y,\, \sigma \in \mathfrak{S}(\sC)\}$, and let $\Delta^{\circ}_{\mathbb{B}_{\sigma}}= \Delta_{\mathbb{B}_{\sigma}}\setminus \{star(\{e^k_i\}): e^k_i \in V_{\sigma}\}$.
Then
\begin{align}
\Delta_{\cT_\sigma} \cong \Delta^{\circ}_{\mathbb{B}_{\sigma}},\end{align}
where $\cong$ is an isomorphism of abstract simplicial complexes induced by the map $\sC \mapsto \omega_{\sigma}(\sC)$.
\end{lemma}

\begin{proof}
Given $Y \in \Delta_{\cT_\sigma}$, we let $\omega_{\sigma}(Y)$ denote $\{\omega_{\sigma}(\sC):\sC \in Y\}$ where $\omega_{\sigma}(\sC)$ is identified with the set for which it is an indicator vector.
We first show that given $Y \in \Delta_{\cT_\sigma}$, $\omega_{\sigma}(Y) \in \Delta^{\circ}_{\mathbb{B}_{\sigma}}$.
It is clear that for $e^k_i \in V_{\sigma}$, there exists no $\sC \in Y$, such that $\omega_{\sigma}(\sC)= e^k_i$, hence $\omega_{\sigma}(Y) \notin star(\{e_i^k\})$.  It remains to check that $\omega_{\sigma}(Y)$ is a nested set, and we verify the conditions of Definition \ref{nestedsetcomplexdef}.
For verifying condition (\ref{nestedcon1}), we let $\sC_i,\sC_j \in Y$ and observe that if $\sC_i\rightarrow \sC_j$ then $\omega_{\sigma}(\sC_i)\subseteq \omega_{\sigma}(\sC_j)$.
On the other hand, if $\sC_i\sim \sC_j$, then $\omega_{\sigma}(\sC_i)\cap \omega_{\sigma}(\sC_j)= \emptyset$ by the contrapositive of Lemma \ref{intersectionimplication}.

For verifying condition (\ref {nestedcon2}), suppose that $N=\{ \sC_i:1\leq i \leq k\}\subseteq Y$ with $k\geq 2$, are such that the $\omega_{\sigma}(\sC_i)$ are disjoint, then we must show that $\bigcup_{i=1}^k\omega_{\sigma}(\sC_i) \notin \mathbb{B}_{\sigma}$. 
Suppose for contradiction that $\bigcup_{i=1}^k\omega_{\sigma}(\sC_i) = \omega_{\sigma}(\sC) \in \mathbb{B}_{\sigma}$.
The collection of nonzero $\pi(\omega_{\sigma}(\sC_i))$ are disjoint, and $\bigcup_{i=1}^k\pi(\omega_{\sigma}(\sC_i)) = \pi(\omega_{\sigma}(\sC))$. By induction on $n$, we can conclude that there is at most one $\sC_i$ 
which is not type 1, and $\pi(\omega_{\sigma}(\sC_i))=\pi(\omega_{\sigma}(\sC))$.
Let $\sC_j \in N$ which is type 1, $A_j\in \sC_j^n$ the corresponding fusion $n$-bracket, and $u^n_{\sigma(s)}, u^n_{\sigma(s+1)} \in A_j^n$ so that $\omega_{\sigma}(\sC_j)^n_s =1$.
There must exists some $A \in \sC^{n}$ such that $u^n_{\sigma(s)}, u^n_{\sigma(s+1)} \in A^n$ as $\omega_{\sigma}(\sC)^n_s =1$.
But we know that $\pi(A)$ is not a singleton $(n-1)$-bracket thus there exists some $u^{{n-1}}_{\sigma(t)}, u^{n-1}_{\sigma(t+1)} \in A^{n-1}$ such that, without loss of generality, $\pi(u^n_{\sigma(s)})= \pi(u^n_{\sigma(s+1)}) = u^{{n-1}}_{\sigma(t)}$.
We have that $\omega_{\sigma}(\sC)^{n-1}_t =1$,
thus $\omega_{\sigma}(\sC_i)^{n-1}_t =1$ and $u^{n-1}_{\sigma(t)}, u^{n-1}_{\sigma(t+1)}$ belong to an essential bracket in $\sC^{n-1}_i$, and there exists an essential bracket $A \in \sC^n$ with $u^n_{\sigma(s)} \in A^n$. This contradicts the assumption that $\sC_i \sim \sC_j$.

It follows that all of the $\sC_i$ are type 1, and $\pi(\sC_i) = \pi(\sC_j)$ for $i,j$ with $1\leq i,j\leq k$.
Thus $\sC$ must also be type 1.
Take the minimum $s$ such that the fusion bracket $A_1 \in \sC^n$ has $u^n_{\sigma(s)} \in A_1^n$.
Let $\sC_i$ be such that there exists some $A_2 \in \sC_i^n$ with $u^n_{\sigma(s)} \in A_2$. 
Let $u^n_{\sigma(t)} \in A_2$ with $t$ maximum.
Because $\sC_i \neq \sC$, it must be that  $u^n_{\sigma(t+1)} \in A_1$ implying $\omega_{\sigma}(\sC)^n_t=1$.
So, there exists some $j$ such that $\omega_{\sigma}(\sC_j)^n_t=1$ implying there exists some $A_3 \in \sC_j^n$ with $u^n_{\sigma(t)}, u^n_{\sigma(t+1)} \in A_3^n$, but then $\sC_i \nsim \sC_j$, a contradiction.
Hence $\bigcup_{i=1}^k\omega_{\sigma}(\sC_i) \notin \mathbb{B}_{\sigma}$.\footnote{This last step of this argument, considering the case when all $\sC_i$ are type 1, corresponds to the classical verification of the statement that the rays of the wonderful associahedral fan form the indicator vectors for a building set.}

We have proven that the image of $\Delta_{\cT_\sigma}$ is a subcomplex of $\Delta^{\circ}_{\mathbb{B}_{\sigma}}$.
To show that the image is all of $\Delta^{\circ}_{\mathbb{B}_{\sigma}}$, we will take $\tau \in \Delta^{\circ}_{\mathbb{B}_{\sigma}}$ and show that the preimage of $\tau$ under the map induced by $ \sC \mapsto \omega_{\sigma}(\sC)$ is in $\Delta_{\cT_\sigma}$.
First note that the preimage is well-defined: if $\sC, \sC' \in \fX({\cT})^{\sigma}$ and $\omega_{\sigma}(\sC)=\omega_{\sigma}(\sC')$, then $\sC=\sC'$.
Take  $\omega_{\sigma}(\sC_i),\omega_{\sigma}(\sC_j) \in \tau$.  If $\omega_{\sigma}(\sC_i) \subseteq \omega_{\sigma}(\sC_j)$ then $\sC_i\rightarrow \sC_j$.  We claim that if $\omega_{\sigma}(\sC_i) \cap \omega_{\sigma}(\sC_j)  = \emptyset$, then $\sC_i \sim \sC_j$.  If $\sC_i \nsim \sC_j$, then by Lemma \ref{zetajoin} we know that there exists some $\sC_3\in \wh X_{\cT}^{\sigma}$ such that $\omega_{\sigma}(\sC_1) \cup \omega_{\sigma}(\sC_2) = \omega_{\sigma}(\sC_3)$, but this contradicts the assumption that $\tau$ is a nested set.
By Proposition \ref{flagnessresult}, these conditions are sufficient to ensure that the preimage of $\tau$ is a nested collection of collisions in $\Delta_{\cT_\sigma}$.
\end{proof}

\begin{proposition}
\label{nestomaps}
Let $\cF^{\circ}(\mathbb{B}_{\sigma}) = \cF(\mathbb{B}_{\sigma}) \setminus \{star(\{e_i\}): i \in V_{\sigma}\}$.
Then
\begin{align}
\Gamma_{\omega_\sigma}:\cK^{\met}(\Delta_\cT)|_{U_{\sigma}}\rightarrow \cF^{\circ}(\mathbb{B}_{\sigma})
\end{align}
and \begin{align}
\Gamma_{\omega_\sigma}^{\rho}:\cF(\Delta_\cT)|_{\Gamma(U_{\sigma})}\rightarrow \cF^{\circ}(\mathbb{B}_{\sigma})
\end{align}
are piecewise-linear isomorphisms, and $\Gamma_{\omega_\sigma}^{\rho}$ is unimodular on each cone of \\ $\cF(\Delta_\cT)|_{\Gamma(U_{\sigma})}$.
\end{proposition}

We will need the following statement.

\begin{lemma}
\label{smoothnestohedra}\cite{postnikovgp,zelevinsky2005nested, feichtner2004matroid}
Given a building set $\mathbb{B}$, the associated nestohedral fan $\cF({\mathbb{B}})$ is smooth.
\end{lemma}

\begin{proof}(\ref{nestomaps})
It follows by Lemma \ref{omegametric} that these maps are indeed injective piecewise-linear maps defined on the charts of the shuffle atlas for the metric $n$-bracketing complex.
It follows by Lemma \ref{nestoisom} that when we apply these same maps to the nestohedral atlas the image is as described.
It remains to check that $\Gamma_{\omega_\sigma}^{\rho}$ is unimodular on each cone.
We know by Lemma \ref{smooth} that each cone of $\cF(\Delta_\cT)|_{U_{\sigma}}$ is smooth and by Lemma \ref{smoothnestohedra} each cone of $\cF^{\circ}(\mathbb{B}_{\sigma})$ is smooth.
Because $\Gamma_{\omega_\sigma}^{\rho}$ maps the generators for unimodular cones to generators for unimodular cones, the desired statement follows.
\end{proof}

\begin{corollary}\label{unionbraicconeslemma}
For each $n$-bracketing $\sB$ and each $\sigma \in \mathfrak{S}(\sB)$, the cone $\tau_{\omega_{\sigma}}(\sB)$ is a union of cones in the braid arrangement.
\end{corollary}

\begin{proof}
The cone $\tau_{\omega_{\sigma}}(\sB)$ is triangulated by the collection of cones from $\cF(\mathbb{B}_\sigma)$ which it contains.
Each cone in $\cF(\mathbb{B}_\sigma)$ is itself a union of cones from the braid arrangement.
By transitivity, $\tau_{\omega_{\sigma}}(\sB)$ is a union of cones in the braid arrangement.
\end{proof}

As a direct application we have the following corollary. 

\begin{corollary}
\label{omegasigmaunimodular}
The maps $\Gamma_{\omega_\sigma}^{\rho}$ is unimodular on each cone of the velocity fan in the shuffle chart $\Gamma(U_\sigma)$.
\end{corollary}

\begin{proof}
Let $\sB \in \cK(\cT)$ such that $\sigma \in \mathfrak{S}(\sB)$.
Let $Y \in \Delta_\cT$ be a facet such that $\tau(Y) \subseteq \tau(\sB)$.
By Proposition \ref{nestomaps}, $\Gamma_{\omega_\sigma}^{\rho}$ is unimodular on each cone of $\cF(\Delta_\cT)|_{\Gamma(U_{\sigma})}$, and by Lemma \ref{omegametric}, we know that $\Gamma_{\omega_\sigma}^{\rho}$ is linear on all of $\tau(\sB)$.
\end{proof}

\begin{remark}
The image of the permutahedral velocity fan restricted to $\Gamma(U_{\sigma})$ under $\Gamma_{\omega_\sigma}^{\rho}$ is equal to the braid arrangement triangulation of $\cF^{\circ}(\mathbb{B}_{\sigma})$.
\null\hfill$\triangle$
\end{remark}

\begin{remark}
Postnikov's $\mathbb{B}$-trees associated to $\Delta^{\circ}_{\mathbb{B}_{\sigma}}$ are precisely the trees obtained in Theorem \ref{treecombtheorem}.
\null\hfill$\triangle$
\end{remark}

\begin{remark}
In this section when proving that the image of a cone in the velocity fan is a union of cones in the braid arrangement, we did so by appealing either to the triangulated or permutahedral velocity fan highlighting the role of these structures.
Alternately, we could have proven these statements directly by appealing to the defining inequalities of preposet cones described by Postnikov-Reiner-Williams \cite{postnikov2008faces}.
\null\hfill$\triangle$
\end{remark}

\begin{remark}
The velocity fan always contains a distinguished chamber of the braid arrangement, which we now describe.
For this remark, we will break with convention and denote the coordinates of a vectors ${\bf v} \in \mathbb{R}^m$ by ${\bf v} = (v_1, \ldots, v_m)$, i.e.\ we will not use superscripts.
There is a distinguished $n$-bracketing, which we will denote $\sB_{\cT}$, and distinct extended collisions $\sC_{\cT,\,1}, \ldots ,\sC_{\cT,\,m} \in \fX(\cT)$ such that for each $i$, $\sC_{\cT,\,i}\leq \sB_{\cT}$ and 
\begin{align}
\rho(\sC_{\cT,\,i})_k =
\begin{cases}
1 \,\,\,\text{if}\,\,\, k\leq i \,\,\,\text{and} \\
0 \,\,\,\text{if}\,\,\, i < k. \\
\end{cases}
\end{align}

In particular, we always take $\sC_{\cT,\,m} = \sB_{\min}(\cT)$. 
Let $u^n_{t_n}$ be the largest labeled vertex of depth $n$ in $\cT$.
By induction, a bracketing $\sB_{\cT'}$ as described above exists for the tree 
$\cT' =\cT\setminus \{u^n_{t_n}\}$.
For describing $\sB_{\cT}$, we construct a collision $\sC_{\cT,\,i}$ from each extended collision $\sC_{\cT',i}$.
For each $A' \in \sC_{\cT',i}^{n-1}$ with $\pi(u^n_{t_n}) \in A'^{\,n-1}$, we introduce an $n$-bracket $A \in \sC_{\cT',i}^{n}$ with $A^n = \{u^n_{t_n}\}$.
We extend the height partial order for $\sC_{\cT',i}$ by making $A$ larger than any other $n$-bracket $\wt A$ with $\pi({\wt A})=\pi(A)$.
It is straightforward to check that the $\sC_{\cT,\,i}$ are compatible, and have coordinates as described.
\null\hfill$\triangle$
\end{remark}

\section{The local nested fiber product structure of the velocity fan}
\label{localsection}

Combinatorially, each facet of an associahedron factors as a product of two smaller associahedra.
This observation is essential for understanding the operadic structure of the associahedron.
The second author characterized the recursive fiber product description of 2-associahedra: combinatorially, each facet in a 2-associahedron factors as a product of a smaller 2-associahedron with a fiber product of 2-associahedra over a 1-associahedron \cite{b:2-associahedra}.
This description was utilized by the second author and Carmeli \cite{bottman_carmeli} for understanding the relative 2-operad structure of 2-associahedra.
In this section, we extend this description to categorical $n$-associahedra via iterated fiber products\footnote{We expect this description to be an essential step in verifying that categorical $n$-associahedra form an example of a relative $n$-operad naturally viewed as an $\omega$-operad in the sense of Batanin \cite{batanin1998monoidal}.}, and we investigate the extent to which this recursive structure is realized geometrically by the velocity fan.

The recursive structure of the associahedron is fully realized in the geometry of the wonderful associahedral fan $\cF_n$:
the star of a ray in $\cF_n$ factors as a product of two smaller the wonderful associahedral fans.\footnote{It is perhaps an under-appreciated fact that it is impossible to realize this recursive structure via a polytopal realization of the associahedron.
More precisely, it is impossible to construct a sequence of polytopes $\{P_i\}$ such that $P_i$ realizes the $i$th associahedron and each facet of $P_i$ factors as a product $P_j \times P_k$ with $j,k <i$.
In fact, it is already impossible to do this in dimension 3 with polytopes $P_0,P_1,P_2,P_3$.
In recent work, Masuda--Thomas--Tonks--Vallette \cite{masuda2021diagonal} showed that each facet of Loday's associahedron factors as a product of weighted versions of smaller Loday associahedra.
This refines the above observation each facet of Loday's associahedron factors as a product of smaller polytopes which are normally equivalent to Loday associahedra.}
We will see in Theorem \ref{localvelocity} that this nice property partially extends to our velocity fans in an interesting way.
We begin with some preliminary results on the theory of fiber products of sets, posets, cone complexes, and fans.

The primary reference for fiber products of fans and cone complexes, of which we are aware, is the relatively recent work of Molcho \cite{molcho2021universal} on universal stacky semistable reduction building on work of Abramovich-Karu \cite{abramovich1997weak}.
Fiber products of polytopes have been investigated in \cite{buczynska2007geometry}\cite{sullivant2007toric}\cite{haase2021existence},
and pullbacks of diagrams of polyhedral objects appear in \cite{adiprasito2018log,adiprasito2018semistable}.

\subsection{Fiber products of sets, posets, cone complexes, and fans}
\label{localvelocitysubsection}

\

\begin{definition}
{\bf Fiber products of sets.}
Let $X_i$ for $1\leq i \leq n$ and $Y$ be sets with maps $f_i : X_i \rightarrow Y$, then the fiber product of the $X_i$ over $Y$, written
\begin{align}
\prod_{1\leq i \leq n}^Y X_i
\end{align}
is the set
\begin{align}
\{ (y,x_1, \dots, x_n)\in Y \times \prod_{1\leq i \leq n}X_i : \forall \,\, 1\leq i\leq n, f_i(x_i)= y \},
\end{align}
\null\hfill$\triangle$
\end{definition}

\begin{remark}
Note that the fiber product of a collection of empty sets over a set $X$ is equal to $X$.  This consideration is relevant for fiber products of $n$-associahedra.
\null\hfill$\triangle$
\end{remark}

We review that fiber products exist in the categories of posets, cone complexes, and fans.

\begin{definition}
Let $(P,\leq_P)$ and $(Q,\leq_Q)$ be partially ordered sets.
A morphism in the category of posets is a map $f:P\rightarrow Q$ such that if $x,y \in P$ and $x\leq_P y$ then $f(x) \leq_Q f(y)$.
\null\hfill$\triangle$
\end{definition}

\begin{definition}\label{fiberproductsofposets}
{\bf Fiber products of partially ordered sets.}
Let $(P_i, \leq_i)$ for $1\leq i \leq n$ and $(Q, \leq_Q)$ be posets with poset morphisms $f_i : P_i \rightarrow Q$, then the fiber product of the $P_i$ over $Q$, written
\begin{align}
\prod_{1\leq i \leq n}^Q P_i
\end{align}
is the fiber product of the underlying sets equipped with the following partial order:
\begin{align}
(q,p_1, \dots, p_n), (q',p_1', \dots, p_n') \in \prod_{1\leq i \leq n}^Q P_i
\end{align}
then
\begin{align}
(q,p_1, \dots, p_n) \leq (q',p_1', \dots, p_n')
\end{align}
if $p_i\leq p_i'$ for all $1\leq i \leq n$.
\null\hfill$\triangle$
\end{definition}

\noindent
Note that because the $f_i$ are morphisms in the definition above, the condition $p_i\leq p_i'$ ensures that $q\leq q'$.

\begin{proposition}[Fiber products of conical sets]
For $0\leq i \leq s$, let $\tau_i$ be a conical set.
Suppose that $f_i:\tau_i\rightarrow \tau_0$ are linear transformations.

Then the set theoretic fiber product of the $\tau_i$ over $\tau$ 
\begin{align}
\prod_{1\leq i\leq s}^{\tau}\tau_i 
\end{align}
is naturally equipped with the structure of a conical set.
\end{proposition}

\begin{proof}
Let ${\bf x} = \prod_{1\leq i\leq s}{\bf x}_i, \, {\bf y} = \prod_{1\leq i\leq s}{\bf y}_i$ be points in $\prod_{1\leq i\leq s}^{\tau}\tau_i$ and ${\bf z},{\bf w} \in \tau$ such that $f_i({\bf x}_i)={\bf z}$ and $f_i({\bf y}_i)={\bf w}$ for all $1\leq i \leq s$.
Let $\lambda, \mu \in \mathbb{R}_{\geq 0}$.
Then for each $i$ with $1\leq i\leq s$, we have
\begin{align}
f_i(\lambda{\bf x}_i+\mu{\bf y}_i)=\lambda f_i({\bf x}_i)+\mu f_i({\bf y}_i)=\lambda {\bf z}+\mu{\bf w}.
\end{align}
Therefore $\prod_{1\leq i\leq s}^{\tau}\tau_i$ is closed under taking conical combinations.
\end{proof}

In the special case of cones, fiber products of conical sets have the following geometric description.

\begin{definition}[Fiber products of cones]\label{fiberproductofcones}
For $0\leq i \leq s$, let $\tau_i$ be a cone in $\mathbb{R}^{m_i}$ having coordinates ${\bf x}^i=(x^i_1, \ldots, x^i_{m_i})$.
Suppose that for $0\leq i\leq s$, we have linear transformations $f_i:\mathbb{R}^{m_i}\rightarrow \mathbb{R}^{m_0}$ are such that $f_i(\tau_i) \subseteq \tau_0$ and $f_0$ is the identity map.
Let $T = (\tau_0, \tau_1, \ldots, \tau_s)$ and $f = (f_0, f_1, \ldots, f_s)$.
Let $V_f$ be the linear space in $\prod_{i=0}^s \mathbb{R}^{m_i}$ cut out by the equations $f_i({\bf x}^i) = f_j({\bf x}^j)$ for all $0\leq i, j\leq s$.

Then we define the \emph{fiber product of the $\tau_i$ over $\tau$} is the cone 
\begin{align}
\prod_{1\leq i\leq s}^{\tau}\tau_i \coloneqq \left(\prod_{i=0}^s \tau_i\right) \cap V_f.
\end{align}
\null\hfill$\triangle$
\end{definition}

The following Lemma characterizes of the faces of a cone which is obtained as the intersection of another cone with a linear space.
\begin{lemma}\label{facesofconescapV}
Let $C_1 \subset \mathbb{R}^n$ be a polyhedral cone and $V \subset \mathbb{R}^n$ be a linear space.
Suppose that $C_2 =C_1 \cap V$.
Let $\cF_1$ and $\cF_2$ be the faces of $C_1$ and $C_2$, respectively.
Then 
\begin{align}
\cF_2 = \{\tau\cap V: \tau \in \cF_1\}.
\end{align}
Moreover, there are bijections
\begin{align}
\cF_2 \, \longleftrightarrow \, \{\tau \cap V: \tau \in \cF_1,\, V\,\, {\rm \,\,intersects} \,\,\tau \,\, {\rm in\,\,its\,\, relative\,\, interior}\}
\end{align}
\begin{align}
\longleftrightarrow \, \{\tau \cap V: \tau \in \cF_1,\, \forall\, \tau' \in \cF_1,\, \tau'< \tau,\, \tau' \cap V \neq \tau \cap V \}.
\end{align}
\end{lemma}

\begin{proof}
This can be verified by appealing to polarity and linear programming.
\end{proof}

\begin{lemma}\label{prodofconecomplexes}
For $1\leq i \leq s$, let $\cF_i$ be cone complexes.
Then $\cG=\prod_{1\leq i\leq s}\cF_i$ is a cone complex.
\end{lemma}

\begin{proof}
This is straightforward to check.
\end{proof}

\begin{definition}[Fiber products of cone complexes]
\label{fiberconecomplexes}
For $1\leq i \leq s$, let $\cF_i$ be cone complexes.
Let $\cF$ be a cone complex. 
Suppose that $f_i:\cF_i\rightarrow \cF$ are morphisms of cone complexes, i.e.\ each $f_i$ is a piecewise-linear transformation such that for each conical set $\tau_i \in \cF_i$ there exists some conical set $\tau \in \cF$ with $f_i(\tau_i)\subseteq \tau$.

Then we define the \emph{fiber product of the $\cF_i$ over $\cF$} as 
\begin{align}
\prod_{1\leq i\leq s}^{\cF}\cF_i \coloneqq \Bigl\{\prod_{1\leq i\leq s}^{\tau} \tau_i: \tau \in \cF, \tau_i \in \cF_i, f_i(\tau_i)\subseteq \tau \Bigr\}.
\end{align}
\null\hfill$\triangle$
\end{definition}

\begin{proposition}\label{conecomplexesclosedunderfibprodprop}
Cone complexes are closed under taking fiber products.
\end{proposition}

\begin{proof}
Let $\cG=\cF \times \prod_{1\leq i\leq s}\cF_i$, and let $\cH=\prod_{1\leq i\leq s}^{\cF}\cF_i$ be a fiber product of cone complexes as in Definition \ref{fiberconecomplexes}.
By Lemma \ref{prodofconecomplexes}, $\cG$ is a cone complex.
We now verify the conditions of a cone complex for $\cH$.
Let $X(\cH)$ be the set of points in $\cH$.
Let
\begin{align}
\phi:\cG \rightarrow \cH
\end{align}
be the natural surjection with $\phi(\tau) = \tau \cap X(\cH)$ for each $\tau \in \cG$.
By piecewise-linearity of the $f_i$, if we take $\tau = (\tau_0, \tau_1, \ldots, \tau_s)$, then we can alternately describe $\phi(\tau) = \prod_{1\leq i \leq s}^{\tau_0} \tau_i$.
By the definition of a cone complex, there exists a linear isomorphism $f_{\tau}$ such that $f_{ \tau}({\tau})$ is a polyhedral cone in some Euclidean space.
We can take $f_\tau=(f_{\tau_0}, \ldots, f_{\tau_s})$ to be the tuple of such maps for the constituent conical sets. 
Thus, given a conical set $\tau \in \cG$, Definition \ref{fiberproductofcones} demonstrates that $f_{\tau}(\phi(\tau))=f_{\tau}(\tau) \cap V_f$ for the linear space $V_f$ described there, where $f = (id, f_{\tau_0}\circ f_1 \circ f_{\tau_1}^{-1},\ldots, f_{\tau_0}\circ f_s \circ f_{\tau_s}^{-1})$, and $id$ is the identity map. 

Let $\tau_1$ be a conical set in $\cH$.
Let $\tau_2$ be a face of $\tau_1$.
We wish to show that $\tau_2 \in \cH$.
Let $\tau'_1 \in \cG$ be a conical set such that $\phi(\tau'_1)=\tau_1$.
By Lemma \ref{facesofconescapV}, and the observation of the previous paragraph (noting that we can take the same linear space $V_f$ for both $f_{\tau'_1}(\tau'_1)$ and $f_{\tau'_2}(\tau'_2)$), we know that there exists some $\tau'_2 \in \cG$ which is a face of $\tau'_1$ such that $\phi(\tau'_2)=\tau_2$.

 Next let $\tau_1, \tau_2 \in \cH$.
 We wish to show that $\tau_3 = \tau_1 \cap \tau_2$ is a union of faces of both $\tau_1$ and $\tau_2$.
 Let $\tau'_1, \tau'_2 \in \cG$ be conical sets such that $
 \phi(\tau'_1)=\tau_1$ and $\phi(\tau'_2)=\tau_2$.
 The map $\phi$ commutes with taking unions and intersections, hence $\phi(\tau'_1 \cap \tau'_2) = \tau_1 \cap \tau_2 = \tau_3$.
 By Lemma \ref{prodofconecomplexes}, we see that $\tau'_1 \cap \tau'_2$ is a finite union of faces of both $\tau'_1$ and $\tau'_2$, and the desired statement now follows from Lemma \ref{facesofconescapV}. 

\end{proof}

\begin{definition}
Let $\cF$ and $\cG$ be cone complexes and $f$ a linear transformation such that for each conical set $\sigma \in \cF$ there exists some conical set $\tau \in \cG$ such that $f(\sigma)= \tau$, then $f$ is \emph{combinatorially weakly semistable}.
\null\hfill$\triangle$
\end{definition}

Molcho says that that a map of cone complexes is \emph{weakly semistable} if in addition to being combinatorially weakly semistable, the map also induces a surjection of the underlying lattices, and notes that this definition is a slight weakening of the notion of weak semistability described by Abramovich-Karu \cite{abramovich1997weak}.
We remark that the projection maps described in Definitions \ref{fibervelocityfans} and \ref{fiberproductmetricbracketingdef} do satisfy Molcho's lattice condition and thus are weakly semistable.

The following lemma is implicit in \cite{molcho2021universal} and was rediscovered by the authors.
We note that this version applies to fiber products of nonrational fans, which might be of interested to polyhedral geometers.

\begin{lemma}\label{faceposetfiberproduct}
Let $\cF_i$ for $1\leq i \leq s$ and $\cF$ be cone complexes as in Definition \ref{fiberconecomplexes}, and ${\overline f}_i:\cP(\cF_i)\rightarrow \cP(\cF)$ is the induced map on face posets.
If the $f_i$ are all combinatorially weakly semistable then the face poset of the fiber product of cone complexes is the fiber product of the face posets of the cone complexes: 
\begin{align}
\cP\,\Bigg(\, \prod_{1\leq i\leq s}^{\cF}\cF_i \, \Bigg) \, = \, \prod_{1\leq i\leq s}^{\cP(\cF)}\cP(\cF_i).
\end{align}
\end{lemma}

\begin{proof}
We implicitly identify individual conical sets with cones.
By Lemma \ref{facesofconescapV}, there is a bijection between cones of the fiber product and cones in the product of the $\cF_i$ such that the linear space in question $V_f$ intersects their relative interior.
Thus, in order to determine the face poset of the fiber product, we should understand when $V_f$ intersects the relative interior of a cone in the product of the $\cF_i$.

 Let $\tau_i \in \cF_i$ for $1\leq i\leq s$, and $\tau \in \cF$ such that $f_i(\tau_i)\subseteq \tau$ for each $i$.
Suppose further that $f_i(\tau_i) = \tau$ for each $i$.
If this occurs then for any (some) point ${\bf v}$ in the relative interior of $\tau$ there exists a point ${\bf v}_i$ in the relative interior of $\tau_i$ such that $f_i({\bf v}_i) ={\bf v} $.
To see this, note that combinatorial weak semistability gives that $f_i$ surjects the rays of $\tau_i$ to the rays of $\tau$.
A point is in the relative interior of a cone if and only if it can be expressed as a positive combination of all of the ray generators.
The relative interior of $(\tau, \tau_1, \dots, \tau_s)$ equal to the product of the relative interiors of each of these cones, and the statement follows. 
 
Conversely, suppose there exist $1 \leq j,k \leq s$ with $f_j(\tau_j) \neq f_k(\tau_k)$, then the corresponding linear space $V_f$ does not intersect the relative interior of a cone in the product of the $\cF_i$ as, without loss of generality, $f_j(\tau_j)$ does not contain a point in the relative interior of $f_k(\tau_k)$. 
\end{proof}

The geometric description of fiber products of cones from Definition \ref{fiberproductofcones} extends to the category of fans equipped with morphisms which are globally linear, and this has interesting consequences for fans, which don't apply to general cone complexes.

\begin{definition}[Fiber products of fans]
\label{fiberfans}
For $1\leq i \leq s$, let $\cF_i$ be a fan in $\mathbb{R}^{m_i}$.
Suppose that $\cF_{i}$ has coordinates ${\bf x}^i = (x_1^i,\ldots, x_{m_i}^i)$.
Let $\cF$ be a fan in $\mathbb{R}^n$ and suppose that $f_i:\mathbb{R}^{m_i}\rightarrow \mathbb{R}^{n}$ are linear transformations such that for each cone $\tau_i \in \cF_i$ there exists some cone $\tau \in \cF$ such that $f_i(\tau_i)\subseteq \tau$.
Let $V_f$ be the linear space in $\prod_{i=1}^s \mathbb{R}^{m_i}$ cut out by the equations $f_i({\bf x}^i) = f_j({\bf x}^j)$ for all $1\leq i, j\leq s$.

Then we define the \emph{fiber product of the $\cF_i$ over $\cF$} as 
\begin{align}
\prod_{1\leq i\leq s}^{\cF}\cF_i \coloneqq \Bigl\{ \left(\prod_{i=1}^s \tau_i\right) \cap V_f: \tau_i \in \cF_i\Bigr\}. 
\end{align}
\null\hfill$\triangle$
\end{definition}

\begin{corollary}
\label{fibercomplete}
Let $\cF_i$ and $\cF$ be fans as in the statement of Definition \ref{fiberfans}.
If each $\cF_i$ is a complete fan, then
\begin{align}
\prod_{1\leq i \leq s}^{\cF} \cF_{i}
\end{align}
is a complete fan.
\end{corollary}

\begin{proof}
The fiber product above
is the intersection of $\prod_{i =1}^s \cF_{i} $ with a linear space, hence it is complete.
\end{proof}

An application of Corollary \ref{fibercomplete} is described in Remark \ref{completenessviafiberproduct}.

\begin{remark}
Fiber products of projective, i.e.\ polytopal, fans are projective.
This will not be utilized.
Furthermore, the fiber products of generalized permutahedra with respect to coordinate projection maps are again generalized permutahedra.
This was independently observed in the work of Sanchez \cite{sanchez2023derived}, and a corresponding statement for Bergman fans was observed by Francois-Hampe \cite{franccois2013universal}.
\null\hfill$\triangle$
\end{remark}

The following two lemmas are straightforward to verify.

\begin{lemma}\label{projectingredundantcoordinates}
Let $\cF$ be a fan in $\mathbb{R}^m$.
Suppose that that there exists some pair of indices $1 \leq i<j \leq m$ such that for every point ${\bf v} \in \cF$, we have ${\bf v}_i = {\bf v}_j$.
Let $\pi_i$ be the map which projects away from the $i$th coordinate.
Then $\pi_i:\mathbb{R}^m \rightarrow \mathbb{R}^{m-1}$ is a linear isomorphism from $\cF$ onto its image.
\end{lemma}

\begin{lemma}\label{isofiberconecomplexlemma}
Let $\cF,\cG,\cF_i, \cG_i$ be cone complexes with $1\leq i \leq s$.
Suppose that we have piecewise-linear isomorphisms $f_i:\cF_i\rightarrow \cF, g_i:\cG_i \rightarrow \cG, \psi_i:\cF_i\rightarrow \cG_i, \psi:\cF\rightarrow \cG$.
Suppose that for each $i$ we have $g_i \circ \psi_i = \psi \circ f_i$.
Then the given maps induce a piecewise-linear isomorphism
\begin{align}
\prod_{1\leq i\leq s}^{\cF}\cF_i \cong \prod_{1\leq i\leq s}^{\cG}\cG_i.
\end{align}
\end{lemma}

\subsection{Fiber products of \texorpdfstring{$n$}{n}-associahedra}\label{fiberproductsn-associahedrasubsection}

\

In this subsection we define fiber products of categorical $n$-associahedra and use these fiber products to describe the local structure of categorical $n$-associahedra.

Recall, from Definition \ref{restrictbracket}, we denote the $n$-associahedron $\cK(\cT|_{A})$ for some $k$-bracket $A$ as $\cK({A})$. 

\begin{definition}[Fiber products of categorical $n$-associahedra]
\label{fiberproductnassociahedra}
Let $\cT$ be a rooted plane tree and $\sC \in \fX(\cT)$.
We recursively define the iterated fiber product $\cK(\cT|_{\sC})$. For $A \in \sC^n$ essential\footnote{Recall from Definition \ref{essentialbrackets} that $A\in \sC^n$ is essential if and only if it is nontrivial, but for $k<n$, essential brackets may also be certain maximal brackets.}, let $\cK(\cT|_{\sC,\,{A}})=\cK({A})$.
Next, let $A$ be an essential $k$-bracket in $\sC^k$ for $k <n$.
We denote
\begin{align}
\cK(\cT|_{\sC,\,{A}}) \coloneqq \prod_{
\substack{
A'\in \sC^{k+1},
\\
A' \,\,\text{essential},
\\ \pi(A')=A}}^{\cK(A)}\cK(\cT|_{\sC,\,{A'}})
\end{align}
with the maps from the factors to the base defined presently.
Suppose by induction that we have defined maps $\Pi_{A'}:\cK(\cT|_{\sC,\,{A'}})\rightarrow \cK(A')$ for each $A'\in \sC^{k+1}$.
We compose with the map $\pi:\cK(A')\rightarrow \cK(A)$ to get the desired map $\pi\circ\Pi_{A'}:\cK(\cT|_{\sC,\,{A'}})\rightarrow \cK(A)$.
Finally, in order to extend the induction, we define $\Pi_{A}:\cK(\cT|_{\sC,\,{A}})\rightarrow \cK(A)$ to be the product of the maps $\pi\circ\Pi_{A'}$ over all $A' \in \sC^{k+1}$ with $A'$ essential, and $\pi(A') = A$.\footnote{We remark that the map $\Pi_{A}$ may not be a product of copies of the map $\pi^{n-k}$ as there may exists some $j$-bracket $A'$ with $j\geq k$ living over $A$ such that there exist no $j+1$ bracket which projects to $A'$.}

Let $A_{\bullet}$ be the fusion bracket for $\sC$.
We define

\begin{align}
\cK(\cT|_{\sC}) \coloneqq \cK(\cT|_{\sC,\,A_{\bullet}}).
\end{align}
\null\hfill$\triangle$
\end{definition}

We refer to elements of $\cK(\cT|_{\sC})$ as \emph{fiber products of $n$-bracketings}.
By repeated application of Definition \ref{fiberproductsofposets}, $\cK(\cT|_{\sC})$ is naturally equipped with the structure of a poset.

\begin{definition}
\label{facetterminology}
Let $\sC \in \fX(\cT)$, then we define the \emph{facet} associated to $\sC$ to be the interval $\cK(\cT)_{\sC} \coloneqq [\,\sC, \star]\subseteq \cK(\cT)$.
\null\hfill$\triangle$
\end{definition}

The following lemma describes the face poset of a facet of an $n$-associahedron.

\begin{lemma}\footnote{To make things look nicer in this section, we have written $\cT/\sC$ as $\cT/_{\sC}$.}
\label{facetposet}
Let $\sC \in \fX(\cT)$, then the corresponding facet of the $n$-associahedron admits a canonical isomorphism
\begin{align}
\cK(\cT)_{\sC} \cong \cK(\cT/_{\sC}) \times \cK(\cT|_{\sC}).
\end{align}
\end{lemma}

\begin{figure}[ht]
\centering
\def\svgwidth{0.7\textwidth}
%% Creator: Inkscape 1.2 (dc2aeda, 2022-05-15), www.inkscape.org
%% Accompanies image file 'localfiberproduct.pdf' (pdf, eps, ps)
\begingroup%
  \makeatletter%
  \providecommand\color[2][]{%
    \errmessage{(Inkscape) Color is used for the text in Inkscape, but the package 'color.sty' is not loaded}%
    \renewcommand\color[2][]{}%
  }%
  \providecommand\transparent[1]{%
    \errmessage{(Inkscape) Transparency is used (non-zero) for the text in Inkscape, but the package 'transparent.sty' is not loaded}%
    \renewcommand\transparent[1]{}%
  }%
  \providecommand\rotatebox[2]{#2}%
  \newcommand*\fsize{\dimexpr\f@size pt\relax}%
  \newcommand*\lineheight[1]{\fontsize{\fsize}{#1\fsize}\selectfont}%
  \ifx\svgwidth\undefined%
    \setlength{\unitlength}{295.06012538bp}%
    \ifx\svgscale\undefined%
      \relax%
    \else%
      \setlength{\unitlength}{\unitlength * \real{\svgscale}}%
    \fi%
  \else%
    \setlength{\unitlength}{\svgwidth}%
  \fi%
  \global\let\svgwidth\undefined%
  \global\let\svgscale\undefined%
  \makeatother%
  \begin{picture}(1,0.61116653)%
    \lineheight{1}%
    \setlength\tabcolsep{0pt}%
    \put(0,0){\includegraphics[width=\unitlength,page=1]{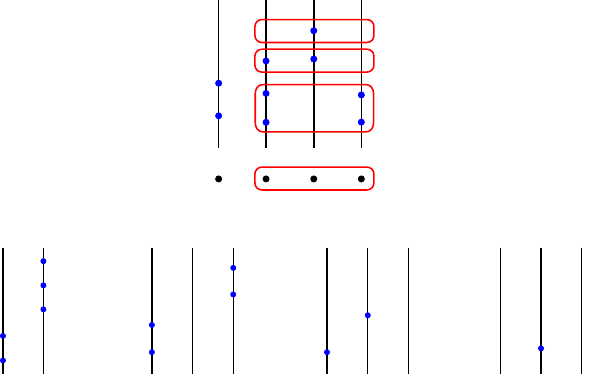}}%
    \put(0.12460948,0.09897314){\makebox(0,0)[lt]{\lineheight{1.25}\smash{\begin{tabular}[t]{l}$\times$\end{tabular}}}}%
    \put(0.40686451,0.09897314){\makebox(0,0)[lt]{\lineheight{1.25}\smash{\begin{tabular}[t]{l}$\times$\end{tabular}}}}%
    \put(0.69506398,0.09897314){\makebox(0,0)[lt]{\lineheight{1.25}\smash{\begin{tabular}[t]{l}$\times$\end{tabular}}}}%
    \put(0,0){\includegraphics[width=\unitlength,page=2]{localfiberproduct.pdf}}%
  \end{picture}%
\endgroup%

\caption{A collision $\sC$ in a 2-associahedron, and the corresponding factorization of interval $\cK(\cT)_{\sC}=[\sC,\star] \subseteq \cK(\cT)$ into a smaller 2-associahedron $\cK(\cT/_\sC)$ on the left times the fiber product of three 2-associahedra over a single 1-associahedron $\cK(\cT|_{\sC})$ on the right.
} 
\label{localfiberproductfig}
\end{figure}

\begin{proof}
We produce a poset isomorphism
\begin{align}
{\Psi}: \cK(\cT)_{\sC} \rightarrow \cK(\cT/_{\sC}) \times \cK(\cT|_{\sC}).
\end{align}

Let $\sB \in \cK(\cT)_{\sC}$, and let $A \in \sB^k$ be nontrivial, then one of following mutually exclusive cases holds:
\begin{enumerate}

\item\label{posetmap0}
There exists $A'\in \sC^k$ with $A' = A$.

\item\label{posetmap1} For every essential bracket $A'\in \sC^k$, we have that $A \cap A'=\emptyset$. 

\item\label{posetmap2} There exists some nontrivial $A'\in \sC^k$ such that $A' \subsetneq A$.

\item\label{posetmap3} 
There exists an essential bracket $A'$ in $\sC^k$, 
 such that $A \subsetneq A'$.
\end{enumerate}

Let $\sB^+$ be the collection of $k$ brackets $A$ in $\sB$ from cases (\ref{posetmap0}), (\ref{posetmap1}), and (\ref{posetmap2}) above for all $0\leq k \leq n$, and height partial order inherited from $\sB$.
Let $\sB^-$ be the set of $k$ brackets $A$ in $\sB$ from cases (\ref{posetmap0}) and (\ref{posetmap3}) above for all $0\leq k \leq n$, and height partial order inherited from $\sB$.
Then
\begin{enumerate}
\item The collections $\sB^+$ and $\sB^-$ are both $n$-bracketings.

\item There exists a canonical map $\Psi^+$ from $\sB^+$ to an $n$-bracketing in $\cK(\cT/_{\sC})$.

\item There exists a canonical map $\Psi^-$ from $\sB^-$ to a fiber product of $n$-bracketings in $\cK(\cT|_{\sC})$.
\end{enumerate} 

For verifying statement (1) above, one can check that the (\textsc{nested}) and (\textsc{partition}) properties are preserved for $\sB^+$ and $\sB^-$.
For statement (2), recall Definition \ref{quotientbracketingdef}.
Statement (3) follows from the (\textsc{partition}) property for $\sB^-$.

Let $D$ be the map which sends the $n$-bracketing $\sB$ to the ordered pair $(\sB^+,\sB^-)$ of $n$-bracketings, let $\wt{\Psi}\coloneqq\Psi^+\times \Psi^-$, and let $\Psi \coloneqq\wt{\Psi}\circ D: \cK(\cT)_{\sC} \rightarrow \cK(\cT/_{\sC}) \times \cK(\cT|_{\sC})$.
Then $\Psi$ is a poset isomorphism. 
\end{proof}

\begin{remark}
The bracketings $\sB^+$ and $\sB^-$ can alternately be defined using collisions:

\begin{align}
\sB^+ = \bigvee_{
\substack{
\wt \sC \,\in \,\fX(\cT),
\\
{\wt \sC}\leq \sB,
\\
{\sC \rightarrow \wt \sC \,\,\text{or}\,\,\sC \sim \wt \sC
}
}} {\wt \sC}
 \,\,\,\,\,\,\,\,\text{and}\,\,\,\,\,\,\,\, 
\sB^- = \bigvee_{
\substack{
\wt \sC \,\in \,\fX(\cT),
\\
{\wt \sC}\leq \sB,
\\
{\wt \sC \rightarrow \sC
}
}} {\wt \sC} .
\end{align}

This perspective is further developed in Lemma \ref{intervalatoms}.
\null\hfill$\triangle$
\end{remark}

\begin{remark}
By Lemma \ref{facetposet}, the iterated fiber product $\cK(\cT|_{\sC})$ may be considered as the \emph{restriction} of $\cK(\cT)$ to $\sC$, hence our choice of notation.
\null\hfill$\triangle$
\end{remark}

\begin{remark}
In order to extend the Lemma \ref{facetposet} to obtain a similar description of the full face lattice of an $n$-associahedron, we first observe what a ridge (facet of a facet) looks like.
A facet of $\cK(\cT/_{\sC}) \times \cK(\cT|_{\sC})$ must be either a facet of $\cK(\cT/_{\sC})$ or $\cK(\cT|_{\sC})$.
In the former case, we can apply Lemma \ref{facetposet}.
In the latter case, it can be shown that a facet of $\cK(\cT|_{\sC})$ factors as a product of two smaller fiber products of $n$-associahedra.
Thus, while a facet of an $n$-associahedron does not factor as a product of two smaller $n$-associahedra, a facet of a fiber product of $n$-associahedra factors as a product of two smaller fiber products of $n$-associahedra. 
This suggests that the category of fiber products of $n$-associahedra may be more well-behaved than the category of $n$-associahedra.
\null\hfill$\triangle$
\end{remark}

\subsection{Fiber products of velocity fans}

\

\begin{definition}[Fiber products of velocity fans]
\label{fibervelocityfans}
Let $\cT$ be a rooted plane tree and $\sC \in \fX(\cT)$.
We recursively define the iterated fiber product $\cF(\cT|_{\sC})$.
For $A \in \sC^n$ essential, let $\cF(\cT|_{\sC,\,{A}})=\cF({A})$.
Next, let $A$ be an essential $k$-bracket in $\sC^k$ for $k <n$.
We denote
\begin{align}
{\wt \cF}(\cT|_{\sC,\,{A}}) \coloneqq \prod_{
\substack{
A'\in \sC^{k+1},
\\
A' \text{essential},
\\ \pi(A')=A}}^{\cF(A)}\cF(\cT|_{\sC,\,{A'}})
\end{align}
We define maps from the factors to the base as in Definition \ref{fiberproductnassociahedra}, but replace the combinatorial maps $\pi$ with the corresponding geometric projections.
An element of ${\wt \cF}(\cT|_{\sC,\,{A}})$ can be viewed as a tuple of vectors.
The initial parts of the vectors in a tuple agree with the image in the base.
Lemma \ref{projectingredundantcoordinates} allows us to delete the initial parts of these vectors leaving the coordinates of the base $\cF(A)$, and we define the resulting collection as $\cF(\cT|_{\sC,\,{A}}).$\footnote{This choice to remove redundant coordinates does not seem strictly necessary, but has been done primarily for dimensionality reasons, e.g.\ as discussed in Remark \ref{completenessviafiberproduct}, the dimension of $\cF(\cT|_{\sC})$ will now be equal to the dimension of the ambient space.}

Let $A_{\bullet}$ be the fusion bracket for $\sC$.
We define
\begin{align}
\cF(\cT|_{\sC}) \coloneqq \cF(\cT|_{\sC,\,A_{\bullet}})
\end{align}
\null\hfill$\triangle$
\end{definition}

\begin{lemma}\label{fibervelocityfanfaceposetlemma}
The face poset of $\cF(\cT|_{\sC})$ is equal to $\cK(\cT|_{\sC})$.
\end{lemma}

\begin{proof}
The maps from the factors to the base in Definition \ref{fibervelocityfans} are combinatorially weakly semistable, hence the statement follows by Theorem \ref{mainvelocitystatement} and iterated application of Lemma \ref{faceposetfiberproduct}.
\end{proof}

Recall the definition of the localization of a cone complex a face given in Definition \ref{localizationforabtractfandef}.

\begin{remark}\label{completenessviafiberproduct}
We can utilize fiber products of fans to give an alternate proof of the completeness property for the velocity fan.
A fan is complete if and only if the stars of all of its rays are complete.
If $\cF_1$ and $\cF_2$ are fans of the same dimension with the same dimensional lineality spaces and the same face posets, then $\cF_1$ is complete if and only if $\cF_2$ is complete.
The fans $\cF(\cT)_{{\rho}(\sC)}$ and $\cF(\cT/_{\sC}) \times \cF(\cT|_{\sC})$ are indeed fans of the same dimension with the same dimensional lineality spaces, and by Lemmas \ref{facetposet} and \ref{fibervelocityfanfaceposetlemma} they have isomorphic face posets.
Furthermore, we know by induction and Corollary \ref{fibercomplete} that $\cF(\cT/_{\sC}) \times \cF(\cT|_{\sC})$ is complete, thus $\cF(\cT)_{{\rho}(\sC)}$ is complete, implying $\cF(\cT)$ is complete.
\null\hfill$\triangle$
\end{remark}

In the previous remark, it was important that we allow $\cF_1$ and $\cF_2$ have the same face poset, but not necessarily be equal.
Indeed, $\cF(\cT)_{{\rho}(\sC)}$ and $\cF(\cT/_{\sC}) \times \cF(\cT|_{\sC})$ are not equal (if $\cT$ is not concentrated).
We proceed towards a description of the relationship between these fans.
We begin with a lemma clarifying what are the rays of the localization $\cF(\cT)_{\rho(\sC)}$ for $\sC \in \fX(\cT)$.
This is essentially a combinatorial result about $\cK(\cT)_{\sC}$ and potentially could have appeared in subsection \ref{collisionsasatoms}, although we leverage Theorem \ref{mainvelocitystatement} for simplifying our argumentation.

\begin{lemma}\label{intervalatoms}
Let $ \sC, {\wt \sC} \in \fX(\cT)$.
The poset interval $\cK(\cT)_{\sC}$ is atomic, and ${\wt \sC} \vee \sC$ defines an atom of $\cK(\cT)_{\sC}$ if and only if
\begin{enumerate}
\item\label{atoms1} $\sC \sim {\wt \sC}$,

\item\label{atoms2} For $\{\sC_1, \sC_2\} = \{\sC,\wt \sC\}$ with $\sC_1\rightarrow {\sC_2}$,

\begin{enumerate}
\item\label{atoms2a} $\sC_1$ and $\sC_2$ have different fusion brackets, or

\item\label{atoms2b} $\sC_1$ and $\sC_2$ have the same fusion brackets, and the largest $k\leq n$ such that $\pi^{n-k}(\sC_2) \neq \pi^{n-k}({\sC_1})$, there exists a unique essential $k$-bracket $A \in {\sC_2}^k$ such that $A \notin \sC_1^k$.
\end{enumerate}
\end{enumerate}
\end{lemma}

\begin{proof}
The fact that $\cK(\cT)_{\sC}$ is atomic follows as it is the face poset of a fan $\cF(\cT)_{{\rho}(\sC)}$.
In order to determine the atoms, we see that by Lemma \ref{raysoflocalizationlemma} they must correspond to the rays of $\cF(\cT)_{{\rho}(\sC)}$.
These rays correspond to rays ${\rho}({\wt \sC})$ of $\cF(\cT)$ which together with ${\rho}(\sC)$ and ${\bf 1}$ generate a cone of $\cF(\cT)$.
This occurs precisely when $ {\wt \sC}$ and $\sC$ are the unique collisions contained in ${\wt \sC} \vee \sC$.
It is straightforward to verify that this occurs precisely in the cases listed above.
\end{proof}

The following is the main result of this section.

\begin{theorem}\label{localvelocity}
Let ${\rho}(\sC)$ be a ray of the velocity fan $\cF(\cT)$.
There exists a piecewise-unimodular isomorphism 
\begin{align}
\Theta_{\sC}:\,\,\, \cF(\cT)_{{\rho}(\sC)} \rightarrow \cF(\cT/_{\sC}) \times \cF(\cT|_{\sC})
\end{align}
such that $\Theta_{\sC}^{-1}$ restricted to $\cF(\cT/_{\sC}) \times {\bf 0}$ is linear and agrees with $(P_{\sigma}^T)^{-1}\times {\bf 0}$ restricted to $\cF(\cT/_{\sC}) \times {\bf 0}$, for $\sigma$ a compatible $\cT$ shuffle for $\sC$, and with the coordinates ordered according to $\sigma$.
\end{theorem}

The map $\Theta_{\sC}$ will be defined as the unique piecewise-linear function induced by a certain map sending ray and lineality space generators of $\cF(\cT)_{{\rho}(\sC)}$ to ray and lineality space generators of $\cF(\cT/_{\sC}) \times \cF(\cT|_{\sC})$.

We begin by clarifying the standard ray and lineality space generators for $\cF(\cT)_{{\rho}(\sC)}$, $\cF(\cT/_{\sC})$, and $ \cF(\cT|_{\sC})$.
We take the lineality space generators in $\cF(\cT)_{{\rho}(\sC)}$ to be ${\bf 1}$ and ${\rho}(\sC)$.
We take the ray generators for $\cF(\cT)_{{\rho}(\sC)}$ to be the set of ray generators ${\rho}(\sC')$ from $\cF(\cT)$ which satisfy the conditions from Lemma \ref{intervalatoms}.
The fan $\cF(\cT/_{\sC})$ is a velocity fan so we take the standard ray and lineality space generators for this velocity fan.
We denote the all-ones vector which generates its lineality space as ${\bf 1}/_{\sC}$.

The lineality space generator for $\cF(\cT|_{\sC})$ will be ${\bf 1}|_{\sC}$, i.e.\ the all-ones vector on its coordinates. We utilize Lemma \ref{intervalatoms} for defining the ray generators for $\cF(\cT|_{\sC})$: if $ {\wt \sC} \in \fX(\cT)$ with ${\wt \sC}\rightarrow \sC$ satisfies condition (\ref{atoms2}) from Lemma \ref{intervalatoms}, then the restriction of ${\wt \sC}$ to any bracket $A \in \sC^k$ either gives a extended collision there or is the collection of singleton brackets in $A$.
In the former case, we take the corresponding ray or lineality space generator for $\cF(A)$.
In the latter case, we take ${\bf 0}$.
We define the ray generator associated to $\wt \sC$ to be the concatenation of these vectors.
We extend the ray and lineality space generators of $\cF(\cT/_{\sC})$ and $\cF(\cT|_{\sC})$ to ray and lineality space generators of $\cF(\cT/_{\sC}) \times \cF(\cT|_{\sC})$ by taking the product with ${\bf 0}$ on the appropriate side.

We define a map $\wh{\Theta}_\sC$ from ray and lineality space generators of $\cF(\cT)_{{\rho}(\sC)}$, described in Lemma \ref{intervalatoms}, to the ray and lineality space generators of $\cF(\cT/_{\sC}) \times \cF(\cT|_{\sC})$ as
\begin{align}
\wh{\Theta}_\sC({\rho}(\sC'))=
\begin{cases}
\,\, ({\rho}(\sC'),0) & \text{if} \,\, \sC \sim \sC', \\
\,\, ({\rho}(\sC'/\sC),{\bf 1}_{\sC}) & \text{if} \,\, \sC \rightarrow \sC', \\
\,\, (0,{\rho}(\sC')) & \text{if} \,\, \sC' \rightarrow \sC.
\end{cases}
\end{align}

We take the convention that if $\sC' = \sC$, then $\rho(\sC'/\sC) = {\bf 0}$, so the second and third cases agree.
Note that if $\sC' = \sB_{\min}$, then $\wh{\Theta}_\sC({\rho}(\sB_{\min}))= \wh{\Theta}_\sC({\bf 1})=({\bf 1}/_{\sC},{\bf 1}|_{\sC})$. We define $\Theta_{\sC}$ to be the unique piecewise-linear map which extends $\wh{\Theta}_\sC$. We will prove that $\Theta_{\sC}$ is well-defined, bijective, and unimodular on cones.
Following the paradigm of this paper, our analysis of these maps will go by way of metric $n$-bracketings.

\

\subsection{Fiber products of metric \texorpdfstring{$n$}{n}-bracketing complexes}

\

Recall Definition \ref{c-metricdef} and Lemma \ref{C-metriclemma}.
\begin{lemma}\label{localvelocitytolocalmetric}
Let $\sB \in \cK(\cT)$, and let $\sC_0, \dots, \sC_k$ be extended collisions with $\sC_i \leq \sB$ for all $i$. 
Let $\sC_0 = \sB_\min$, $\sC_1 = \sC$, and take $\lambda_i \in \mathbb{R}$ with $\lambda_i \geq 0$ for $i\geq 2$.
Given a $\sC$-metric $n$-bracketing $\ell_{\sB}$ such that $\ell_{\sB} = \sum _{i =0}^k \lambda_i \ell(\sC_i)$, we define
\begin{align}
\Gamma_{\sC}(\ell_{\sB}) \coloneqq \sum _{i =0}^k \lambda_i \rho(\sC_i).
\end{align}
Then $\Gamma_{\sC}: \cK^{\met}(\cT)_{\sC} \rightarrow \cF(\cT)_{{\rho}(\sC)}$ is a piecewise-linear isomorphism.
\end{lemma}

\begin{proof}
This follows by adding and subtracting ${\rho}(\sC)$, one of the lineality space generators for $\cF(\cT)_{{\rho}(\sC)}$, as in the proof of Lemma \ref{C-metriclemma}, and applying Proposition \ref{mainvelocitythm}.
\end{proof}

Given a $k$-bracket $A$ on $\cT$, let $\cK^{\met}({A})\coloneqq \cK^{\met}(\cT|_A)$.

\begin{definition}[Fiber products of metric $n$-bracketing complexes]
\label{fiberproductmetricbracketingdef}
Let $\cT$ be a rooted plane tree and $\sC \in \fX(\cT)$.
We recursively define the iterated fiber product $\cK^{\met}(\cT|_{\sC})$. For $A \in \sC^n$ essential, let $\cK^{\met}(\cT|_{\sC,\,{A}})=\cK^{\met}({A})$.
Next, let $A$ be an essential $k$-bracket in $\sC^k$ for $k <n$.
We denote

\begin{align}
{\cK^{\met}}(\cT|_{\sC,\,{A}}) \coloneqq \prod_{
\substack{
A'\in \sC^{k+1},
\\
A' \text{essential},
\\ \pi(A')=A}}^{{\cK^{\met}}(A)}{\cK^{\met}}(\cT|_{\sC,\,{A'}})
\end{align}

We define maps from the factors to the base as in Definition \ref{fiberproductnassociahedra}, but replace the combinatorial maps $\pi$ with the corresponding metric $n$-bracketing projections.

Let $A_{\bullet}$ be the fusion bracket for $\sC$.
We define

\begin{align}
{\cK^{\met}}(\cT|_{\sC}) \coloneqq {\cK^{\met}}(\cT|_{\sC,\,A_{\bullet}})\,.
\end{align}

We refer to elements of ${\cK^{\met}}(\cT|_{\sC})$ as \emph{fiber products of metric $n$-bracketings}. \null\hfill$\triangle$
\end{definition}

\begin{lemma}
\label{fibermetricbracketingfaceposetlemma}
The face poset of ${\cK^{\met}}(\cT|_{\sC})$ is equal to $\cK(\cT|_{\sC})$.
\end{lemma}

\begin{proof}
The maps from the factors to the base in Definition \ref{fiberproductmetricbracketingdef} are combinatorially weakly semistable, hence the statement follows by iterated application of Lemma \ref{faceposetfiberproduct}.
\end{proof}

One can canonically produce ray and lineality space generators for ${\cK^{\met}}(\cT|_{\sC})$ mimicking the construction for ${\cF}(\cT|_{\sC})$.

\begin{lemma}
There is a piecewise-linear isomorphism $\Gamma_{\cT|_{\sC}}:\cK^{\met}(\cT|_{\sC})\rightarrow {\cF}(\cT|_{\sC})$ induced by sending the ray and lineality space generators of ${\cK^{\met}}(\cT|_{\sC})$ to the ray and lineality space generators of $\cF(\cT|_{\sC})$.
\end{lemma}
 
\begin{proof}
The statement follows by iteratively applying Lemma \ref{isofiberconecomplexlemma}.
\end{proof}

\begin{definition}
Let $\sC' \in \fX(\cT)$ be as in Lemma \ref{intervalatoms}. We define a map ${\wh\Psi}^{\met}$ from the ray and lineality space generators of $\cK^{\met}(\cT)_{\sC}$ to the ray and lineality space generators for $\cK^{\met}(\cT/_\sC)\times  {\cK^{\met}}(\cT|_{\sC})$:
\begin{align}
{\wh \Psi}^{\met}(\ell(\sC'))\coloneqq
\begin{cases}
\,\, (\ell(\sC'),0) & \text{ if}\,\, \sC' \sim \sC \\
\,\, (\ell(\sC'/\sC),\ell(\sC)) & 
\,\,\text{if} \,\, \sC \rightarrow \sC' \\
\,\, (0,\ell(\sC')) & \text{ if}\,\, \sC' \rightarrow \sC,
\end{cases}
\end{align}
where $(0,\ell(\sC'))$ in the third case denotes the product of the restrictions of $\ell(\sC')$ to each essential bracket in $\sC$.
\null\hfill$\triangle$
\end{definition}

\begin{lemma}\label{localmetricdecomposition}
The map ${\wh \Psi}^{\met}$ extends to a piecewise-linear isomorphism \begin{align}
\Psi^{\met}:\cK^{\met}(\cT)_{\sC}\rightarrow \cK^{\met}(\cT/_\sC)\times  {\cK^{\met}}(\cT|_{\sC}).
\end{align}
\end{lemma}

\begin{proof}
Combining Lemma \ref{C-metriclemma} and Lemma \ref{raysoflocalizationlemma}, we find that, given a $\sC$-metric $n$-bracketing $\ell_{\sB} \in \cK^{\met}(\cT)_{\sC}$, there exists compatible $\sC_0, \ldots, \sC_k \in \wh{\fX}(\cT)$ such that $\sC_0 = \sB_{\min}$, $\sC_1 = \sC$, for each $i\geq 2$, $\sC_i$ satisfies the conditions of Lemma \ref{intervalatoms},  and there exist $\lambda_i \in \mathbb{R}$ with $\lambda_i \geq 0$ for $i\geq 2$ such that 
\begin{align}
\ell_{\sB} = \sum_{i=0}^k \lambda_i \ell(\sC_i),
\end{align}
and we define
\begin{align}
\Psi^{\met}(\ell_{\sB}) \coloneqq \sum_{i=0}^k \lambda_i {\wh \Psi}^{\met}(\ell(\sC_i)).
\end{align}

We prove that $\Psi^{\met}$ is linear on each conical set in $\cK^{\met}(\cT)_{\sC}$.  Let $\sC_1, \dots, \sC_k$ and $\sC'_0, \dots ,\sC'_{l}$ be compatible collisions which satisfies the conditions of Lemma \ref{intervalatoms}.
Take $\sC_0 = \sC'_0 = \sB_\min$, and $\sC_1 = \sC'_1 = \sC$.
We prove that there exist $\lambda_i, \gamma_i \in \mathbb{R}$ with $\lambda_i, \gamma_i \geq 0$ for $i\geq 2$ such that
\begin{align}
\label{c-metriclinear}
\sum _{i =0}^k \lambda_i \ell(\sC_i) = \sum _{i =0}^l \gamma_i \ell(\sC'_i)
\end{align}
if and only if
\begin{align}\label{c-metriclinear2}
\sum _{i=0}^k \lambda_i \Psi^{\met}(\sC_i) = \sum _{i =0}^l \gamma_i \Psi^{\met}(\sC'_i).
\end{align}

Using the characterization of $\sC$-metric $n$-bracketings in Lemma \ref{C-metriclemma}, it is clear, as in Lemmas \ref{fusionmetricbracketingdecomp} and \ref{metricpartitionlemma}, that we can group the terms in the sums in equation (\ref{c-metriclinear}) according to their fusion brackets and the resulting sums are $\sC$-metric $n$-bracketings. Moreover we have an equality of $\sC$-metric $n$-bracketings in the sum (\ref{c-metriclinear}) if and only if the restricted sums for each fusion bracket agree.  Thus it suffices to prove that the  ``if and only if" statement holds for $\sC$-metric $n$-bracketings whose underlying $n$-bracketings have the same fusion bracket.

 Let $A$ be the fusion bracket for $\sC$.  Let $A'$ be the fusion bracket for the sum in question.  It must be that either
 \begin{enumerate}
 \item\label{metricmap1}  ${\overline D}(A) \cap {\overline D}(A') = \emptyset$,
 \item\label{metricmap2}  there exists some $k\geq 0$ such that $\pi^{k}(A')\subsetneq A$,
    \item\label{metricmap3} there exists some $k\geq 0$ such that $\pi^{k}(A)\subsetneq A'$,
     \item\label{metricmap4}  $A=A'$.
 \end{enumerate}
 For the cases (\ref{metricmap1}) and (\ref{metricmap2}), the equivalence is trivial.  Suppose we are in case $(\ref{metricmap3})$.  By Lemma \ref{equalityofcoefficientsums}, we have that $ \sum_{i=0}^k\lambda_i = \sum_{i=0}^l\gamma_i$, and the equivalence follows.  

 Suppose we are in case (\ref{metricmap4}).  Here we further subdivide the sum into two parts.  The first corresponding to collisions $\sC'$ with $\sC' \rightarrow \sC$, and the second corresponding to collisions $\sC'$ with $\sC \rightarrow \sC'$ and $\sC' \neq \sC$.  It again suffices to establish the equivalence for the two sums separately.  The case $\sC' \rightarrow \sC$ is straightforward; check each essential bracket in $\sC$ individually.  For the case $\sC \rightarrow \sC'$ with $\sC' \neq \sC$, the equivalence is again clear as $ \sum_{i=0}^k\lambda_i = \sum_{i=0}^l\gamma_i$.

 \end{proof}

\subsection{Proof of Theorem \ref{localvelocity}}

\

\begin{lemma}
The map $\wh{\Theta}_{\sC}$ extends to a a piecewise-linear isomorphism
\begin{align}
\Theta_{\sC}:\,\, \cF(\cT)_{{\rho}(\sC)} \rightarrow \cF(\cT/_{\sC}) \times \cF(\cT|_{\sC}).
\end{align}
\end{lemma}

\begin{proof}
To verify that $\wh{\Theta}_{\sC}$ indeed extends to a piecewise-linear function $\Theta_{\sC}$, we claim that we can interpret $\Theta_{\sC}$ as a composition of three piecewise-linear isomorphisms 
\begin{align}
(\,\Gamma_{\cT/\sC}\times\Gamma_{\cT|_{\sC}})\circ \Psi^{\met} \circ \Gamma_{\sC}^{-1},
\end{align}
where $\Gamma_{\cT/\sC}$ is the map $\Gamma$ from Definition \ref{Gammamap} associated to the $n$-associahedron $\cK(\cT/\sC)$.  It is straightforward to verify that the above map  agrees with $\wh{\Theta}_{\sC}$ at the level of ray and lineality space generators, and the claim follows.
\end{proof}

\begin{lemma}\label{fibersmooth}
The map $\Theta_{\sC}$ is unimodular on each cone of $\cF(\cT)_{{\rho}(\sC)}$.
\end{lemma}

\begin{proof}
Because the reduced permutahedral velocity fan ${\overline \cF}(\mathscr{O}(\cT))$ is a unimodular triangulation of ${\overline \cF}(\cT)$, we know that ${\overline \cF}(\mathscr{O}(\cT))_{{\rho}(\sC)}$ induces a unimodular triangulation of ${\overline \cF}(\cT)_{{\rho}(\sC)}$.
The map $\Theta_{\sC}$ applied to ${\cF}(\mathscr{O}(\cT))_{{\rho}(\sC)}$ induces a subdivision of ${\cF}(\cT/_{\sC}) \times {\cF}(\cT|_{\sC})$.
For checking that $\Theta_{\sC}$ is unimodular, it suffices to check that given a cone $\tau \in {\cF}(\mathscr{O}(\cT))_{{\rho}(\sC)}$, its image $\Theta_{\sC}(\tau)$ is unimodular.

We can see that $\Theta_{\sC}(\tau)$ factors as the product of two cones $\tau_1\times \tau_2$ in the subdivision of $\cF(\cT/_{\sC}) \times \cF(\cT|_{\sC})$.  The subdivision of $\cF(\cT/_{\sC})$ is equal to ${\cF}(\mathscr{O}(\cT/_{\sC}))$, hence $\tau_1$ is unimodular.  It remains to prove that $\tau_2$ is unimodular.
Construct a matrix whose columns correspond to ray and lineality space generators for $\tau_2$ as in Lemma \ref{smooth} and Proposition \ref{permutahedralvelocitymainprop}.
Order the columns corresponding to rays according to the order of their preimages in $\cF(\cT)$, and order the rows so that the entries corresponding to any particular essential $k$-bracket $A \in \sC^k$ form a consecutive set.
When we restrict to such a set of rows, we find the ray and lineality space generators of a cone in the triangulated velocity fan of the $k$-associahedron $\cK(A)$, although there may be repetitions of columns.  We can now proceed as in the argument for Lemma \ref{smooth} by first performing the corresponding column subtractions.
The appropriate replacement for the adjoint permutation is a block matrix with blocks give by $P_{\sigma}^T$ associated to the various $\cF(A)$ where $A$ ranges over all essential brackets in $\sC$.

We have chosen to work with the permutahedral velocity fan in this argument, rather than the triangulated velocity fan, because the recursive structure is less complicated: after quotienting a generalized collision by another generalized collision, the image is again a generalized collision.
With this observation in hand, we can perform the Laplace expansion along the desired row and apply induction.
\end{proof}

\begin{lemma}\label{linearpartof7.26lemma}
The function $\Theta_{\sC}$ is such that $\Theta_{\sC}^{-1}$ restricted to $\cF(\cT/_{\sC}) \times {\bf 0}$ is linear and agrees with $(P_{\sigma}^T)^{-1}\times {\bf 0}$ for $\sigma$ a compatible $\cT$-shuffle for $\sC$.  
\end{lemma}

\begin{proof}
In order to verify this statement, we note that at the level of rays, if $\sC'$ satisfied the conditions of Lemma \ref{intervalatoms} and $\sC \rightarrow \sC'$, then $\Theta_{\sC}({\rho}(\sC')-{\rho}(\sC))={\rho}(\sC'/\sC)$, and if $\sC \sim \sC'$, then $\Theta_{\sC}({\rho}(\sC'))= {\rho}(\sC'/\sC)$.
By Lemma \ref{contractionlemma}, $\Theta_{\sC}^{-1}$ restricted to ray and lineality space generators of $\cF(\cT/_{\sC}) \times {\bf 0}$ agrees with $(P_{\sigma}^T)^{-1}\times {\bf 0}\,$, with the coordinates ordered according to $\sigma$.  Because there is at most one piecewise-linear map which extends a given map on ray and lineality space generators of a fan, the two functions must agree.  
\end{proof}

\section{Concentrated \texorpdfstring{$n$}{n}-associahedra as generalized permutahedra}
\label{s:concentrated_realization}

We investigate a natural class of categorical $n$-associahedra which we call \emph{concentrated}.
These are precisely the $n$-associahedra associated to trees which admit no nontrivial $\cT$-shuffles.
This section serves two primary purposes, first to show how various results from earlier sections specialize nicely for concentrated $n$-associahedra, and second to provide a polytopal realization of concentrated $n$-associahedra advertising future work on projectivity for velocity fans.

We observe that, for a concentrated $n$-associahedron, the velocity fan is a coarsening of the braid arrangement, the triangulated velocity fan is the normal fan of a nestohedron, and the permutahedral velocity fan is always the braid arrangement.
We produce a combinatorially meaningful support function on the velocity fan of any concentrated $n$-associahedron which gives a polyhedral realization of this $n$-associahedron as a generalized permutahedron.
We reexpress these polyhedral realizations as positive Minkowski sums of standard simplices and we provide an explicit vertex presentation.  We observe that the simplices in Postnikov's Minkowski sum description of a nestohedron, specialized to the case of the triangulated concentrated $n$-associahedra, form a super set of the simplices in our Minkowski sum description of the concentrated $n$-associahedra.
We demonstrate that the constrainahedra of the second and third authors can be recovered as a Minkowski sum of concentrated $n$-associahedra ranging over all concentrated trees $\cT$ having a specified number of vertices at each depth.
This suggests a theory of \emph{hypershuffle products of generalized permutahedra} extending the shuffle products of Chapoton and Pilaud.

\begin{definition}
A depth $n$ tree $\cT$ is \emph{concentrated} if for every $k<n$, there exists a single $u \in V(\cT)$ of depth $k$ with $D(u) \neq \emptyset$.
An $n$-associahedron $\cK_\cT$ is \emph{concentrated} if $\cT$ is concentrated.
\null\hfill$\triangle$
\end{definition}

\begin{remark}
For 2-associahedra, a tree $\cT$ is concentrated if all points are concentrated on a single line in the corresponding arrangement.
The class of concentrated $n$-associahedra includes the 1-associahedra, which correspond to trees $\cT$ of depth 1, and multiplihedra, which correspond to concentrated trees of depth 2 with exactly two depth 1 vertices.
\null\hfill$\triangle$
\end{remark}

\begin{remark}
A tree $\cT$ is concentrated if and only if the identity permutation is the only $\cT$-shuffle.
\null\hfill$\triangle$
\end{remark}

\begin{lemma}\label{rho=zeta}
Let $\cK(\cT)$ be a concentrated $n$-associahedron, and let $\sC$ an extended collision in $\cK(\cT)$ so that $\mathfrak{S}(\sC) = \{e\}$, where $e$ is the identity permutation.
Then $\rho(\sC)=\zeta(\sC) = \omega_e(\sC)$.
\end{lemma}

\begin{proof}
Because $\cT$ admits no nontrivial shuffles, we know that $\rho(\sC)$ is a 0-1 vector.
A vertex $u^k_i$ belongs to a nontrivial $k$-bracket with another vertex $u^k_j$ where $i<j$, i.e.\ $\zeta(\sC)^k_i =1$, if and only if $u^k_i$ belongs to a nontrivial $k$-bracket with $u^k_{i+1}$, i.e.\ $\rho(\sC)^k_i=1$, if and only if $\omega_e(\sC)^k_i=1$.
\end{proof}

\begin{lemma}\label{concentratednnestohedral}
Let $\cK(\cT)$ be a concentrated $n$-associahedron.
The velocity fan of $\cK(\cT)$ is a coarsening of the braid arrangement, and the triangulated velocity fan of $\cK(\cT)$ is the normal fan of a nestohedron.
\end{lemma}

\begin{proof}
There are various ways to see this.
For example, by Lemma \ref{rho=zeta}, we know that $\rho(\sB)=\zeta(\sB)$, hence the piecewise-linear map from Theorem \ref{permutahedroidthm} is the identity map.
This shows that each cone in the velocity fan is union of cones in the braid arrangement, i.e.\ the velocity fan is a coarsening of the braid arrangement.
The corresponding statement for the triangulated velocity fan follows as the nestohedral atlas for the triangulated velocity fan consists of a single chart corresponding to the identity permutation $e$, the set $V_{e} = \emptyset$, and the map $\Gamma_{\omega_e}^{\rho}$ described in Proposition \ref{nestomaps} is the identity map.
\end{proof}

We recall the definition of graph associahedra \cite{graphassoc,postnikovgp}, which form a special class of nestohedra.

\begin{definition}
Let $G=(V,E)$ be an undirected graph. A \emph{tube} is a subset of $B\subseteq V$ such that the induced subgraph $G[B]$ is connected.
The \emph {graphical building set} of $G$ is the collection of all tubes of $G$.
\null\hfill$\triangle$
\end{definition}

\begin{definition}
The \emph{graph associahedron} for a graph $G$ is the nestohedron associated to the graphical building set of $G$.
\null\hfill$\triangle$
\end{definition}

\begin{lemma}
The triangulated velocity fans of concentrated $n$-associahedra are the normal fans of certain graph associahedra.
\end{lemma}

\begin{proof}
Given concentrated tree $\cT$, we first construct the desired graph $G({\cT})$\footnote{The notation $G_{\cT}$ is used for the collision mixed graph.} (see Figure \ref{graphassociafig} for an example).
We identify the vertices of $G({\cT})$ with the coordinates $x_i^k$ in the velocity fan for $\cT$. Vertices of the same depth $x^{k}_i$ and $x^k_j$ are connected by an edge if $j = i+1$, and for $k > l$ vertices $x^k_i$ and $x^l_j$ are connected by an edge if the vertex $u^k_i$ is an descendant of $u^l_j$ or $u^l_{j-1}$ in $\cT$.
Graph associahedra are nestohedra, and triangulated concentrated $n$-associahedra are nestohedra by Lemma \ref{concentratednnestohedral}, where the building sets are described.
Therefore, one is left to verify that the blocks of the two building sets are the same.  Indeed, collisions, as presented by their 0-1 vectors, are in obvious bijection with the tubes on $G(\cT)$.
\end{proof}

\begin{figure}
\def\svgwidth{0.8\textwidth}
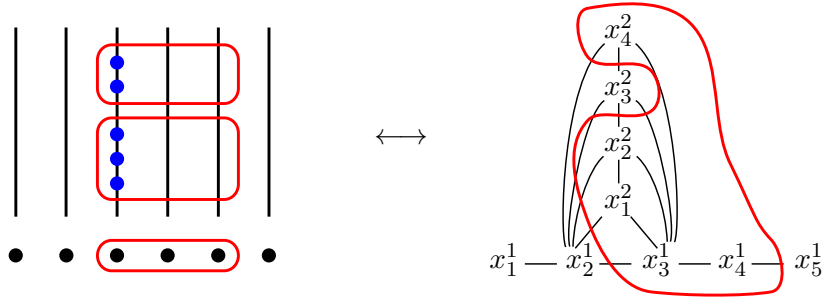
\centering
\caption{
\label{graphassociafig}
On the left: a collision in a concentrated 2-associahedron.  On the right: the corresponding graph and tube.}
\end{figure}

Recall, the permutahedral velocity fan defined in \ref{permutahedral_assoc} is a finer triangulation of the velocity fan.

\begin{lemma}
Permutahedral velocity fans for concentrated $n$-associahedra are braid arrangements.
\end{lemma}
\begin{proof}
For a concentrated $n$-associahedron, the generalized collisions $\sD$ from Definition \ref{generalized_collision} have associated set of ray generators $\{\rho(\sD):\sD \in \mathfrak{P}(\cT)\}$ equal to the set of all 0-1 vectors in $\mathbb{R}^m$ different from ${\bf 1}$.
The partial order for generalized collisions described in Definition \ref{perm_partial} specializes to the usual inclusion order on the 0-1 vectors, and the result follows.
\end{proof}

\begin{lemma}
\label{coarsening}
The velocity fan of an $n$-associahedron $\cK(\cT)$ is a coarsening of the braid arrangement if and only if $\cK(\cT)$ is either concentrated or 1-dimensional.
\end{lemma}

\begin{proof}
By Lemma \ref{concentratednnestohedral} and Proposition \ref{braidonedimfan}, we know that velocity fans of concentrated and 1-dimensional $n$-associahedra are coarsenings of the braid arrangement.  Conversely, suppose that $\cK(\cT)$ is neither concentrated nor 1-dimensional.
Then we can construct a collision $\sC$ such that $\rho(\sC)$ has three distinct values, and thus cannot be a ray of the braid arrangement.
To see this, $\cT$ must have a vertex with two disjoint paths of length at least two leaving from it.
We can use these paths to construct a collision $\sC$ such that $\rho(\sC)$ has a single entry 1, a single entry 2, and the nonempty set of remaining entries are all 0.
\end{proof}

\begin{remark}
These are not the only categorical $n$-associahedra which can be realized as generalized permutahedra.
For example, the 2-associahedron $W_{110}$ is a hexagon and thus can be realized as a generalized permutahedron although not by one whose normal fan is the velocity fan.  Asymptotically, the fraction of $n$-associahedra containing an octagonal face approaches 1, thus almost all $n$-associahedra cannot be realized as generalized permutahedra.
\end{remark}

\begin{remark}
If $\cK(\cT)$ is a concentrated $n$-associahedron and $\rho(\sC)$ is a ray of the velocity fan $\cF(\cT)$, then the piecewise unimodular isomorphism afforded by Theorem \ref{localvelocity} is the identity map.
\null\hfill$\triangle$
\end{remark}

\begin{remark}
Concentrated $n$-associahedra provide a natural generalization of the 2-associahedra where all the points are on the left-most line.
(In the notation of \cite{b:2-associahedra}, these are the 2-associahedra $W_{k0\cdots0}$.)
These have an important interpretation in symplectic geometry.
Recall that in \cite{b:realization}, the first author realized 2-associahedra as \emph{witch curves}, which are the domains for the pseudoholomorphic quilts whose counts define the structure maps in the symplectic $(A_\infty,2)$-category \textsf{Symp} \cite{abouzaid_bottman}.
\textsf{Symp} is intended to be an implementation of functoriality of the Fukaya category with respect to Lagrangian correspondences.
If we restrict ourselves to symplectic manifolds satisfying geometric conditions that are sufficient to exclude figure eight bubbling, it is possible to implement functoriality for the Fukaya category in a simpler fashion, using quilted disks with boundary marked points.
Such quilted disks are exactly what these left-concentrated 2-associahedra parametrize.
\null\hfill$\triangle$
\end{remark}

\subsection{Collision units for polytopes}

\

In this subsection we realize each concentrated $n$-associahedron as generalized permutahedron whose normal fan is the velocity fan.
We proceed by constructing a support function on the velocity fan of a concentrated $n$-associahedron.

Let $\cF$ be a complete fan in $\mathbb{R}^m$, and $h:\mathbb{R}^m \rightarrow \mathbb{R}$ a piecewise-linear function on $\cF$.
We say that we say that $h$ is \emph{strongly convex} if, for any $\mathbf{u},\mathbf{v} \in \mathbb{R}$ which do not live in a common chamber of $\cF$, we have
\begin{align}
\label{eq:convex_function}
h(\mathbf{u}+\mathbf{v})
<
h(\mathbf{u})+h(\mathbf{v}).
\end{align}

Given a polytope $P\subseteq \mathbb{R}^m$ (not necessarily full-dimensional), its \emph{support function} is $h_P:\mathbb{R}^m \rightarrow \mathbb{R}$, where for ${\bf v} \in \mathbb{R}^n$,
\begin{align*}
h_P(\mathbf{v})
\coloneqq
\min_{p \in P} \langle v, p\rangle.
\end{align*}
The support function $h_P$ is a piecewise-linear strongly convex function on the normal fan of $P$.  Conversely, any strongly convex piecewise-linear function $h$ on a complete fan $\cF$ is the support function of a polytope $P$ with normal fan $\cF$: every ray generator gives rise to a facet inequality, and generators of the lineality space give rise to equations.
We call a fan $\cF$ \emph{projective} if there exists a polytope $P$ whose normal fan is $\cF$.

\begin{definition}
Two polytopes are \emph{normally equivalent} if their normal fans coincide. 
\null\hfill$\triangle$
\end{definition}

For a concentrated tree $\cT$, we now introduce a collection $U (\cT)$ of combinatorial objects called \emph{collision units} which will be utilized for defining support functions and vertex coordinates.

\begin{definition}
\label{collision_unit}

A \emph{collision unit} $\bf{c}$ is a list of length $n$, with entries $\mathbf{c}_k$ such that
\begin{enumerate}
\item for each $k$,  either $\mathbf{c}_k=\varnothing$ or $\mathbf{c}_k= \{ u^k_i,u^k_j \}$, a pair of distinct vertices of depth $k$, 

\item $\bf{c}$ has at least one nonempty entry, and

\item if $\bf{c}$ has at least two nonempty entries, then each entry $ \{ u^k_i,u^k_j \}$ is of the form $\{u^k_{i},u^k_{i+1}\}$.
\null\hfill$\triangle$
\end{enumerate}
\end{definition}

\begin{definition}
For a given concentrated tree $\cT$, the collection $U(\cT)$ consists of all the collision units as defined above.
\null\hfill$\triangle$
\end{definition}

\begin{example}
In a concentrated 2-associahedron, collision units are lists $(\{ u^1_i,u^1_j\},\varnothing)$ (pairs of lines), or lists $(\varnothing, \{ u^2_i,u^2_j\})$ (pairs of points), or lists $(\{ u^1_i,u^1_{i+1}\}, \{ u^2_j,u^2_{j+1}\})$ (four-tuples of two lines and two points, where both the lines and the point are neighbors).
\null\hfill$\triangle$
\end{example}

\begin{remark}
Collision units are introduced for extending the classical realization of the associahedron as a generalized permutahedron: Shnider-Sternberg \cite{shnider1993quantum} describe the value of the support function on a ray of the wonderful associahedral fan corresponding to a single nontrivial 1-bracket $A$ as the number of pairs of points contained in that bracket. Loday \cite{Lodayassociahedron} 
provides the vertex description, which we generalize in \S\ref{vertices} also using collision units. 
\null\hfill$\triangle$
\end{remark}

Let $\bf{c}(\varnothing)$ denote the list having no nonempty entries.
By the definition above, $\bf{c}(\varnothing)$ is not a collision unit, but it arises naturally when erasing entries in collision units.

\begin{definition}
For a collision unit $\mathbf{c} \in U(\cT)$, its \emph{level} $l(\mathbf{c})$ is defined as the maximal index such that $\mathbf{c}_{l (\mathbf{c}) } \neq \varnothing$.
We let $\pi: U(\cT) \to U(\pi(\cT)) \cup \{\mathbf{c}(\varnothing)\}$, denote the map which erases the last entry of a collision unit $\bf{c}$.
\null\hfill$\triangle$
\end{definition}

\begin{definition}
\label{unitcontain}
Let $A = (A^1, \ldots, A^k)$ be a $k$-bracket.
We say that a collision unit $\mathbf{c}$ of level $k$ is \emph{contained} in $A$ if $\mathbf{c}_i \subseteq A^i$ for all $i \leq k$.
Let $U(A)$ denote the set of level $k$ collision units contained in $A$.
\null\hfill$\triangle$
\end{definition}

\begin{proposition}
\label{explicit-support}
For a $k$-bracket $A = (A^1, \ldots, A^k)$, the cardinality of the set $U(A)$ is equal to 
\begin{align}
\sum_{i=1}^k \binom{|A^i|}{2} + \prod_{i=1}^k |A^i| - \sum_{i=1}^k |A^i| + k -1 .
\end{align}
\end{proposition}

\begin{proof}
The summand $\sum_{i=1}^k \binom{|A^i|}{2}$ counts the collision units in $A$ with exactly one nonempty entry, where there is no neighboring condition imposed on the entries.
The number of other collision units in $A$ is clearly equal to 
$$ \sum_{I \subseteq \{ 1, \ldots , k \}, |I| \geq 2 } \prod_{i \in I} (|A^i|-1).$$

We thus need to verify the equality
$$\sum_{I \subseteq \{ 1, \ldots , k \}, |I| \geq 2 } \prod_{i \in I} (|A^i|-1) = \prod_{i=1}^k |A^i| - \sum_{i=1}^k |A^i| + k -1.$$

Moving $-\sum_{i=1}^k |A^i| + k -1$ to the left, we obtain the sum over all the subsets including the empty one:
$$ \sum_{I \subseteq \{ 1, \ldots , k \} } \prod_{i \in I} (|A^i|-1) = \prod_{i=1}^k |A^i|,$$
and the formula above follows immediately from the multinomial theorem.
\end{proof}

We are now equipped to define the support function on the velocity fan. We first define it on metric $n$-bracketings, and then precompose with $\Gamma$, which is an isomorphism by Proposition \ref{mainvelocitythm}.

\begin{definition}
\label{supportfunction}
Define $h^{\met}_\cT:
\cK^{\met}(\cT)\rightarrow \mathbb{R}$ as
\begin{align}
\label{supfunction}
h^{\met}_\cT \bigl( \ell_{\sB}\bigr) \coloneqq 
\sum_{\substack{1 \leq k \leq n \\ A \in \sB^k}} \ell_{\sB}(A) | U(A)| ,
\end{align}
and define $h_\cT:
\cF(\cT)\rightarrow \mathbb{R}$ as $h_\cT = h^{\met}_\cT \circ \Gamma^{-1}.$
\null\hfill$\triangle$
\end{definition}

\begin{remark}
Prop \ref{explicit-support} allows us to write the definition of the support function by an explicit formula without referring to collision units.
However, the framework of collision units will be very useful for establishing strong convexity in the following subsection.
\null\hfill$\triangle$
\end{remark}

\begin{figure}[H]
\centering
\def\svgwidth{0.7\textwidth}
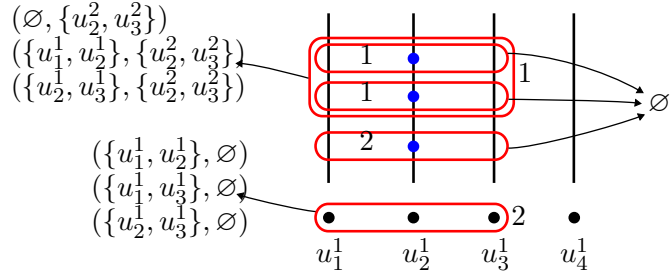
\caption{Here we depict a metric 2-bracketing and indicate the containment of collision units in individual brackets.
The value of $h_{\cT}$ is therefore equal to $1 \times 3 + 2 \times 3 = 9$.}
\end{figure}

\begin{lemma}
For a concentrated $n$-associahedron $\cK (\cT)$, the function $h_\cT$ is piecewise-liner on the velocity fan $\cF(\cT)$.
\end{lemma}

\begin{proof}
By Proposition \ref{mainvelocitystatement},  it suffices to verify that $h^{\met}_\cT$ is piecewise-linear on $\cK^{\met}(\cT)$.  It is clear that the function $h^{\met}_\cT$ agrees on the intersection of any two conical sets in $\cK^{\met}(\cT)$, so it remains to verify that $h^{\met}_\cT$ is linear on each conical set in $\cK^{\met}(\cT)$.
By Definition \ref{def:ops_on_metric_n-bracketings}, the addition of metric $n$-bracketings with compatible underlying bracketings is obtained by unioning bracket sets and adding weights on the same bracket.
The above support function is defined as a sum of values, that are computed for individual $k$-brackets, regardless of the $n$-bracketing to which these $k$-brackets belong. The result follows. 
\end{proof}

Note that being defined individually on every bracket is a sufficient, but perhaps nonnecessary condition for a function to be linear on cones of the velocity fan.

\begin{remark}
Possible modifications of the support function include removing the condition of vertices being neighbors in the Definition \ref{collision_unit} of a collision unit, or assigning a positive integral weight to every individual collision unit and counting it with such a weight.
These modifications would give other integral functions, that are linear on the cones of the velocity fan in the concentrated case, and strictly convex.
The function $h_\cT$ described in this section is the minimal one having values in positive integers.
\null\hfill$\triangle$
\end{remark}

We now prove the following lemma.

\begin{lemma}
The function $h_\cT$, introduced in Definition \ref{supportfunction}, is strongly convex, that is, for ${\mathbf u}$ and ${\mathbf v}$ from different chambers of the velocity fan,
\begin{align}
h_\cT(\mathbf{u} + \mathbf{v})
<
h_\cT(\mathbf{u}) + h_\cT(\mathbf{v})
\end{align}
\end{lemma}

We need the following technical definition.

\begin{definition}
\label{c-coordinate}
Let $\bf{c}$ be a collision unit with $ \{ u^k_a,u^k_b \} \in \mathbf{c}$.
Say that $x^k_i$ is a $\bf{c}$-coordinate if $i \in [a,b)$.
\null\hfill$\triangle$
\end{definition}

Let $\sC \in \mathfrak{X}(\cT)$.
From Lemma \ref{rho=zeta}, we know that $\bf{c}$ is contained in some bracket of $\sC$ as in Definition \ref{unitcontain} whenever $\rho(\sC)^k_i = 1$ for all $\bf{c}$-coordinates $x^k_i$. 
 Since the velocity fan for a concentrated $n$-associahedron $\cK(\cT)$ is a refinement of the braid arrangement by Lemma \ref{coarsening}, it suffices to prove supermodularity of $h_\cT$, which we now recall (see \cite{ardila2020coxeter,Castillo-Liu}).
Let $\rho_S$ and $\rho_T$ be ray generators, that are indicator vectors for some sets $S$ and  $T$ which are subsets in the set of all the coordinates of the fan.  We denote by $\rho_{S \cup T}$ and $\rho_{ S \cap T}$ the indicator vectors of, respectively, union and intersection of these sets and note that $\rho_S + \rho_T = \rho_{S \cup T} + \rho_{S \cap T}$. 
For any rays $\rho_S$ and $\rho_T$, the vectors $\rho_{S \cup T}$ and $\rho_{S \cap T}$ belong to a common chamber of the braid arrangement, hence by Lemma \ref{coarsening}, they belong to a common chamber of the velocity fan. 

\begin{remark}
We emphasize that it is possible to have two collisions $\sC_1$ and $\sC_2$ with some nontrivial brackets $A_1 \in \sC_1^k$ and $A_2 \in \sC_2^k$ such that $A^k_1 \cap A^k_2 \neq \emptyset$, and $\rho_S = \rho(\sC_1)$, $\rho_T = \rho(\sC_2)$ with $S\cap T= \emptyset$.  For example, take the 1-associahedron, and $\sC_1$ and $\sC_2$ determined by the brackets $\{u_i,u_{i+1}\}$ and $\{u_{i+1},u_{i+2}\}$, respectively.
\null\hfill$\triangle$
 \end{remark}

In the proof of the Lemma below, we let $\sC(\rho)$ denote the collision corresponding to a ray generator $\rho \in \cF(\cT)$.

\begin{lemma}
For $\rho_S$ and $\rho_T$ ray generators from different chambers of the velocity fan,
\begin{align}
h_\cT(\rho_S) + h_\cT(\rho_T) < h_\cT(\rho_{S \cup T}) + h_\cT(\rho_{S \cap T})
\end{align}
\end{lemma}

\begin{proof}
Both sides count some collision units, once or twice.
We show that every collision unit counted on the left is also counted on the right at least as many times, and there exists a collision unit counted on the right that does not appear on the left.

Let $\bf{c}$ be a collision unit counted twice on the left.
This means that all $\bf{c}$-coordinates are 1 in both $\rho_S$ and $\rho_T$, thus also in $\rho_{S \cup T}$ and $\rho_{S \cap T}$, so $\bf{c}$ is also counted twice on the right.
Now let $\bf{c}$ be a collision unit counted at least once on the left.
This means that either all $\bf{c}$-coordinates are 1 in $\rho_S$, or in $\rho_T$, so also in $\rho_{S \cup T}$ , so $\bf{c}$ is also counted on the right.

We now construct a collision unit counted on the right but not on the left, working by induction. 
Case $n = 0$ is trivial. Assume the lemma for $(n-1)$-bracketings.
If $\pi(\sC(\rho_S))$ and $\pi(\sC(\rho_T))$ are incompatible, then in particular they are collisions in $\cK(\pi(\cT))$ (note that it is not true for general projections of collisions).
Then by induction there exists a collision unit in $\mathbf{c}^{-} \in U(\pi(\cT))$ counted only on the right of the inequality 
\begin{align}
h_{\pi(\cT)}(\pi(\rho_S)) + h_{\pi(\cT)}(\pi(\rho_T)) < h_{\pi(\cT)}(\pi(\rho_{S \cup T})) + h_{\pi(\cT)}(\pi(\rho_{S \cap T}))
\end{align}
We lift it to the desired collision unit in $\mathbf{c} \in U(\cT)$ by appending $\varnothing$.

Now let $\pi(\sC(\rho_S))$ and $\pi(\sC(\rho_T))$ be compatible.
Then there exist two incompatible $n$-brackets $A_S = (A^1_S, \ldots, A^n_S) \in \sC^n(\rho_S)$ and $A_T = (A^1_T, \ldots, A^n_T) \in \sC^n(\rho_T)$.
There are two possibilities: either $A^n_S$ and $A^n_T$ are not nested, or 
$A^n_S \subsetneq A^n_T$, while $\pi(A_T) \subsetneq \pi (A_S)$ (or vice versa).
In the first case, there exists a pair of depth $n$ vertices $u^n_a$ and $u^n_b$ with $u^n_a \in A^n_S$, $u^n_a \notin A^n_T$ and $u^n_b \in A^n_T$, $u^n_b \notin A^n_S$, and the desired collision unit consists of this pair and $\varnothing$ elsewhere.\footnote{Note that this argument is the reason why we needed to allow nonneighboring vertices in collision units with only one nonempty entry.}

Suppose we are in the second case so that $A^n_S \subsetneq A^n_T$, while $\pi(A_T) \subsetneq \pi (A_S)$, and take $\mathbf{c}^-$ some collision unit inside $U(\pi(A_S)) \sqcup U(\pi^2(A_S)) \ldots \sqcup U(\pi^{n-1}(A_S))$ but not in $U(\pi(A_T)) \sqcup U(\pi^2(A_T)) \ldots \sqcup U(\pi^{n-1}(A_T))$. 
Let $(u^n_a,u^n_{a+1})$ be a neighboring pair of points such that both are in $A^n_T$ and at least one of them is not in $A^n_S$.
Extending $\mathbf{c}^-$ by  sufficiently many $\varnothing$ and this pair, we obtain the desired collision unit $\mathbf{c}$.
\end{proof}

\begin{remark}
In the proof above we use the fact that for a concentrated $n$-associahedron, incompatibility of two collisions implies existence of two incompatible brackets.
This statement is false for a general $n$-associahedron.
\null\hfill$\triangle$
\end{remark}

\begin{figure}[H]
\centering
\def\svgwidth{0.75\textwidth}
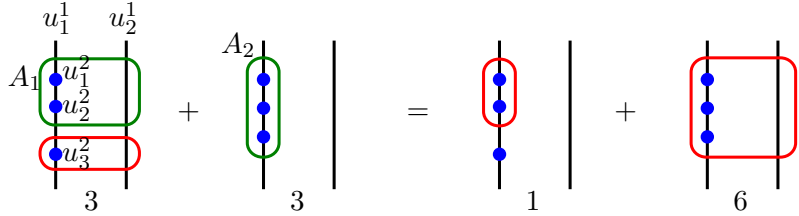
\caption{
This picture illustrates the procedure in the proof above.
We are in the case $A_1^2 \subsetneq A_2^2$, $\pi(A_2) \subsetneq \pi(A_1)$, so $\mathbf c^- = (\{u_1^1,u_2^1)\})$, $\mathbf c = (\{u_1^1,u_2^1\}, \{u_2^2,u_3^2\})$.
}
\end{figure}

We have thus proved the following.

\begin{theorem}
\label{polytopes}
Velocity fans for concentrated $n$-associahedra are projective.
\end{theorem}
 
Given a concentrated rooted plane tree $\cT$, we let $P(\cT)$ denote the polytope determined by the support function $h_{\cT}$.

\begin{remark}
Loday's associahedron \cite{Lodayassociahedron} is a special case of $P(\cT)$, as is
Forcey's multiplihedron \cite{Forcey} (to be precise, a realization of the multiplihedron due to \cite{Ardilla-Doker} is a special case -- Forcey's realization is a projection of this realization).
\null\hfill$\triangle$
\end{remark}

\begin{figure}[H]
\centering
\def\svgwidth{0.75\textwidth}
\includegraphics[width=0.75\textwidth]{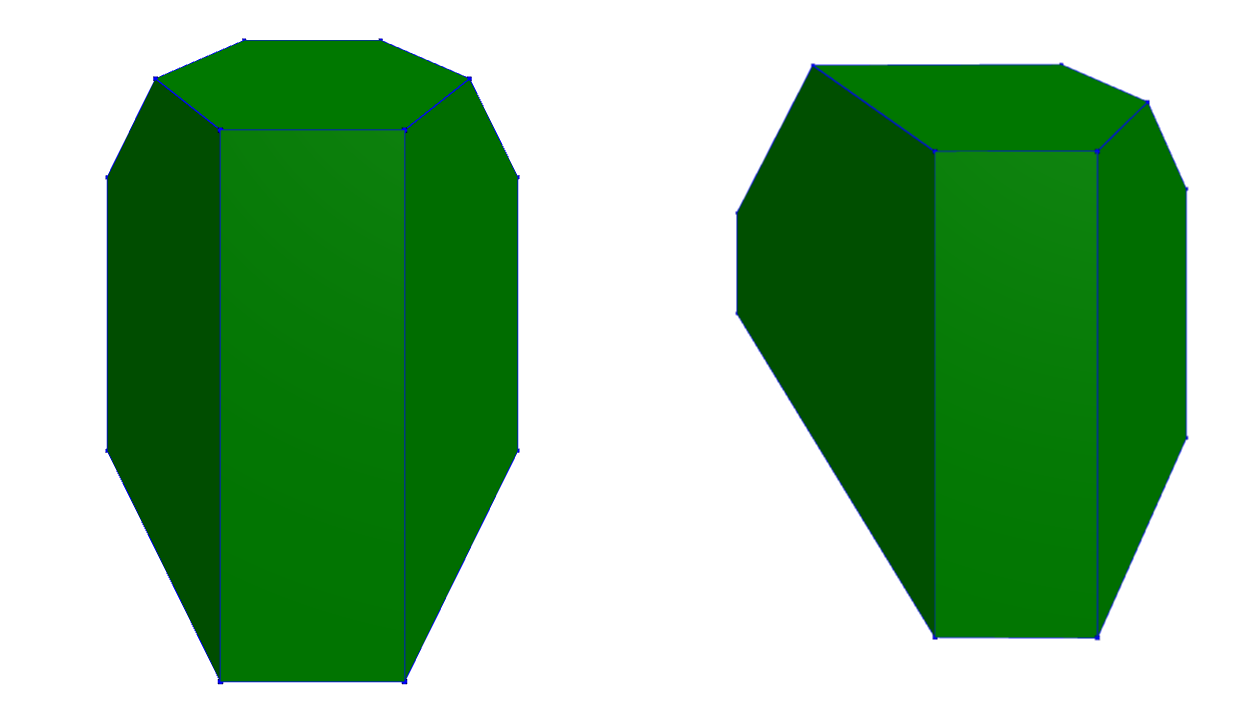}
\caption{
On the left: the polytope $P(3,0,0)$, which is also isomorphic to $P(0,2,0,0)$.  On the right: the polytope $P(2,0,0,0)$.
}
\end{figure}

 We will refine our description of this realization by expressing it as a Minkowski sums of simplices, and by providing vertex coordinates.

\subsection{Minkowski sums of simplices}

\

Any generalized permutahedron $P$ admits a description as a signed Minkowski sum of the standard simplices (see \cite{danilov2000cores,postnikovgp, Ardilla-Doker-Benedetti}):
\begin{align}
P = \sum_{S}\mathbf{y}_S(P) \Delta(S),
\end{align}

where $S$ ranges over all the subsets of the ground set $[n]$,  $\mathbf{y}_S(P) \in \mathbb{R}$, and $\Delta(S)$ denotes the convex hull of the standard basis vectors indicating elements of $S$.  The class of \emph{$Y$-generalized permutahedra}, also known as 
\emph{hypergraph polytopes}, are those generalized permutahedra for which the coefficients in the sum satisfy $\mathbf{y}_S(P) \geq 0$, i.e. they are positive Minkowski sums of simplices. 
 In this subsection, we describe concentrated $n$-associahedra as hypergraph polytopes: for a concentrated tree $\cT$, we associate a simplex to every collision unit and take the sum of all such simplices.

\begin{definition}
\label{simplices}
For a collision unit $\mathbf{c}$, let $A(\mathbf{c})$ be the smallest bracket containing $\mathbf{c}$, and let $\sC(\mathbf{c})$ be the smallest collision containing $A(\mathbf{c})$.
The set $S(\mathbf{c})$ is defined as the support of the ray of $\sC(\mathbf{c})$.
The simplex associated to $\mathbf{c}$ is thus $\Delta(S(\mathbf{c}))$.
\null\hfill$\triangle$
\end{definition}

Note that the smallest collision containing a bracket is well-defined in the concentrated case, but not in general.

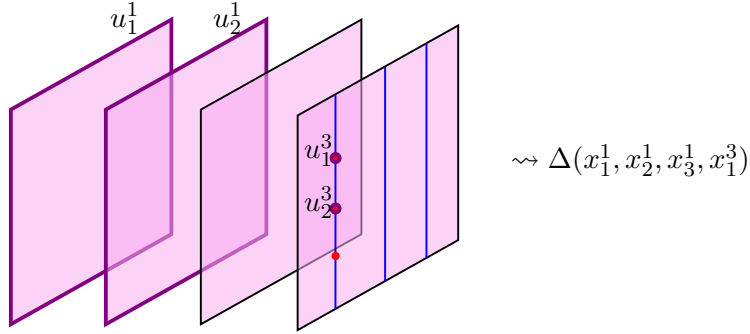
\begin{figure}[ht]
\centering
\def\svgwidth{0.75\textwidth}
%% Creator: Inkscape 1.2 (dc2aeda, 2022-05-15), www.inkscape.org
%% Accompanies image file 'simplex_for_unit.pdf' (pdf, eps, ps)
\begingroup%
  \makeatletter%
  \providecommand\color[2][]{%
    \errmessage{(Inkscape) Color is used for the text in Inkscape, but the package 'color.sty' is not loaded}%
    \renewcommand\color[2][]{}%
  }%
  \providecommand\transparent[1]{%
    \errmessage{(Inkscape) Transparency is used (non-zero) for the text in Inkscape, but the package 'transparent.sty' is not loaded}%
    \renewcommand\transparent[1]{}%
  }%
  \providecommand\rotatebox[2]{#2}%
  \newcommand*\fsize{\dimexpr\f@size pt\relax}%
  \newcommand*\lineheight[1]{\fontsize{\fsize}{#1\fsize}\selectfont}%
  \ifx\svgwidth\undefined%
    \setlength{\unitlength}{311.52477493bp}%
    \ifx\svgscale\undefined%
      \relax%
    \else%
      \setlength{\unitlength}{\unitlength * \real{\svgscale}}%
    \fi%
  \else%
    \setlength{\unitlength}{\svgwidth}%
  \fi%
  \global\let\svgwidth\undefined%
  \global\let\svgscale\undefined%
  \makeatother%
  \begin{picture}(1,0.40051698)%
    \lineheight{1}%
    \setlength\tabcolsep{0pt}%
    \put(0,0){\includegraphics[width=\unitlength,page=1]{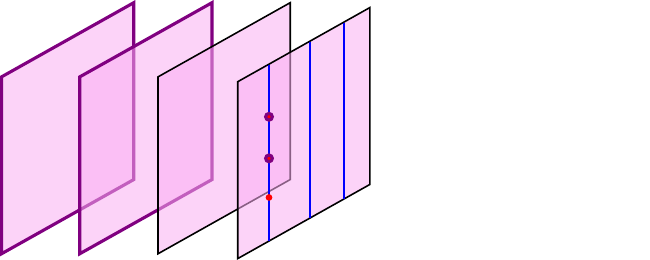}}%
    \put(0.12853595,0.38738083){\makebox(0,0)[lt]{\lineheight{1.25}\smash{\begin{tabular}[t]{l}$u_1^1$\end{tabular}}}}%
    \put(0.25933262,0.38738083){\makebox(0,0)[lt]{\lineheight{1.25}\smash{\begin{tabular}[t]{l}$u_2^1$\end{tabular}}}}%
    \put(0.37476413,0.22003185){\makebox(0,0)[lt]{\lineheight{1.25}\smash{\begin{tabular}[t]{l}$u_1^3$\end{tabular}}}}%
    \put(0.37476413,0.15369673){\makebox(0,0)[lt]{\lineheight{1.25}\smash{\begin{tabular}[t]{l}$u_2^3$\end{tabular}}}}%
    \put(0.63699865,0.20759445){\makebox(0,0)[lt]{\lineheight{1.25}\smash{\begin{tabular}[t]{l}$\rightsquigarrow \Delta(x_1^1,x_2^1,x_3^1,x_1^3)$\end{tabular}}}}%
  \end{picture}%
\endgroup%

\caption{
\label{fig:collision_unit_simplex}
Here we depict a collision unit $\mathbf{c} = (\{u^1_1,u^1_2\}, \varnothing, \{u^3_1, u^3_2\})$ in a concentrated 3-associahedron, and a simplex associated to it by the above construction.
} 
\end{figure}

\begin{theorem}
The realized concentrated $n$-associahedron $P(\cT)$ has $\mathbf{y}_S(P(\cT)) = 1$ when $S = S(\mathbf{c})$ for some collision unit, and 0 otherwise. 
\end{theorem}

\begin{proof}
The polytope $P(\cT)$ is a generalized permutahedron.
The Minkowski sum of simplices as described above is also a generalized permutahedron.
To compare two generalized permutahedra, it suffices to compare their support functions on all the 0-1-vectors.
Any 0-1 vector $\rho$ can be written as a sum of 0-1 vectors $\rho_1+\ldots + \rho_r$, where $\rho_i$ are ray generators in the same chamber of the velocity fan.
One then observes that a simplex $\Delta(S(\mathbf{c}))$ is contained in the support of $\rho$ if and only if the collision unit $\mathbf{c}$ is contained in a bracket of a collision corresponding to one of the $\rho_i$.
\end{proof}

\begin{figure}[ht]
\centering
\def\svgwidth{0.65\textwidth}
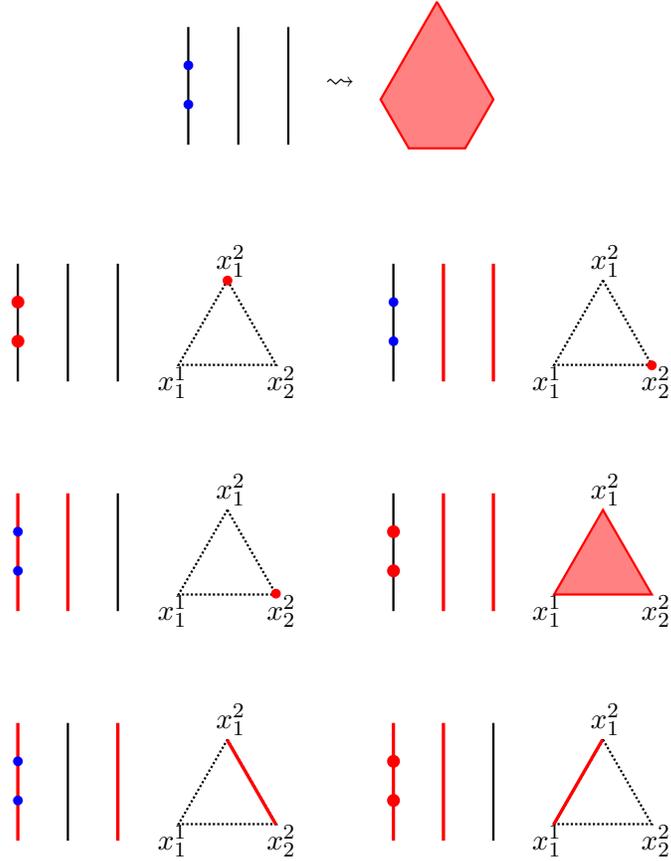
\caption{This is how the pentagon 
$\cK_{2,0,0}$ is obtained as a Minkowski sum of simplices that correspond to its six collision units.
\label{fig:minkowski_decomp}
}
\end{figure}

\begin{proposition}
The realized concentrated $n$-associahedron $P(\cT)$ is a Minkowski-summand of the nestohedron associated to $\cK(\Delta_{\cT})$.
\end{proposition}

\begin{proof}
The simplices in the Minkowski sum realization of a concentrated $n$-associahedron $\cK(\cT)$ are a subset of the simplices in the Minkowski sum realization of the nestohedron whose normal fan is, by Lemma \ref{concentratednnestohedral}, the triangulated velocity fan $\cF(\Delta_{\cT})$: indeed, the latter set has a simplex for every collision, while the first set has simplices only for collisions of the form $\sC = \sC(\mathbf{c})$.
\end{proof}

\begin{remark}
One can associate a simplex to {\em every} bracket just as in Definition \ref{simplices}, and sum these simplices.
By doing so, one would obtain another realization of $\cK(\cT)$ with the same fan, but with a nonminimal support function.
\null\hfill$\triangle$
\end{remark}

\subsection{Recovering the constrainahedra}

\

We now describe the precise relationship between families of concentrated $n$-associahedra, and constrainahedra.
Constrainahedra were envisioned by the second author, informally described in \cite{Tierney} and formally defined in \cite{bottman2022constrainahedra}.
When realized in terms of ``flappy trees'', the constrainahedra will be a key ingredient in equipping the Fukaya category with a monoidal $A_\infty$-structure, in the context of SYZ mirror symmetry.
(See \cite[\S5.4]{abouzaid_bottman}.)
In \cite{chapoton2022shuffles} a normally equivalent presentation was given as shuffle products of associahedra, that is, as Minkowski sums of Cartesian products of associahedra with zonotopes associated to complete bipartite graphs, thus giving explicit Minkowski coefficients.\footnote{In the recent work of Black-L\"utjeharms-Sanyal  \cite{black2024linear}, a rather different realization of constrainahedra is provided as pivot rule polytopes of products of simplices.}  For the purposes of this subsection, we take this Minkowski sum description as the starting point for defining constrainahedron.

\begin{definition}
\label{simplices_constr}
 Fix a sequence of nonnegative integers ${\bf t} = (t_1, \ldots, t_n)$. We define a polytope $C({\bf t})$ in the space with coordinates $x^k_i$ with $1\leq k \leq n$  and $1 \leq i \leq t_k$.
The polytope is defined as having Minkowski coordinates $\mathbf{y}_S(C({\bf t})) = 1$ when $S = \{(k,i)\ldots, (k,j) \}$ for some $ 1 \leq k \leq n$ and $1 \leq i < j \leq t_k$ (these correspond to associahedral summands) or when $S = \{(k,i),(k',j) \}$ for some $1 \leq k \neq k' \leq n$, $1 \leq i \leq t_k$, $1 \leq j \leq t_{k'}$ (these correspond to zonotopal summands), and $0$ otherwise.  We define a \emph{constrainahedron of profile ${\bf t}$} to be any polytope which is normally equivalent to $C({\bf t})$.
\null\hfill$\triangle$
\end{definition}

Combinatorially, faces of constrainahedra correspond to iterated collisions in a grid of orthogonal hyperplanes, as explained in \cite{bottman2022constrainahedra}.
Similarly to the velocity fan, a facet normal indicates which hyperplanes collided.

The value of the support function on such a ray is equal to the number of \emph{constrainahedral collision units} participating in this collision, where the definition of a collision unit is repeated verbatim with vertices of $\cT$ replaced by labels of hyperplanes.

\begin{definition}
Fix a sequence of nonnegative integers ${\bf t} = (t_1, \ldots, t_n)$.
A tree $\cT$ (and the corresponding $n$-associahedron) is said to have profile ${\bf t}$ if it has $t_k+1$ vertices of depth $k$. 
\null\hfill$\triangle$
\end{definition}

In the proof of the next theorem, we utilize the criterion for when the normal fan of a generalized permutahedron $P$ coarsens the normal fan of a generalized permutahedron $Q$.
The criterion combines the interpretation of chambers in the normal fan as domains of linearity for the support function, and the fact that face normals of generalized permutahedra are 0-1 vectors.

\begin{proposition}
\label{coarsening_criterion}
Let $P$ and $Q$ be generalized permutahedra with support functions $f_P$ and $f_Q$, respectively.  The normal fan of $P$ coarsens the normal fan of $Q$, when for any pair of 0-1-vectors $\mathbf{v}$ and $\mathbf{w}$ the strict inequality $f_P(\mathbf{v})+f_P(\mathbf{w}) < f_P(\mathbf{v} \lor \mathbf{w}) + f_P(\mathbf{v} \land \mathbf{w})$ implies the strict inequality $f_Q(\mathbf{v})+f_Q(\mathbf{w}) < f_Q(\mathbf{v} \lor \mathbf{w}) + f_Q(\mathbf{v} \land \mathbf{w})$.
\end{proposition}

The following theorem demonstrates how we can recover constrainahedra in our setting.  The theorem is perhaps surprising as there are rays in constrainahedra which are not present in any of the concentrated $n$-associahedra of the same profile.

\begin{theorem}
Fix a tuple ${\bf t} = (t_1,\ldots, t_n)$ of nonnegative integers.  The polytope 
$$ C'({\bf t}) =\sum_\cT P(\cT),$$
 where $\cT$ ranges over all the concentrated trees of profile ${\bf t}$, is a constrainahedron of profile ${\bf t}$. 
\end{theorem}

\begin{proof}
We prove that $C'({\bf t})$ is normally equivalent to $C({\bf t})$. The simplices appearing in $C'({\bf t})$ are described by \ref{simplices}.
The simplices appearing in $C({\bf t})$, as described in \ref{simplices_constr}, are a subset of those; they can be realized as $\Delta(S(\mathbf{c}))$ for appropriately chosen $\cT$ and $\mathbf{c}$.
It suffices to show that the extra simplices added to obtain $C'(\mathbf{t})$ do not change the normal equivalence class, that is, that their normal fans are coarsenings of the normal fan of $C({\bf t})$.
In the notation of Proposition \ref{coarsening_criterion}, take $Q= C({\bf t})$, and take $P=\Delta(S(\mathbf{c}))$ one of those extra simplices, corresponding to a collision unit $\mathbf{c}$ as in \ref{simplices}. 
For a 0-1 vector $\mathbf{v}$, the value $f_P(\mathbf{v})$ is 1 if $S$ is included in the support set of $\mathbf{v}$, and 0 otherwise.
The strict inequality $f_P(\mathbf{v})+f_P(\mathbf{w}) < f_P(\mathbf{v} \lor \mathbf{w}) + f_P(\mathbf{v} \land \mathbf{w})$ means that $S$ is in the support of $\mathbf{v} \lor \mathbf{w}$, but not $\mathbf{v}$ or $\mathbf{w}$ independently.
Then the inequality $f_Q(\mathbf{v})+f_Q(\mathbf{w}) < f_Q(\mathbf{v} \lor \mathbf{w}) + f_Q(\mathbf{v} \land \mathbf{w})$ holds, because $\mathbf{c}$ viewed as a constrainahedral collision unit contributes only to the right-hand side, namely to the summand $f_Q(\mathbf{v} \lor \mathbf{w})$.  
\end{proof}

\begin{example}
In the setting of the previous proof, let ${\bf c}$ be the collision unit $(\{u^1_1,u^1_2\}, \varnothing, \{u^3_1, u^3_2\})$ from the $3$-associahedron from Figure \ref{fig:collision_unit_simplex} so that  $P = \Delta (x^1_1,x^1_2,x^1_3,x^3_1)$. The coordinates for the whole fan are $x^1_1$, $x^1_2$, $x^1_3$, $x^2_1$, $x^2_2$, $x^3_1$, $x^3_2$. The inequality $f_P(\mathbf{v})+f_P(\mathbf{w}) < f_P(\mathbf{v} \lor \mathbf{w}) + f_P(\mathbf{v} \land \mathbf{w})$ is strict, for example, for $\mathbf{v} = (1,1,1,1,0,0,0)$ and $\mathbf{w} = (1,0,0,0,0,1,0)$, as $S({\bf c})$ is only included in the support of $\mathbf{v} \lor \mathbf{w}$. Then for $Q = C(3,2,2)$ the inequality $f_Q(\mathbf{v})+f_Q(\mathbf{w}) < f_Q(\mathbf{v} \lor \mathbf{w}) + f_Q(\mathbf{v} \land \mathbf{w})$ is strict (namely it is $10+3 < 18+1$), and if we read $u^k_i$ as labels of hyperplanes for $C(3,2,2)$, then the same ${\bf c}$ contributes only to $f_Q(\mathbf{v} \lor \mathbf{w})$.
\null\hfill$\triangle$
\end{example}

The above construction indicates that concentrated $n$-associahedra can be obtained from Cartesian products of usual associahedra by Minkowski-adding certain {\em hypergraph polytopes} \cite{Benedetti-Bergeron-Machacek}.
Indeed, the simplices not from associahedra --- those of the form $\Delta(S({\bf c}))$ for ${\bf c}$ having more than one nonempty entry --- correspond to some hyperedges. 
Chapoton and Pilaud in \cite{chapoton2022shuffles} introduced an operation of taking \emph{shuffle products} of generalized permutahedra, which consists of taking Cartesian product and then adding a graphical zonotope associated to the complete bipartite graph.
Shuffle products of associahedra are constrainahedra.
The results of this section indicate that the operation of shuffling could be generalized to a certain operation of {\em hypershuffling}, to be described and studied elsewhere.

\subsection{Vertex presentation}
\label{vertices}

\

In this section we discuss explicit formulas for vertex coordinates, generalizing Loday's vertex description for associahedra \cite{Lodayassociahedron} and, a bit less directly, Forcey's vertex description for multiplihedra \cite{Forcey}. 

Consider a maximal $n$-bracketing $\sB$.
We now define what it means for a collision unit to {\em contribute} to a certain coordinate for the corresponding vertex.

\begin{definition}
\label{add}
Let $x^k_i$ be a coordinate in the velocity fan. Denote by $\sB_{x^k_i}$ the collection of $l$-brackets for all $l \geq k$ that contain $u^k_{i}$ and $u^k_{i+1}$ and do not contain any other bracket containing $u^k_{i}$ and $u^k_{i+1}$ (i.e.\ are containment-minimal among such).
A collision unit $\bf{c}$ \emph{contributes} to $x^k_i$ in $\sB$ if it is contained in some $l$-bracket $A \in \sB_{x^k_i}$ in the sense of \ref{unitcontain} and is not contained in  any other $l$-bracket $A' \subsetneq A$.
\null\hfill$\triangle$
\end{definition}

 We now provide an informal rephrasing of the definition above for the reader's convenience. Suppose that the consecutive affine spaces in our concentrated tree arrangement are identified one by one until we reach a single complete flag.  Then we can associate a maximal $n$-bracketing $\sB$ which encodes this process of identifying affine spaces, up to order of disjoint identifications.  A collision unit ${
 \bf c}$ contributes to $x^k_i$ in $\sB$ if the moment when the corresponding consecutive pair of spaces $L_i^k$ and $L_{i+1}^k$ are identified is the moment when all of the affine spaces present in ${
 \bf c}$ are contracted to a single flag of affine spaces.

 Denote by $X^k_i$ the set of collision units contributing to $x^k_i$.
 
\begin{definition}
The point ${\bf v}(\sB) \in \mathbb{R}^m$  corresponding to the maximal bracketing $\sB$ is defined as having the coordinates 
\begin{align}
\label{vertex_coord}
{\bf v}(\sB)^k_i = |X^k_i|
\end{align}
\null\hfill$\triangle$
\end{definition}

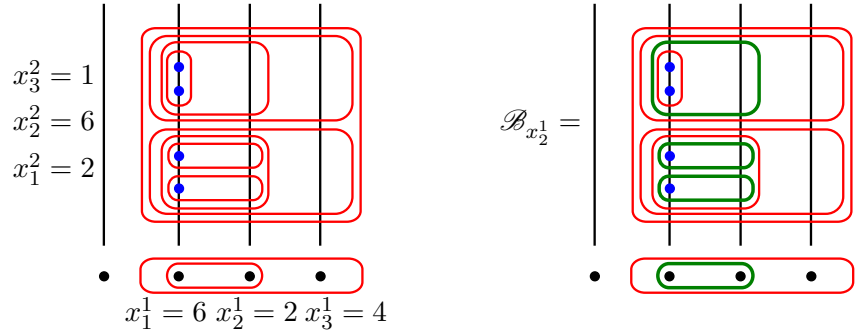
\begin{figure}[H]
\centering
\def\svgwidth{0.8\textwidth}
%% Creator: Inkscape 1.2 (dc2aeda, 2022-05-15), www.inkscape.org
%% Accompanies image file 'vertex.pdf' (pdf, eps, ps)
\begingroup%
  \makeatletter%
  \providecommand\color[2][]{%
    \errmessage{(Inkscape) Color is used for the text in Inkscape, but the package 'color.sty' is not loaded}%
    \renewcommand\color[2][]{}%
  }%
  \providecommand\transparent[1]{%
    \errmessage{(Inkscape) Transparency is used (non-zero) for the text in Inkscape, but the package 'transparent.sty' is not loaded}%
    \renewcommand\transparent[1]{}%
  }%
  \providecommand\rotatebox[2]{#2}%
  \newcommand*\fsize{\dimexpr\f@size pt\relax}%
  \newcommand*\lineheight[1]{\fontsize{\fsize}{#1\fsize}\selectfont}%
  \ifx\svgwidth\undefined%
    \setlength{\unitlength}{272.17327517bp}%
    \ifx\svgscale\undefined%
      \relax%
    \else%
      \setlength{\unitlength}{\unitlength * \real{\svgscale}}%
    \fi%
  \else%
    \setlength{\unitlength}{\svgwidth}%
  \fi%
  \global\let\svgwidth\undefined%
  \global\let\svgscale\undefined%
  \makeatother%
  \begin{picture}(1,0.38045829)%
    \lineheight{1}%
    \setlength\tabcolsep{0pt}%
    \put(0,0){\includegraphics[width=\unitlength,page=1]{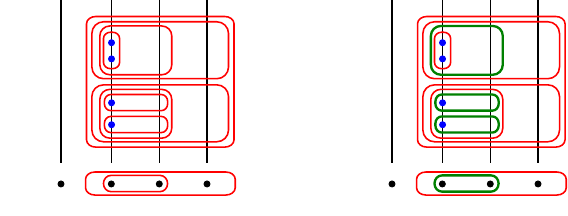}}%
    \put(0.57813983,0.22966858){\makebox(0,0)[lt]{\lineheight{1.25}\smash{\begin{tabular}[t]{l}$\sB_{x_2^1}=$\end{tabular}}}}%
    \put(0.13213479,0.00259055){\makebox(0,0)[lt]{\lineheight{1.25}\smash{\begin{tabular}[t]{l}$x_1^1=6$\end{tabular}}}}%
    \put(0.24048179,0.00259055){\makebox(0,0)[lt]{\lineheight{1.25}\smash{\begin{tabular}[t]{l}$x_2^1=2$\end{tabular}}}}%
    \put(0.34651028,0.00259055){\makebox(0,0)[lt]{\lineheight{1.25}\smash{\begin{tabular}[t]{l}$x_3^1=4$\end{tabular}}}}%
    \put(-0.00111228,0.17382207){\makebox(0,0)[lt]{\lineheight{1.25}\smash{\begin{tabular}[t]{l}$x_1^2=2$\end{tabular}}}}%
    \put(-0.00111228,0.23066426){\makebox(0,0)[lt]{\lineheight{1.25}\smash{\begin{tabular}[t]{l}$x_2^2=6$\end{tabular}}}}%
    \put(-0.00111228,0.28479012){\makebox(0,0)[lt]{\lineheight{1.25}\smash{\begin{tabular}[t]{l}$x_3^2=1$\end{tabular}}}}%
  \end{picture}%
\endgroup%

\caption{
Here is an illustration of the above formula for a particular maximal 2-bracketing in $\cK_{0,4,0,0}$.}
\end{figure}

We now prove the following theorem.

\begin{theorem}
The convex hull of the points ${\bf v}(\sB)$ for all maximal $n$-bracketings $\sB$ is the polytope $P(\cT)$.
\end{theorem}

\begin{proof}
The proof is provided by Lemmas \ref{v-brack}-\ref{faceineq}.
\end{proof}

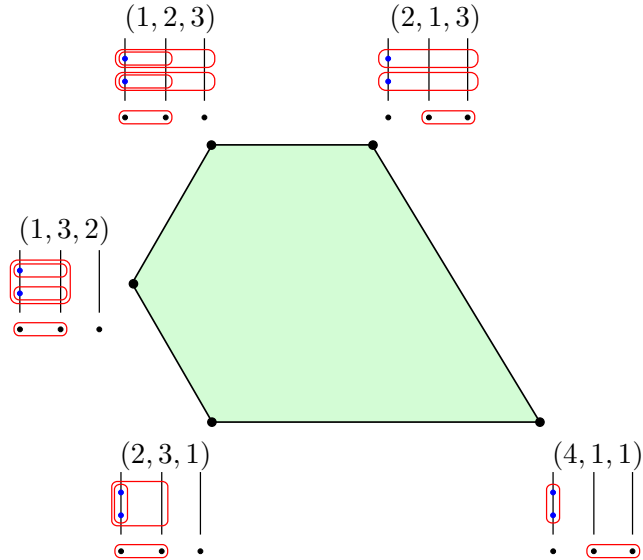
\begin{figure}[H]
\centering
\def\svgwidth{0.6\textwidth}
%% Creator: Inkscape 1.2 (dc2aeda, 2022-05-15), www.inkscape.org
%% Accompanies image file 'vertices.pdf' (pdf, eps, ps)
\begingroup%
  \makeatletter%
  \providecommand\color[2][]{%
    \errmessage{(Inkscape) Color is used for the text in Inkscape, but the package 'color.sty' is not loaded}%
    \renewcommand\color[2][]{}%
  }%
  \providecommand\transparent[1]{%
    \errmessage{(Inkscape) Transparency is used (non-zero) for the text in Inkscape, but the package 'transparent.sty' is not loaded}%
    \renewcommand\transparent[1]{}%
  }%
  \providecommand\rotatebox[2]{#2}%
  \newcommand*\fsize{\dimexpr\f@size pt\relax}%
  \newcommand*\lineheight[1]{\fontsize{\fsize}{#1\fsize}\selectfont}%
  \ifx\svgwidth\undefined%
    \setlength{\unitlength}{281.69950374bp}%
    \ifx\svgscale\undefined%
      \relax%
    \else%
      \setlength{\unitlength}{\unitlength * \real{\svgscale}}%
    \fi%
  \else%
    \setlength{\unitlength}{\svgwidth}%
  \fi%
  \global\let\svgwidth\undefined%
  \global\let\svgscale\undefined%
  \makeatother%
  \begin{picture}(1,0.85517638)%
    \lineheight{1}%
    \setlength\tabcolsep{0pt}%
    \put(0,0){\includegraphics[width=\unitlength,page=1]{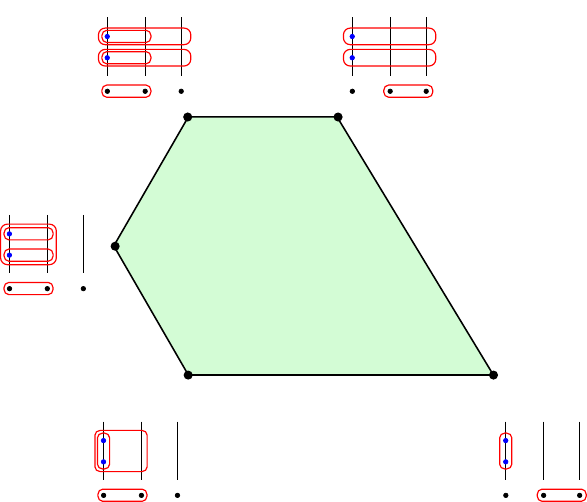}}%
    \put(0.18215878,0.84424941){\makebox(0,0)[lt]{\lineheight{1.25}\smash{\begin{tabular}[t]{l}$(1,2,3)$\end{tabular}}}}%
    \put(0.6024155,0.84424941){\makebox(0,0)[lt]{\lineheight{1.25}\smash{\begin{tabular}[t]{l}$(2,1,3)$\end{tabular}}}}%
    \put(0.01405603,0.50804399){\makebox(0,0)[lt]{\lineheight{1.25}\smash{\begin{tabular}[t]{l}$(1,3,2)$\end{tabular}}}}%
    \put(0.17441717,0.1519319){\makebox(0,0)[lt]{\lineheight{1.25}\smash{\begin{tabular}[t]{l}$(2,3,1)$\end{tabular}}}}%
    \put(0.86009912,0.1519319){\makebox(0,0)[lt]{\lineheight{1.25}\smash{\begin{tabular}[t]{l}$(4,1,1)$\end{tabular}}}}%
  \end{picture}%
\endgroup%

\caption{
The vertex coordinates for the pentagon $\cK_{2,0,0}$.
}
\end{figure}

We begin with a general lemma on the structure of maximal $n$-bracketings in a concentrated $n$-associahedra.
See also Remark \ref{rmk:max_2-bracketings_dimension_formula} for a connection with the second author's dimension formula for 2-bracketings.
 
\begin{definition}
Let $\sB$ be a maximal $n$-bracketing. For $A \in \sB^k$, we say that $A' \in \sB^k$ is its {\em child} if $A' \subsetneq A$ and there does not exist $A'' \in \sB^k$ with $A' \subsetneq A'' \subsetneq A$.
\null\hfill$\triangle$
\end{definition}

\begin{lemma}
\label{v-brack}
Let $\sB$ be a maximal $n$-bracketing in a concentrated $n$-associahedron.
Consider a nonsingleton bracket $A = (A^1, \ldots, A^k) \in \sB^k$ for some $k \leq n$.
Then one of the following holds:
\begin{enumerate}
\item $A$ has one child subbracket  $A_{\text{thin}} = (A^1_{\text{thin}}, \ldots , A^k_{\text{thin}})$, and $A^m_{\text{thin}} \neq A^m$ for exactly one $m<k$.
\item $A$ has two children subbrackets $A_1 = (A^1_1, \ldots, A^k_1)$ and $A_2 = (A^1_2, \ldots, A^k_2)$, and $A^i = A^i_1 = A^i_2 $ for all $i<k$ and $A^k = A^k_1 \bigsqcup A^k_2$.
\end{enumerate}

\end{lemma}

\begin{proof}
 For a $k$-bracket $A$, let us consider its children and rule out all the possibilities except for the ones listed above.

\begin{itemize}
\item If $A$ has three or more children, then there exists $\sB' > \sB$ obtained by adding a bracket that joins two or more of these children.
\item Suppose $A$ has two children, $A_1 = (A^1_1, \ldots, A^k_1)$ and $A_2 = (A^1_2, \ldots, A^k_2)$.
Then $A^k = A^k_1 \bigsqcup A^k_2$ by {\sc (partition)} in \ref{partition}.
If $A^i \neq A^i_1$ for some $i<k$, then there exists $\sB' > \sB$, obtained by adding a bracket that joins $A_1$ and $A_2$ but is still smaller than $A$ on depth $i$.
\item Suppose $A$ has one child, $A_{\text{thin}} = (A^1_{\text{thin}}, \ldots , A^k_{\text{thin}})$.
Then $A^k_{\text{thin}} = A^k$ by {\sc (partition)} in \ref{partition}.
Suppose there are two indices $i_1,i_2<k$ such that $A^{i_j}_{\text{thin}} \neq A^{i_j}$.
Then there exists $\sB' > \sB$ obtained by adding a bracket $A_{\text{thick}}$, where $A^i_{\text{thick}} = A^i_{\text{thin}}$ for all $i$ except for $i_1$, and $A^{i_1}_{\text{thick}} = A^{i_1}$.
\end{itemize}
\end{proof}

\begin{remark}
\label{rmk:max_2-bracketings_dimension_formula}
In \cite[Lemma 3.5(a)]{b:2-associahedra}, the second author proved a dimension formula for tree-pairs, which are equivalent data to 2-bracketings.
In the notation of that paper, this formula reads:
\begin{align}
\label{eq:tree-pair_dimension_formula}
d(2T)
=&
\sum_{\substack{\alpha \in V^1_\comp(T_b)
\incom(\alpha) = (\beta)}} \bigl(\#\incom(\beta) - 2\bigr)
+ \sum_{\alpha \in V^{\geq 2}_\comp(T_b)} \left(\Bigl(\sum_{\beta \in \incom(\alpha)} \#\incom(\beta)\Bigr)-1\right)
\\
&+ \sum_{\rho \in (T_s)_\inte} \bigl(\#\incom(\rho) - 2\bigr).
\nonumber
\end{align}
A tree-pair corresponds to a maximum 2-bracketing exactly when $d(2T) = 0$.
Each of the summands in \eqref{eq:tree-pair_dimension_formula} is nonnegative, so if $d(2T) = 0$, each of these summands must be zero.
The first two sums cannot both be empty, hence there must be at least one zero summand in either the first or second sum.
This zero summand corresponds to either case (1) or (2) in Lemma \ref{v-brack}, depending whether it appears in the second or first sum, respectively.
\null\hfill$\triangle$
\end{remark}

Basing on the cases and the notation of Lemma \ref{v-brack}, we define centers of brackets.

\begin{definition}
\label{centers}
Let $A$ be a $k$-bracket as in case (1) in \ref{v-brack}.
Then its {\em center} is the coordinate $x^r_i$ such that, of the pair $u^r_i$ and $u^r_{i+1}$, both are in $A^r$, but only one is in $A^{r}_{\text{thin}}$.  Let $A$ be a $k$-bracket as in case (2) in \ref{v-brack}.
Then its {\em center} is the coordinate $x^k_i$, where the pair $u^k_i$ and $u^k_{i+1}$ has $u^k_i \in A^k_1$ and $u^k_{i+1} \in A^k_2$.
\null\hfill$\triangle$
\end{definition}

\begin{lemma}
The point ${\bf v}(\sB)$ for any maximal $n$-bracketing $\sB \in \cK(\cT)$ belongs to the affine hyperplane, where the sum of all coordinates equals $ |U(\cT)|$.
\end{lemma}

\begin{proof}
We show that the sum of all coordinates counts every collision unit once.
If a collision unit $\mathbf{c}$ is contained in two different $k$-brackets of the $n$-bracketing $\sB$, then these brackets intersect, and by the {\sc (nested)} condition, one of these brackets contains another.
Therefore, for every $\mathbf{c}$ there exists a unique containment-minimal bracket of $\sB$ in which $\mathbf{c}$ is contained.
The collision unit $\mathbf{c}$ then contributes to the center of this bracket.
\end{proof}

Now let $\sC \in \mathfrak{X}(\cT)$.
Denote by $U(\sC)$ the set of collision units contained in all brackets of $\sC$, and denote by $S$ the support of the 0-1 vector $\rho(\sC)$.

\begin{lemma}
\label{faceineq}
For all maximal $n$-bracketings $\sB$ it holds that
\begin{align}
\sum_{x^k_i \in S} {\bf v}(\sB)^k_i \geq |U(\sC)| 
\end{align}
with equality if and only if $\sC\leq \sB$.
\end{lemma}

\begin{proof}
We show that the collision units of $U(\sC)$ always contributes to one of the coordinates in the left-hand side of above expression, and in case when $\sB$ refines $\sC$, there are no other collision units contributing to these coordinates.

A collision unit $\mathbf{c}$ can only contribute to one of $\mathbf{c}$-coordinates in the sense of \ref{c-coordinate}.
The left-hand side of the above expression lists all of $\mathbf{c}$-coordinates for collision units $\mathbf{c} \in U(\sC)$.
Summing the values of the corresponding coordinates counts all the collision units contributing to these coordinates, and we have shown that in the very least these are all the collision units of $U(\sC)$, thus proving the inequality.

We now prove the equation. Assume first $\sC \leq \sB$. We need to see, that collision units from $U(\cT) \backslash U(\sC)$ do not contribute to the coordinates listed on the left-hand side.
Assume, to the contrary, there exists such a collision unit $\mathbf{c} \notin U(\sC)$.  Any bracket to which it may contribute, as in Definition \ref{add}, intersects nontrivially with a bracket of $\sC$, which cannot happen in an $n$-bracketing.

Assume now that $\sC$ is not compatible with $\sB$; we construct a collision unit $\mathbf{c} \notin U(\sC)$ contributing to one of the coordinates in $S$ for $\mathbf{v}(\sB)$. Let $A \in \sB^k$ be a $k$-bracket incompatible with some $k$-bracket in $\sC^k$; such an $A$ necessarily exists because $\cT$ is concentrated. Let $A' \in \sB^k$ be the smallest bracket having center $x^r_i \in S$ such that $A \subseteq A'$ ($A'$ may or may not coincide with $A$). There are two cases to consider. 
First assume that there exists $\mathbf{c'} \in U(A) \sqcup U(\pi(A)) \sqcup U(\pi^2(A)) \ldots \sqcup U(\pi^{k-1}(A))$ which is not in $U(\sC)$ and has $\mathbf{c}'_{r} = \varnothing$. Then the collision unit $\mathbf{c}$, which is obtained by adding $\{u^r_i,u^r_{i+1}\}$ as the $r$th entry to $\mathbf{c}'$, is not in $U(\sC)$ and is contained in $A'$, thus contributing to $x^r_i$ on the right-hand side. 
Now assume that all collision units which are in $U(A) \sqcup U(\pi(A)) \sqcup U(\pi^2(A)) \ldots \sqcup U(\pi^{k-1}(A))$, but not in $U(\sC)$, have nonempty $r$th entry. This implies that there exists a collision unit $\mathbf{c'} \in U(A) \sqcup U(\pi(A)) \sqcup U(\pi^2(A)) \ldots \sqcup U(\pi^{k-1}(A))$ with $\mathbf{c'} \notin U(\sC)$, which has only $r$th entry nonempty so that $\mathbf{c}'_r = \{u^r_j,u^r_l\}$. We either have $j<l<i$ or $i+1 < j <l$, and we assume the former order (the argument for the latter order is the similar).  The desired collision unit $\mathbf{c}$ is obtained from $\mathbf{c'}$ by replacing $\{ u^r_j,u^r_l \}$ with $ \{ u^r_j,u^r_{i+1} \}$.
\end{proof}

\bibliographystyle{alpha}
\small
\bibliography{biblio}

\end{document}